\documentclass[11pt,reqno]{amsart}
\pdfoutput=1

\usepackage{amsmath}
\usepackage{amsfonts}
\usepackage{amssymb}
\usepackage{enumerate}
\usepackage{amstext}
\usepackage{mathtools}
\usepackage{amsbsy}
\usepackage{amsopn}
\usepackage{bbm,amsthm}
\usepackage{amscd}
\usepackage[pdftex]{color}
\usepackage{amsxtra}
\usepackage{upref}
\usepackage{epstopdf}
\usepackage{graphicx,color}
\usepackage{hyperref}
\usepackage{longtable}
\usepackage{graphicx}
\usepackage[francais, english]{babel}
\usepackage{soul}
\usepackage{array}
\usepackage{todonotes}
\usepackage{comment}
\usepackage{tikz-cd}

\DeclareFontFamily{OML}{rsfs}{\skewchar\font'177}
\DeclareFontShape{OML}{rsfs}{m}{n}{ <5> <6> rsfs5 <7> <8> <9> rsfs7
  <10> <10.95> <12> <14.4> <17.28> <20.74> <24.88> rsfs10 }{}
\DeclareMathAlphabet{\mathfs}{OML}{rsfs}{m}{n}

\newtheorem{theorem}{Theorem}
\newtheorem*{maintheorem}{Main Theorem}

\newtheorem{lemma}[theorem]{Lemma}
\newtheorem{proposition}[theorem]{Proposition}
\newtheorem{corollary}[theorem]{Corollary}

\theoremstyle{definition}

\theoremstyle{remark}
\newtheorem{remark}[theorem]{\bf Remark}

\numberwithin{equation}{section}
\numberwithin{theorem}{section}

\newcommand{\intav}[1]{\mathchoice {\mathop{\vrule width 6pt height 3 pt depth  -2.5pt
\kern -8pt \intop}\nolimits_{\kern -6pt#1}} {\mathop{\vrule width
5pt height 3  pt depth -2.6pt \kern -6pt \intop}\nolimits_{#1}}
{\mathop{\vrule width 5pt height 3 pt depth -2.6pt \kern -6pt
\intop}\nolimits_{#1}} {\mathop{\vrule width 5pt height 3 pt depth
-2.6pt \kern -6pt \intop}\nolimits_{#1}}}

\newcommand{\intavl}[1]{\mathchoice {\mathop{\vrule width 6pt height 3 pt depth  -2.5pt
\kern -8pt \intop}\limits_{\kern -6pt#1}} {\mathop{\vrule width 5pt
height 3  pt depth -2.6pt \kern -6pt \intop}\nolimits_{#1}}
{\mathop{\vrule width 5pt height 3 pt depth -2.6pt \kern -6pt
\intop}\nolimits_{#1}} {\mathop{\vrule width 5pt height 3 pt depth
-2.6pt \kern -6pt \intop}\nolimits_{#1}}}

\newcommand{\un}{\underline}

\newcommand{\ve}{\varepsilon}
\newcommand{\wt}{\widetilde}
\newcommand{\wh}{\widehat}
\newcommand{\vf}{\varphi}

\newcommand{\R}{\mathbb{R}}
\newcommand{\N}{\mathbb{N}}

\newcommand{\Z}{\mathbb{Z}}

\newcommand{\Hol}[1]{{\textrm{H\"ol}}_{#1}}

\newcommand{\Tau}{\mathcal{T}} 
\newcommand{\kH}{\mathfrak{H}}
\newcommand{\kr}{\mathfrak{r}}

\renewcommand{\exp}[1]{{\rm exp}_{#1}}

\newcommand{\Lip}{{\rm Lip}}
\newcommand{\Sas}{d_{\rm Sas}}
\newcommand{\inj}{{\rm inj}}
\newcommand{\nuh}{{\rm NUH}}
\newcommand{\simN}{\stackrel{N}\sim}
\newcommand{\Sh}{\wh\Lambda}

\newcommand{\vertiii}[1]{{\left\vert\kern-0.2ex\left\vert\kern-0.2ex\left\vert #1 
    \right\vert\kern-0.2ex\right\vert\kern-0.2ex\right\vert}}

\makeatletter
\newsavebox{\@brx}
\newcommand{\llangle}[1][]{\savebox{\@brx}{\(\m@th{#1\langle}\)}%
  \mathopen{\copy\@brx\kern-0.5\wd\@brx\usebox{\@brx}}}
\newcommand{\rrangle}[1][]{\savebox{\@brx}{\(\m@th{#1\rangle}\)}%
  \mathclose{\copy\@brx\kern-0.5\wd\@brx\usebox{\@brx}}}
\makeatother

\title[Symbolic dynamics of flows in high dimension]{Symbolic dynamics for non-uniformly\\ hyperbolic flows in high dimension}

\author[Y. Lima]{Yuri Lima}
\address{Instituto de Matemática e Estatística, Universidade de São Paulo, Rua do Matão, 1010, Cidade Universitária, 05508-090. São Paulo -- SP, Brazil}
\email{yurilima@gmail.com}

\author[J. Mongez]{Juan Carlos Mongez}
\address{Instituto de Matemática e Estatística, Universidade de São Paulo, Rua do Matão, 1010, Cidade Universitária, 05508-090. São Paulo -- SP, Brazil}
\email{mongez@usp.br}

\author[J. Nascimento]{João Paulo Nascimento }
\address{Departamento de Matem\'atica -- CCN, Universidade Federal do Piauí, 64049-550. Teresina -- PI, Brazil}
\email{jpaulosousan@gmail.com}

\date{\today}
\thanks{YL was supported by 
CNPq/MCTI/FNDCT project 406750/2021-1,
FUNCAP grant UNI-0210-00288.01.00/23,
and Instituto Serrapilheira, grant
``Jangada Din\^{a}mica: Impulsionando Sistemas Din\^{a}micos na Regi\~{a}o Nordeste''. JCM was partially supported by FAPESP-Bolsa No Pa\'is P\'os-Doutorado No. 2024/22609-1 and CNPq-Brazil-Bolsa de Pós-doutorado Júnior No. 175065/2023-3. JPN was supported by CAPES}

\begin{document}

\begin{abstract}
We construct symbolic dynamics for flows with positive speed in any dimension: for each $\chi>0$, we code a set that has full measure for every invariant probability measure which is $\chi$--hyperbolic. In particular, the coded set contains all hyperbolic periodic orbits with Lyapunov exponents outside of $[-\chi,\chi]$.
This extends the recent work of Buzzi, Crovisier, and Lima for three dimensional flows with positive speed \cite{BCL23}. As an application, we code homoclinic classes of measures by suspensions of irreducible countable Markov shifts, and prove that each such class has at most one probability measure that maximizes the entropy.
\end{abstract}

\maketitle

\tableofcontents

\section{Introduction}

Since the work of Sarig \cite{Sarig-JAMS} and the recent developments/applications of Markov partitions for non-uniformly hyperbolic systems, it has become clear that Markov partitions that code uncountably many invariant measures provide more information than Markov partitions for a single (or countably many) measure. The Markov partition for surface diffeomorphisms constructed by Sarig indeed has this desired property, as it codes all recurrent\footnote{Recurrence is defined in terms of a list of non-uniform hyperbolicity parameters.} points with some non-uniform hyperbolicity greater than a fixed threshold $\chi>0$. This was later extended to diffeomorphisms in any dimension by Ben Ovadia \cite{O18}. For flows with positive speed, the first construction of Markov partitions by Lima and Sarig only coded, for three dimensional flows, countably many ergodic hyperbolic measures at the same time \cite{Lima-Sarig} (the original statement only codes one measure at a time, but the arguments easily apply to code countably many). Although this is satisfactory to obtain many statistical consequences, there are applications that require a Markov partition coding more (uncountably many) measures. It took almost ten years until an improved result for flows was obtained, by Buzzi, Crovisier and Lima \cite{BCL23}. They constructed, as in the case of diffeomorphisms, a Markov partition that codes all recurrent points with some non-uniform hyperbolicity greater than a fixed threshold $\chi>0$. As communicated by Dinowitz \cite{Dinowitz}, the result in \cite{BCL23} is essential to make multifractal analysis for flows. 

The present paper goes in the same direction of \cite{BCL23} and constructs, for flows with positive speed {\em in any dimension}, a Markov partition that codes all recurrent points with some non-uniform hyperbolicity greater than a fixed threshold $\chi>0$. Let us be more precise.
Let $M$ be a smooth closed finite dimensional manifold
and $X$ be a $C^{1+\beta}$ vector field on $M$ with $\beta>0$
which is non-singular, i.e. $X_p\neq 0$ for all $p\in M$, and let $\vf=\{\vf^t\}_{t\in\R}$ be the flow generated by $X$.
We describe the coded set in terms of $\vf$--invariant probability measures, as follows.  
Let $\chi>0$.

\medskip
\noindent
{\sc $\chi$--hyperbolic measure:} A $\vf$--invariant probability measure $\mu$ on $M$
is {\em $\chi$--hyperbolic} if $\mu$--a.e. point has all Lyapunov exponents in directions not parallel to $X$ outside of the interval $[-\chi,\chi]$.

\begin{maintheorem}\label{maintheorem}
Let $X$ be a non-singular $C^{1+\beta}$ vector field ($\beta>0$) on a closed manifold $M$, as above. For each $\chi>0$, there exists a locally compact topological Markov flow
$(\Sigma_r,\sigma_r)$ and a map $\pi_r:\Sigma_r\to M$ such that
$\pi_r\circ \sigma_r^t=\vf^t\circ\pi_r$ for all $t\in\R$, and satisfying:
\begin{enumerate}[{\rm (1)}]
\item The roof function $r$ and the projection $\pi_r$ are H\"older continuous.
\item $\pi_r[\Sigma_r^\#]$ has full measure for every $\chi$--hyperbolic measure on $M$.
\item $\pi_r$ is finite-to-one on $\Sigma_r^\#$, i.e. $\operatorname{Card}(\{z\in \Sigma_r^\#:\pi_r(z)=x\})<\infty$, for all $x\in \pi_r[\Sigma_r^\#]$.
\end{enumerate}
\end{maintheorem}

\noindent
A more precise version of the Main Theorem is stated in 
Theorem~\ref{t.main}.
\smallskip

A topological Markov flow is the unit speed vertical flow on a suspension space whose
basis is a topological Markov shift and whose roof function is continuous, everywhere positive and uniformly bounded.
We can endow $(\Sigma_r,\sigma_r)$
with a natural metric, called the {\em Bowen-Walters metric}, that makes $\sigma_r$ a continuous flow.
It is for this metric that $\pi_r$ is H\"older continuous.
The set $\Sigma_r^\#$ is the {\em regular}  set of $(\Sigma_r,\sigma_r)$, consisting of all elements of
$\Sigma_r$ for which the symbolic coordinate has a symbol repeating infinitely often in the future
and a symbol repeating infinitely often in the past. See Section \ref{Section-Preliminaries} for the definitions.

The Main Theorem provides a {\em single} symbolic extension that codes all
$\chi$--hyperbolic measures at the same time, and that is finite-to-one almost everywhere.
This improves on the result for flows by
Araujo, Lima and Poletti \cite{ALP} and by Lima and Poletti \cite{Lima-Poletti},
whose codings depend on the choice of a measure (or countably many measures), and extends to any dimension the work of Buzzi, Crovisier and Lima \cite{BCL23}.

In applications, it is useful to work with {\em irreducible} Markov shifts
since, among other properties, they are topologically transitive and 
carry at most one equilibrium state for each H\"older continuous potential
(see Section~\ref{subsection-symbolic}).
This is related to the notion of homoclinically related measures and of
\emph{homoclinic classes of measures}, defined in Section~\ref{sec.homoclinic}.
In this context, we prove the following theorem.

\begin{theorem}\label{thm.homoclinic}
In the setting of the Main Theorem,
let $\mu$ be an ergodic hyperbolic measure.
Then $\Sigma_{r}$ contains an irreducible component
${\Sigma'_{r}}$ which lifts all ergodic $\chi$--hyperbolic measures $\nu$ that are homoclinically related to $\mu$.
\end{theorem}

Observe that we {\em do not} require $\mu$ to be $\chi$--hyperbolic, hence $\mu$ itself might not be lifted.
The important thing that $\mu$ provides is a set of invariant measures, the homoclinic class of measures
related to $\mu$, and we focus on the subset of these measures that in addition are $\chi$--hyperbolic.
Note that if $\nu$ is $\chi$--hyperbolic and homoclinically related to $\mu$, then we can
change $\mu$ to $\nu$ in the statement of Theorem \ref{thm.homoclinic} and obtain the same conclusion.

This implies the following result for equilibrium states. Call a continuous potential $\psi:M\to\R$ {\em admissible} if $\psi\circ\pi_r:\Sigma_r\to\R$ is H\"older continuous, for every $\pi_r$ satisfying the Main Theorem. Clearly, every H\"older continuous $\psi$ is admissible.

\begin{corollary}\label{cor.local-uniq}
In the setting of the Main Theorem,
let $\mu$ be an ergodic hyperbolic measure, and let $\psi:M\to \mathbb{R}$ be an admissible potential.
Then there is at most one hyperbolic measure $\nu$ which is homoclinically related to $\mu$ and satisfies 
$$
h(\vf,\nu)+\int \psi  d\nu=\sup\left\{h(\vf,\eta) +\int \psi  d\eta:\eta \text{ \em hyperbolic and homoclinically related to $\mu$}\right\}.
$$
\end{corollary}

The above corollary {\em does not} claim the uniqueness of the equilibrium state for the potential $\psi$, since there could be non-hyperbolic measures achieving the above supremum. But for contexts where equilibrium states are known to be hyperbolic, Corollary \ref{cor.local-uniq} applies to provide the uniqueness of the equilibrium state. This is known, for instance, for geodesic flows on rank one manifolds \cite{Knieper-Rank-One-Entropy}, and it was recently reproved \cite{BCFT}. Using this information, Lima and Poletti  obtained a proof using symbolic dynamics of the uniqueness of the measure of maximal entropy for geodesic flows on rank one manifolds \cite{Lima-Poletti}, as well as results of \cite{BCFT}. Since our work extends the work of Lima and Poletti, it can be used to reprove the aforementioned results as well. We point out that, over the past ten years, there has been significant progress on the uniqueness of measures of maximal entropy for geodesic flows, see for instance~\cite{Gelfert-Ruggiero,CKP-20,CKW, Gelfert-Ruggiero-23,Mamani-24,Mamani-Ruggiero,pacifico-yang-yang,Wu}.

The field of symbolic dynamics has been extremely successful in analyzing systems displaying hyperbolic
behavior. Its modern history includes (but is not restricted to) the construction of Markov partitions in various 
uniformly and non-uniformly hyperbolic settings:
\begin{enumerate}[$\circ$]
\item Adler and Weiss for two dimensional hyperbolic toral automorphisms \cite{Adler-Weiss-PNAS}.
\item Sina{\u\i} for Anosov diffeomorphisms \cite{Sinai-Construction-of-MP}.
\item Ratner for Anosov flows \cite{Ratner-MP-three-dimensions,Ratner-MP-n-dimensions}.
\item Bowen for Axiom A diffeomorphisms \cite{Bowen-MP-Axiom-A,Bowen-LNM}
and Axiom A flows without fixed points \cite{Bowen-Symbolic-Flows}.
\item Katok for sets approximating hyperbolic measures of diffeomorphisms \cite{KatokIHES}.
\item Hofbauer \cite{Hofbauer-PMM} and Buzzi \cite{Buzzi-IJM,Buzzi-Invent} for piecewise maps on the interval and beyond.
\item Sarig for surface diffeomorphisms \cite{Sarig-JAMS}.
\item Lima and Matheus for surface maps with singularities, e.g. billiards \cite{Lima-Matheus}.
\item Ben Ovadia for diffeomorphisms in any dimension \cite{O18}.
\item Lima and Sarig for three dimensional flows without fixed points \cite{Lima-Sarig}.
\item Lima for one-dimensional maps \cite{Lima-AIHP}.
\item Araujo, Lima, and Poletti for non-invertible maps with singularities in any dimension \cite{ALP}.
\item Buzzi, Crovisier and Sarig for homoclinic classes of measures for diffeomorphisms in any dimension \cite{BCS-MME}.
\item Lima, Obata and Poletti for homoclinic classes of measures for non-invertible maps in any dimension \cite{LOP-24}.
\end{enumerate}

It is also relevant to mention previous related work that dealt with homoclinic classes of measures:
\begin{enumerate}[$\circ$]
\item Rodriguez Hertz et al introduced ergodic homoclinic classes of hyperbolic periodic
points, and studied SRB measures for surface diffeomorphisms \cite{HHTU-CMP}. This was recently extended for flows by de Jesus, Espitia and Ponce \cite{JEP2025}.
\item Buzzi, Crovisier and Sarig introduced homoclinic classes of
measures and proved that the Markov partitions of
Sarig \cite{Sarig-JAMS} and Ben Ovadia \cite{O18} code
homoclinic classes by  irreducible countable topological Markov shifts,
and each homoclinic class supports at most one equilibrium state
for each admissible potential, see \cite[Thm. 3.1, Cor. 3.3]{BCS-MME}.
\item Buzzi, Crovisier and Lima studied homoclinic classes of measures
for 3--dimensional flows with positive speed \cite{BCL23}.
\item Lima and Poletti studied homoclinic classes of measures for geodesic flows on rank one manifolds
\cite{Lima-Poletti}.
\end{enumerate}
We also point out that, very recently, Zang proved that a $C^\infty$ flow with positive speed
on a closed three-dimensional manifold has finitely many measures of maximal entropy \cite{Yuntao2025}.

\subsection{Method of proof}\label{ss-method-proof}

We build on the work of Buzzi, Crovisier and Lima \cite{BCL23}, which was inspired by the work of Sarig \cite{Sarig-JAMS}, Lima and Sarig \cite{Lima-Sarig}, and Bowen \cite{Bowen-Symbolic-Flows}. The main steps of the construction are:
\begin{enumerate}[(1)]
\item Construct two global Poincar\'e sections $\Lambda,\widehat{\Lambda}$ such that
$\Lambda\subset\widehat{\Lambda}$. We use $\Lambda$ as the reference section
for the construction, and $\widehat{\Lambda}$ as a security section.
\item Let $f:\Lambda\to\Lambda$ be the Poincar\'e return map of $\Lambda$.
If $\mu$ is $\chi$--hyperbolic and $\nu$ is the measure induced on $\Lambda$,
then $\nu$--almost every $x\in\Lambda$ has a Pesin chart
$\Psi_x:[-Q(x),Q(x)]^2\to \widehat\Lambda$ whose sizes along the orbit satisfy $\lim\tfrac{1}{n}\log Q(f^n(x))=0$.
Note that the center of the chart is in $\Lambda$, while the image is on
$\widehat\Lambda$. Local changes of coordinates by linear maps of norm $Q^{-1}$
allow to conjugate $f$ to a uniformly hyperbolic map.
\item Introduce $\ve$--double charts $\Psi_x^{p^s,p^u}$, which are versions of Pesin charts
which separately control the local stable and local unstable hyperbolicity at $x$. Define the
transition between $\ve$--double charts so that the parameters $p^s,p^u$ are {\em almost maximal}.
\item Construct a countable collection $\mathfs A$ of $\ve$--double charts that are dense
in the space of all $\ve$--double charts. The notion of denseness is defined in terms of
finitely many parameters of $x$. Using pseudo-orbits, shadowing
and the graph transform method, the collection $\mathfs A$ defines
a Markov cover $\mathfs Z$. In general, $\mathfs Z$ defines a symbolic coding that
is {\em usually infinite-to-one}. Fortunately, $\mathfs Z$ is locally finite (this crucial property is not direct and requires proof, which is long and delicate). 
\item $\mathfs Z$ satisfies a Markov property: for every $x\in\bigcup_{Z\in \mathfs Z}Z$
there is $k>0$ such that $f^k(x)$ satisfies a Markov property in the stable direction and
$\ell>0$ such that $f^{-\ell}(x)$ satisfies a Markov property in the unstable direction. The values
of $k,\ell$ are uniformly bounded.
\item The local finiteness of $\mathfs Z$ and the uniform bounds on $k,\ell$ allow 
to apply a refinement method to obtain a countable Markov partition, which 
defines a topological Markov flow $(\Sigma_r,\sigma_r)$ and a map
$\pi_r:\Sigma_r\to M$ satisfying the Main Theorem.
\end{enumerate}
Similarly to Bowen \cite{Bowen-Symbolic-Flows} and analogous to \cite{BCL23}, a good return of the center of
a chart is a return to $\Lambda$. The ideas of \cite{Bowen-Symbolic-Flows} are also used in steps
(5) and (6).

Steps (1), (3) and most of (4) are performed as in \cite{BCL23}. Due to the high dimension, step (2) requires a different approach, similar to \cite{O18,ALP}. 

The final part of step (4), that $\mathfs Z$ is locally finite, constitutes an important and substantial part of our work, and also requires a different approach from \cite{BCL23}. For that, we follow the techniques developed in \cite{O18,ALP}. We note
that the parameters estimating the rates of contraction/expansion along stable/unstable directions are defined in terms of integrals, but fortunately these integrals can be approximated by sums. Implementing this approximation requires, among other things, obtaining estimates on the rates of contraction/expansion along invariant manifolds finer than those obtained in \cite{BCL23}, so that the integrals can be indeed well approximated by sums. The development of the techniques of \cite{O18,ALP} to flows is, in our opinion, the main novelty of this work.

Similarly to \cite{BCL23}, our work solves the {\em parsing problem} for flows with positive speed in any dimension. This problem is related to the fact that there is no canonical way to parse a flow orbit into good returns,
hence a single orbit might be cut into different ways.
In order to address this important issue, \cite{BCL23} obtains an intrinsic solution to the  {\em inverse problem}, whose conclusion is that the parameters of the $\ve$--double
charts coding an orbit are defined ``up to bounded error''.
By intrinsic we mean that \cite{BCL23} compares the parameters of the $\ve$--double charts directly with those of the coded orbit.

As expected and claimed above, much of this work is related to \cite{BCL23}. This text strikes a balance between including all necessary details and avoiding a verbatim reproduction of the arguments in \cite{BCL23}. Therefore, 
to keep the text self-contained, many statements and proofs are similar to \cite{BCL23}, but we have made an effort to highlight the novel contributions of our work.

\medskip
\noindent
{\em{Acknowledgements.}} We are thankful to J. Buzzi, S. Crovisier, M. Poletti, J. Yang and the anonymous
referees for the suggestions that greatly improved the quality of the manuscript. During the preparation of this manuscript,
the authors were affiliated with the Universidade Federal do Ceará, and we gratefully acknowledge its support.

\subsection{Preliminaries}\label{Section-Preliminaries}

\subsubsection{Symbolic dynamics}\label{subsection-symbolic}
Let $\mathfs G=(V,E)$ be an oriented graph, where $V,E$ are the vertex and edge sets.
We denote edges by $v\to w$, and assume that $V$ is countable.

\medskip
\noindent
{\sc Topological Markov shift (TMS):} It is a pair $(\Sigma,\sigma)$
where
$$
\Sigma:=\{\text{$\Z$--indexed paths on $\mathfs G$}\}=
\left\{\un{v}=\{v_n\}_{n\in\Z}\in V^{\Z}:v_n\to v_{n+1}, \forall n\in\Z\right\}
$$
is the symbolic space and $\sigma:\Sigma\to\Sigma$, $[\sigma(\un v)]_n=v_{n+1}$, is the {\em left shift}. 
We endow $\Sigma$ with the distance $d(\un v,\un w):={\rm exp}[-\inf\{|n|\in\Z:v_n\neq w_n\}]$.
The {\em regular set} of $\Sigma$ is
$$
\Sigma^\#:=\left\{\un v\in\Sigma:\exists v,w\in V\text{ s.t. }\begin{array}{l}v_n=v\text{ for infinitely many }n>0\\
v_n=w\text{ for infinitely many }n<0
\end{array}\right\}.
$$

\medskip
We only consider TMS that are \emph{locally compact}, i.e.
for all $v\in V$ the number of ingoing edges $u\to v$ and outgoing edges $v\to w$ is finite.

\medskip
Given $(\Sigma,\sigma)$ a TMS, let $r:\Sigma\to(0,+\infty)$ be a continuous function.
For $n\geq 0$, let
$r_n=r+r\circ\sigma+\cdots+r\circ \sigma^{n-1}$ be $n$--th {\em Birkhoff sum} of $r$,
and extend this definition for $n<0$
in the unique way such that the {\em cocycle identity} holds: $r_{m+n}=r_m+r_n\circ\sigma^m$, $\forall m,n\in\Z$.

\medskip
\noindent
{\sc Topological Markov flow (TMF):} The TMF defined
by $(\Sigma,\sigma)$ and the \emph{roof function} $r$ is the pair $(\Sigma_r,\sigma_r)$ where
$\Sigma_r:=\{(\un v,t):\un v\in\Sigma, 0\leq t<r(\un v)\}$
and $\sigma_r:\Sigma_r\to\Sigma_r$ is the flow on $\Sigma_r$ given by
$\sigma_r^t(\un v,t')=(\sigma^n(\un v),t'+t-r_n(\un v))$, where
$n$ is the unique integer such that $r_n(\un v)\leq t'+t<r_{n+1}(\un v)$.
We endow $\Sigma_r$ with a natural metric $d_r(\cdot,\cdot)$,
called the {\em Bowen-Walters metric}, such that $\sigma_r$ is a continuous flow \cite{Bowen-Walters-Metric,Lima-Sarig}.
The {\em regular set} of $(\Sigma_r,\sigma_r)$ is $\Sigma_r^\#=\{(\un v,t)\in\Sigma_r:\un v\in \Sigma^\#\}$.

\medskip
In other words, $\sigma_r$ is the unit speed vertical flow on $\Sigma_r$ with the identification
$(\un v,r(\un v))\sim (\sigma(\un v),0)$. 
The roof functions we will consider will be H\"older continuous.
In this case, there exist $\kappa,C>0$ such that $d_r(\sigma_r^t(z),\sigma_r^{t}(z'))\leq C d_r(z,z')^\kappa$
for all $|t|\leq 1$ and $z,z'\in\Sigma_r$, see \cite[Lemma 5.8]{Lima-Sarig}.

\medskip
\noindent
{\sc Irreducible component:}
If $\Sigma$ is a countable Markov shift defined by an oriented graph
$\mathfs{G}=(V,E)$, its \emph{irreducible components} are the subshifts $\Sigma'\subset \Sigma$ over
maximal subsets $V'\subset V$ satisfying the following condition:
$$\forall v,w\in V',\;\exists \un v\in \Sigma \text{ and } n\geq 1\text{ such that } v_0=v \text{ and } v_n=w.$$
An irreducible component $\Sigma'_r$ of a suspended shift $\Sigma_r$ is a set of
elements $(\un v,t)\in \Sigma_r$ with $\un v$ in an irreducible component $\Sigma'$ of $\Sigma$.

\subsubsection{Metrics}\label{s.metric}

If $M$ is a smooth Riemannian manifold, we denote by
$d_M$ the distance induced by the Riemannian metric.
The Riemannian metric induces a Riemannian metric $\Sas(\cdot,\cdot)$ on $TM$,
called the {\em Sasaki metric}, see e.g. \cite[\S2]{Burns-Masur-Wilkinson}.
For nearby small vectors, the Sasaki metric is
almost a product metric in the following sense. Given a geodesic $\gamma$ in $M$
joining $y$ to $x$, let $P_\gamma:T_yM\to T_xM$
be the parallel transport along $\gamma$. If $v\in T_xM$, $w\in T_yM$ then
$\Sas(v,w)\asymp d(x,y)+\|v-P_\gamma w\|$ as $\Sas(v,w)\to 0$, see e.g.
\cite[Appendix A]{Burns-Masur-Wilkinson}. The rate of convergence depends on the
curvature tensor of the metric on $M$.

Given an open set $U\subset \R^n$ and $h:U\to \R^m$,
let $\|h\|_{C^0}:=\sup_{x\in U}\|h(x)\|$ denote the $C^0$ norm of $h$. For $0<\beta\leq 1$,
let $\Hol{\beta}(h):=\sup\frac{\|h(x)-h(y)\|}{\|x-y\|^\beta}$ 
where the supremum ranges over distinct elements $x,y\in U$. Note that $\Hol{1}(h)$ is a
Lipschitz constant of $h$, that we will also denote by $\Lip(h)$.
If $h$ is differentiable, let
$\|h\|_{C^1}:=\|h\|_{C^0}+\|dh\|_{C^0}$ denote its $C^1$ norm, and
$\|h\|_{C^{1+\beta}}:=\|h\|_{C^1}+\Hol{\beta}(dh)$ its $C^{1+\beta}$ norm.

For any $x,y$ close to some point $z$ in a Riemannian manifold $M$,
the parallel transport along the shortest geodesic between $x$ and $y$
induces a linear map $P_{x,y}:T_xM\to T_yM$.
To  any linear map $A\colon T_xM\to T_yM$, one associates a map
$\widetilde A :=P_{y,z} \circ A\circ P_{z,x}$.
By definition, $\widetilde{A}$ depends on $z$ but different basepoints $z$ define
maps that differ from $\widetilde{A}$ by pre and post composition with isometries.
In particular, $\|\widetilde{A}\|$ does not depend on the choice of $z$.

\subsubsection{Notations.}
For $a,b,\varepsilon>0$, we write $a=e^{\pm\ve}b$ when $e^{-\ve}\leq \frac{a}{b}\leq e^\ve$.
We also write $a\wedge b:=\min (a,b)$.
We write $\bigsqcup A_n$ to represent the {\em disjoint union} of sets $A_n$.

We fix a smooth Riemannian manifold $M$ of dimension $d+1$.

\subsection{Standing assumptions}\label{ss.standing}

Let $M$ be a closed smooth Riemannian manifold of dimension $d+1$, and let $X:M\to TM$ be 
a $C^{1+\beta}$ vector field such that $X(x)\neq 0$, $\forall x\in M$, and let $\varphi=(\varphi^t)_{t\in \R}$ 
be the flow generated by $X$.
We will denote the value of the vector field $X$ at $x$ also by $X_x$.
Given a set $Y\subset M$ and an interval $I\subset\R$, write $\vf^I(Y):=\bigcup_{t\in I}\vf^t(Y)$.

Since obtaining a coding for the flow generated by $X$ is equivalent to obtaining a coding for 
the flow generated by $cX$ for some 
$c>0$, we assume from now on that $\|\nabla X\|_0\leq1$ (just change $X$ to $cX$ for $c>0$ small enough)\footnote{ The notation $\nabla X$ represents the covariant differential, i.e. for each $x\in M$ 
we have a linear map $\nabla X(x):T_xM\to T_xM$ defined by $[\nabla X(x)](Y)=\nabla_Y X$.}.
This assumption
avoids the introduction of some multiplicative constants.
For instance, since an application of the Gr\"onwall inequality implies that $\|d\vf^t\|\leq e^{\|\nabla X\|_0 |t|}$ for all $t\in\R$
(see e.g. \cite{Kunzinger-flow}), we will simply write that $\|d\vf^t\|\leq e^{|t|}$, $\forall t\in\R$.
Another consequence is that every Lyapunov exponent of $\vf$ has absolute value
at most $1$, hence we can take $\chi\in (0,1)$ in the definition of $\chi$--hyperbolicity.

\section{Poincaré Sections}\label{Section-sections}

The goal of this section is to:
\begin{enumerate}[$\circ$]
\item Construct two sections $\Lambda$ and $\wh{\Lambda}$ with controlled geometrical properties such that $\Lambda \subset \wh\Lambda$, $d(\Lambda, \partial\wh\Lambda) >0$, and the trajectories under $\varphi$ of every point in $M$ intersects $\Lambda$ after some time $\rho \ll 1$. The smaller section $\Lambda$ defines a return map $f$ and a return time $r$. From now on, we call $\Lambda$ the reference section and $\wh\Lambda$ the security section.
\item Define the induced linear Poincaré flow $\Phi$, which is a flow that describes the infinitesimal behavior of $\varphi$ in the directions transverse to $X$.
\item Introduce the holonomy maps $g_x^+, g_x^-$ for each $x \in \Lambda$, which are local and {\em continuous} versions of Poincaré return maps. In Section \ref{sec-NUH-locus}, we will construct dynamically relevant systems of coordinates for these maps.
\end{enumerate}

\subsection{Transverse discs and flow boxes}
Let us start with the following definition.

\medskip
\noindent
{\sc $\rho$--transverse disc:} A codimension one open disc $D \subset M$ is {\em $\rho$--transverse} if:
\begin{enumerate}[$\circ$]
    \item $D$ is compactly contained in a $C^\infty$ codimension one submanifold of $M$.
    \item $\mathrm{diam}(D) < 4 \rho$.
    \item For every $x \in D$, $\angle (X (x), T_x D^\perp ) < \rho$.
\end{enumerate}

In other words, a $\rho$--transverse disc is a small codimension one submanifold that is almost orthogonal to $X$. It is easy to build $\rho$--transverse discs, e.g. using the tubular neighborhood theorem: by this theorem, $\varphi$ is conjugated in local charts to the vertical flow $\psi$ on $\R^n \times \R$ given by $\psi_t(x, t_0)=(x,t_0+t)$. If $\rho '$ is small enough, then the image of $B(0, \rho ') \times \{ t_0 \}$ under the local chart is a $\rho$--transverse disc.

\medskip
\noindent
{\sc Flow box:} Every $\rho$--transverse disc $D$ defines a {\em flow box} $\varphi^{[-4 \rho, 4 \rho]} D $.

\medskip
The assumption that $X$ has positive speed implies that if $\rho > 0$ is small enough then the map $(y,t) \in D \times [-4 \rho, 4 \rho] \mapsto \varphi^t (y)$ is a diffeomorphism onto the flow box $\varphi^{[-4 \rho, 4 \rho]} D$. We denote its inverse by $x \in \varphi^{[-4 \rho, 4 \rho]} D \mapsto (\mathfrak{q}_D (x), \mathfrak{t}_D (x))$, where $\mathfrak{q}_D: \varphi^{[-4 \rho, 4 \rho]} D \rightarrow D$ and $\mathfrak{t}_D: \varphi^{[-4 \rho, 4 \rho]} D \rightarrow [- 4 \rho, 4 \rho]$.

\begin{lemma}\label{lemma-local-coord}
There is $\rho_0 = \rho_0 (M,X) >0$ such that for all $\rho_0$--transverse discs $D$ and $D'$:
\begin{enumerate}[{\rm (1)}]
    \item The maps $\mathfrak{q}_D$, $\mathfrak{t}_D$ are $C^{1+\beta}$.
    \item The map $\mathfrak{q}_D$ has a Lipschitz constant smaller than $2$.
    \item If $D'$ intersects the flow box $\varphi^{[-4 \rho, 4 \rho]} D$, then the restriction to $D'$ of the map $\mathfrak{t}_D$ has a Lipschitz constant smaller than $1$.
\end{enumerate}
\end{lemma}

When $M$ has dimension three, this is \cite[Lemma 2.1]{BCL23}, and the proof in higher dimension goes without change.

\subsection{Proper sections and Poincaré return maps}

We start with some definitions.

\medskip
\noindent
{\sc Proper section:}
A \emph{proper section of size $\rho$} is a finite union $\Lambda = \bigcup_{i=1}^n D_i$ of $\rho$--transverse discs $D_1, \ldots , D_n$ such that:
\begin{enumerate}[(1)]
\item {\sc Cover:} $M = \bigcup_{i=1}^n \varphi^{[0, \rho)} D_i$.
\item {\sc Partial order:} For all $i \neq j$, at least one of the sets $\overline{D_i} \cap \varphi^{[0, 4 \rho]} \overline{D_j}$ or $\overline{D_j} \cap \varphi^{[0, 4 \rho]} \overline{D_i}$ is empty; in particular $\overline{D_i} \cap \overline{D_j} = \emptyset$.
\end{enumerate}

\medskip
The {\em return time function} $r_{\Lambda}: \Lambda \rightarrow (0, \rho )$ is defined by $r_\Lambda(x) := \inf \{ t>0 : \varphi^t (x) \in \Lambda \}$.

\medskip
\noindent
{\sc Poincaré return map:} The \emph{Poincaré return map} of a proper section $\Lambda$ is the map $f_\Lambda: \Lambda \rightarrow \Lambda$ defined by $f_\Lambda (x) := \varphi^{r_\Lambda (x)} (x)$.

\medskip
In the sequel, we fix $\rho < \min \{ 0.25, \rho_0 \}$ small and consider two proper sections $\Lambda, \wh\Lambda$ of size $\rho/2$ such that $\Lambda \subset \wh\Lambda$ and $d_M (\Lambda, \partial \wh\Lambda) > 0$. We let $d = d_{\wh\Lambda}$ be the metric on $\wh\Lambda$ defined by the induced Riemannian metric on $\wh\Lambda$. For $x \in \wh\Lambda$ and $r > 0$, we use the notation:
\begin{enumerate}[$\circ$]
    \item $B(x,r) \subset \wh\Lambda$ for the ball in the distance $d$ with center $x$ and radius $r$;
    \item $B_x [r] \subset T_x \wh\Lambda$ for the ball with center 0 and radius $r$.
\end{enumerate}

Since flow boxes are $C^{1+\beta}$, there is $L=L(\wh\Lambda)>0$ such that for any transverse disc $D_i$ defining $\wh\Lambda$ the maps $\mathfrak q_{D_i}, \mathfrak t_{D_i}$ satisfy
$\text{H\"ol}_{\beta} (d \mathfrak q_{D_i}) < L$ and  $\text{H\"ol}_{\beta} (d\mathfrak t_{D_i}) < L$.

\subsection{Exponential maps}\label{exponential map}

Given $x \in \wh\Lambda$, let $\inj(x)$ denote the injectivity radius of $\wh\Lambda$ at $x$, and let $\exp{x}$ be the \emph{exponential map} of $\wh\Lambda$ at $x$, wherever it can be defined. Below we list the properties of $\exp{x}$ that we will use.

\medskip
\noindent
{\sc Regularity of $\exp{x}$:} There is $\mathfrak{r} \in (0, \rho)$ such that for every $x \in \Lambda$ the following properties hold on the ball $B_x := B(x, 2 \mathfrak{r}) \subset \wh\Lambda$:
\begin{enumerate}[ii\, )]
\item[(Exp1)] If $y \in B_x$, then $\inj(y) \geq 2 \mathfrak{r}$, the map $\exp{y}^{-1} : B_x \rightarrow T_y \wh\Lambda$ is a diffeomorphism onto its image, and for all $v \in T_x \Sh$, $w \in T_y \Sh$ with $\| v \|, \| w \| \leq 2 \mathfrak{r}$ it holds
$$\tfrac{1}{2} (d(x,y) + \|v - P_{y,x} w \| ) \leq \Sas(v,w) \leq 2(d(x,y) + \| v -P_{y,x} w \| ),$$
where $P_{y,x}$ is the parallel transport along the minimizing geodesic joining $y$ to $x$.
\item[(Exp2)] If $y_1, y_2 \in B_x$, then $d(\exp{y_1} v_1, \exp{y_2} v_2) \leq 2 \Sas(v_1, v_2)$ for $\| v_1 \|, \| v_2 \| \leq 2 \mathfrak{r}$, and $\Sas(\exp{y_1}^{-1} z_1, \exp{y_2}^{-1} z_2) \leq 2[d(y_1, y_2) + d(z_1, z_2)]$ for $z_1, z_2 \in B_x$ when the expression is defined. In particular, $\| d(\exp{x})_v\| \leq 2$ for $\| v \| \leq 2 \mathfrak{r}$ and $\| d(\exp{x}^{-1})_y \| \leq 2$ for $y \in B_x$.
\end{enumerate}

\medskip
Conditions (Exp1)--(Exp2) express that the exponential maps and their inverses are well-defined and have Lipschitz constants at most 2 in balls of radius $2 \mathfrak{r}$. The existence of the constant $\mathfrak{r}$ follows from a compactness argument, using that $d_M (\Lambda, \partial \Sh) > 0$ and that $d(\exp{x})_0$ is the identity map.

The next two assumptions require some regularity on the derivatives of $\exp{x}$. For $x, x' \in \Sh$, let $\mathfs{L}_{x,x'} := \{ A: T_x \Sh \rightarrow T_{x'} \Sh \text{ is linear}\}$ and $\mathfs{L}_x := \mathfs{L}_{x,x}$. In particular, $P_{y,x}$ considered in (Exp1) belongs to $\mathfs{L}_{y,x}$. Given $y \in B_x, z \in B_{x'}$ and $A \in \mathfs{L}_{y,z}$, define $\wt{A} \in \mathfs{L}_{x,x'}$ by $\wt{A} := P_{z,x'} \circ A \circ P_{x,y}$. The norm $\| \wt{A} \|$ does not depend on the choice of $x, x'$. 
\label{geodesic-triangles} If $A_i \in \mathfs{L}_{y_i,z_i}$, then $\| \wt{A_1} - \wt{A_2} \|$ does depend on the choice of $x,x'$, but if we change the basepoints $x,x'$ to $w, w'$ then the respective differences differ by precompositions and postcompositions whose norms have the order of the areas of the geodesic triangles formed by $x, w, y_i$ and by $x',w',z_i$, which will be negligible to our estimates. Hence we are free to consider $\wt A$ without making an explicit choice of $x,x'$.

For $x \in \Lambda$, define the map $\tau = \tau_x : B_x \times B_x \rightarrow \mathfs{L}_x$ by $\tau (y,z) = \wt{d(\exp{y}^{-1})_z}$, where we use the identification $T_v (T_y \Sh) \cong T_y \Sh$ for all $v \in T_y \Sh$.

\medskip
\noindent
{\sc Regularity of $d\exp{x}$:} There is $\mathfrak{K}>1$ such that for all $x \in \Lambda$ the following holds:
\begin{enumerate}[ii\, )]
\item[(Exp3)] If $y_1, y_2 \in B_x$ then $\|\wt{d(\exp{y_1})_{v_1}} - \wt{d(\exp{y_2})_{v_2}} \| \leq \mathfrak{K} \Sas(v_1, v_2)$, for all $\| v_1 \|, \| v_2 \| \leq 2 \mathfrak{r}$, and $\| \tau (y_1, z_1) - \tau (y_2, z_2) \|\leq  \mathfrak{K} [d(y_1, y_2) + d(z_1,z_2)]$ for all $z_1, z_2 \in B_x$.
\item[(Exp4)] If $y_1, y_2 \in B_x$ then the map $\tau (y_1, \cdot) - \tau (y_2, \cdot) : B_x \rightarrow \mathfs{L}_x$  has Lipschitz constant $\leq \mathfrak{K} d(y_1, y_2)$.
\end{enumerate}

\medskip
Condition (Exp3) bounds the Lipschitz constants of the derivatives of $\exp{x}$, and (Exp4) bounds the Lipschitz constants of the second derivatives of $\exp{x}$. The existence of $\mathfrak{K}$ is guaranteed whenever the curvature tensor of $\Sh$ is uniformly bounded, and this happens because $\Sh$ is a finite union of $\rho$--transverse discs.

\subsection{Induced linear Poincaré flows}\label{section-linear-flow}

The classical linear Poincaré flow is the $\R$--cocycle induced by $d \varphi$ in the bundle orthogonal to $X$. Here we employ the version considered in \cite{BCL23}: we fix a 1--form $\theta$ and consider parallel projections to $X$ onto the bundle $\mathrm{Ker}(\theta)$. The first step is to choose a suitable 1--form.

\begin{lemma}\label{Def-of-1-form}
If $\Sh$ is a proper section of size $\rho/2$, then there exists a 1--form $\theta$ on $M$ such that:
\begin{enumerate}[{\rm (1)}]
    \item $\theta ( X(x)) = 1$ and $\angle ( X(x),\mathrm{Ker} (\theta_x)^{\perp}) < \rho, \forall x \in M$.
    \item $\mathrm{Ker}(\theta_x) = T_x \Sh, \forall x \in \Sh$.
\end{enumerate}
\end{lemma}

When $M$ has dimension three, this is \cite[Lemma 2.2]{BCL23}, and the proof in higher dimension goes without change.

From now on, we fix a 1--form $\theta$ satisfying Lemma \ref{Def-of-1-form}, and then introduce the $d$--dimensional bundle 
$$N := \bigsqcup_{x \in M}\mathrm{Ker}(\theta_x).$$
For each $x \in M$, let $\mathfrak{p}_x : T_x M \rightarrow N_x$ be the projection to $N_x$ parallel to $X(x)$. By Lemma \ref{Def-of-1-form}(1), for all $x \in M$ we have:
$$\| \mathfrak{p}_x \| = \tfrac{1}{\cos \angle (X(x),\mathrm{Ker} (\theta_x)^{\perp})} < \tfrac{1}{\cos \rho} < 1 + \rho.$$

\medskip
\noindent
{\sc Induced linear Poincaré flow:} The \emph{linear Poincaré flow of $\varphi$ induced by $\theta$} is the flow $\Phi = \{ \Phi^t \}_{t \in \R} : N \rightarrow N$ defined by $\Phi^t (v) = \mathfrak{p}_{\varphi^t (x)} [ d \varphi^t_x (v)]$ for $v \in N_x$.

\medskip
We will usually omit the subscripts $x$ and $\varphi^t (x)$, as they will become evident in the context. It is clear that $\Phi$ is Hölder continuous and satisfies $\| \Phi_x^t \| \leq \| \mathfrak{p}_{\varphi^t (x)} \|\cdot \|d \varphi_x^t \| \leq (1+\rho) e^{|t|} < e^{\rho + |t|}, \forall t \in \R$. Therefore:
\begin{equation}\label{Norm of ILPF}
\|\Phi^t\|=e^{\pm (\rho+|t|)},\ \forall t\in\R,\ \text{and} \ \| \Phi^t \| = e^{\pm 3 \rho},\ \forall |t| \leq 2 \rho.  
\end{equation}

\begin{lemma}\label{lemma-poincare-flow}
The following holds.
\begin{enumerate}[{\rm (1)}]
\item $\Phi$ is a flow: $\Phi^{t+t} = \Phi^t \circ \Phi^{t'}$ for all $t, t' \in \R$.
\item If $D \subset \Sh$ is a transverse disc, then for all $x \in D$ it holds $d(\mathfrak{q}_D)_x = \mathfrak{p}_x$.
\end{enumerate}
\end{lemma}

When $M$ has dimension three, this is \cite[Lemma 2.3]{BCL23}, and the proof in higher dimension goes without change.

\subsection{Holonomy maps}

So far, we have fixed two proper sections $\Lambda, \Sh$ of size $\rho /2$. From now on, we write $f:=f_{\Lambda}$. The maps $f, r_{\Lambda}$ usually admit discontinuities, caused by the boundaries of $\Lambda$, thus we introduce a family of local diffeomorphisms related to $f$. Recall that $\mathfrak{r} > 0$ is a fixed small parameter, and that $B_x := B(x, 2 \mathfrak{r})$. Write $\Sh = \bigcup_{i=1}^N D_i$ as the disjoint union of $\rho$--transverse discs $D_i$, and let $\mathfrak q_{D_i}$ as before. By Lemma \ref{lemma-local-coord}(2), $\mathrm{Lip}(\mathfrak{q}_{D_i}) < 2$.

Assume that $x, \varphi^t (x) \in \Lambda$ for some $|t| < \rho$ with $x \in D_i$ and $\varphi^t (x) \in D_j$. The restrictions $\mathfrak{q}_{D_j}\restriction_{B_x}$ and $\mathfrak{q}_{D_i}\restriction_{B_{\varphi^t (x)}}$ are diffeomorphisms onto their images, and one is the inverse of the other whenever the compositions make sense. When this occurs, we call these restrictions \emph{holonomy maps}.

\begin{lemma}\label{Lemma-map-g}
 Under the above conditions, the holonomy map $\mathfrak{q}_{D_j}\restriction_{B_x}$ is a $2$--bi-Lipschitz $C^{1+\beta}$ diffeomorphism onto its image, and its derivative at every $y \in B_x$ is $\Phi^s\restriction_{N_y}$, where $|s|<\rho$ satisfies $\mathfrak q_{D_j}(y)=\vf^s(y)$. 
\end{lemma}

\begin{proof}
When $M$ has dimension three, this is \cite[Lemma 2.3]{BCL23}, and the proof in higher dimension goes without change, as follows.
Write $g=\mathfrak{q}_{D_j}\restriction_{B_x}$. The first statement is direct from Lemma \ref{lemma-local-coord}(2). Now, since $g=\mathfrak{q}_{D_j} \circ \varphi^s$, Lemma \ref{lemma-poincare-flow}(2) implies that $dg_y = d(\mathfrak{q}_{D_j})_{\varphi^s (y)} \circ d\varphi_y^s\restriction_{T_y \Sh} = \mathfrak{p}_{\varphi^s(y)} \circ d \varphi_y^s\restriction_{N_y}= \Phi^s \restriction_{N_y}$.
\end{proof}

We are interested in a particular class of holonomy maps, defined as follows. Let $0 <t, t' < \rho$ such that $f(x) = \varphi^t (x) \in D_j$ and $f^{-1}(x) = \varphi^{-t'}(x) \in D_k$.

\medskip
\noindent
{\sc Holonomy maps:} The \emph{forward holonomy map} at $x$ is $g_x^+ := \mathfrak{q}_{D_j}\restriction_{B_x}$. The \emph{backward holonomy map} at $x$ is $g_x^- := \mathfrak{q}_{D_k} \restriction_{B_x}$.

\medskip
Clearly $(g_x^+)^{-1} = g_{f(x)}^-$. It is important to stress that $g_x^+$ might differ from $f$ and from the Poincaré return map to $\Sh$, e.g. we might have $y\approx x$ with $\vf^{t''}(y)\in\wh\Lambda$ for $0<t''\ll t$, in which case $g_x^+(y)\not= f_{\wh\Lambda}(y)$.

\section{The non-uniformly hyperbolic locus}\label{sec-NUH-locus}

Up to now, we have fixed the following objects: a flow $\varphi$ with positive speed, parameters $\chi,\rho>0$, sections $\Lambda,\wh\Lambda$ and a 1--form $\theta$. In this section, we will:
\begin{enumerate}[(1)]
\item Define the set $\nuh$ of points that exhibit a hyperbolicity of strength at least $\chi$. We fix $\varepsilon>0$ small
enough, and associate to each $x\in \nuh$ a number $Q(x)\in (0,1)$ that approaches zero 
as the quality of the hyperbolicity at $x$ deteriorates.
\item Introduce numbers $q(x)\in [0,1)$, that measure how fast $Q(\varphi^t(x))$
decreases to zero as $|t|\to\infty$.
We also associates analogous number $q^s(x)$ and $q^u(x)$
for future and past orbits.
\item Define the set $\nuh^\#$ of points $x\in\nuh$ whose hyperbolicity satisfies a recurrence property: there is $c(x)>0$ such that
$q(\varphi^t(x))>c(x)$ for some values of $t$ arbitrarily close to $\pm\infty$.
As we will prove, this set carries all $\chi$--hyperbolic measures.
\item Define Pesin charts $\Psi_x$ for each $x\in \Lambda\cap\nuh$.
We then prove that, in Pesin charts, the holonomy maps $g_x^{\pm}$ are close to hyperbolic linear maps.
\end{enumerate}

\subsection{Non-uniformly hyperbolic locus}\label{Def-NUH3}
We define a set with some non-uniform hyperbolicity.

\medskip
\noindent
{\sc Non-uniformly hyperbolic locus $\nuh=\nuh(\vf,\chi,\rho,\theta)$:}  It is the $\vf$--invariant set of points $x\in M$ for which there is a splitting $N_x=N^s_x\oplus N^u_x$ such that:
\begin{enumerate}[(NUH1)]
\item For every $v\in N_x^s$:
$$\liminf_{t \to +\infty} \tfrac{1}{t} \log \| \Phi^{-t}v \| > 0 \ \mbox{ and }\ \limsup_{t \to +\infty} \tfrac{1}{t} \log \| \Phi^{t}v \| \leq-\chi.
$$

\item  For every $w\in N^u_x $:
    $$\liminf_{t \to +\infty} \tfrac{1}{t} \log \| \Phi^{t}w \| > 0 \ \mbox{ and }\ \limsup_{t \to +\infty} \tfrac{1}{t} \log \| \Phi^{-t}w \|\leq -\chi.
$$

\item The parameters $s(x) =\displaystyle\sup_{v \in N^s_x\atop{\|v\|=1}} S(x,v)$  and $u(x) = \displaystyle\sup_{w\in N^u_{x}\atop{\|w\|=1}} U(x,w)$ are finite, where:
\begin{align*}
&S(x,v)^2 = 4e^{2\rho}\int_{0}^ \infty e^{2\chi t } \|  \Phi^{t}v\|^2dt \ \text{ and}\ \
U(x,w)^2 =4e^{2\rho}\int_{0}^ \infty e^{2\chi t } \|  \Phi^{-t}w\|^2 dt.
\end{align*}

\end{enumerate}

\medskip
It is clear that $N^s_x,N^u_x$ satisfying the above assumptions are unique. Recalling that $\chi \in (0,1)$, the estimate before \eqref{Norm of ILPF} gives that
\begin{align}\label{estimate-integral}
\int_0^{\infty}e^{2\chi t}\|\Phi^t v\|^2 dt\geq \int_0^{\infty}e^{2\chi t}e^{-2\rho -2t}dt=
\tfrac{e^{-2\rho}}{2(1-\chi)}>\tfrac{e^{-2\rho}}{2}
\end{align}
for all $v \in N^s_x$ unitary, hence for each $x\in \nuh$ we have $s(x),u(x)\in [\sqrt{2}, \infty)$.

\begin{proposition}\label{Prop-NUH}
    If $\mu$ is a $\chi$--hyperbolic probability measure, then $\mu[\nuh]=1$.
\end{proposition}

\begin{proof}
The proof is adapted from \cite[Lemma 3.1]{BCL23}.
Fix a $\chi$--hyperbolic measure $\mu$. By the Oseledets theorem, there is a set $X\subset M$
with $\mu[X]=1$ such that for all $x\in X$ there is a $d\vf$--invariant decomposition $E^s_x\oplus \langle X_x\rangle\oplus E^u_x=T_xM$ satisfying:
\begin{enumerate}[(1)]
\item $\lim\limits_{t \to \pm \infty} 
\frac{1}{t} \log \| d\varphi^t v^s \| < -\chi$ and $\lim\limits_{t \to \pm \infty} \frac{1}{t} \log \| d\varphi^t v^u \| > \chi$ for all $v^{s/u}\in E^{s/u}_x\setminus\{0\}$.
\item $\lim\limits_{t \to \pm \infty} \frac{1}{t} \log |\sin \angle (E_{\varphi^t(x)}^{s/u}, X_{\varphi^t(x)})| = 0$.
\end{enumerate}
Let $P=P_x$ be the projection to $N_x$ parallel to $X_x$, and define
$N^{s/u}_x=P(E^{s/u}_x)$. We claim that for all $x\in X$ these subspaces satisfy conditions (NUH1)--(NUH3). Once this is proved, it follows that $X\subset\nuh$, and so $\mu[\nuh]=1$. By symmetry, we just need to check conditions (NUH1)--(NUH3) for $N^s_x$. For that, it is enough to prove that if $x\in X$ then
$$\lim_{t\to\pm\infty}\tfrac{1}{t}\log\|\Phi^tn^s_x\|<-\chi
$$
for all $n^s_x\in N^s_x\backslash\{0\}$ (this automatically implies $s(x)<\infty$). We prove this in the sequel.

Fix $x\in X$ and $e^s_x\in E^s_x$ unitary. Let $n^s_x$ be a unitary vector parallel to $P(e^s_x)$. There are scalars 
$\gamma=\gamma(e^s_x)\neq 0$, $\delta=\delta(e^s_x)$ such that
\begin{equation}\label{definition-normal-vectors}
n^s_x=\gamma e^s_x+\delta X_x.
\end{equation}
In the case that $X_x\perp N_x$, we have 
$\gamma=\pm\tfrac{1}{\sin\angle (X_x,e^{s}_x)}$. In the general case, since by construction $\angle(X_x,N_x^\perp)<\rho$ (see Lemma \ref{Def-of-1-form}), we get that $\gamma=\pm\tfrac{e^{\pm 4\rho}}{\sin\angle (X_x,e^{s}_x)}$. Writing $e^s_{x,t}=\tfrac{d\vf^t e^s_x}{\|d\vf^t e^s_x\|}$ and defining $n^s_{x,t}$ as the unitary vector parallel to $P_{\vf^t(x)}(e^s_{x,t})$, we can similarly define $\gamma_t=\gamma(e^s_{x,t})\neq 0, \delta_t=\delta(e^s_{x,t})$ and obtain that 
$\gamma_t=\pm\tfrac{e^{\pm 4\rho}}{\sin\angle (X_{\vf^t(x)},e^{s}_{x,t})}$. Hence condition (2) implies that
$$
\lim\limits_{t\to\pm\infty}\tfrac{1}{t}\log|\gamma_t|=0.
$$

We claim that
$\|\Phi^tn^s_x\|=\tfrac{|\gamma|}{|\gamma_t|}\|d\vf^t e^s_x\|$ for all $x\in X$.
By (\ref{definition-normal-vectors}),
\begin{align*}
d\vf^t n^s_x=d\vf^t[\gamma e^s_x+\delta X_x]=
\gamma\|d\vf^t e^s_x\|e^s_{x,t}+\delta X_{\vf^t(x)}
\end{align*}
and so
\begin{align*}
&\ \Phi^tn^s_x={\mathfrak p}_{\vf^t(x)}[\gamma\|d\vf^t e^s_x\|e^s_{x,t}+\delta X_{\vf^t(x)}]=
\gamma\|d\vf^t e^s_x\|{\mathfrak p}_{\vf^t(x)}(e^s_{x,t})\\
&=\gamma\|d\vf^t e^s_x\|{\mathfrak p}_{\vf^t(x)}\left[\tfrac{1}{\gamma_t} n^s_{x,t}-
\tfrac{\delta_t}{\gamma_t}X_{\vf^t(x)}\right]=
\tfrac{\gamma}{\gamma_t}\|d\vf^t e^s_x\|n^s_{x,t}.
\end{align*}
Taking norms, we obtain that $\|\Phi^tn^s_x\|=\tfrac{|\gamma|}{|\gamma_t|}\|d\vf^t e^s_x\|$.
Hence
$$\lim_{t\to\pm\infty}\tfrac{1}{t}\log\|\Phi^tn^s_x\|
=\lim_{t\to\pm\infty}\tfrac{1}{t}\log\|d\vf^t e^s_x\|<-\chi.
$$
The proof is now complete.
\end{proof}

\subsection{Oseledets-Pesin reduction.}\label{Section-reduction}

In this section we construct a diagonalization for the restriction of the cocycle $\Phi$ to the set $\nuh$. We begin with the following definition.

\medskip
\noindent
{\sc Lyapunov inner product:} For each $x\in\nuh$, define an inner product $\llangle \cdot,\cdot\rrangle$ on $N_x$, which we
call {\em Lyapunov inner product}, by the following identities:

\begin{enumerate}[$\circ$]
    \item  for $v_1,v_2\in N^s_x$:
$$\llangle v_1,v_2\rrangle =4 e^{2\rho}\int_0^\infty e^{2\chi t}\langle \Phi^tv_1,\Phi^t v_2\rangle dt.$$
    \item  for $v_1,v_2\in N^u_x$:
$$\llangle v_1,v_2\rrangle =4 e^{2\rho}\int_0^\infty e^{2\chi t}\langle \Phi^{-t}v_1,\Phi^{-t} v_2\rangle dt.$$
\item for $v_1\in N^s_x$ and $v_2 \in N^u_x$: $\llangle v_1,v_2 \rrangle = 0$.
\end{enumerate}

\medskip
By conditions (NUH1)--(NUH3), the above integrals are finite. Let $\vertiii{\cdot}$ denote the norm induced by $\llangle\cdot,\cdot\rrangle$. Since $\vertiii{\cdot}$ uniquely defines $\llangle\cdot,\cdot\rrangle$, we call $\vertiii{\cdot}$ the {\em Lyapunov norm}. There is a relation between this norm and the Riemannian norm of $M$: if $v\in N^{s/u}$ then by (\ref{estimate-integral}) we have
$$
\vertiii{v}^2=4 e^{2\rho}\int_0^\infty e^{2\chi t}\|\Phi^{\pm t} v\|^2dt
\geq 4 e^{2\rho}\int_0^\infty e^{2\chi t}e^{-2\rho-2 |t|}\|v\|^2dt
=\tfrac{2}{1-\chi}\|v\|^2\cdot
$$
In particular, $\vertiii{v}^2\geq 2\|v\|^2$.

For $x \in \nuh$, let $d_s(x), d_u(x) \in \mathbb{N}$ be the dimensions of $N^s_{x}$, $N^u_{x}$ respectively. Since the splitting $N^s \oplus N^u$ is $\Phi$--invariant, the functions $d_s$, $d_u$ are $\vf$--invariant. In the following, we denote the canonical inner product in $\mathbb{R}^d$ by $\langle \cdot, \cdot \rangle_{\mathbb{R}^d}$.

\medskip
\noindent
{\sc Linear map $C(x)$:} For $x \in \nuh$, define $C(x):\mathbb{R}^d \to N_x$ as a linear map satisfying the following conditions:
\begin{enumerate}[$\circ$]
    \item $C(x)$ sends the subspace $\mathbb{R}^{d_s(x)} \times \{0\}$ to $N^s_{x}$ and $\{0\} \times \mathbb{R}^{d_u(x)}$ to $N^u_{x}$.
    \item $\langle v, w \rangle_{\mathbb{R}^d}=\llangle C(x)v, C(x)w \rrangle$ for all $v, w \in \mathbb{R}^d$, i.e. $C(x)$ is an isometry between $(\mathbb{R}^n, \langle\cdot, \cdot\rangle_{\mathbb{R}^d})$ and $(N_x, \llangle\cdot, \cdot\rrangle)$.
\end{enumerate}

\medskip
The map $C(x)$ is not uniquely defined (for instance, rotations inside $N^s_{x}$, $N^u_{x}$ preserve the above properties), but we can define it such that $x \in \nuh \mapsto C(x)$ is a measurable map, see e.g. \cite[Footnote at page 48]{O18}. Below we list the main properties of $C(x)$.

\begin{lemma}\label{Lemma-linear-reduction}
    The following holds for all $x \in\nuh$.
\begin{enumerate}[{\rm (1)}]
    \item $\|C(x)\|\leq 1$ and 
\[
\|C(x)^{-1}\|=\sup_{v \in N_x\setminus\{0\}} \frac{\vertiii{v}}{\|v\|}=
\sup_{v^s+v^u\in N^s_x \oplus N^u_x\atop{\|v^s + v^u\|}\neq 0} \frac{\sqrt{\vertiii{v^s}^2 + \vertiii{v^u}^2}}{\|v^s +v^u\|}\cdot
\]
    \item {\sc Oseledets-Pesin reduction:} For all $0 \leq t \leq 2 \rho$, the map $D(x,t)=C(\vf^t(x))^{-1} \circ \Phi^t \circ C(x)$ has the block form
$$
\begin{bmatrix}
    D_s (x,t) && \\ && D_u (x,t)
\end{bmatrix}$$
where $D_s(x,t)$ is a $d_s(x)\times d_s(x)$ matrix
with $e^{-4 \rho} < \|D_s(x,t)v_1\|< e^{-\chi t}$ for all unit vectors $v_1\in \R^{d_s(x)}$,
and  $D_u(x,t)$ is a $d_u(x)\times d_u(x)$ matrix with
$e^{\chi t}<\|D_u(x,t)v_2\|<e^{4 \rho}$ for all unit vectors $v_2\in \R^{d_u(x)}$.
    \item $\frac{\|C(\varphi^t(x))^{-1}\|}{\|C(x)^{-1}\|} = e^{\pm 2(\rho + |t|)}$ for all $t\in\R$; in particular, $\frac{\|C(\varphi^t(x))^{-1}\|}{\|C(x)^{-1}\|} =e^{\pm 6 \rho}$ for all $|t|\leq 2\rho$. 
\end{enumerate}
\end{lemma}

The proof is in Appendix \ref{Appendix-proofs}. Part (2) is known as Oseledets-Pesin reduction, and constitutes a diagonalization of $\Phi$.

Observe that within ${\rm NUH}$, we have defined a Lyapunov inner product that depends on a parameter $\chi > 0$. This construction is made possible by conditions (NUH1)--(NUH3). For any $\chi' < \chi$, we can carry out a similar construction inside the larger set ${\rm NUH}{\chi'}$, which produces a different Lyapunov inner product, denoted by $\vertiii{\cdot}_{\chi'}$. Since ${\rm NUH} \subset {\rm NUH}_{\chi'}$, the inner product $\vertiii{\cdot}_{\chi'}$ is also defined on ${\rm NUH}$. Note that $\vertiii{\cdot}_{\chi'} \leq \vertiii{\cdot}$, which means that $\vertiii{\cdot}$ induces a stronger norm than $\vertiii{\cdot}_{\chi'}$. For each $ x \in {\rm NUH}_{\chi'}$, we define a linear map $C_{\chi'}( x)$ associated with the inner product $\vertiii{\cdot}_{\chi'}$. The map $C_{\chi'}(x)$ satisfies a version of Lemma \ref{Lemma-linear-reduction} adapted to the weaker exponent $\chi'$, involving a corresponding block matrix $D_{\chi'}(x,t)$. Observe that $D(x,t)$ exhibits stronger hyperbolicity rates than $D_{\chi'}(x,t)$. Moreover, since $\vertiii{\cdot}_{\chi'} \leq \vertiii{\cdot}$, it follows from part (1) that
$\|C_{\chi'}( x)^{-1}\|\leq \|C( x)^{-1}\|$.

The Lyapunov inner product $\vertiii{\cdot}$ will be used throughout most of the paper. However, there is a subtle point where it becomes necessary to work with the weaker inner product $\vertiii{\cdot}_{\chi'}$ (see Section \ref{Section-finite}), following an approach successfully applied by Ben Ovadia \cite{Ova20}.

\subsection{ Quantification of hyperbolicity: the parameters $ Q(x), q(x), q^{s/u}(x) $}

We now introduce another positive parameter $\ve \ll \rho,\chi$. 
\footnote{How small $\ve$ is in comparison to $\rho,\chi$ depends on a finite number of inequalities.
In practice, we first choose $\rho=\rho(\vf,\chi)>0$ small enough satisfying a finite number of inequalities, and then
we choose $\ve=\ve(\vf,\chi,\rho)>0$ satisfying an additional finite number of inequalities.}

\medskip
\noindent
{\sc Parameter $Q(x)$:} For $x\in\nuh$, let
$$
Q(x)= \ve^{6/\beta}\|C( x)^{-1}\|^{-48/\beta}.
$$
\medskip

This parameter depends on $\ve > 0$, but for simplicity, we will omit this dependence from the notation. It is important to note that most of the results remain valid as long as $\ve > 0$ is sufficiently small. The term $\ve^{6/\beta}$ helps absorb multiplicative constants, while $\|C( x)^{-1}\|$ reflects the rate of hyperbolicity. The longer it takes for hyperbolic behavior to manifest, the larger this quantity becomes.
By Lemma \ref{Lemma-linear-reduction}(3), we have
\begin{equation}\label{Relation-Q(x)}
    \frac{Q(\varphi^t(x))}{Q(x)} = e^{\pm\frac{288\rho}{\beta}}, \ \ \forall x \in \text{NUH}, \forall\, 0 \leq t \leq 2\rho.
\end{equation}

Moreover , we have the following bounds for $Q( x)$
for $\ve>0$ small enough:
\begin{align}\label{estimates-Q}
\begin{array}{l}
Q( x)\leq \ve^{6/\beta}, \ \|C( x)^{-1}\|Q( x)^{\beta/48}\leq \ve^{1/8},
\  \|C( f( x))^{-1}\|Q( x)^{\beta/12}\leq \ve^{1/4}.
\end{array}
\end{align}

One could also define a parameter $Q_{\chi'}(x)$ by replacing $\|C( x)^{-1}\|$ with $\|C_{\chi'}( x)^{-1}\|$. In this case, for any $ x \in {\rm NUH}$ and $\chi' < \chi$, we have the inequality $Q( x) \leq Q_{\chi'}( x)$. However, since all our estimates will be carried out using the smaller value $Q( x)$, the parameter $Q_{\chi'}$ plays no essential role in this paper.

In order to have a better dynamical understanding of trajectories, we focus on the orbits of $\nuh$ with some recurrence with respect to the parameter $Q$. Having that in mind, we introduce the following additional parameters.

\medskip
\noindent
{\sc Parameters $q(x), q^s(x), q^u(x) $:}
For $x \in \nuh$, define:
\begin{align*}
q(x) &:= \ve \inf \{ e^{\ve |t|} Q(\varphi^t(x)) : t \in \mathbb{R} \},\\
q^s(x)&:= \ve \inf \{ e^{\ve |t|} Q(\varphi^t(x)) : t \geq 0 \},\\
q^u(x)&:= \ve \inf \{ e^{\ve |t|} Q(\varphi^t(x)) : t \leq 0 \}.
\end{align*}

\medskip
We have $0 \leq q(x), q^s(x), q^u(x) \leq \ve Q(x)$, and so these parameters are much smaller than $Q(x)$ itself. Furthermore, $q^s(x) \land q^u(x) = q(x)$. The families $\{ q^s(\varphi^t(x)) \}_{t \in \mathbb{R}}$ and $\{ q^u(\varphi^t(x)) \}_{t \in \mathbb{R}}$ represent {\em local quantifications of hyperbolicity} for the invariant directions along the orbit $\{ \varphi^t(x) \}_{t \in \mathbb{R}}$. The following lemma states a slow variation property of $q$. Its proof may be found in \cite[Lemma 3.4]{BCL23}.

\begin{lemma}\label{Lemma-q}
For all $x\in\nuh$ and $t\in \R$, it holds $q(\varphi^t(x)) = e^{\pm \ve |t|} q(x)$.
\end{lemma}

\subsection{ The recurrently non-uniformly hyperbolic locus $\nuh^\#$ }\mbox{}

\medskip
\noindent
{\sc Recurrently non-uniformly hyperbolic locus $ \nuh^\# = \nuh^\#(\varphi, \chi, \rho, \theta, \ve) $:} It is the invariant set of points \( x \in \text{NUH} \) such that:
\begin{enumerate}[(NUH3)]
    \item[(NUH4)] $q(x) > 0 $.
    \item[(NUH5)] $\limsup\limits_{t \to +\infty} q(\varphi^t(x)) > 0  \mbox{ and }  \limsup\limits_{t \to -\infty} q(\varphi^t(x)) > 0 .$
\end{enumerate}

Under condition (NUH4), Lemma \ref{Lemma-q} implies that $q(\varphi^t(x))>0$ for all $t \in \mathbb{R}$, and condition (NUH5) requires that these values do not degenerate to zero in the limit. Similarly to $\nuh$, the set $\nuh^\#$ carries all $\chi$--hyperbolic measures.

\begin{proposition}\label{Prop-adaptedness}
     If $\mu$ is a $\varphi$--invariant probability measure with $\mu[\nuh] = 1$, then $\mu[\nuh^\#]=1$. In particular, if $\mu$ is $\chi$--hyperbolic then $\mu[\nuh^\#]=1$.
\end{proposition}

The proof is the same as \cite[Proposition 3.5]{BCL23}.

\subsection{ The $\mathbb{Z}$--indexed versions of $q^{s/u}(x)$: the parameters $ p^{s/u}(x) $}\label{section-Z-indexed}

To simplify the analysis of $q^{s/u}(x)$, we define a discrete time approximation of these numbers, which we call $ \mathbb{Z}$--indexed versions of $q^{s/u}(x)$. The benefit of working with the $\Z$--indexed versions is that they satisfy explicit recursive formulas.   Recall that we have fixed $\Lambda$ a proper section of size $\rho/2$ with Poincaré return time denoted by $r_\Lambda$. In particular, $0 < \inf(r_\Lambda) \leq \sup(r_\Lambda) \leq \rho/2 $.

\medskip
\noindent
{\sc $\mathbb{Z}$--indexed versions of $q^s,q^u$:} Let $ x \in \nuh$. For each sequence $ \Tau= \{t_n\}_{n \in \mathbb{Z}} $ of real numbers with $ \frac{1}{2} \inf(r_\Lambda) \leq t_{n+1} - t_n \leq 2 \sup(r_\Lambda)$, define:
\begin{align*}
p^s(x, \Tau, n) &:= \ve \inf \{ e^{\ve (t_m - t_n)} Q(\varphi^{t_m}(x)) : m \geq n \},\\
p^u(x, \Tau, n) &:= \ve \inf \{ e^{\ve (t_n - t_m)} Q(\varphi^{t_m}(x)) : m \leq n \}.
\end{align*}

\medskip
Note that $p^{s/u}(x, \Tau, n) \geq q^{s/u}(\varphi^{t_n}(x))$. Given that the choice of $\Tau$ will be clear from the context, we will simplify the notation $p^{s/u}(x, \Tau, n)$ writing $p^{s/u}(\varphi^{t_n}(x))$. 
The next proposition shows that the values $p^{s/u}(\varphi^{t_n}(x))$ are not very sensitive to the choice of $\Tau$.

\begin{proposition}\label{Prop-Z-par}
    The following holds for all $x \in \nuh^\#$ and $\Tau = \{t_n\}_{n \in \mathbb{Z}}$ with $\frac{1}{2} \inf(r_\Lambda) \leq t_{n+1} - t_n \leq 2 \sup(r_\Lambda)$.
\begin{enumerate}[{\rm (1)}]
    \item {\sc Robustness:} Let $\kH := \ve \rho + \frac{288 \rho}{\beta}$. For all $n \in \Z$ and $t \in [t_n, t_{n+1}]$, it holds:
    $$
    \frac{p^{s/u}(\varphi^{t_n}(x))}{q^{s/u}(\varphi^t(x))} = e^{\pm \kH}.
   $$
    \item {\sc Greedy algorithm:} For all $n \in \Z$ it holds:
    \begin{align*}
    p^s(\varphi^{t_n}(x)) &= \min \left\{ e^{\ve (t_{n+1} - t_n)} p^s(\varphi^{t_{n+1}}(x)), \ve Q(\varphi^{t_n}(x)) \right\}\\
    p^u(\varphi^{t_n}(x)) &= \min \left\{ e^{\ve (t_n - t_{n-1})} p^u(\varphi^{t_{n-1}}(x)), \ve Q(\varphi^{t_n}(x)) \right\}.
    \end{align*}
    In particular:
    \begin{align*}
    \ve Q(\varphi^{t_n}(x)) &\geq p^s(\varphi^{t_n}(x)) \geq e^{-\ve (t_n - t_m)} p^s(\varphi^{t_m}(x)),\ \forall n \geq m,\\
    \ve Q(\varphi^{t_n}(x)) &\geq p^u(\varphi^{t_n}(x)) \geq e^{-\ve (t_m - t_n)} p^u(\varphi^{t_m}(x)),\ \forall m \geq n.
    \end{align*}
    \item {\sc Maximality:} $p^s(\varphi^{t_n}(x)) = \ve Q(\varphi^{t_n}(x))$ for infinitely many $n>0$, and $p^u(\varphi^{t_n}(x)) = \ve Q(\varphi^{t_n}(x))$ for infinitely many $n < 0$.
\end{enumerate}
\end{proposition}

The proof is the same as \cite[Proposition 3.6]{BCL23}.

\subsection{Pesin charts $\Psi_x$} 

For $x \in \nuh$ and $r>0$, let $R[x,r]= B^{d_s(x)}(r) \times B^{d_u(x)}(r) \subset \R^{d_s(x)} \times \R^{d_u(x)}=\R^d$ be the product of the $d_s(x)$ and $d_u(x)$ dimensional balls of radii $r$. As the dependence on $x$ is only on the dimensions $d_{s/u}(x)$ and the arguments will usually fix these values, we will simply write $R[r]$. We define Pesin charts for $x\in \Lambda \cap \nuh$.

\medskip
\noindent
{\sc Pesin chart at $x$:} It is the map $\Psi:R[\kr]\to \wh{\Lambda}$ defined by $\Psi_x:=\exp{x} \circ C(x)$.

\medskip
As in \cite{BCL23}, the center $x$ of the Pesin chart $\Psi_x$ always belongs to the reference section $\Lambda$, although its image can intersect $\wh{\Lambda}\setminus\Lambda$. The next lemma collects the basic properties of $\Psi_x$. This result is the higher dimensional version of \cite[Lemma 3.7]{BCL23}, and the proof is the same.

\begin{lemma}\label{lemma3.7}
    For all $x \in \Lambda\cap\nuh$, the Pesin chart $\Psi_x$ is a diffeomorphism onto its image and satisfies:
    \begin{enumerate}[{\rm (1)}]
        \item $\Psi_x$ is $2$--Lipschitz and $\Psi^{-1}_x$ is $2\|C(x)^{-1}\|$--Lipschitz.
        \item $\|\wt{d(\Psi_x)}_{v_1}-\wt{d(\Psi_x)}_{v_2}\|\leq \mathfrak{K}\|v_1-v_2\|$ for all $v_1,v_2 \in R[\kr]$.
    \end{enumerate}
\end{lemma}

\subsection{ Holonomy maps $g^{\pm}$ in Pesin charts} 

At the scale $Q(x)$, we can control $g_x^{\pm}$ in Pesin charts, representing it as a small perturbation of a hyperbolic linear map.

\begin{theorem}\label{Thm-non-linear-Pesin}
The following holds for all $\ve>0$ small enough. For all $x\in\Lambda\cap\nuh$
the map $f_x^+:=\Psi_{f(x)}^{-1}\circ g_x^+\circ\Psi_x$ is well-defined on
$R[10Q(x)]$ and satisfies:
\begin{enumerate}[{\rm (1)}]
\item $d(f_x^+)_0=C(f(x))^{-1}\circ \Phi^{r_\Lambda(x)}\circ C(x)$ and $e^{-4 \rho} <m(d(f_x^+)_0)\leq \|d(f_x^+)_0\|< e^{4 \rho}$.
\item $f_x^+=\begin{bmatrix} D_s(x) & 0 \\ 0 & D_u(x) \end{bmatrix}+H$ where:
\begin{enumerate}[{\rm (a)}]
\item $e^{-4 \rho}< \| D_s(x) \|,\|D_u(x)^{-1}\| <e^{-\chi r_\Lambda(x)}$, cf. Lemma \ref{Lemma-linear-reduction}{\rm (2).}
\item $H(0)=0$ and $dH_0=0$.
\item $\|H\|_{C^{1+\frac{\beta}{2}}}<\ve$.
\end{enumerate}
\end{enumerate}
A similar statement holds for $f_x^-:=\Psi_x^{-1}\circ g_{f(x)}^-\circ \Psi_{f(x)}$.
\end{theorem}

Above, $m(T)$ denotes the co-norm of linear transformation $T$. The proof of Theorem \ref{Thm-non-linear-Pesin} is the same as \cite[Theorem 3.8]{BCL23}. The details are in Appendix \ref{Appendix-proofs}.

\subsection{The overlap condition}\label{section-overlap}

In this section, we consider a notion, called {\em $\ve$--overlap}, that allows to control the change of coordinates from $\Psi_x$ to $\Psi_y$ when $x,y$
are ``sufficiently close''. This notion was introduced in \cite{Sarig-JAMS} for surface diffeomorphisms.  
We will make extensive use of Pesin
charts with different domains.

\medskip
\noindent
{\sc Pesin chart $\Psi_x^\eta$:} It is restriction of $\Psi_x$ to $R[\eta]$, where $0<\eta\leq Q(x)$.
\medskip

Recall that $d$ is the distance on $\widehat \Lambda$
associated to the induced Riemannian metric.

\medskip

\noindent
{\sc $\ve$--overlap:} We say that two Pesin charts $\Psi_{x_1}^{\eta_1},\Psi_{x_2}^{\eta_2}$
{\em $\ve$--overlap} if $\tfrac{\eta_1}{\eta_2}=e^{\pm\ve}$, $\operatorname{dim}(N^s_{x_1})=\operatorname{dim}(N^s_{x_2})$, and
$d(x_1,x_2)+\|\widetilde{C(x_1)}-\widetilde{C(x_2)}\|<(\eta_1\eta_2)^4$.
In particular,  $x_1,x_2$ belong to the same local connected component of $\wh\Lambda$.
When this happens, we write $\Psi_{x_1}^{\eta_1}\overset{\ve}{\approx}\Psi_{x_2}^{\eta_2}$.

\medskip
In other words, for us ``sufficiently close'' means that both $x_1,x_2$ and
$C(x_1),C(x_2)$ are very close, the invariant splittings have the same dimension and the domains considered for Pesin charts are almost the same. The lemma below is an auxiliary technical result.

\begin{lemma}
The following holds for $\ve>0$ small. If
$\Psi_{x_1}^{\eta_1}\overset{\ve}{\approx}\Psi_{x_2}^{\eta_2}$, then
$$\Psi_{x_i}(R[10Q(x_i)])\subset B_{x_1}\cap B_{x_2}\ \text{ for } i=1,2.$$
\end{lemma}

In particular, we can apply (Exp3) for either $x_1$ or $x_2$. The proof of this result is the same as \cite[Lemma 3.9]{BCL23}. Next, we compare Pesin charts when an $\ve$--overlap holds.

\begin{proposition}\label{Lemma-overlap}
The following holds for $\ve>0$ small.
If $\Psi_{x_1}^{\eta_1}\overset{\ve}{\approx}\Psi_{x_2}^{\eta_2}$ and $C_i = \widetilde{C(x_i)}$ for $i=1,2$, then:
\begin{enumerate}[{\rm (1)}]
\item {\sc Control of $C^{-1}$:} $\| C_1^{-1} - C_2^{-1} \| < (\eta_1 \eta_2)^3$ and $\frac{\|C_1^{-1} \|}{\| C_2^{-1} \|} = e^{\pm (\eta_1 \eta_2)^3}$. 
\item {\sc Control of $Q$:} $\frac{Q(x_1)}{Q(x_2)} = e^{(\eta_1 \eta_2)^2}$.
\item {\sc Overlap:} $\Psi_{x_i}(R[e^{-2\ve}\eta_i])\subset \Psi_{x_j}(R[\eta_j])$ for $i,j=1,2$.
\item {\sc Change of coordinates:} For $i,j=1,2$, the map $\Psi_{x_i}^{-1}\circ\Psi_{x_j}$
is well-defined in $R[\mathfrak r]$,
and $\|\Psi_{x_i}^{-1}\circ\Psi_{x_j}-{\rm Id}\|_{C^2}<\ve(\eta_1\eta_2)^2$
where the norm is taken in $R[\mathfrak r]$.
\end{enumerate}
\end{proposition}

This proposition, first proved by Sarig for surface diffeomorphisms \cite{Sarig-JAMS}, in the above form is motivated by \cite[Prop. 4.1]{ALP} and \cite[Prop. 3.10]{BCL23}. We include its proof in Appendix
\ref{Appendix-proofs}.

\subsection{The maps $f_{x,y}^+,f_{x,y}^-$}

Let $x,y\in\Lambda\cap\nuh$ such that $\Psi_{f(x)}^{\eta}\overset{\ve}{\approx}\Psi_y^{\eta'}$.
In this section, we change $\Psi_{f(x)}$ by $\Psi_y$ in the definition of $f_x^+$ and obtain a result
similar to Theorem \ref{Thm-non-linear-Pesin}.

\medskip
\noindent
{\sc The maps $f_{x,y}^+$ and $f_{x,y}^-$:} If $\Psi_{f(x)}^{\eta}\overset{\ve}{\approx}\Psi_y^{\eta'}$,
define the map $f_{x,y}^+:=\Psi_y^{-1}\circ g_x^+\circ \Psi_x$.
If $\Psi_{x}^{\eta}\overset{\ve}{\approx}\Psi_{f^{-1}(y)}^{\eta'}$, define
$f_{x,y}^-:=\Psi_x^{-1}\circ g_y^-\circ \Psi_y$. 

\medskip
As claimed, the next result is a version of Theorem \ref{Thm-non-linear-Pesin} for the maps $f_{x,y}^{\pm}$.

\begin{theorem}\label{Thm-non-linear-Pesin-2}
The following holds for all $\ve>0$ small enough.
If $x,y\in\Lambda\cap\nuh$ and $\Psi_{f(x)}^{\eta}\overset{\ve}{\approx}\Psi_{y}^{\eta'}$, then
$f_{x,y}^+$ is well-defined on $R[10Q(x)]$ and can be written as
$f_{x,y}^+=\begin{bmatrix} D_s & 0 \\ 0 & D_u \end{bmatrix}+H$ where:
\begin{enumerate}[{\rm (1)}]
\item $e^{-4\rho}< \| D_s \|,\|D_u^{-1}\| < e^{- \chi r_\Lambda (x)}$, cf. Lemma \ref{Lemma-linear-reduction}{\rm (2)}.
\item $\|H(0)\|<\ve\eta$, $\|dH_0\|<\ve\eta^{\beta/3}$,
$\Hol{\beta/3}(dH)<\ve$.
\end{enumerate}
If $\Psi_{x}^{\eta}\overset{\ve}{\approx}\Psi_{f^{-1}(y)}^{\eta'}$
then a similar statement holds for $f_{x,y}^-$.
\end{theorem}

\begin{proof}
Note that $f_{x,y}^+=(\Psi_y^{-1}\circ\Psi_{f(x)})\circ f_x^+=:g\circ f_x^+$ is a
small perturbation of $f_x^+$. By Theorem \ref{Thm-non-linear-Pesin},
$$
f_x^+(0)=0,\ \|d(f_x^+)\|_{C^0}<2e^{4\rho},\ \|d(f_x^+)_v-d(f_x^+)_w\|\leq \ve\|v-w\|^{\beta/2},\ \forall v,w\in R[10Q(x)],
$$
where the $C^0$ norm is taken in $R[10Q(x)]$
and, by Proposition \ref{Lemma-overlap}(4),
$$
\|g-{\rm Id}\|<\ve(\eta\eta')^2,\ \|d(g-{\rm Id})\|_{C^0}<\ve(\eta\eta')^2,\ \|dg_v-dg_w\|\leq\ve(\eta\eta')^2\|v-w\|^{\beta/2}
$$
for $v,w\in R[\mathfrak r]$, where the $C^0$ norm is taken in $R[\mathfrak r]$.

We begin showing that $f_{x,y}^+$ is well-defined on $R[10Q(x)]$. For small enough $\ve>0$ we have
$f_x^+(R[10Q(x)])\subset B(0,20 \sqrt{2} e^{4\rho} Q(x))\subset R[\mathfrak r]$
since $20 \sqrt{2} e^{4\rho} Q(x)< 40e^{4\rho}\ve^{6/ \beta}<\mathfrak r$.
By Proposition \ref{Lemma-overlap}(4), $f_{x,y}^+$ is well-defined.
 
We define $D_s=D_s(x), D_u=D_u(x)$ as in Theorem \ref{Thm-non-linear-Pesin}. Part (1) is clear, so we focus on part (2).
We have $\|H(0)\|=\|g(0)\|<\ve(\eta\eta')^2<\ve\eta$
and for $\ve>0$ small enough:
$$
\|dH_0\|= \|dg_0\circ d(f_x^+)_0-d(f_x^+)_0\|\leq \|d(g-{\rm Id})_0\|\|d(f_x^+)_0\|
<\ve(\eta\eta')^2 e^{4\rho} <\ve\eta^{\beta/3}.
$$
Now, since $f_x^+(R[10Q(x)])\subset R[\mathfrak r]$, if $\ve>0$ is small then
for $v,w\in R[10Q(x)]$:
\begin{align*}
&\ \|dH_v-dH_w\|=\|dg_{f_x^+(v)}\circ d(f_x^+)_v-dg_{f_x^+(w)}\circ d(f_x^+)_w\|\\
&\leq \|dg_{f_x^+(v)}-dg_{f_x^+(w)}\|\cdot\|d(f_x^+)_v\|+\|dg_{f_x^+(w)}\|\cdot\|d(f_x^+)_v-d(f_x^+)_w\|\\
&\leq \ve(\eta\eta')^2\|f_x^+(v)-f_x^+(w)\|^{\beta/2}\|d(f_x^+)\|_{C^0}+\ve\|dg\|_{C^0}\|v-w\|^{\beta/2}\\
&\leq \left[\ve(\eta\eta')^2\|d(f_x^+)\|_{C^0}^{1+\beta/2}+ 20 \sqrt{2} \ve\|dg\|_{C^0}Q(x)^{\beta/6}\right]\|v-w\|^{\beta/3}\\
&\leq \left[\ve^{24/\beta}(2e^{4\rho})^{1+\beta/2}+ 40\ve\right]\ve\|v-w\|^{\beta/3}<\ve\|v-w\|^{\beta/3}.
\end{align*}
The proof is now complete.
\end{proof}

\section{Invariant manifolds and shadowing}

In the previous sections, we have fixed $\vf,\chi,\rho,\Lambda,\widehat\Lambda,\theta,\ve$, where
$\rho,\ve$ are small parameters.

\subsection{Pseudo-orbits}\label{ss.pseudo.orbits}\mbox{}

\medskip
\noindent
{\sc $\ve$--double chart:} An {\em $\ve$--double chart} is a pair of Pesin charts
$\Psi_x^{p^s,p^u}=(\Psi_x^{p^s},\Psi_x^{p^u})$ where
$0<p^s,p^u\leq \ve Q(x)$.

\medskip
The purpose of the parameters $p^s/p^u$ is to measure the hyperbolicity at $x$ in the stable/unstable directions.

\medskip
\noindent
{\sc Transition time:}
Given two $\ve$--double charts $v=\Psi_x^{p^s,p^u}$ and $w=\Psi_y^{q^s,q^u}$ such that $\Psi_{f(x)}^{q^s\wedge q^u}\overset{\ve}{\approx}\Psi_y^{q^s\wedge q^u}$ and
$\Psi_{f^{-1}(y)}^{p^s\wedge p^u}\overset{\ve}{\approx}\Psi_x^{p^s\wedge p^u}$, we define the {\em transition time} $T(v,w)$ by
$$
\min\left\{\min\{T^+(z):z\in \Psi_x(R[\tfrac{1}{20}(p^s\wedge p^u)])\},
\\ \min\{-T^-(z):z\in \Psi_y(R[\tfrac{1}{20}(q^s\wedge q^u)])\}\right\},
$$
where $T^+:B_x\to\R$ and $T^-:B_y\to\R$ are the $C^{1+\beta}$ functions satisfying
$g_x^+=\vf^{T^+}$, $g_{y}^{-}=\vf^{T^-}$ with $T^+(x)=r_\Lambda(x)$ and $T^-(y)=-r_\Lambda(f^{-1}(y))$. 

\medskip
\noindent
{\sc Edge $v\overset{\ve}{\rightarrow}w$:} Given two $\ve$--double charts $v=\Psi_x^{p^s,p^u}$ and
$w=\Psi_y^{q^s,q^u}$, we draw an {\em edge} from $v$ to $w$ if the following conditions are satisfied:
\medskip
\begin{enumerate}[iii\,]
\item[(GPO1)] $\Psi_{f(x)}^{q^s\wedge q^u}\overset{\ve}{\approx}\Psi_y^{q^s\wedge q^u}$ and
$\Psi_{f^{-1}(y)}^{p^s\wedge p^u}\overset{\ve}{\approx}\Psi_x^{p^s\wedge p^u}$.
\smallskip
\item[(GPO2)] The estimates below hold:
\begin{align}
\label{gpo2-a}&
e^{-\ve p^s}\min\{e^{\ve T(v,w)}q^s,e^{-\ve}\ve Q(x)\}\leq p^s\leq \min\{e^{\ve T(v,w)}q^s,\ve Q(x)\}\\
\label{gpo2-b}&
e^{-\ve q^u}\min\{e^{\ve T(v,w)}p^u,e^{-\ve}\ve Q(y)\}\leq q^u\leq \min\{e^{\ve T(v,w)}p^u,\ve Q(y)\}.
\end{align}
\end{enumerate}
We denote the edge by $v\overset{\ve}{\rightarrow}w$.

\begin{remark}\label{rmk-time}
If $v\overset{\ve}{\rightarrow} w$ then by Theorem \ref{Thm-non-linear-Pesin-2} 
we have
$$
g_y^{-}(\Psi_y(R[\tfrac{1}{20}(q^s\wedge q^u)]))\subset \Psi_x(R[\tfrac{1}{15}(p^s\wedge p^u)])
$$
and so $T(v,w)=T^+(z)$ for some $z\in \Psi_x(R[\tfrac{1}{15}(p^s\wedge p^u)])$.
In particular, $T(v,w)\leq \rho$.
\end{remark}

\medskip
Condition (GPO1) is a ``nearest neighbor condition''. Condition (GPO2) is a greedy recursion which, among other things, implies that the measurement of hyperbolicity at the $\ve$--double charts $v$ and $w$ are ``as large as possible''.
At the moment, these explanations might seem vague, but they will make sense in the proof of Theorem \ref{Thm-coarse-graining} (Coarse graining) and Theorem \ref{Thm-inverse} (Inverse theorem).

\medskip
\noindent
{\sc $\ve$--generalized pseudo-orbit ($\ve$--gpo):} An {\em $\ve$--generalized pseudo-orbit ($\ve$--gpo)}
is a sequence $\un{v}=\{v_n\}_{n\in\Z}$ of $\ve$--double charts
such that $v_n\overset{\ve}{\rightarrow}v_{n+1}$ for all $n\in\Z$. We say that
$\un v$ is {\em regular} if there are $v,w$ such that $v_n=v$ for infinitely many $n>0$ and
$v_n=w$ for infinitely many $n<0$.

\medskip
\noindent
{\sc Positive and negative $\ve$--gpo:}
A {\em positive $\ve$--gpo} is a sequence $\un{v}^+=\{v_n\}_{n\geq 0}$ of $\ve$--double charts
such that $v_n\overset{\ve}{\rightarrow}v_{n+1}$ for all $n\geq 0$. A {\em negative $\ve$--gpo}
is a sequence $\un{v}^-=\{v_n\}_{n\leq 0}$ of $\ve$--double charts
such that $v_n\overset{\ve}{\rightarrow}v_{n+1}$ for all $n\leq -1$.

\begin{lemma}\label{Lemma-minimum}
If $v=\Psi_x^{p^s,p^u},w=\Psi_y^{q^s,q^u}$ are $\ve$--double charts
satisfying {\rm (GPO2)} then $\tfrac{p^s\wedge p^u}{q^s\wedge q^u}=e^{\pm 2\ve}$.
\end{lemma}

The statement and proof are the same as \cite[Lemma 4.2]{BCL23}.

\subsection{Graph transforms and invariant manifolds}\label{ss.graph.transform}

Let $v=\Psi_x^{p^s,p^u}$ be an $\ve$--double chart. Recall
that $B^{d_{s/u}(x)}(r)$ denotes the open ball of center $0\in \R^{d_{s/u}(x)}$ and radius $r$. Let $B^{d_{s/u}(x)}[r]$ denote the respective closed ball.

\medskip
\noindent
{\sc Admissible manifolds:} An {\em $s$--admissible manifold at $v$} is a set
of the form 
$$
V=\Psi_x\{(t,F(t)): t \in B^{d_s(x)} [p^s]\}
$$ where $F:B^{d_s(x)} [p^s] \to\R^{d_u(x)}$ is a $C^{1+\beta/3}$ function
such that:
\begin{enumerate}
\item[(AM1)] $\|F(0)\|\leq 10^{-3}(p^s\wedge p^u)$.
\item[(AM2)] $\|dF_0\|\leq \tfrac{1}{2}(p^s\wedge p^u)^{\beta/3}$.
\item[(AM3)] $\|dF\|_{C^0}+\Hol{\beta/3}(dF)\leq\tfrac{1}{2}$ where the norms are taken in $B^{d_s(x)}[p^s]$.
\end{enumerate}
The function $F$ is called the \emph{representing function} of $V$.
Similarly, a {\em $u$--admissible manifold at $v$} is a set
of the form $\Psi_x\{(G(t),t): t \in B^{d_u (x)} [p^u] \}$ where $G:B^{d_u (x)} [p^u] \to\R^{d_s(x)}$ is a $C^{1+\beta/3}$ function
satisfying (AM1)--(AM3), with norms taken in $B^{d_u (x)} [p^u]$.

\medskip
If $V_1,V_2$ are two $s$--admissible manifolds at $v$, with representing functions
$F_1,F_2$, for $i\geq 0$ define $ d_{C^i}(V_1,V_2):=\|F_1-F_2\|_{C^i}$ where the
norm is taken in $B^{d_s (x)} [p^s]$. The same applies to $u$--admissible manifolds.

In the sequel, we introduce {\em graph transforms}, which is the tool we will use to construct invariant 
manifolds. The proofs are adaptations of \cite{Sarig-JAMS, ALP,BCL23}, hence we will discuss them when necessary. The main result of this section, Theorem \ref{Thm-stable-manifolds},
lists the basic properties of invariant manifolds. 
Given an $\ve$--double chart $v=\Psi_x^{p^s,p^u}$, we denote
by $\mathfs M^s(v)$ the set of its $s$--admissible manifolds.

\medskip
\noindent
{\sc The graph transform $\mathfs F_{v,w}^s$:} To any edge
$v\overset{\ve}{\rightarrow}w$ between $\ve$--double charts $v=\Psi_x^{p^s,p^u}$ and $w=\Psi_y^{q^s,q^u}$,
we associate the {\em graph transform} $\mathfs F_{v,w}^s:\mathfs M^s(w)\to\mathfs M^s(v)$
as being the map that sends an $s$--admissible manifold at $w$ with representing function $F:B^{d_s (x)} [q^s]\to\R^{d_u(x)}$ to the unique
$s$--admissible  manifold at $v$ with representing function $G:B^{d_s (x)} [p^s]\to\R^{d_u(x)}$ such that
$\{(t,G(t)): t \in B^{d_s (x)} [p^s]\}\subset f_{x,y}^-\{(t,F(t)): t \in B^{d_s (x)} [q^s]\}$.

\begin{lemma}\label{Prop-graph-transform}
If $\ve>0$ is small enough, then $\mathfs F_{v,w}^s$ is well-defined for any edge $v\overset{\ve}{\rightarrow}w$.
Furthermore, if
$V_1,V_2\in \mathfs M^s(w)$ then:
\begin{enumerate}[{\rm (1)}]
\item $ d_{C^0}(\mathfs F_{v,w}^s(V_1),\mathfs F_{v,w}^s(V_2))\leq e^{-\chi\inf(r_\Lambda)/2} d_{C^0}(V_1,V_2)$.
\item $ d_{C^1}(\mathfs F_{v,w}^s(V_1),\mathfs F_{v,w}^s(V_2))\leq
e^{-\chi\inf(r_\Lambda)/2}( d_{C^1}(V_1,V_2)+ d_{C^0}(V_1,V_2)^{\beta/3})$.
\end{enumerate}
\end{lemma}

\medskip
\noindent
{\sc The stable manifold of a positive $\ve$--gpo:} The {\em stable manifold} of a positive 
$\ve$--gpo $\un v^+=\{v_n\}_{n\geq 0}$ is 
$$
V^s[\un v^+]:=\lim_{n\to\infty}(\mathfs F_{v_0,v_1}^s\circ\cdots\circ\mathfs F_{v_{n-2},v_{n-1}}^s\circ\mathfs F_{v_{n-1},v_n}^s)(V_n)
$$
for some (any) choice  $(V_n)_{n\geq 0}$ with $V_n\in \mathfs M^s(v_n)$. The convergence occurs
in the $C^1$ topology.

\medskip

Similarly, we introduce the {\em unstable manifold} $V^u[\un v^-]$ of a negative $\ve$--gpo. The proof of the good definition of $V^{s/u}[\un v^{\pm}]$ as well as of Lemma \ref{Prop-graph-transform} follows from \cite[Appendix]{ALP}, taking the parameters $p = q^s,  \widetilde{p} = p^s, \wt \eta = p^s \wedge p^u,  \eta = q^s \wedge q^u$.

We now state the basic properties of $V^{s/u}[\un v^\pm]$.
We need to introduce some notation. 

\medskip
\noindent
{\sc Maps $G_n$ and $\tau_{n}$:}  
Given ${\un v}^+=\{\Psi_{x_n}^{p^s_n,p^u_n}\}_{n\geq 0}$
a positive $\ve$--gpo, write $g_{x_n}^+=\vf^{T_n}$ where $T_n:B_{x_n}\to\R$
is a $C^{1+\beta}$ function with $T_n(x_n)=r_\Lambda(x_n)$. Let $G_0={\rm Id}$ and
$G_n:=g_{x_{n-1}}^+\circ\cdots\circ g_{x_0}^+$, $n\geq 1$. For $n\geq 0$, define $\tau_n:V^s[\un v^+]\to\R$ by
$$
\tau_n(x):=\sum_{k=0}^{n-1}T_k(G_k(x)),
$$
equal to the flow displacement of $x$
under the maps $g_{x_0}^+,g_{x_1}^+,\ldots, g_{x_{n-1}}^+$.

\medskip
Although $G_n$ and $\tau_n$ depend on $\un v^+$, we will only mention this dependence when more than one positive $\ve$--gpo is considered, in which case we will write $G_{\un v^+,n}$ and $\tau_{\un v^+,n}$.

\begin{theorem}[Stable manifold theorem]\label{Thm-stable-manifolds}
The following holds for all $\ve>0$ small enough.
Let ${\un v}^+=\{v_n\}_{n\geq 0}=\{\Psi_{x_n}^{p^s_n,p^u_n}\}_{n\geq 0}$
be a positive $\ve$--gpo.

\begin{enumerate}[{\rm (1)}]
\item {\sc Admissibility:} The set $V^s[{\un v}^+]$ is an $s$--admissible manifold at $v_0$, equal to
$$
V^s[{\un v}^+]=\{x\in \Psi_{x_0}(R[p^s_0]):(g^+_{x_{n-1}}\circ\dots\circ g^+_{x_0})(x)\in \Psi_{x_n}(R[10Q(x_n)]),\,\forall n\geq 0\}.
$$

\item {\sc Invariance:} $g_{x_0}^+(V^s[\{v_n\}_{n\geq 0}])\subset V^s[\{v_n\}_{n\geq 1}]$.
\smallskip
\item {\sc Hyperbolicity:} For all $y,y'$ in $V^s[\un v^+]$ and all $n\geq 0$: 
$$d(G_n(y),G_n(y'))
\leq 2 d(\Psi^{-1}_{x_0}(y),\Psi^{-1}_{x_0}(y'))\;e^{-\frac{2\chi}{3}\tau_n(y)}.$$
For any unit vector $w$ tangent to $V^s[{\un v}^+]$ at a point $y$ and all $n\geq 0$:
\begin{align*}
\|d(G_n)_yw\|&\leq 8 \| C(x_0)^{-1} \|\; e^{-\frac{2\chi}{3} \tau_n(y)}\ \text{ and}\\
\|d(G_{-n})_yw\|&\geq 
\tfrac{1}{8}(p^s_0\wedge p^u_0)^{\frac{\beta}{12}}\; e^{\frac{2\chi}{3}(-\tau_{-n}(y))-\frac{\beta\ve}{6}n}.
\end{align*}
\smallskip
\item {\sc H\"older property:} The map $\un v^+\mapsto V^s[\un v^+]$ is H\"older continuous:\\
There are $K>0$ and $\theta\in(0,1)$ such that for all $N\geq 0$, if $\un v^+,\un w^+$ are positive $\ve$--gpo's
with $v_n=w_n$ for $n=0,\ldots,N$
then $ d_{C^1}(V^s[\un v^+],V^s[\un w^+])\leq K\theta^N$.
\end{enumerate}
The submanifold $V^s[\un v^+]$ is called \emph{local stable manifold} of $\un v^+$.
A similar statement holds for the unstable manifold $V^u[\un v^-]$ of
a negative $\varepsilon$-gpo $\un v^-$.
\end{theorem}

\medskip

The above theorem is a strengthening of the Pesin stable manifold theorem \cite{Pesin-Izvestia-1976}.
Its statement is similar to \cite{Sarig-JAMS} and \cite{O18}, and its proof is performed exactly as in
\cite[Prop. 3.12 and 4.4]{O18}, noting that in Pesin charts the composition
$g^+_{x_{n-1}}\circ\dots\circ g^+_{x_0}$ is represented by $f_{x_{n-1},x_n}^+\circ\cdots\circ f_{x_0,x_1}^+$.
Since each $f_{x_i,x_{i+1}}^+$ is hyperbolic (Theorem \ref{Thm-non-linear-Pesin-2})
and each $\mathfs F_{v_i,v_{i+1}}^s$ is contracting (Lemma \ref{Prop-graph-transform}), the proof follows.
We note that the second estimate of part (3) is proved as in \cite[Prop. 6.5]{Sarig-JAMS},
see also the proof of \cite[Prop. 4.11]{ALP}.

\begin{remark}
Let us be more precise on how to obtain part (3). Proceeding as in \cite[Prop. 4.4]{O18} and applying Theorem \ref{Thm-non-linear-Pesin-2}, we get that, in Pesin charts, the distance of the images of $y,y'$ under $G_1$ contracts at least by
$e^{-\chi r_\Lambda(x_0)}+O(\ve)<e^{-\frac{3\chi}{4}r_\Lambda(x_0)}$. Since $r_\Lambda(x_0)=\tau_1(x_0)$ and $\tau_1$ is 1--Lipschitz (Lemma \ref{lemma-local-coord}), we have $|r_\Lambda(x_0)-\tau_1(y)|\leq d(x_0,y)\ll \ve$ and so $\left|\frac{r_\Lambda(x_0)}{\tau_1(y)}-1\right|<\tfrac{\ve}{\inf(r_\Lambda)}$ which is smaller than $1/9$ for small $\ve$. Therefore $e^{-\frac{3\chi}{4}r_\Lambda(x_0)}<e^{-\frac{2\chi}{3}\tau_1(y)}$.
\end{remark}

\subsection{Shadowing}
We say that an $\ve$--gpo $\{\Psi_{x_n}^{p^s_n,p^u_n}\}_{n\in\Z}$ {\em shadows}
a point $x\in \widehat \Lambda$ if:
$$
(g^+_{x_{n-1}}\circ\dots\circ g^+_{x_0})(x)\in \Psi_{x_n}(R[p^s_n\wedge p^u_n])\text{ for all $n\geq 0$},$$
$$(g^-_{x_{n+1}}\circ\dots\circ g^-_{x_0})(x)\in \Psi_{x_n}(R[p^s_n\wedge p^u_n])\text{ for all $n\leq 0$}.
$$

\medskip
An important property is the following result.

\begin{proposition}\label{Prop-shadowing}
If $\ve$ is small enough, then every $\ve$--gpo $\un v$ shadows a unique point
$\{x\}=V^s[\un v]\cap V^u[\un v]$.
\end{proposition}

\medskip
The proof uses the following basic property of admissible manifolds.

\begin{lemma}\label{Lemma-admissible-manifolds}
The following holds for all $\ve>0$ small enough. If $v=\Psi_x^{p^s,p^u}$ is an $\ve$--double chart,
then for every $V^{s/u}\in\mathfs M^{s/u}(v)$ it holds that $V^s$ and $V^u$ intersect at a single point
$P\in \Psi_x\left(R\left[\tfrac{1}{500}(p^s\wedge p^u)\right]\right)$.
\end{lemma}

This statement and its proof are similar to \cite[Lemma 4.7]{BCL23}.

\begin{proof}[Proof of Proposition~\ref{Prop-shadowing}]
We give a sketch of proof, since it is the same as \cite[Theorem 4.2]{Sarig-JAMS}. Let $\un v=\{v_n\}_{n\in\Z}=\{\Psi_{x_n}^{p^s_n,p^u_n}\}_{n\in\Z}$ be an $\ve$--gpo.
\begin{enumerate}[$\circ$]
\item By Theorem \ref{Thm-stable-manifolds}(1),
any point shadowed by $\un v$ must lie in $V^s[\{v_n\}_{n\geq 0}]\cap V^u[\{v_n\}_{n\leq 0}]$.
By Lemma \ref{Lemma-admissible-manifolds}, this intersection is a single point $\{x\}$.
We claim that $\un v$ shadows $x$. 
\item The definition of shadowing is equivalent to the following weaker definition: 
$\un v$ shadows $x$ if and only if
$$(g^+_{x_{n-1}}\circ\dots\circ g^+_{x_0})(x)\in \Psi_{x_n}(R[10Q(x_n)])\text{ for all $n\geq 0$},$$
$$(g^-_{x_{n+1}}\circ\dots\circ g^-_{x_0})(x)\in \Psi_{x_n}(R[10Q(x_n)])\text{ for all $n\leq 0$}.$$
\item By Theorem \ref{Thm-stable-manifolds}(2), if $n\geq 0$ then 
$(g^+_{x_{n-1}}\circ\dots\circ g^+_{x_0})(x)\in V^s[\{v_{n+k}\}_{k\geq 0}]\subset \Psi_{x_n}(R[10Q(x_n)])$, 
and $(g^-_{x_{n+1}}\circ\dots\circ g^-_{x_0})(x)\in V^u[\{v_{n+k}\}_{k\leq 0}]\subset \Psi_{x_n}(R[10Q(x_n)])$, and so the weaker definition of shadowing holds.
\end{enumerate}
This concludes the proof.
\end{proof}

\subsection{Additional properties}

So far, we have defined stable/unstable manifolds using holonomy maps. It is important to give them an intrinsic characterization in terms of the flow. This is the content of the next result.

\begin{proposition}\label{Prop-center-stable}
The following holds for all $\ve>0$ small enough. Let $\un v=\{v_n\}_{n\geq 0}$ be a positive $\ve$--gpo
with $v_0=\Psi_x^{p^s,p^u}$, and let $F:B^{d_s (x)}[p^s] \to\R^{d_u(x)}$ be the representing function of $V^s=V^s[\un v^+]$.
Then there exists a function $\Delta:B^{d_s (x)}[p^s] \to\R$ with $\Delta(0)=0$ such that
the set $\wt V^s:=\{\vf^{\Delta(w)}[\Psi_x(w,F(w))]: w \in B^{d_s (x)}[p^s]\}$ satisfies
$d(\vf^t(\wt y),\vf^t(\wt z))\leq e^{-\frac{\chi\inf(r_\Lambda)}{4\sup(r_\Lambda)}t}$
for all $\wt y,\wt z\in \wt V^s$ and all $t\geq 0$.
An analogous statement holds for negative $\ve$--gpo's. 
\end{proposition}

In other words, $V^s$ ``lifts'' to a set $\wt V^s$ that contracts in the future {\em under the flow}. The statement above and its proof are similar to \cite[Proposition 4.8]{BCL23}. We include the proof in Appendix \ref{Appendix-proofs}. The choice of $\Delta(0)=0$ above is arbitrary: given $y=\Psi_x(t,F(t))\in V^s$,
we can choose $\Delta$ so that $\Delta(y)=0$. The resulting set
$\widetilde V^s\ni y$ also satisfies Proposition \ref{Prop-center-stable}. 

We finish this section proving another property about invariant manifolds, whose proof can also be found in Appendix \ref{Appendix-proofs}.

\begin{proposition}\label{Prop-disjointness}
The following holds for $\ve>0$ small enough.
Let $\un v^+=\{v_n\}_{n\geq 0}$ and $\un w=\{w_n\}_{n\geq 0}$ be positive $\ve$--gpo's,
with $v_0=\Psi_x^{p^s,p^u}$ and $w_0=\Psi_x^{q^s,q^u}$. Then 
either $V^s[\un v^+],V^s[\un w^+]$ are disjoint or one contains the other.
\end{proposition}

\section{First coding}\label{section-coarse-graining}

In the previous sections, we have fixed $\vf,\chi,\rho,\Lambda,\widehat\Lambda,\theta,\ve$ such that
$\ve\ll \rho\ll 1$, and we have constructed invariant manifolds for $\ve$--gpo's. 
We also defined shadowing. In this section, we:
\begin{enumerate}[$\circ$]
\item Choose a countable family of $\ve$--double 
charts whose $\ve$--gpo's they define shadow the whole set $\Lambda\cap\nuh^\#$.
\item Introduce a first coding, that is usually infinite-to-one (and so {\em does not} satisfy the Main Theorem). 
\end{enumerate}
Once this is performed, the last sections will show how to pass from this first coding to a coding satisfying the Main Theorem.

\subsection{Coarse graining}

This section constitutes an important contribution first developed in \cite{BCL23}, that cannot be obtained
using the methods of \cite{Sarig-JAMS,Lima-Sarig,Lima-Matheus}. 
Broadly speaking, conditions (GPO1)--(GPO2) considered in these latter works are not flexible enough to shadow all orbits of $\Lambda\cap\nuh^\#$ in an ``efficient'' way. As introduced in \cite{BCL23}, one approach to bypass this difficulty is to consider more flexible versions for (GPO1)--(GPO2). Gladly, the methods of \cite{BCL23} in this regard can be reproduced ipsis literis in higher dimension. Since this discussion is a main step in the proof, we have decided to include it in full detail. The final result we will obtain in this section is the following.

\begin{theorem}[Coarse graining]\label{Thm-coarse-graining}
For all $0<\ve\ll \rho\ll 1$, there exists a countable family $\mathfs A$ of $\ve$--double charts
with the following properties:
\begin{enumerate}[{\rm (1)}]
\item {\sc Discreteness:} For all $t>0$, the set $\{\Psi_x^{p^s,p^u}\in\mathfs A:p^s,p^u>t\}$ is finite.
\item {\sc Sufficiency:} If $x\in\Lambda\cap\nuh^\#$ then there is a regular $\ve$--gpo
$\un v\in{\mathfs A}^{\Z}$ that shadows $x$.
\item {\sc Relevance:} For each $v\in \mathfs A$, $\exists\un{v}\in\mathfs A^\Z$ an $\ve$--gpo
with $v_0=v$ that shadows a point in $\Lambda\cap\nuh^\#$.
\end{enumerate}
\end{theorem}

Recall that $\un v=\{v_n\}_{n\in\Z}$ is regular if there are $v,w$ such that
$v_n=v$ for infinitely many $n>0$ and $v_n=w$ for infinitely many $n<0$.
According to Proposition \ref{Prop-adaptedness} and part (2) above,
the regular $\ve$--gpo's in $\mathfs A$ shadow almost every point with respect to every
$\chi$--hyperbolic measure.

\begin{proof}
The proof, which is essentially the same as \cite[Theorem 5.1]{BCL23}, follows a similar strategy of \cite{Sarig-JAMS,Lima-Sarig}
but the implementation is significantly harder since, as already mentioned, the definition of edge considered here is more complicated.

\medskip
Let $\N_0=\N\cup\{0\}$, and let $X:=\Lambda^3\times {\rm GL}(d,\R)^3\times (0,1] \times \{0,1,\ldots,d\}$.
For each $x\in\Lambda\cap\nuh^\#$, let
$\Gamma(x)=(\un x,\un C,\un Q, d)\in X$ with
\begin{align*}
\un x=(f^{-1}(x),x,f(x)),\ \un C=(C(f^{-1}(x)),C(x),C(f(x))),\ \un Q=(Q(x),q(x)), \ d_s = d_s(x).
\end{align*}
Let $Y=\{\Gamma(x):x\in\Lambda\cap\nuh^\#\}$. We want to find a countable dense subset
of $Y$. Since the maps $x\mapsto C(x),Q(x),q(x),d_s(x)$ are usually just measurable,
we apply a precompactness argument.
For each $\un{\ell}=(\ell_{-1},\ell_0,\ell_1)\in\N_0^3$, $m,j\in\N_0$ and $k \in \{0,1,\ldots,d\}$, define
$$
Y_{\un \ell,m,j,k}:=\left\{\Gamma(x)\in Y:
\begin{array}{cl}
e^{\ell_i}\leq\|C(f^i(x))^{-1}\|<e^{\ell_i+1},&-1\leq i\leq 1\\
e^{-m-1}\leq Q(x)< e^{-m}&\\
e^{-j-1}\leq q(x)< e^{-j}&\\
d_s(x) = k
\end{array}
\right\}.
$$

\medskip
\noindent
{\sc Claim 1:} $Y=\bigcup\limits_{\un\ell\in\N_0^3,m,j\in\N_0\atop{0\leq k\leq d}}Y_{\un\ell,m,j,k}$, and each
$Y_{\un\ell,m,j,k}$ is precompact in $X$.

\begin{proof}[Proof of Claim $1$.] 
The first statement is clear, so we show the precompactness.
Fix $\un\ell\in \N_0^3$, $m,j\in\N_0$, $0\leq k\leq d$, and take $\Gamma(x)\in Y_{\un\ell,m,j,k}$. We have
$\un x\in \Lambda^3$, a precompact subset of $M^3$.
For $|i|\leq 1$, $C(f^i(x))$ is an element of ${\rm GL}(d,\R)$ with norm $\leq 1$ and
inverse norm $<e^{\ell_i+1}$, hence it belongs to a compact subset of ${\rm GL}(d,\R)$.
This guarantees that $\un C$ belongs to a compact subset of ${\rm GL}(d,\R)^3$. Also,
$\un Q\in [e^{-m-1},1]\times [e^{-j-1},1]$, the product of compact subintervals of $(0,1]$. Using that the product of precompact sets is precompact, the claim is proved.
\end{proof}

By Claim 1, there exists a finite set
$Z_{\un\ell,m,j,k}\subset Y_{\un\ell,m,j,k}$ such that for every $\Gamma(x)\in Y_{\un\ell,m,j,k}$
there exists $\Gamma(y)\in Z_{\un\ell,m,j,k}$ with:
\begin{enumerate}[{\rm (a)}]
\item $ d(f^i(x),f^i(y))+\|\widetilde{C(f^i(x))}-\widetilde{C(f^i(y))}\|<\tfrac{1}{4}q(x)^8$,
$|i|\leq 1$.
\item $\tfrac{Q(x)}{Q(y)}=e^{\pm \ve/3}$ and $\tfrac{q(x)}{q(y)}=e^{\pm \ve/3}$.
\item $d_s(x) = d_s(y).$
\end{enumerate}
Condition (a) implies that $f^i(x),f^i(y)$ belong to the same disc of $\Lambda$, for $|i|\leq 1$.
For $\eta>0$, let $I_{\ve,\eta}:=\{e^{-\ve^2\eta i}:i\geq 0\}$, a countable discrete set whose ``thickness''
depends on $\eta$.

\medskip
\noindent
{\sc The alphabet $\mathfs A$:} Let $\mathfs A$ be the countable family of $\Psi_x^{p^s,p^u}$ such that:
\begin{enumerate}[i i)]
\item[(CG1)] $\Gamma(x)\in Z_{\un\ell,m,j,k}$ for some
$(\un\ell,m,j,k)\in\N_0^3\times \N_0\times \N_0 \times \{0,1,\ldots,d\}$.
\item[(CG2)] $0<p^s,p^u\leq \ve Q(x)$ and $p^s,p^u\in I_{\ve,q(x)}$.
\item[(CG3)] $e^{-\mathfrak H-1}\leq \tfrac{p^s\wedge p^u}{q(x)}\leq e^{\mathfrak H+1}$, where $\mathfrak H$
is given by Proposition \ref{Prop-Z-par}(1).
\end{enumerate}

\medskip
\noindent
{\em Proof of discreteness.}
Fix $t>0$, and let $\Psi_x^{p^s,p^u}\in\mathfs A$ with $p^s,p^u>t$.
If $\Gamma(x)\in Z_{\un\ell,m,j,k}$ then:
\begin{enumerate}[$\circ$]
\item Finiteness of $\un\ell$: since $e^{\ell_0}\leq \|C(x)^{-1}\|<Q(x)^{-1}<t^{-1}$,
we have $\ell_0<|\log t|$. By Lemma \ref{Lemma-linear-reduction}(3), for $i=\pm 1$ we also have
$$
e^{\ell_i}\leq \|C(f^i(x))^{-1}\|\leq e^{6 \rho} \|C(x)^{-1}\|
<e^{6\rho}t^{-1},
$$
hence $\ell_{-1},\ell_1<6\rho+|\log t|=:T_t$, which is bigger than $|\log t|$.
\item Finiteness of $m$: $e^{-m}>Q(x)>t$, hence $m<|\log t|$.
\item Finiteness of $j$: $e^{-j}>q(x)\geq e^{-\mathfrak H-1}(p^s\wedge p^u) >e^{-\mathfrak H-1} t$,
hence $j\leq |\log t|+\mathfrak H+1$.
\end{enumerate}
The finiteness of $k$ is obvious. Therefore
\begin{align*}
\#\left\{\Gamma(x):
\Psi_x^{p^s,p^u}\in\mathfs A\text{ s.t. }p^s,p^u>t\right\}\leq
\sum_{j=0}^{\lceil |\log t|+\mathfrak H\rceil+1}\sum_{m=0}^{\lceil |\log t|\rceil}
\sum_{-1\leq i\leq 1\atop{\ell_i=0}}^{T_t} \sum_{k=0}^{d} \# Z_{\un\ell,m,j,k}
\end{align*}
is the finite sum of finite terms, hence finite. For each such $\Gamma(x)$,
$$
\#\{(p^s,p^u):\Psi_x^{p^s,p^u}\in\mathfs A\text{ s.t. }p^s,p^u>t\}\leq (\# I_{\ve,q(x)}\cap (t,1))^2$$
is finite, hence 
\begin{align*}
\#\left\{\Psi_x^{p^s,p^u}\in\mathfs A:p^s,p^u>t\right\}\leq 
\sum_{j=0}^{\lceil |\log t|+\mathfrak H\rceil+1}\sum_{m=0}^{\lceil |\log t|\rceil} \sum_{-1\leq i\leq 1\atop{\ell_i=0}}^{T_t}
\sum_{k=0}^{d} \sum_{\Gamma(x)\in Z_{\un\ell,m,j,k}}(\# I_{\ve,q(x)}\cap (t,1))^2
\end{align*}
is the finite sum of finite terms, hence finite. This proves the discreteness of $\mathfs A$.

\medskip
\noindent
{\em Proof of sufficiency.}
Let $x\in\Lambda\cap\nuh^\#$. Take $(\ell_i)_{i\in\Z},(m_i)_{i\in\Z},(j_i)_{i\in\Z}$ and $k$ such that:
\begin{align*}
& \|C(f^i(x))^{-1}\|\in [e^{\ell_i},e^{\ell_i+1}),\,Q(f^i(x))\in [e^{-m_i-1},e^{-m_i}),\\
&q(f^i(x))\in[e^{-j_i-1},e^{-j_i}),\, k=d_s(x).
\end{align*}
For $n\in\Z$, let $\un\ell^{(n)}=(\ell_{n-1},\ell_n,\ell_{n+1})$. Then $\Gamma(f^n(x))\in Y_{\un\ell^{(n)},m_n,j_n,k}$.
Take $\Gamma(x_n)\in Z_{\un\ell^{(n)},m_n,j_n,k}$ such that:
\begin{enumerate}[aaa)]
\item[(${\rm a}_n$)] $ d(f^i(f^n(x)),f^i(x_n))+
\|\widetilde{C(f^i(f^n(x)))}-\widetilde{C(f^i(x_n))}\|<\tfrac{1}{4}q(f^n(x))^8$, $|i|\leq 1$.
\item[(${\rm b}_n$)] $\tfrac{Q(f^n(x))}{Q(x_n)}=e^{\pm\ve/3}$ and $\tfrac{q(f^n(x))}{q(x_n)}=e^{\pm\ve/3}$.
\end{enumerate}
From now on the proof differs from \cite{Sarig-JAMS,Lima-Sarig,Lima-Matheus}, which constitutes the development made in \cite{BCL23}, and that does not rely on the dimension of $M$.
Take $\{t_n\}_{n\in\Z}$ such that $f^n(x)=\vf^{t_n}(x)$, with $t_0=0$ and $g_{x_n}^+[f^n(x)]=\vf^{t_{n+1}-t_n}[f^n(x)]$.
Define
\begin{align*}
P_n^s&:=\ve\inf\{e^{\ve|t_{n+k}-t_n|}Q(x_{n+k}):k\geq 0\},\\
P_n^u&:=\ve\inf\{e^{\ve|t_{n+k}-t_{n}|}Q(x_{n+k}):k\leq 0\}.
\end{align*}
There is no reason for $\Psi_{x_n}^{P^s_n,P^u_n}$ belonging to $\mathfs A$ nor
for $\{\Psi_{x_n}^{P^s_n,P^u_n}\}_{n\in\Z}$ being an $\ve$--gpo. Indeed,
with the above definitions 
one of the inequalities in (GPO2) holds in the reverse direction!
To satisfy (GPO2), we will slightly decrease each $P^s_n,P^u_n$.
Below we show how to make this ``surgery'' for $P^s_n$ (the method for
$P^u_n$ is symmetric).

Start noting the greedy recursion $P^s_n=\min\{e^{\ve(t_{n+1}-t_n)}P^s_{n+1},\ve Q(x_n)\}$
and that
\begin{align*}
&P^s_n=e^{\pm \ve / 3}\ve\inf\{e^{\ve|t_{n+k}-t_n|}Q(f^{n+k}(x)):k\geq 0\}
=e^{\pm \ve /3}p^s(x,\mathcal T,n)=e^{\pm\left(\mathfrak H+\frac{\ve}{3}\right)}q^s(f^n(x)),
\end{align*}
by (${\rm b}_n$) above and Proposition \ref{Prop-Z-par}(1), where $\mathcal T=\{t_n\}_{n\in\Z}$.
We fix $\lambda:=\exp{}[\ve^{1.5}]$
and divide the indices $n\in\Z$ into two groups:
\begin{center}
$n$ is {\em growing} if $P^s_n\geq \lambda P^s_{n+1}$
and it is {\em maximal} otherwise.
\end{center}
Note that $\lambda$ has an exponent 
with order smaller than $\ve$. The definition of growing/maximal indices is motivated by the following: the parameter
$P^s_n$ gives a choice on the size of the stable manifold at $x_n$,
therefore we expect $P^s_n$ to be larger than $P^s_{n+1}$ at least by a multiplicative
factor bigger than $\lambda$, unless it reaches the maximal size $\ve Q(x_n)$.
In the first case the index is growing, and in the second it is maximal.
Assuming that $\ve>0$ is sufficiently small, we note two properties of this notion:
\begin{enumerate}[$\circ$]
\item If $n$ is maximal then $P^s_n=\ve Q(x_n)$: otherwise, the greedy recursion gives
$$
P^s_n=e^{\ve(t_{n+1}-t_n)}P^s_{n+1}\geq e^{\ve\inf(r_\Lambda)}P^s_{n+1}>\lambda P^s_{n+1},$$
which contradicts the assumption that $n$ is maximal.
\item There are infinitely many maximal indices $n>0$,
and infinitely many maximal indices $n<0$: the first claim follows exactly as in the proof of
Proposition \ref{Prop-Z-par}(3)  (remember we are assuming that $x\in\nuh^\#$
and so $\limsup\limits_{n\to+\infty} P^s_n>0$).
The second claim follows from direct computation: if there is $n_0$ such that every $n<n_0$ is growing then
$P^s_n\geq \lambda^{n_0-n}P^s_{n_0}$ for all $n<n_0$, which cannot hold since $\lambda^{n_0-n}\to\infty$
as $n\to-\infty$.
\end{enumerate}

We define $p^s_n=a_n P^s_n$ where $e^{-\ve}<a_n\leq 1$ are appropriately chosen.
We first define $a_n$ for the maximal indices $n\in\Z$ as the largest value in $(0,1]$ with
$a_nP^s_n\in I_{\ve,q(x_n)}$.
In particular, $e^{-\ve^2 q(x_n)}\leq a_n\leq 1$. Then we define $a_n$ for the growing indices.
Fix two consecutive maximal indices $n<m$ and define $a_{n+1},\ldots,a_{m-1}$ with a backwards 
induction as follows.
If $n<k<m$ and $a_{k+1}$ is well-defined then we choose $a_k$ largest as possible satisfying:
\begin{enumerate}[aa)]
\item[(i)] $e^{-\frac{\ve}{4}P^s_k}a_{k+1}\leq e^{\frac{\ve}{4}P^s_k}a_k\leq a_{k+1}$;
\item[(ii)] $a_k P^s_k\in I_{\ve,q(x_k)}$.
\end{enumerate}
This choice is possible because the interval $(e^{-\frac{\ve}{4}P^s_k}a_{k+1},a_{k+1}]$
intersects $I_{\ve,q(x_k)}$,
since $\tfrac{\ve}{4}P^s_k\geq \tfrac{\ve}{4}e^{-\left(\mathfrak H+\frac{\ve}{3}\right)}q^s(f^k(x))\geq
\tfrac{\ve}{4}e^{-\left(\mathfrak H+\frac{\ve}{3}\right)}q(f^k(x))
\geq \tfrac{\ve}{4}e^{-\left(\mathfrak H+\frac{2\ve}{3}\right)}q(x_k)>\ve^2 q(x_k)$.
The first condition implies that $0<a_{n+1}\leq\cdots\leq a_{m-1}\leq a_m\leq 1$.
The maximality on the choice of $a_k$ indeed implies the inequality
$e^{-\ve^2 q(x_k)}a_{k+1}\leq e^{\frac{\ve}{4}P^s_k}a_k\leq a_{k+1}$
for every growing $k$ (this is stronger than (i)).

Before continuing, we collect some estimates relating $q(x_k),P^s_k,p^s_k$.
Fix two consecutive maximal indices $n<m$. Then the following holds for all $\ve>0$ small enough:
\begin{enumerate}[$\circ$]
\item $\displaystyle\sum_{k=n+1}^m P^s_k<\ve^{\frac{6}{\beta}-1}$:
every $k=n+1,\ldots,m-1$ is growing,
thus $P^s_k\leq \lambda^{n+1-k}P^s_{n+1}$ for $k=n+1,\ldots,m$. Therefore
$$
\sum_{k=n+1}^m P^s_k\leq P^s_{n+1}\sum_{i=0}^{m-n-1}\lambda^{-i}<\ve^{\frac{6}{\beta}+1}\frac{1}{1-\lambda^{-1}}<
2\ve^{\frac{6}{\beta}-0.5}<\ve^{\frac{6}{\beta}-1},
$$
since $\lim\limits_{\ve\to 0}\tfrac{\ve^{1.5}}{1-\lambda^{-1}}=1$.
\item $\displaystyle\sum_{k=n+1}^m q(x_k)<\ve^{\frac{6}{\beta}-1}$: by the previous item,
$$
\sum_{k=n+1}^m q(x_k)\leq e^{\mathfrak H+\frac{2\ve}{3}}\sum_{k=n+1}^m P^s_k
<2e^{\mathfrak H+\frac{2\ve}{3}}\ve^{\frac{6}{\beta}-0.5}<\ve^{\frac{6}{\beta}-1}.
$$
\item $a_{n+1}>\lambda^{-1}$: using that $a_m\geq e^{-\ve^2 q(x_m)}>e^{-\ve P^s_m}$ and that
$e^{-\ve P^s_k}a_{k+1}\leq a_k$ for every growing $k$, we have
$$
a_{n+1}\geq \exp{}\left[-\ve\sum_{k=n+1}^{m-1}P^s_k\right]a_m\geq \exp{}\left[-\ve\sum_{k=n+1}^{m}P^s_k\right]
> \exp{}\left[-\ve^{\frac{6}{\beta}}\right]>\lambda^{-1},
$$
since $\ve^{\frac{6}{\beta}}<\ve^{1.5}$. 
\end{enumerate}
In particular, $a_k>\lambda^{-1}>e^{-\ve}$ for all $k\in\Z$.

\medskip
\noindent
{\sc Claim 2:} $\Psi_{x_n}^{p^s_n,p^u_n}\in\mathfs A$ for all $n\in\Z$.

\begin{proof}[Proof of Claim $2$.] We have to check (CG1)--(CG3).

\medskip
\noindent
(CG1) By definition, $\Gamma(x_n)\in Z_{\un\ell^{(n)},m_n,j_n,k}$.

\medskip
\noindent
(CG2) We have $p^s_n\leq P^s_n\leq \ve Q(x_n)$, and the same holds for $p^u_n$.
By definition, $p^s_n,p^u_n\in I_{\ve,q(x_n)}$.

\medskip
\noindent
(CG3) The proof of this in \cite{BCL23} had a mistake, so we take the chance to correct it. By definition, $P^s_n=e^{\pm\left(\mathfrak H+\frac{\ve}{3}\right)}q^s(f^n(x))$ and $p^{s/u}_n=e^{\pm\ve}P^{s/u}_n$, hence
$\tfrac{p^{s/u}_n}{q^{s/u}(f^n(x))}=e^{\pm\left(\mathfrak H+\frac{4\ve}{3}\right)}$. This implies that
$\tfrac{p^s_n\wedge p^u_n}{q(f^n(x))}=e^{\pm\left(\mathfrak H+\frac{4\ve}{3}\right)}$. Since by (${\rm b}_n$) we have
$\tfrac{q(f^n(x))}{q(x_n)}=e^{\pm\ve/3}$, it follows that 
$\tfrac{p^s_n\wedge p^u_n}{q(x_n)}=e^{\pm\left(\mathfrak H+2\ve\right)}$.
\end{proof}

\medskip
\noindent
{\sc Claim 3:} $\Psi_{x_n}^{p^s_n,p^u_n}\overset{\ve}{\rightarrow}\Psi_{x_{n+1}}^{p^s_{n+1},p^u_{n+1}}$
for all $n\in\Z$.

\begin{proof}[Proof of Claim $3$.] We have to check (GPO1)--(GPO2).

\medskip
\noindent
(GPO1) By (${\rm a}_n$) with $i=1$ and (${\rm a}_{n+1}$) with $i=0$, we have
\begin{align*}
&\ d(f(x_n),x_{n+1})+\|\widetilde{C(f(x_n))}-\widetilde{C(x_{n+1})}\|\\
&\leq  d(f^{n+1}(x),f(x_n))+
\|\widetilde{C(f^{n+1}(x))}-\widetilde{C(f(x_n))}\|\\
&\ \ \ \,+ d(f^{n+1}(x),x_{n+1})+
\|\widetilde{C(f^{n+1}(x))}-\widetilde{C(x_{n+1})}\|\\
&<\tfrac{1}{4}q(f^n(x))^8+\tfrac{1}{4}q(f^{n+1}(x))^8
\overset{!}{\leq} \tfrac{1}{4}(1+e^{8\ve})q(f^{n+1}(x))^8\\
&\overset{!!}{\leq} \tfrac{1}{4}e^{8\mathfrak H+\frac{32\ve}{3}}(1+e^{8\ve})(p^s_{n+1}\wedge p^u_{n+1})^8
\overset{!!!}{<}(p^s_{n+1}\wedge p^u_{n+1})^8,
\end{align*}
where in $\overset{!}{\leq}$ we used Lemma \ref{Lemma-q}, in $\overset{!!}{\leq}$ we used
(${\rm b}_n$) and the estimate used to prove (CG3) in the previous paragraph, and in $\overset{!!!}{<}$
we used that $\tfrac{1}{4}e^{8\mathfrak H+\frac{32\ve}{3}}(1+e^{8\ve})<1$ when $\ve,\rho>0$ are sufficiently small.
This proves that
$\Psi_{f(x_n)}^{p^s_{n+1}\wedge p^u_{n+1}}\overset{\ve}{\approx}\Psi_{x_{n+1}}^{p^s_{n+1}\wedge p^u_{n+1}}$.
Similarly, we prove that
$\Psi_{f^{-1}(x_{n+1})}^{p^s_n\wedge p^u_n}\overset{\ve}{\approx}\Psi_{x_n}^{p^s_n\wedge p^u_n}$.

\medskip
\noindent
(GPO2) We show that relation (\ref{gpo2-a}) holds for all $k\in\Z$:
$$
e^{-\ve p^s_k}\min\{e^{\ve T(v_k,v_{k+1})}p^s_{k+1},e^{-\ve}\ve Q(x_k)\}\leq p^s_k\leq
\min\{e^{\ve T(v_k,v_{k+1})}p^s_{k+1},\ve Q(x_k)\}.
$$
Relation (\ref{gpo2-b}) is proved similarly. For ease of notation, write $T_k=T(v_k,v_{k+1})$
and $\Delta_k=(t_{k+1}-t_k)-T_k$. Since $T_k$ is the minimal time, we have $\Delta_k\geq 0$.
Using Lemma \ref{lemma-local-coord}(3), condition (${\rm a}_n$) and Remark
\ref{rmk-time}, we also have the following upper bound for $\Delta_k$:
$$
\Delta_k\leq {\rm diam}(R[\tfrac{1}{15}(p^s_k\wedge p^u_k)])=\tfrac{\sqrt{2}}{15}(p^s_k\wedge p^u_k)
< \tfrac{p^s_k}{4}\cdot 
$$
We fix two consecutive maximal indices $n<m$ and prove the above inequality for $k=n,\ldots,m-1$.
We divide the proof into two cases: $k=n$ and $k\neq n$.
Assume first that $k=n$. For $\ve>0$ small enough (remember $a_{n+1}>\lambda^{-1}$),
$$
e^{\ve T_n}p^s_{n+1}=e^{\ve T_n} a_{n+1}P^s_{n+1}>\exp{}\left[\inf(r_\Lambda)\ve-\ve^{1.5}\right]P^s_{n+1}
>\lambda P^s_{n+1}>P^s_n=\ve Q(x_n).
$$
Therefore
$$
e^{-\ve p^s_n}\min\{e^{\ve T_n}p^s_{n+1},e^{-\ve}\ve Q(x_n)\}=e^{-\ve p^s_n}e^{-\ve}\ve Q(x_n)
<e^{-\ve}\ve Q(x_n)<a_nP^s_n=p^s_n
$$
and
$$
\min\{e^{\ve T_n}p^s_{n+1},\ve Q(x_n)\}=\ve Q(x_n)=P^s_n\geq p^s_n.
$$
This proves (\ref{gpo2-a}) for $k=n$.

Now let $k\neq n$, and call
${\rm I}=\min\{e^{\ve T_k}p^s_{k+1},e^{-\ve}\ve Q(x_k)\}$,
${\rm II}=\min\{e^{\ve T_k}p^s_{k+1},\ve Q(x_k)\}$.
We wish to show that $e^{-\ve p^s_k}{\rm I}\leq p^s_k\leq {\rm II}$.
Since $a_{k+1}\geq e^{-\ve\Delta_k}a_{k+1}>\exp{}\left[-\ve \tfrac{p^s_k}{4}-\ve^{1.5}\right]>\exp{}[-\ve]$,
we have
\begin{align*}
&\ {\rm I}=\min\{e^{-\ve\Delta_k}a_{k+1}e^{\ve (t_{k+1}-t_k)}P^s_{k+1},e^{-\ve}\ve Q(x_k)\}\\
&\leq a_{k+1}\min\{e^{\ve (t_{k+1}-t_k)}P^s_{k+1},\ve Q(x_k)\}=a_{k+1}P^s_k.
\end{align*}
Therefore $e^{-\ve p^s_k}{\rm I}\leq e^{-\frac{\ve}{2}P^s_k}a_{k+1}P^s_k\leq a_kP^s_k=p^s_k$,
where in the second inequality we used property (i) in the definition of $a_k$.

For the other inequality, start observing that
$$
p^s_k=a_kP^s_k = a_k\min\{e^{\ve(t_{k+1}-t_k)}P^s_{k+1},\ve Q(x_k)\}=
\min\{e^{\ve(t_{k+1}-t_k)}a_kP^s_{k+1},a_k \ve Q(x_k)\}.
$$
Clearly $a_k \ve Q(x_k)\leq \ve Q(x_k)$. Using that $\Delta_k\leq \tfrac{P^s_k}{4}$,
we have $e^{\ve\Delta_k}a_k\leq e^{\frac{\ve}{4}P^s_k}a_k\leq a_{k+1}$, where in the last
passage we used property (i) in the definition of $a_k$. Hence
$$
e^{\ve(t_{k+1}-t_k)}a_kP^s_{k+1}=e^{\ve T_k}e^{\ve \Delta_k}a_k P^s_{k+1}\leq e^{\ve T_k}a_{k+1}P^s_{k+1}
=e^{\ve T_k}p^s_{k+1}.
$$
The conclusion is that
$p^s_k\leq {\rm II}$.
The proof of Claim 3 is now complete.
\end{proof}

\medskip
\noindent
{\sc Claim 4:} $\{\Psi_{x_n}^{p^s_n,p^u_n}\}_{n\in\Z}$ is regular.

\begin{proof}[Proof of Claim $4$.]
Since $x\in\nuh^\#$ and $\tfrac{p^s_n\wedge p^u_n}{q(f^n(x))}=e^{\pm(\mathfrak H+1)}$,
we have
$\limsup\limits_{n\to+\infty}(p^s_n\wedge p^u_n)>0$ and $\limsup\limits_{n\to-\infty}(p^s_n\wedge p^u_n)>0$.
By the discreteness of $\mathfs A$, it follows that $\Psi_{x_n}^{p^s_n,p^u_n}$ repeats infinitely
often in the future and infinitely often in the past.
\end{proof}

\medskip
\noindent
{\sc Claim 5:} $\{\Psi_{x_n}^{p^s_n,p^u_n}\}_{n\in\Z}$ shadows $x$.

\begin{proof}[Proof of Claim $5$.]
By (${\rm a}_n$) with $i=0$, we have
$\Psi_{f^n(x)}^{p^s_n\wedge p^u_n}\overset{\ve}{\approx}\Psi_{x_n}^{p^s_n\wedge p^u_n}$, hence 
by Proposition \ref{Lemma-overlap}(3) we have $f^n(x)=\Psi_{f^n(x)}(0)\in \Psi_{x_n}(R[p^s_n\wedge p^u_n])$,
thus $\{\Psi_{x_n}^{p^s_n,p^u_n}\}_{n\in\Z}$ shadows $x$.
This concludes the proof of sufficiency.
\end{proof}

\medskip
\noindent
{\em Proof of relevance.} The alphabet $\mathfs A$ need not satisfy
the relevance condition a priori, but we can easily reduce it to a sub-alphabet $\mathfs A'$ satisfying (1)--(3).
Call $v\in\mathfs A$ relevant if there is $\un v\in\mathfs A^\Z$ with $v_0=v$ such that $\un{v}$ shadows
a point in $\Lambda\cap\nuh^\#$. Since $\nuh^\#$ is $\vf$--invariant, every $v_i$ is relevant.
Hence $\mathfs A'=\{v\in\mathfs A:v\text{ is relevant}\}$ is discrete
because $\mathfs A'\subset\mathfs A$, it is sufficient and relevant by definition.
\end{proof}

\subsection{First coding}\label{ss.first.coding}

Let $\Sigma$ be the TMS associated to the graph with vertex set $\mathfs A$ given by
Theorem \ref{Thm-coarse-graining} and
edges $v\overset{\ve}{\to}w$. An element $\un v\in\Sigma$ is an $\ve$--gpo,
so let $\pi:\Sigma\to \widehat \Lambda$ by
$$
\{\pi(\un v)\}:=V^s[\un v]\cap V^u[\un v].
$$
The main properties of the triple $(\Sigma,\sigma,\pi)$ are listed below.

\begin{proposition}\label{Prop-pi}
The following holds for all $0<\ve\ll \rho\ll 1$.
\begin{enumerate}[{\rm (1)}]
\item Each $v\in\mathfs A$ has finite ingoing and outgoing degree, hence $\Sigma$ is locally compact.
\item $\pi:\Sigma\to \widehat \Lambda$ is H\"older continuous.
\item $\pi[\Sigma^\#]\supset\Lambda\cap\nuh^\#$.
\end{enumerate} 
\end{proposition}

Part (1) follows from Lemma \ref{Lemma-minimum} and Theorem \ref{Thm-coarse-graining}(1),
part (2) follows from Theorem \ref{Thm-stable-manifolds}(4),
and part (3) follows from Theorem \ref{Thm-coarse-graining}(2).
It is important noticing that $(\Sigma,\sigma,\pi)$ is {\em not} the TMS that satisfies the Main Theorem,
since $\pi$ might be (and usually is) infinite-to-one.
We use $\pi$ to induce a locally finite cover of $\Lambda\cap\nuh^\#$, which will then be
refined to generate a new TMS whose TMF is the one satisfying the Main Theorem.

We end this section introducing the TMF defined by $(\Sigma,\sigma,\pi)$.
Remember that $r_\Lambda:\Lambda\to(0,\rho/2)$ is the first return time to $\Lambda$.

\medskip
\noindent
{\sc The roof function $r:\Sigma\to (0,\rho)$:} Given $\un v=\{\Psi_{x_n}^{p^s_n,p^u_n}\}_{n\in\Z}\in\Sigma$,
let $x=\pi(\un v)$ and assume that $x_1$ belongs to the disc $D\subset{\widehat \Lambda}$.
Define $r(\un v):=-\mathfrak t_D(x)=r_\Lambda(x_0)-\mathfrak t_D[\vf^{r_{\Lambda}(x_0)}(x)]$.

\medskip
Since $g_{x_0}^+=\mathfrak q_D$,
$r(\un v)$ is the time increment for $\vf$ between the points $\pi(\un v)$ and $g_{x_0}^+[\pi(\un v)]$.
In particular, $\vf^{r(\un v)}[\pi(\un v)]=\pi[\sigma(\un v)]$ belongs to $\widehat\Lambda$ but not necessarily to
$\Lambda$. (Note: even if $\pi(\un v),\vf^{r(\un v)}[\pi(\un v)]\in\Lambda$, the values of 
$r(\un v)$ and $r_\Lambda[\pi(\un v)]$ may differ.)

\medskip
\noindent
{\sc The triple $(\Sigma_r,\sigma_r,\pi_r)$:} We take $(\Sigma_r,\sigma_r)$
to be the TMF associated to the TMS $(\Sigma,\sigma)$ and roof function $r$,
and $\pi_r:\Sigma_r\to M$ to be the map defined by $\pi_r[(\un v,t)]=\vf^t[\pi(\un v)]$.

\medskip
The next proposition lists the main properties of $(\Sigma_r,\sigma_r,\pi_r)$.

\begin{proposition}\label{Prop-pi_R}
The following holds for all $0<\ve\ll \rho\ll 1$.
\begin{enumerate}[{\rm (1)}]
\item $\pi_r\circ\sigma_r^t=\vf^t\circ\pi_r$, for all $t\in\R$.
\item $\pi_r$ is H\"older continuous with respect to the Bowen-Walters distance.
\item $\pi_r[\Sigma_r^\#]\supset\nuh^\#$.
\end{enumerate}
\end{proposition}

\begin{proof}
Part (1) is direct from the definition of $\pi_r$. The proof of part (2) uses Proposition \ref{Prop-pi}(2),
and follows by the same methods used in the proof of \cite[Lemma 5.9]{Lima-Sarig}.
To prove part (3), let $S:=\Sigma^\#\times\{0\}\subset\Sigma_r^\#$.
By Proposition \ref{Prop-pi}(3), $\pi_r(S)\supset\Lambda\cap\nuh^\#$.
Since $\pi_r[\Sigma_r^\#]=\bigcup\limits_{t\in\R}\vf^t[\pi_r(S)]$ and
$\nuh^\#=\bigcup\limits_{t\in\R}\vf^t[\Lambda\cap\nuh^\#]$,
we get that $\pi_r[\Sigma_r^\#]\supset\nuh^\#$. 
\end{proof}

\section{Inverse theorem}\label{Section-inverse}

Up to now, we have constructed a first coding $\pi:\Sigma\to \widehat\Lambda$, but it is usually infinite-to-one. Our next goal is to understand how $\pi$ loses injectivity:
if $\un v\in\Sigma$ and $x=\pi(\un v)$, what is the relation between the parameters defining $\un v$ and
those associated to the orbit of $x$? We analyze this question as an {\em inverse problem}: 
fixed $x\in \widehat\Lambda$,
the parameters of $\un v$ are defined ``up to bounded error''. The answer to this inverse problem is what
we call an {\em inverse theorem}. From now on, we require that
$\un v\in\Sigma^\#$, where $\Sigma^\#$ is the {\em regular set} of $\Sigma$:
$$
\Sigma^\#:=\left\{\un v\in\Sigma:\exists v,w\in V\text{ s.t. }\begin{array}{l}v_n=v\text{ for infinitely many }n>0\\
v_n=w\text{ for infinitely many }n<0
\end{array}\right\}.
$$

\medskip
Recall that $r:\Sigma\to (0,\rho)$ is the roof function, defined before
Proposition \ref{Prop-pi_R}. Let $r_n=\sum_{i=0}^{n-1}r\circ\sigma^i$ denote its $n$--th Birkhoff sum with respect to the shift map
$\sigma:\Sigma\to\Sigma$. Let $\un v=\{\Psi_{x_n}^{p^s_n,p^u_n}\}_{n\in\Z}\in\Sigma$,
and let $x=\pi(\un v)$. Then:
\begin{enumerate}[$\circ$]
\item $\vf^{r_n(\un v)}(x)=\pi[\sigma^n(\un v)]$, a point in $\widehat\Lambda$ that is close to $x_n$.
\item $g_{x_n}^+[\vf^{r_n(\un v)}(x)]=\vf^{r_{n+1}(\un v)}(x)$.
\end{enumerate}
Let $p^{s/u}(\vf^{r_n(\un v)}(x))$ be the $\Z$--indexed version of the parameter $q^{s/u}$
with respect to the sequence of times $\{r_n(\un v)\}_{n\in\Z}$ (see Section \ref{section-Z-indexed} for the definition).

\begin{theorem}[Inverse theorem]\label{Thm-inverse}
The following holds for all $0<\ve\ll \rho\ll 1$.
If $\un v=\{\Psi_{x_n}^{p^s_n,p^u_n}\}_{n\in\Z}\in\Sigma^\#$
and $x=\pi(\un v)$, then $x\in\nuh^\#$ and the following are true.
\begin{enumerate}[{\rm (1)}]
\item $d(\vf^{r_n(\un v)}(x),x_n)<50^{-1}(p^s_n\wedge p^u_n)$.
\item $\tfrac{\|C(x_n)^{-1}\|}{\|C(\vf^{r_n(\un v)}(x))^{-1}\|}=e^{\pm 2\sqrt{\ve}}$.
\item $\tfrac{Q(x_n)}{Q(\vf^{r_n(\un v)}(x))}=e^{\pm \sqrt[3]{\ve}}$.
\item $\tfrac{p^s_n}{p^s(\vf^{r_n(\un v)}(x))}=e^{\pm\sqrt[3]{\ve}}$ and
$\tfrac{p^u_n}{p^u(\vf^{r_n(\un v)}(x))}=e^{\pm\sqrt[3]{\ve}}$.
\item $\Psi_{\vf^{r_n(\un v)}(x)}^{-1}\circ\Psi_{x_n}$
can be written in the form $\delta+Ov+\Delta(v)$ for $v\in R[10Q(x_n)]$,
where $\delta\in\R^d$ satisfies $\|{\delta}\|<{50}^{-1}(p^s_n\wedge p^u_n)$,
$O$ is an orthogonal linear map preserving the splitting $\R^{d_s(x)}\times\R^{d_u(x)}$, and $\Delta:R[10Q( x_n)]\to \R^d$ satisfies
$\Delta(0)=0$ and $\|{d\Delta}\|_{C^0}<5\sqrt{\ve}$ on $R[10Q( x_n)]$. The same statement applies to representing $\Psi_{x_n}^{-1}\circ\Psi_{\vf^{r_n(\un v)}(x)}$ in $R[10Q(\vf^{r_n(\un v)}(x))]$.
\end{enumerate}
\end{theorem}

The above theorem compares the parameters of an $\ve$--gpo with the parameters of its shadowed point.

\subsection{Identification of invariant subspaces}

To analyze the inverse problem, we need to compare the parameters of the charts. One key aspect of this comparison regards the hyperbolicity parameters, which are linked to the Lyapunov inner product (see Section \ref{Section-reduction}). In general, given $x, y \in {\rm NUH}$, it is not immediately clear how to compare the Lyapunov inner product norm between vectors in $N^{s/u}_{x}$ and those in $N^{s/u}_{y}$. However, when $x$ defines an $\ve$–double chart $\Psi_{x}^{p^s,p^u}$ and $y$ lies in a stable or unstable set at $\Psi_{x}^{p^s,p^u}$, Ben Ovadia introduced a canonical method for making this comparison, using the representing function in the chart \cite{O18}. This method was also used in \cite{ALP}, which we will follow. The idea is to first identify subspaces within the chart and then transfer this identification to the manifold.

Let $\Psi_{ x}^{p^s,p^u}$ be an $\ve$--double chart, let $V$ be an $s$--admissible
manifold at $\Psi_{ x}^{p^s,p^u}.$
In the sequel, we use the notation $T V$ to represent the tangent bundle
of $V$ as a subset of $TM$. Writing $d_{s/u}=d_{s/u}(x)$, recall that
$$
V=\Psi_{ x}\{(v_1,G(v_1)):v_1\in B^{d_s}[p^s]\}
$$
where $G:B^{d_s}[p^s]\to \R^{d_u}$ is a $C^{1+\beta/3}$ function satisfying (AM1)--(AM3).
Fix $ y\in  V$, and write $y=\Psi_{ x}(z)$ where $z=(v_1,G(v_1))$.
Denote the tangent space to the graph of $G$ at $z$ by
$T_{z}{\rm Graph}(G)$. We have
$$
T_{z}{\rm Graph}(G)=\left\{\begin{bmatrix} w \\ (dG)_{v_1}w\end{bmatrix}:w\in \R^{d_s}\right\},
$$
which is canonically isomorphic to $\R^{d_s}\times\{0\}$ via the map 
$$
\begin{array}{rcl}
\iota_s\ :\  \R^{d_s}\times\{0\}& \xrightarrow{\hspace*{8mm}}& T_{z}{\rm Graph}(G)\\
&&\\
\begin{bmatrix} w \\ 0\end{bmatrix}& \xmapsto{\hspace*{8mm}} & \begin{bmatrix} w \\ (dG)_{v_1}w\end{bmatrix}.\\
\end{array}
$$
Recalling that $N^s_{ x}=(d\Psi_{ x})_0[\R^{d_s}\times\{0\}]$ and 
$T_{ y} V=(d\Psi_{ x})_z[T_{z}{\rm Graph}(G)]$, we have the following definition.

\medskip
\noindent
{\sc The map $\Theta^s_{ x, y}$:} We define $\Theta^s_{ x, y}:N^s_{ x}\to T_{ y} V$ 
as the composition of the linear maps
$$
\Theta^s_{ x, y}:= (d\Psi_{ x})_z\circ \iota_s \circ [(d\Psi_{ x})_0]^{-1}.
$$

\medskip
In other words, $\Theta^s_{ x, y}$ is defined to make the diagram below commute
$$
\begin{tikzcd}[row sep = huge, column sep = huge]
\mathbb{R}^{d_s} \times \{ 0 \} \arrow{r}{\iota_s} \arrow{d}[swap]{d(\Psi_x)_0} & T_z{\rm Graph}(G) \arrow{d}{d(\Psi_x)_z} \\ N^s_x \arrow{r}[swap]{\Theta_{x,y}^s} & T_yV
\end{tikzcd}
$$
and so it has the explicit formula
$$
\Theta^s_{ x, y}\left((d\Psi_{ x})_{0}\begin{bmatrix} w \\ 0\end{bmatrix}\right)=
(d\Psi_{ x})_{z}\begin{bmatrix} w \\ (dG)_{v_1}w\end{bmatrix}.
$$

A similar definition holds when $ y$ belongs to a $u$--admissible manifold at
$\Psi_{ x}^{p^s,p^u}$. Let
$$
V=\Psi_{ x}\{(G(v_2),v_2):v_2\in B^{d_u}[p^u]\}
$$
where $G:B^{d_u}[p^u]\to \R^{d_s}$ is a $C^{1+\beta/3}$ function satisfying (AM1)--(AM3). For $ y\in V$,
write $y=\Psi_{ x}(z)$ where $z=(G(v_2),v_2)$. We have
$$
T_{z}{\rm Graph}(G)=\left\{\begin{bmatrix} (dG)_{v_2}w \\ w\end{bmatrix}:w\in \R^{d_u}\right\},
$$
which is canonically isomorphic to $\{0\}\times\R^{d_u}$ via the map
$$
\begin{array}{rcl}
\iota_u\ :\  \{0\}\times\R^{d_u}& \xrightarrow{\hspace*{8mm}}& T_{z}{\rm Graph}(G)\\
&&\\
\begin{bmatrix} 0 \\ w\end{bmatrix}& \xmapsto{\hspace*{8mm}} & \begin{bmatrix} (dG)_{v_2}w \\ w\end{bmatrix}.\\
\end{array}
$$

\medskip
\noindent
{\sc The map $\Theta^u_{ x, y}$:} We define $\Theta^u_{ x, y}:N^u_{ x}\to T_{ y} V$ 
as the composition of the linear maps
$$
\Theta^u_{ x, y}:= (d\Psi_{ x})_z\circ \iota_u \circ [(d\Psi_{ x})_0]^{-1}.
$$

\medskip
Similarly, $\Theta^u_{ x, y}$ is defined to make the following diagram to commute
$$
\begin{tikzcd}[row sep = huge, column sep = huge]
\{ 0 \} \times \mathbb{R}^{d_u} \arrow{r}{\iota_u} \arrow{d}[swap]{d(\Psi_x)_0} & T_z{\rm Graph}(G) \arrow{d}{d(\Psi_x)_z} \\ N^u_x \arrow{r}[swap]{\Theta_{x,y}^u} & T_yV
\end{tikzcd}
$$
and it has the formula
$$
\Theta^u_{ x, y}\left((d\Psi_{ x})_{0}\begin{bmatrix} 0 \\ w\end{bmatrix}\right)=
(d\Psi_{ x})_{z}\begin{bmatrix} (dG)_{v_2}w \\ w\end{bmatrix}.
$$

Now assume that $\{ y\}= V^s\cap  V^u$, where $ V^{s/u}$ is a $s/u$--adimissible manifold
at $\Psi_{ x}^{p^s,p^u}$. By the above discussion, we have two linear maps
$\Theta^{s/u}_{ x, y}:N^{s/u}_{ x}\to T_{ y} V^{s/u}$.
Since $N^s_{ x}\oplus N^u_{ x}={N}_{ x}$
and $T_{ y} V^s\oplus T_{ y} V^u={N}_{ y}$, the following definition makes sense.

\medskip
\noindent
{\sc The map $\Theta_{ x, y}$:} We define $\Theta_{ x, y}:{N}_{ x}\to {N}_{ y}$
as the unique linear map s.t. $\Theta_{ x, y}\restriction_{N^s_{ x}}=\Theta^s_{ x, y}$ and 
$\Theta_{ x, y}\restriction_{N^u_{ x}}=\Theta^u_{ x, y}$. More specifically, if $v=v^s+v^u$ with
$v^{s/u}\in N^{s/u}_{ x}$ then 
$$
\Theta_{ x, y}(v):=\Theta^s_{ x, y}(v^s)+\Theta^u_{ x, y}(v^u).
$$

\medskip
We can similarly see $\Theta_{ x, y}$ as a map defined in terms of a commuting diagram.
Let $G,H$ be the representing functions of $V^s,V^u$, let $y=\Psi_{ x}(z)$
with $z=(v_1,v_2)$, and let 
$$
\begin{array}{rcl}
\iota\ \ \ :\ \ \R^d& \xrightarrow{\hspace*{8mm}}& \R^d\\
&&\\
\begin{bmatrix} w_1 \\ w_2\end{bmatrix}& \xmapsto{\hspace*{8mm}} &
\begin{bmatrix} w_1+(dH)_{v_2}w_2 \\ w_2+(dG)_{v_1}w_1\end{bmatrix}.\\
\end{array}
$$
Then we obtain a commuting diagram
$$
\begin{tikzcd}[row sep = huge, column sep = huge]
\mathbb{R}^d \arrow{r}{\iota} \arrow{d}[swap]{d(\Psi_x)_0} & \mathbb{R}^d \arrow{d}{d(\Psi_x)_z} \\ N_x \arrow{r}[swap]{\Theta_{x,y}} & N_y
\end{tikzcd}
$$
and
$$
\Theta_{ x, y}\left((d\Psi_{ x})_{0}\begin{bmatrix} w_1 \\ w_2\end{bmatrix}\right)=
(d\Psi_{ x})_{z}\begin{bmatrix} w_1+(dH)_{v_2}w_2 \\ w_2+(dG)_{v_1}w_1\end{bmatrix}.
$$

The next result proves that the maps defined above are close to parallel transports.

\begin{lemma}\label{ALP-Lemma 6.1}
    Let $\Psi_{x}^{p^s,p^u}$ be an $\varepsilon$-double chart with $\eta = p^s \wedge p^u$. 
    \begin{enumerate}[{\rm (1)}]
        \item Let $V$ be an $s$--admissible manifold at $\Psi_{x}^{p^s,p^u}$, and let $y\in V$. If $y \in \Psi_{x}(B^{d_s(x)}[\eta] \times B^{d_u(x)}[\eta])$, then
        $$
        \|\Theta_{x,y}^{s} - P_{x,y} \|\, , \, \|(\Theta_{x,y}^{s})^{-1} - P_{x,y}^{-1} \| \leq \tfrac{1}{2} \eta^{15\beta/48} 
        $$
        where $P_{x,y}$ is the restriction of the parallel transport from $x$ to $y$ to the subspace $N^s_{x}$. In particular,
        $\frac{\|\Theta_{x,y}^{s}(v)\|}{\|v\|}= \operatorname{exp} \left[ \pm \eta^{15\beta/48} \right]$ for all $v \in N^s_x \backslash \{ 0 \}$.
       An analogous statement holds for $u$--admissible manifolds.
        
        \item If $y \in V^s \cap V^u$ where $V^{s/u}$ is a $s/u$--admissible manifold at $\Psi_x^{p^s,p^u}$, then
        $$
        \|\Theta_{x,y}- P_{x,y} \|\, , \, \|\Theta_{x,y}^{-1}- P_{x,y}^{-1} \| \leq \tfrac{1}{2} \eta^{15\beta/48}. 
        $$
        In particular,
        $\frac{\|\Theta_{x,y}(v)\|}{\|v\|} = \operatorname{exp} \left[ \pm \eta^{15\beta/48} \right]$ for all $v \in N_x \backslash \{0 \}$.
    \end{enumerate}
\end{lemma}

The proof is the same as \cite[Lemma 6.3]{ALP}.

\subsection{Improvement lemma}\label{Section-improvement}

Recall the norm $\vertiii{\cdot}$ induced by the Lyapunov inner product, introduced in Section \ref{Section-reduction}, and the graph transforms $\mathfs{F}^{s/u}$ introduced in Section \ref{ss.graph.transform}.
In the next result, we write $d_{s/u}=d_{s/u}(x)$.

\begin{lemma}[Improvement Lemma]\label{improvement-lemma} 
The following holds for $\ve > 0$ small enough. Let $v \overset{\ve}{\rightarrow} w$ with $v=\Psi_{x_0}^{p^s_0, p^u_0}, w=\Psi_{x_1}^{p^s_1,p^u_1}$, and write $\eta_0 = p^s_0 \wedge p^u_0, \eta_1 = p^s_1 \wedge p^u_1$. Fix $V_1 \in \mathfs{M}^s(w)$ and let $V_0 = \mathfs{F}_{v,w}^s(V_1)$. Let $y_0 \in V_0, \ y_1 = g_{x_0}^+ (y_0) \in V_1$, and assume that
$$
y_0\in B^{d_s}[\eta_0]\times B^{d_u}[\eta_0]
\ \text{ and }\ 
y_1\in B^{d_s}[\eta_1]\times B^{d_u}[\eta_1].
$$
Consider $v_1 \in T_{y_1} V_1$ and $v_0 = d(g_{x_0}^+)^{-1} v_1 \in T_{y_0} V_0$. Finally, let $w_0 \in T_{x_0}M$ such that $v_0 = \Theta_{x_0, y_0}^s (w_0)$ and $w_1 \in T_{x_1}M$ such that $v_1 = \Theta_{x_1, y_1}^s(w_1)$. If $\frac{\vertiii{v_1}}{\vertiii{w_1}} =\operatorname{exp} [ \pm \xi]$ for $\xi \geq \sqrt{\ve}$, then
$$\frac{\vertiii{v_0}}{\vertiii{w_0}} = \operatorname{exp} \left[\pm\left(\xi - Q(x_0)^{\beta/4}\right)\right].$$
Note that the ratio improves. An analogous statement holds for vectors tangent to $u$--admissible manifolds.
\end{lemma}

\begin{proof}
Write $f := g_{x_0}^+$. We focus on one side of the estimate (the other side is proved similarly). We can assume that $\| v_0 \| = 1$.We also write $g^-_{x_1} = \varphi^{T^-}$ where $T^-$ is a $C^{1+\beta}$ function with $T^-(x_1) = -r_\Lambda(f^{-1}(x_1))$. Then $g^-_{x_1}({x_1}) = \varphi^{T^-({x_1})}({x_1})$ and $g^-_{x_1}({y_1}) = \varphi^{T^-({y_1})}({y_1})$. For simplicity of notation, let $t_0 = -T^-({x_1})$ and $t_1 = -T^-({y_1})$, then $g^-_{x_1}(x_1) = \varphi^{-t_0}({x_1})$ and $g^-_{x_1}({y_1}) = \varphi^{-t_1}({y_1})$. Defining $\varrho=\Phi^{-t_0}(w_1)$, we write
$$
\frac{\vertiii{v_0}}{\vertiii{w_0}} = \frac{\vertiii{v_0}}{\vertiii{\varrho}}\cdot \frac{\vertiii{\varrho}}{\|w_0\|}
$$
and estimate each of the fractions above separately. First, we make some comparisons.
\begin{enumerate}[$\circ$]
    \item Comparison of $w_0, v_0$: by Lemma \ref{ALP-Lemma 6.1}, 
    $$\|P_{y_0,x_0} v_0 - w_0 \| \leq \tfrac{1}{2} \eta_0^{15 \beta/48}.$$
    \item Comparison of $v_0, \varrho$: remembering the definitions of $\varrho$ and $v_0$, we have
    \begin{equation*}
    \begin{split}
        & \left\| \varrho - P_{y_0, f^{-1} (x_1)} v_0 \right\|  = \left\| df^{-1}_{x_1} w_1 - P_{y_0, f^{-1} (x_1)} \circ df^{-1}_{y_1} v_1 \right\| \\
        & = \left\| df^{-1}_{x_1} w_1 - \widetilde{(df^{-1}_{y_1})} [P_{y_1, x_1} v_1]\right \| \\
        & \leq \left\| (df^{-1}_{x_1}) \right\| \left\|w_1 - P_{y_1, x_1} v_1 \right\| + \left\| \widetilde{df^{-1}_{x_1}} - \widetilde{df^{-1}_{y_1}} \right\| \| v_1 \|.
    \end{split}
    \end{equation*}
    By Lemma \ref{ALP-Lemma 6.1}, $\| w_1 - P_{y_1,x_1} v_1 \| \leq \frac{1}{2} \| v_1 \| \eta_1^{15 \beta/48} \leq \| v_1 \| \eta_0^{15 \beta/48}$, where in the last passage we used that $\eta_1\leq e^\ve\eta_0$. Now, using Lemma \ref{lemma-local-coord}(2), the Hölder continuity of $df^{-1}$ with $\Hol{\beta}(df^{-1})\leq \max_{|t|\leq 1}\Hol{\beta}(d\vf^t)<\infty$, 
     and that $d(x_1, y_1) \leq 2 \eta_1 \leq 4 \eta_0$, we get that
\begin{equation*}
\begin{split}
    \| \varrho - P_{y_0, f^{-1}(x_1)} v_0 \| & \leq 2 \| v_1 \| \eta_0^{15 \beta/48} + \Hol{\beta}(df^{-1}) (4 \eta_0)^{\beta} \| v_1 \|
     \leq \tfrac{1}{2} \eta_0^{14 \beta/48},
\end{split}
\end{equation*}
where in the last passage we assume that $\ve >0$ is small enough.
    \item Comparison between $w_0, \varrho$: the two estimates above imply that
\begin{equation*}
\begin{split}
    &\, \| \varrho - P_{x_0, f^{-1} (x_1)} w_0 \| \leq \| \varrho - P_{y_0, f^{-1}(x_1)} v_0 \| + \| P_{y_0, f^{-1}(x_1)} v_0 - P_{x_0, f^{-1}(x_1)} w_0 \| \\
    & = \| \varrho - P_{y_0, f^{-1} (x_1)} v_0 \| + \| P_{y_0,x_0} v_0 - w_0 \| + O(\eta_0) < \eta_0^{14 \beta/48},
\end{split}
\end{equation*}
where the term $O(\eta_0)$ is an upper bound on the area of the geodesic triangle formed by $x_0,y_0,f^{-1}(x_1)$ (see the discussion after (Exp2) in page \pageref{geodesic-triangles}).
\end{enumerate}

Now we estimate the fractions $\frac{\vertiii{v_0}}{\vertiii{\varrho}}$ and $\frac{\vertiii{\varrho}}{\vertiii{w_0}}$.\\

\noindent
{\sc Estimate of $\frac{\|\varrho\|}{\|w_0\|}$}: by the definition of the Lyapunov inner product and Lemma \ref{ALP-Lemma 6.1}(1),

$$
\vertiii{w_0} = \|C(x_0)^{-1} w_0\| \geq \|w_0\| \geq \operatorname{exp} \left[ -\eta_0^{15\beta/48} \right] \geq \tfrac{1}{2}\cdot
$$
Letting $A = C(x_0)^{-1}$ and $B = C(f^{-1}(x_1))^{-1}$, we have
$$
\left| \frac{\vertiii{\varrho}}{\vertiii{w_0}} - 1 \right| \leq 2 \left| \vertiii{\varrho}-\vertiii{w_0}\right| = 2 \left|\|B\varrho\| - \|A w_0\|\right|.
$$  
Now let $\widetilde{B} = B \circ P_{x_0, f^{-1}(x_1)}$, so that $A, \widetilde{B}$ have the same domain and codomain. Then
\begin{align*}
&\left|\|B \varrho\| - \|A w_0\|\right| \leq \left|\|B \varrho\| -\|\widetilde{B} w_0\|\right| +  \left|\|\widetilde{B} w_0\|-\|Aw_0\|\right| \\
&\leq \|B \varrho - B \circ P_{x_0, f^{-1}(x_1)} w_0\| + \|\widetilde{B} w_0 - A w_0\| \\
&\leq \|B\|\left\|\varrho - P_{x_0, f^{-1}(x_1)} w_0\right\| + \left\|\widetilde{B} - A\right\| \|w_0\|.
\end{align*}
By the overlap condition, Proposition \ref{Lemma-overlap}(1) and the comparison between $w_0, \varrho$, the 
last expression above is bounded by $2\|A\| \eta_0^{14\beta/48} + 2\eta_0 \leq 2\ve^{1/8} \eta_0^{13\beta/48} + 2\eta_0 < \frac{1}{4} \eta_0^{\beta/4}$.
Therefore
$\left| \frac{\|\varrho\|}{\|w_0\|} - 1 \right| < \frac{1}{2} \eta_0^{\beta/4}$
and so
$$
\frac{\|\varrho\|}{\|w_0\|}={\rm exp}\left[ \pm \eta_0^{\beta/4} \right].
$$

\noindent
{\sc Estimate of $\frac{\vertiii{v_0}}{\vertiii{\varrho}}$:} 
By the last estimate, we need to prove that
\begin{equation}\label{ratio-norms-1}
    \frac{\vertiii{v_0}^2}{\vertiii{\varrho}^2} \leq {\rm exp}\left[2 \xi - 4Q(x_0)^{\beta/4}\right].
\end{equation}

Recall that we are assuming that $\frac{\vertiii{v_1}}{\vertiii{w_1}} \leq e^\xi$ for some $\xi\geq\sqrt{\ve}$. Write
$$I_1 := 4 e^{2 \rho} \int^{t_0}_0 e^{2\chi t} \| \Phi^t v_0 \|^2dt \ \text{ and }\ I_2:= 4 e^{2 \rho} \int^{t_1}_0 e^{2\chi t} \| \Phi^t \varrho \|^2dt.
$$
Then
\begin{align*}
   & \frac{\vertiii{v_0}^2}{\vertiii{\varrho}^2} = \frac{I_1 + 4 e^{2 \rho} \int^\infty_{t_0} e^{2\chi t} \| \Phi^t v_0 \|^2dt}{I_2 + 4 e^{2 \rho} \int^\infty_{t_1} e^{2\chi t} \| \Phi^t \varrho \|^2dt}
    = \frac{I_1+ e^{2 \chi t_0} \vertiii{v_1}^2}{I_2 + e^{2 \chi t_1} \vertiii{w_1}^2}
 \leq \frac{I_1+ e^{2 \chi t_0} e^{2 \xi} \vertiii{w_1}^2}{I_2 + e^{2 \chi t_1} \vertiii{w_1}^2} \\
    & = e^{2 \xi + 2 \chi (t_0 - t_1)} \left( 1 - \frac{I_2 - I_1e^{-2 \xi - 2 \chi (t_0 - t_1)}}{I_2 + e^{2 \chi t_1} \vertiii{w_1}^2} \right)
    = e^{2 \xi + 2 \chi (t_0 - t_1)} \left( 1 - \frac{I_2 - I_1e^{-2 \xi - 2 \chi (t_0 - t_1)}}{\vertiii{\varrho}^2} \right).
\end{align*}
We claim that (\ref{ratio-norms-1}) follows from the estimate
\begin{equation}\label{estimate-for-ratio}
I_2-I_1e^{-2 \xi -2 \chi (t_0 - t_1)} \geq 8 Q(x_0)^{\beta/4} \vertiii{\varrho}^2.
\end{equation}
Indeed, if this is the case, then 
$$
\frac{\vertiii{v_0}^2}{\vertiii{\varrho}^2}\leq e^{2 \xi + 2 \chi (t_0 - t_1)}\left[1 - 8Q(x_0)^{\beta/4}\right] \leq e^{2 \xi + 2 \chi (t_0 - t_1)-8Q(x_0)^{\beta/4}} \leq e^{2 \xi-4Q(x_0)^{\beta/4}},
$$
where in the last two passages we used that $1+x\leq e^x$ for all $x \in \R$ and Lemma \ref{lemma-local-coord}(3) to obtain that $2 \chi |t_0 - t_1| \leq 2d(y_0,f^{-1}(x_1)) \leq 4 Q(x_0)\ll 4 Q(x_0)^{\beta/4}$. Thus, we focus on establishing (\ref{estimate-for-ratio}), which we will prove by estimating each side separately. 

We begin with the left hand side, which we write as 
$(I_2-I_1)+I_1\left[1-e^{-2 \xi -2 \chi (t_0 - t_1)}\right]$.
We have the following estimates for $\ve>0$ small enough:
\begin{enumerate}[$\circ$]
\item $I_1$ has the uniform lower bound
$$
I_1\geq 4e^{2\rho}\int_0^{t_0}e^{2\chi t}e^{-2\rho-2t}dt=
\tfrac{2}{1-\chi}\left[1-e^{2(\chi-1)t_0}\right]\geq
\tfrac{2}{1-\chi}\left[1-e^{2(\chi-1)\inf(r_\Lambda)}\right]
=: C(\chi,\Lambda).
$$
\item Since $2\xi+2\chi(t_0-t_1)\geq 2\sqrt{\ve}-4Q(x_0)\geq \sqrt{\ve}$, we have   
$$
1-e^{-2 \xi -2 \chi (t_0 - t_1)}\geq 1-e^{-\sqrt{\ve}}\geq\tfrac{1}{2}\sqrt{\ve}.
$$
\end{enumerate}
We now estimate $I_2-I_1$ from above. We have
$$
I_2-I_1=4e^{2\rho}\underbrace{\int_0^{t_1}e^{2\chi t}\left(\|\Phi^t\varrho\|^2-\|\Phi^t v_0\|^2\right)dt}_{=: I_3}
+4e^{2\rho}\underbrace{\int_{t_0}^{t_1}e^{2\chi t}\|\Phi^t v_0\|^2dt}_{=:I_4}
$$
and:
\begin{enumerate}[$\circ$]
\item Estimate of $I_3$: noticing that
\begin{align*}
&\left|\|\Phi^t\varrho\|-\|\Phi^tv_0\|\right|=
\left|\|\widetilde{\Phi^t\varrho}\|-\|\widetilde{\Phi^tv_0}\|\right|\leq 
\left\| \widetilde{\Phi^t\varrho}-\widetilde{\Phi^tv_0}\right\|\\
&=\left\|\Phi^t\varrho-\Phi^t P_{y_0,f^{-1}(x_1)}v_0\right\|\leq 
\left\|\Phi^t\right\|\cdot \left\|\varrho-P_{y_0,f^{-1}(x_1)}v_0\right\|
\end{align*}
is bounded by $e^{2\rho}\tfrac{1}{2}\eta_0^{14\beta/48}\ll \ve^{3/2}$, that $\|\varrho\|\in \left[\tfrac{1}{2},2\right]$, and that
$\|\Phi^t\varrho\|+\|\Phi^tv_0\|\leq 3e^{2\rho}$,
we get that
\begin{align*}
|I_3|\leq 3e^{2\rho}\int_0^{t_1}e^{2\chi t}\left|\|\Phi^t\varrho\|-\|\Phi^tv_0\|\right|dt\leq 
3e^{2\rho}\rho e^{4\chi\rho}\ve^{3/2}\ll \ve.
\end{align*}
\item Estimate of $I_4$:
$$
|I_4|\leq |t_1-t_0|e^{2\chi\rho} e^{4\rho}\leq 2e^{2\chi\rho+4\rho}Q(x_0)\ll \ve.
$$
\end{enumerate}
Plugging the estimates together, we conclude that
$$
I_2-I_1e^{-2 \xi -2 \chi (t_0 - t_1)}\geq 
C(\chi,\Lambda)\tfrac{1}{2}\sqrt{\ve}-8e^{2\rho}\ve \geq \ve^{2/3}.
$$

Now we estimate the right hand side in \eqref{estimate-for-ratio}. Since $\vertiii{\varrho} = \|C(f^{-1}(x_1))^{-1} \varrho \| \leq 2 \| C(f^{-1}(x_1))^{-1} \| \leq 4 \| C(x_0)^{-1} \|$, it follows that
$$
8 Q(x_0)^{\beta/4} \vertiii{\varrho}^2 \leq 128 \ve^{3/2}\| C(x_0)^{-1} \|^{-12} \cdot \| C(x_0)^{-1} \|^2 \leq 128\ve^{3/2}\ll \ve.
$$
This completes the proof of \eqref{estimate-for-ratio}, and hence of the lemma. 
\end{proof}

As we just proved, the improvement lemma as stated above consists on an estimate of the ratio of Lyapunov inner norms, and its proof relies on estimating the ratios of the fractions $\tfrac{\vertiii{\varrho}}{\vertiii{w_0}}$ and $\tfrac{\vertiii{v_0}}{\vertiii{\varrho}}$. The first one is very close to one because of the overlap condition. The second one is the ratio that gives the improvement, and its estimate is {\em purely dynamical}, in the sense that it does not depend on the overlap condition. The only fact that we use is that both $\vertiii{v_0}$ and $\vertiii{\varrho}$ are finite, which means that along these directions the flow contracts. 

Therefore, if we consider an actual orbit of the flow (instead of an edge), we can obtain improvements for the Lyapunov inner norms associated to any $\chi'\in(0,\chi)$. This important fact was first implemented in \cite{Ova20}, and later in \cite{ALP}. To properly state it, we need to recall some notation.
Let $W^s = V^s[\un w^+]$, where $\un w^+ = \{\Psi_{y_n}^{q^s_n, q^u_n}\}_{n \geq 0}$ is a positive $\ve$--gpo. For each $n\geq 0$, write $g_{y_n}^+ = \vf^{T_n}$ where $T_n: B_{y_n} \to \R$ is a $C^{1+\beta}$ function satisfying $T_n(y_n) = r_\Lambda(y_n)$, let  
$G_n:= g_{y_{n-1}}^+ \circ \cdots \circ g_{y_0}^+$ with $G_0={\rm Id}$, and $\tau_n: W^s \to \R$ by
$$
\tau_n(x) := \sum_{k=0}^{n -1} T_{k}(G_k(x)),
$$
which represents the total flow time of the point $x$ under the maps $g_{y_0}^+,\ldots,g_{y_{n-1}}^+$.
Fix $x \in W^s \cap {\rm NUH}^\#$, let $\Tau := \{\tau_n(x)\}_{n \geq 0}$ and introduce the parameter $p^s_n := p(x, \Tau, n)$ as defined in Section \ref{section-Z-indexed}.  
Writing $x_n := G_n(x)$, for each $\delta < 1$ consider ${\un v}_{\delta}^+ := \{\Psi_{x_n}^{\delta p^s_n, \delta p^u_n}\}_{n \geq 0}$, which is a sequence of $\ve$-double charts (at this point, the choice of $p^u_n$ is irrelevant). It is very unlikely that ${\un v}_{\delta}^+$ is an $\ve$--gpo, because condition (GPO2) is hardly satisfied with the inclusion of the multiplicative constant $\delta$ (as already observed in Section \ref{ss.graph.transform}), and also because we did not even define $p^u_n$.
Nevertheless, $\{\delta p^s_n\}_{n\geq 0}$ satisfies the conditions
of \cite[Appendix A]{ALP} needed to define stable graph transforms, therefore we can define $V^s[{\un v}^+_\delta]$ to be the stable manifold associated to the sequence $\un v^+_\delta$ via the stable graph transforms.

Similarly, if $x \in W^u \cap {\rm NUH}^\#$, then for every $\delta$ we can define an unstable manifold $V^u[{\un v}_{\delta}^-]$. Recalling that $\vertiii{\cdot}_{\chi'}$ denotes the Lyapunov inner product defined by $\chi'$, 
we are ready to state the corollary.

\begin{corollary}\label{Corollary-improvement}
The following holds for all $\ve>0$ small enough. Given $\chi'\in (0,\chi)$, if $\delta>0$ is small enough in the above notation, then the following statement holds: 
for $y\in V^s[\un v_\delta^+]$, let $v_1\in T_{g_{x_0}^+(y)}V^s[\sigma(\un v_\delta^+)]$
and $v_0={d(g_{x_0}^+)}^{-1}v_1\in T_{y} V^s[\un v_\delta^+]$, and let also $w_0\in T_{x}V^s[\un v_\delta^+]$ s.t. $v_0=\Theta^s_{x,y}(w_0)$
and $w_1\in T_{g_{x_0}^+(x)}V^s[\sigma(\un v_\delta^+)]$ s.t.  $v_1=\Theta^s_{g_{x_0}^+(x),g_{x_0}^+(y)}(w_1)$;
if $\frac{\vertiii{v_1}_{\chi'}}{\vertiii{w_1}_{\chi'}}={\rm exp}[\pm\xi]$
for $\xi\geq {\sqrt{\ve}}$, then 
$$
\frac{\vertiii{v_0}_{\chi'}}{\vertiii{w_0}_{\chi'}}={\rm exp}\left[\pm(\xi-Q( x)^{\beta/4})\right].
$$
An analogous statement holds for unstable manifolds.
\end{corollary}

Observe that the improvement is the same as Lemma \ref{improvement-lemma}, i.e. $Q(x)$ is defined in terms of the parameter $\chi$. 

\begin{proof}
As in Lemma \ref{improvement-lemma}, it is enough to prove the result
for $V^s[\un v^+_\delta]$. Again, we focus on one side of the estimate: assuming that
$\tfrac{\vertiii{v_1}_{\chi'}}{\vertiii{w_1}_{\chi'}}\leq e^\xi$ for $\xi\geq\sqrt{\ve}$,
we will prove that
$$
\tfrac{\vertiii{v_0}_{\chi'}}{\vertiii{w_0}_{\chi'}}\leq {\rm exp}\left[\xi-Q(x)^{\beta/4}\right].
$$
Repeating the calculation made in Lemma \ref{improvement-lemma} to estimate $\tfrac{\vertiii{v_0}}{\vertiii{\varrho}}$, we have
$$
\frac{\vertiii{v_0}_{\chi'}^2}{\vertiii{w_0}_{\chi'}^2}
\leq \frac{I_1+ e^{2 \chi' t_0} e^{2 \xi} \vertiii{w_1}_{\chi'}^2}{I_2 + e^{2 \chi' t_1} \vertiii{w_1}_{\chi'}^2}
= e^{2 \xi + 2 \chi' (t_0 - t_1)} \left( 1 - \frac{I_2 - I_1e^{-2 \xi - 2 \chi'(t_0 - t_1)}}{\vertiii{w_0}_{\chi'}^2} \right).
$$
Therefore, it is enough to prove that
\begin{equation}\label{estimate-for-ratio-Real-orbit}
I_2-I_1e^{-2 \xi -2 \chi'(t_0 - t_1)} \geq 8 Q(x)^{\beta/4} \vertiii{w_0}_{\chi'}^2.
\end{equation}
We begin estimating the left-hand side, which we write as 
$(I_2-I_1)+I_1\left[1-e^{-2 \xi -2 \chi' (t_0 - t_1)}\right]$.
As in the proof of Lemma \ref{improvement-lemma}, we have the following estimates for $\ve > 0$ small enough:

\begin{enumerate}[$\circ$]
\item $I_1$ has the uniform lower bound
$$
I_1\geq
\tfrac{2}{1-\chi'}\left[1-e^{2(\chi'-1)\inf(r_\Lambda)}\right]
=: C(\chi',\Lambda).
$$
\item Since $2\xi+2\chi'(t_0-t_1)\geq 2\sqrt{\ve}-4Q(x)\geq \sqrt{\ve}$, we have   
$$
1-e^{-2 \xi -2 \chi' (t_0 - t_1)}\geq 1-e^{-\sqrt{\ve}}\geq\tfrac{1}{2}\sqrt{\ve}.
$$
\end{enumerate}
We now estimate $I_2-I_1$ from above. We have
$$
I_2-I_1=4e^{2\rho}\underbrace{\int_0^{t_1}e^{2\chi' t}\left(\|\Phi^t w_0\|^2-\|\Phi^t v_0\|^2\right)dt}_{=: I_3}
+4e^{2\rho}\underbrace{\int_{t_0}^{t_1}e^{2\chi' t}\|\Phi^t v_0\|^2dt}_{=:I_4}.
$$
As in the proof of Lemma \ref{improvement-lemma}, we obtain that $|I_3|,|I_4|\ll \ve$.
Plugging the estimates together, we get that
$$
I_2-I_1e^{-2 \xi -2 \chi' (t_0 - t_1)}\geq 
C(\chi',\Lambda)\tfrac{1}{2}\sqrt{\ve}-8e^{2\rho}\ve \geq \ve^{2/3}.
$$

Now we estimate the right hand side in \eqref{estimate-for-ratio-Real-orbit}. Since $\vertiii{w_0}_{\chi'}\leq 
\vertiii{w_0}_{\chi}=\|C(x)^{-1} w_0 \|\leq 2\|C(x)^{-1} \|$, it follows that
$$
8 Q(x)^{\beta/4} \vertiii{w_0}_{\chi'}^2 \leq 32 \ve^{3/2}\| C(x)^{-1} \|^{-12} \cdot \| C(x)^{-1} \|^2 \leq 32\ve^{3/2}\ll \ve.
$$
This completes the proof of \eqref{estimate-for-ratio-Real-orbit}, and hence of the corollary. 
\end{proof}

\subsection{Proof that $ x\in{\rm NUH}$}\label{Section-finite}

The first step in the proof of Theorem \ref{Thm-inverse} is to show that
$ x\in{\rm NUH}$. For that, we first prove that the relevance of each symbol of the alphabet
$\mathfs A$ implies the existence of stable/unstable manifolds where $S/U$ are bounded (recall their definition in Section \ref{Def-NUH3}).
The result below is \cite[Lemma 6.6]{ALP} adapted to our context, which in turn was inspired by \cite[Lemma 4.2 and Corollary 4.3]{Ova20}. 

\begin{lemma}\label{Lemma-finite-norm-1}
Let $ W^s= V^s[\un w^+]$, where $\un w^+$ is a positive $\ve$--gpo.
If $ W^s\cap {\rm NUH}^\#\neq\emptyset$ then
$$
\sup_{y \in W^s}s(y)<\infty.
$$
The same applies to negative $\ve$--gpo's, with respect to the function $u$.
\end{lemma}

\begin{proof}
By symmetry, we just need to prove the statement for positive $\ve$--gpo's.
For $\chi'\in(0,\chi)$, represent
$S_{\chi'}(\cdot,v)$ by $\vertiii{v}_{\chi'}$ and $S(\cdot,v)$ by $\vertiii{v}$.  We note the following straightforward statement.

\medskip
\noindent
{\sc Claim 1:} For every $v$, it holds
$\vertiii{v}=\sup_{\chi'<\chi}\vertiii{v}_{\chi'}$.
In particular, if $\vertiii{v}_{\chi'}\leq L$ for every $v\in T W^s$ with $\|v\|=1$ and every $\chi'\in (0,\chi)$,
then $\sup\limits_{ y\in  W^s}s( y)\leq L$.

\medskip
By assumption, there is $ x\in W^s\cap {\rm NUH}^\#$.
Write $ x_n= G_n(x)$, $p^s_n=q^s( G_n( x))$ and
$p^u_n=q^u(G_n( x))$. We have $\limsup\limits_{n\to\infty}q( x_n)>0$
and so there is $q>0$ and an increasing sequence
$\{n_k\}_{k\geq 0}$ s.t. $q( x_{n_k})\geq q$ for all $k\geq 0$. Using that $q(y) < Q(y)$ and estimates \eqref{estimates-Q}, for every $ y\in{\rm NUH^\#}$ we have
$\|C( y)^{-1}\|<\ve^{1/8}Q( y)^{-\beta/48}<q( y)^{-\beta/48}$ and
so in particular $Q( x_{n_k})>q$ and $\|C( x_{n_k})^{-1}\|<q^{-\beta/48}$ for every $k\geq 0$.

\medskip
Fix $\chi'\in (0,\chi)$, and let $\overline{\chi}=\tfrac{\chi+\chi'}{2}$.
Also, let $\un v=\{v_n\}_{n\in\Z}$ with $v_n=\Psi_{ x_n}^{\delta p^s_n,\delta p^u_n}$, 
where $\delta$ is a positive constant depending on $\chi'$ and $\chi$ satisfying 
$$
\delta^{\beta/3}<\inf\left\{e^{-t\overline{\chi}}-e^{-t\chi}:t\in[\inf(r_\Lambda), 2\rho] \right\}.
$$
Write $ V^s_{\delta,n}= V^s[\{v_\ell\}_{\ell\geq n}]$.
We want to bound $\vertiii{v}_{\chi'}$, uniformly in $v$ with $\|v\|=1$ and $\chi'$.
Instead of $ W^s$, doing this for $V^s_{{\delta},0}$ is simpler,
since inside $V^s_{{\delta},0}$ we have better estimates, as we will see in Claim 3.
Although $ W^s$ is in general not contained in $ V^s_{\delta,0}$, if
$n$ is large then $ G_n( W^s)\subset  V^s_{\delta,n}$.

\medskip
\noindent
{\sc Claim 2:} If $n$ is large enough then $G_n( W^s)\subset  V^s_{\delta,n}$.

\begin{proof}
Since $B^{d_s (x_n)} [\delta p^s_n] \times B^{d_u (x_n)} [\delta p^s_n] \supset B [\delta p^s_n]$, it is enough to prove that if $n$ is large then $G_n (W^s) \subset \Psi_{x_n} (B[\delta p^s_n])$. By Theorem \ref{Thm-stable-manifolds}(3), $G_n(W^s) \subset B(x_n, e^{-\frac{\chi}{2}\inf (r_\Lambda) n})$. Since
$$\Psi_{x_n} (B[\delta p^s_n]) \supset B\left(x_n, \frac{1}{2} \| C(x_n)^{-1} \|^{-1} \delta p^s_n \right),$$
it is enough to prove that for $n$ large enough it holds
$$\frac{1}{2} \| C(x_n)^{-1} \|^{-1} \delta p^s_n > e^{- \frac{\chi}{2}\inf (r_\Lambda)n} \iff  e^{- \frac{\chi}{2}\inf (r_\Lambda)n} \| C(x_n)^{-1} \|\tfrac{1}{p^s_n} < \frac{\delta}{2}\cdot$$
Using that $\| C(x_n)^{-1} \| (p^s_n)^{\beta/48} \leq \| C(x_n) ^{-1} \| Q(x_n)^{\beta/48} < 1$ and that $p^s_n \geq e^{-\ve \sup(r_\Lambda)n} p_0^s$, we obtain that
\begin{align*}
    e^{- \frac{\chi}{2}\inf(r_\Lambda)n} \|C(x_n)^{-1} \| \tfrac{1}{p^s_n} & < e^{- \frac{\chi}{2}\inf (r_\Lambda)n } \left( \tfrac{1}{p^s_n} \right)^{1+\frac{\beta}{48}} \leq e^{- \frac{\chi}{2}\inf (r_\Lambda)n} \left( \tfrac{e^{\ve \sup(r_\Lambda)n}}{p^s_0} \right)^{1+\frac{\beta}{48}} \\
    & = \left( \tfrac{1}{p^s_0} \right)^{1 + \frac{\beta}{48}} e^{-\left[ \frac{\chi}{2}\inf(r_\Lambda) - \ve \sup(r_\Lambda)\left(1 + \frac{\beta}{48}\right) \right] n}
\end{align*}
which, for $\ve > 0$ small enough, converges to zero exponentially fast.
\end{proof}

For $n,\ell\geq 0$, let $G_n^\ell:=g_{x_{n+\ell-1}}^+\circ\cdots\circ g_{x_n}^+$ and $\tau_n^\ell(y):= \sum_{k=n}^{n+\ell -1} T_{k}(G_k(y))$ when defined.

\medskip
\noindent
{\sc Claim 3:} 
For every $v\in T V^s_{\delta,n}$ with $\| v \|=1$
and every $\ell\geq 0$ it holds
$$
\|{dG_n^\ell}v\|\leq 8\|C(x_n)^{-1}\|e^{-\overline{\chi} \tau_n^\ell (x)}.
$$

\begin{proof}
Recalling the definition of $\delta$, proceed exactly as in the proof of \cite[Corollary 4.12]{ALP} where, in view of 
Theorem \ref{Thm-non-linear-Pesin}(2), the inequality (4.2) in \cite{ALP} is substituted by the stronger one
\begin{align*}
&\, \|w_k\|  \leq \left[e^{-\chi r_\Lambda (x_{k-1})}+4\ve(\delta p^s_{k-1})^{\beta/3}\right]\|w_{k-1}\|
\leq\left[e^{-\chi r_\Lambda (x_{k-1})}+\delta^{\beta/3}\right]\|w_{k-1}\| \\
& \leq e^{-\overline{\chi}r_\Lambda(x_{k-1})} \| w_{k-1} \| =e^{-\overline{\chi}T_{k-1}(x_{k-1})} \| w_{k-1} \|,
\end{align*}
thus establishing the result.
\end{proof}

Now we complete the proof of the lemma for $ W^s$.  We fix $v\in T_z V^s_{\delta,n}$ with $\|v\|=1$. For $\ell\geq 0$, write $\tau_\ell:=\tau_n^\ell(z)$ and $\overline{\tau}_\ell:=\tau_n^\ell(x)$, with $\tau_0=\overline{\tau}_0=0$. 
Observing that $\Phi^{\tau_\ell}(v)=dG_n^\ell v$, Claim 3 implies that
\begin{align*}
    &\, \vertiii{v}_{\chi'}^2=4 e^{2\rho}\sum_{\ell\geq 0}\int_{\tau_\ell}^{\tau_{\ell+1}} e^{2\chi't}\|\Phi^t(v)\|^2dt\overset{!}{\leq} 8\rho e^{12\rho} \sum_{\ell\geq 0}e^{2 
\chi' \tau_\ell}  
\|dG_n^{\ell}v\|^2\\
&\leq 8\rho e^{12\rho}\sum_{\ell\geq 0}e^{2\chi'\tau_\ell} \left(64\|C( x_n)^{-1}\|^2e^{-2\overline{\chi}\, \overline{\tau}_\ell}\right)
=512\rho e^{12\rho}\|C( x_n)^{-1}\|^2\sum_{\ell\geq 0}e^{2\chi'(\tau_\ell-\overline{\tau}_\ell) + (\chi'-\chi)\overline{\tau}_\ell},
\end{align*}
where in $\overset{!}{\leq}$ we used estimate \eqref{Norm of ILPF} to get that
\begin{align*}
&\,\int_{\tau_\ell}^{\tau_{\ell+1}} e^{2\chi't}\|\Phi^t(v)\|^2dt
=
e^{2\chi'\tau_\ell}\int_{0}^{\tau_{\ell+1}-\tau_\ell} e^{2\chi't}\|\Phi^t(dG_n^\ell v)\|^2dt\\
&\leq
e^{2\chi'\tau_\ell}\cdot \left(2\rho e^{4\chi'\rho}e^{6\rho}\|dG_n^\ell v\|^2\right)
\leq 2\rho e^{10\rho}e^{2\chi'\tau_\ell}
\|dG_n^\ell v\|^2.
\end{align*}
We estimate the sum $\sum_{\ell\geq 0}e^{2\chi'(\tau_\ell-\overline{\tau}_\ell) + (\chi'-\chi)\overline{\tau}_\ell}$
as follows:
\begin{enumerate}[$\circ$]
\item $\tau_\ell-\overline{\tau_\ell}$ has a uniform upper bound: since $T_k$ is 1--Lipschitz (Lemma \ref{lemma-local-coord}(3)) and the distance $d(G_n^k(x_n),G_n^k(z))$ goes to zero exponentially fast (Theorem \ref{Thm-stable-manifolds}(3)), we have 

\begin{align*}
&\ |\tau_\ell-\overline{\tau}_\ell|\leq \sum_{k=0}^{\ell-1}|T_k(G_n^k(x_n))-T_k(G_n^k(z))| \leq 2Q(x_n)\sum_{k=0}^{\ell-1}e^{-\frac{\chi}{2}\inf(r_{\wh \Lambda})k}\\
&\leq
2Q(x_n)\sum_{k\geq 0}e^{-\frac{\chi}{2}\inf(r_{\wh \Lambda})k}=:T'.
\end{align*}

\item $\sum\limits_{\ell\geq 0}e^{(\chi'-\chi)\overline{\tau}_\ell}\leq\sum\limits_{\ell\geq 0}e^{(\chi'-\chi)\overline{\tau}_\ell}\leq \sum\limits_{\ell\geq 0}e^{-(\chi-\chi')\inf(r_{\wh \Lambda})\ell}=\tfrac{1}{1-e^{-(\chi-\chi')\inf(r_{\wh \Lambda})}}$.
\end{enumerate}
Therefore, for $n=n_k$ we have
$$
\vertiii{v}_{\chi'}^2 \leq \frac{512\rho e^{12\rho}}q^{-\beta/24}e^{2T'}{1-e^{-(\chi-\chi')\inf(r_{\wh \Lambda})}}\,\cdot
$$
Call this bound $L^2$, so that $\vertiii{v}_{\chi'}\leq L$ for every $v\in T V^s_{\delta,n_k}$
with $\|v\|=1$. Note that  $L\to\infty$ as $\chi'\to \chi$, so  the proof is not  complete. To obtain a bound that does not depend on $\chi'$, we will improve the above estimate by applying Corollary \ref{Corollary-improvement}.

Define $\xi\geq \sqrt{\ve}$ by $e^\xi=\max\{\sqrt{2}L,e^{\sqrt{\ve}}\}$.
We claim that
$$
\tfrac{\vertiii{\Theta^s_{x_{n_k},y}(v)}_{\chi'}}{\vertiii{v}_{\chi'}}={\rm exp}[\pm \xi],
\ \ \text{ for all } y\in V^s_{\delta,n_k}\text{ and } v\in N^s_{ x_{n_k}}\backslash\{0\}.
$$
By a normalization, we just need to check this estimate for $\|v\|=1$.
By Lemma \ref{ALP-Lemma 6.1}(1), we have $\tfrac{1}{2}\leq \|\Theta^s_{ x_{n_k}, y}(v)\|\leq 2$
and so 
$$
(\sqrt{2}L)^{-1}=\tfrac{\sqrt{2}/2}{L}\leq\tfrac{\vertiii{\Theta^s_{ x_{n_k}, y}(v)}_{\chi'}}{\vertiii{v}_{\chi'}}
\leq\tfrac{2L}{\sqrt{2}}=\sqrt{2} L.
$$
Now fix $k \geq 1$. Apply Corollary \ref{Corollary-improvement} along the path
$ x_{n_{k-1}}\to\cdots\to  x_{n_k}$.
Since the ratio does not get worse for all transitions $ x_\ell\to  x_{\ell+1}$
and it improves a fixed amount in the last edge $ x_{n_{k-1}}\to  x_{n_{k-1}+1}$,
we conclude that
$$
\tfrac{\vertiii{\Theta^s_{ x_{n_{k-1}}, y}(v)}_{\chi'}}{\vertiii{v}_{\chi'}}={\rm exp}\left[\pm (\xi-q^{\beta/4})\right],
\ \ \text{ for all } y\in V^s_{\delta,n_{k-1}}\text{ and } v\in N^s_{ x_{n_{k-1}}}\backslash\{0\}.
$$
Repeating this procedure until reaching $ x_{n_0}$, we obtain
at least $k$ improvements, as long as the ratio remains outside $[{\rm exp}(-\sqrt{\ve}),{\rm exp}(\sqrt{\ve})]$.
Taking $k\to+\infty$, we conclude that
$\tfrac{\vertiii{\Theta^s_{ x_{n_0}, y}(v)}_{\chi'}}{\vertiii{v}_{\chi'}}={\rm exp}\left[\pm \sqrt{\ve}\right]$
for all $ y\in V^s_{\delta,n_0}$ and $v\in N^s_{ x_{n_0}}\backslash\{0\}$.
Similarly, we obtain that for every $k\geq 0$ it holds
\begin{equation}\label{improved-estimate}
\tfrac{\vertiii{\Theta^s_{ x_{n_k}, y}(v)}_{\chi'}}{\vertiii{v}_{\chi'}}={\rm exp}[\pm \sqrt{\ve}],
\ \ \text{ for all } y\in V^s_{\delta,n_k}\text{ and } v\in N^s_{ x_{n_k}}\backslash\{0\}.
\end{equation}
\medskip
Now fix $k$ large enough s.t. $G_{n_k}( W^s)\subset  V^s_{\delta,n_k}$.
Let $v\in T W^s$ with $\|v\|=1$.
If $w=dG_{n_k}v\in T V^s_{\delta,n_k}$ then
\begin{align*}
&\,\vertiii{v}_{\chi'}^2=4 e^{2\rho}\int_0^\infty e^{2\chi't}\|{\Phi}^tv\|^2=4 e^{2\rho}\int_0^{\tau_{n_k}} e^{2\chi't}\|{\Phi}^tv\|^2dt+4 e^{2\rho}\int_{\tau_{n_k}}^\infty e^{2\chi't}\|{\Phi}^tv\|^2dt\\
&\leq 4 e^{2\rho}\int_0^{\tau_{n_k}} e^{2\chi't}\|{\Phi}^tv\|^2dt+ e^{2\chi'\tau_{n_k}}\|w\|^2. \vertiii{\frac{w}{\|w\|}}_{\chi'}^2\\
&\leq \underbrace{4 e^{2\rho}\int_0^{\tau_{n_k}} e^{2\chi t}\|{\Phi}^tv\|^2dt}_{=:\, {\rm I}}+\underbrace{e^{2\chi\tau_{n_k}}\|w\|^2. \vertiii{\frac{w}{\|w\|}}_{\chi'}^2}_{=:\, {\rm II}}.
\end{align*}
Let $y_0$ be the center of the zeroth chart defining $W^s$.
For $t\in[\tau_\ell,\tau_{\ell+1}]$, Theorem \ref{Thm-stable-manifolds}(3) and estimate \eqref{Norm of ILPF} imply that
$$
\|\Phi^tv\|=\|\Phi^{t-\tau_\ell}dG_\ell v\|=
\left\|\Phi^{t-\tau_\ell}\left(\tfrac{dG_\ell v}{\|dG_\ell v\|}\right)\right\|\cdot \|dG_\ell v\|\leq 8e^{3\rho}\|C(y_0)^{-1}\|e^{-\frac{\chi}{2}\tau_\ell}
$$
and so
$$
\int_{\tau_\ell}^{\tau_{\ell+1}}e^{2\chi t}\|\Phi^t v\|^2dt
\leq \rho e^{2\chi \tau_\ell+ 2\rho}\cdot \left(8e^{3\rho}\|C(y_0)^{-1}\|e^{-\frac{\chi}{2}\tau_\ell}\right)^2
=64\rho e^{8\rho}\|C(y_0)^{-1}\|^2e^{\chi\tau_\ell}.
$$
This implies that 
$$
{\rm I }\leq 256\rho e^{10\rho}\|C(y_0)^{-1}\|^2\sum_{\ell=0}^{n_{k-1}}e^{\chi \tau_\ell}.
$$
To estimate II, write $w\in T_{ y} V^s_{\delta,n_k}$ and define
$\varrho$ by the equality $\tfrac{w}{\|w\|}=\Theta^s_{ x_{n_k}, y}(\varrho)$.
By estimate (\ref{improved-estimate}), Lemma \ref{ALP-Lemma 6.1}(1)
and Lemma \ref{Lemma-linear-reduction}(1), we get that
\begin{align*}
&\, \vertiii{\tfrac{w}{\|w\|}}_{\chi'}\leq e^{\sqrt{\ve}}\vertiii{\varrho}_{\chi'}\leq 2e^{\sqrt{\ve}}\vertiii{\tfrac{\varrho}{\|\varrho\|}}_{\chi'}
\leq 2e^{\sqrt{\ve}}s( x_{n_k})\leq 2e^{\sqrt{\ve}}\|C( x_{n_k})^{-1}\|<2e^{\sqrt{\ve}}q^{-\beta/48}.
\end{align*}
Since $\|w\|\leq 8\|C(y_0)^{-1}\| e^{-\overline{\chi}\tau_{n_k}}$ by Claim 3, we conclude that
$${\rm II }\leq 256\|C(y_0)^{-1}\|^2 e^{(\chi-\chi')\tau_{n_k} 
+2\sqrt{\ve}}q^{-\beta/24}\leq 
256\|C(y_0)^{-1}\|^2e^{\chi\tau_{n_k} 
+2\sqrt{\ve}}q^{-\beta/24}
$$
Plugging the estimates of I and II, it follows that
$$
\vertiii{v}_{\chi'}^2\leq 
256\|C(y_0)^{-1}\|^2\left[\rho e^{10\rho}\sum_{\ell=0}^{n_{k-1}}e^{\chi \tau_\ell}+e^{\chi\tau_{n_k} 
+2\sqrt{\ve}}q^{-\beta/24}\right],
$$
which is independent of $v$ and $\chi'$. 
By Claim 1, the proof is complete.
\end{proof}

Now we proceed to show that, in the notation of Theorem \ref{Thm-inverse}, $\pi(x) \in$ NUH. Recall the relevance property of each $\ve$-double chart of $\mathfs A$, see Theorem \ref{Thm-coarse-graining}.

\begin{proposition}\label{Prop-finite-norm-2}
For every $v_0\in\mathfs A$, there exists a constant $L=L(v_0)$ s.t. the following holds.
If $\un v=\{v_n\}_{n\in\Z}\in\Sigma^\#$ satisfies $v_n=v_0$ for infinitely
many $n>0$ and if $x=\pi(\un v)$, then $s(x)<L$.
The same applies to $u(x)$. Furthermore, $x\in{\rm NUH}$. 
\end{proposition}

\begin{proof}
We continue representing $S(\cdot,v)$ by $\vertiii{v}$.
Write $V^s=V^s[\un v]$, $v_n=\Psi_{x_n}^{p^s_n,p^u_n}$,
$\eta_n=p^s_n\wedge p^u_n$, $d_{s/u}=d_{s/u}(x_n)$ for all $n\in\Z$, and let $\{n_k\}_{k\geq 1}$ be an increasing
sequence s.t. $v_{n_k}=v_0$. Since $v_0$ is relevant, there is $\un w\in\Sigma$
with $w_0=v_0$ s.t. $\pi(\un w)\in {\rm NUH}^\#$. By Lemma \ref{Lemma-finite-norm-1},
if $W^s=V^s[\un w]$ then $\sup\limits_{y\in W^s}s(y)<\infty$.
By Lemma \ref{ALP-Lemma 6.1}(1),
$$
L_0:=\sup\left\{\tfrac{\vertiii{\Theta^s_{x_0,y}(v)}}{\vertiii{v}}:
y\in W^s, y\in\Psi_{x_0}(B^{d_s}[\eta_0]\times B^{d_u}[\eta_0]),v\in N^s_{x_0}\setminus\{0\}
\right\}
$$
is also finite. Define $L_1=\max\{L_0,e^{\sqrt{\ve}}\}>1$.

For each $k\geq 1$, we have $W^s\in {{\mathfs M}}^s(v_{n_k})$.
Starting from $W^s$, apply the stable graph transform along the path
$v_0\overset{\ve}{\to}v_1\overset{\ve}{\to}\cdots\overset{\ve}{\to}v_{n_k}$
to obtain an $s$--admissible manifold at $v_0$, call it $W^s_k$.
Let $F$ be the representing function of $V^s$, and let
$F_k$ be the representing function of $W^s_k$. Since the 
convergence to $V^s$ occurs in the $C^1$ topology, 
$\|{F_k-F}\|_{C^1}\xrightarrow[k\to\infty]{}0$.

\medskip
\noindent
{\sc Claim:} If $\{w_k\}_{k\geq 1}\subset TM$ converges to $w\in TM$ in the Sasaki metric, then
$$
\int_0^\infty e^{2\chi t}\|\Phi^t w\|^2 dt\leq \liminf_{k \rightarrow +\infty}\int_0^\infty e^{2\chi t}\|\Phi^t w_k\|^2 dt.
$$

\begin{proof}[Proof of the claim.]
Define $f,f_k:[0,\infty)\to\R$ by $f(t) = e^{2 \chi t} \|\Phi^t w\|^2$, $f_k (t) = e^{2 \chi t} \| \Phi^t w_k \|^2$ for $k\geq 0$. Since $\{w_k\}_{k\geq 1}$ converges to $w$ and $\Phi^t$ is continuous, $f_k(t)$ converges to $f(t)$ for every $t\geq 0$. By the Fatou lemma, the claim follows.
\end{proof}

Write $x=\Psi_{x_0}(z,F(z))$. Since $x=\pi(\un v)$, 
by Lemma \ref{Lemma-admissible-manifolds} we have that $\|z\|<\eta_0$.
For each $k\geq 1$,
define $y_k$ as the unique element of $W^s_k$ s.t. $y_k=\Psi_{x_0}(z,F_k(z))$.
Fix $v\in N^s_{x_0}$ with $\|{v}\|=1$.
If $v=(d\Psi_{x_0})_0\begin{bmatrix} w \\ 0\end{bmatrix}$ then
$$
\Theta^s_{x_0,x}(v)=(d\Psi_{ x_0})_{(z,F(z))}\begin{bmatrix} w \\ (dF)_z w\end{bmatrix}\ \text{ and }\
\Theta^s_{x_0, y_k}(v)=(d\Psi_{x_0})_{(z,F_k(z))}\begin{bmatrix} w \\ (dF_k)_z w\end{bmatrix}
$$
and so $\Theta^s_{x_0, y_k}(v)\to \Theta^s_{x_0,x}(v)$ in the Sasaki metric. By the claim,
$$
\vertiii{\Theta^s_{x_0,x}(v)}\leq \liminf_{k\to+\infty}\vertiii{\Theta^s_{x_0, y_k}(v)}
$$
and so it is enough to bound the right hand side in the above inequality.

We claim that $G_n(y_k)\in \Psi_{x_n}(B^{d_s}[\eta_n]\times B^{d_u}[\eta_n])$
for $n=0,\ldots,n_k$. To prove this, note that $G_n(y_k)$ belongs to a stable set
and so it is enough to show that, in the charts representation, the first coordinate of
$G_n(y_k)$ belongs to $B^{d_s}[\eta_n]$. The case $n=0$ is true
because $\|z\|<\eta_0$. The proof is by induction, so we just show how to obtain it for $n=1$.
Write $f^+_{x_0,x_1}=D+H$,
where $D=\begin{bmatrix}D_s & 0 \\ 0 & D_u\end{bmatrix}$ is given by Lemma \ref{Lemma-linear-reduction}(2) and 
$H=(H_1,H_2)$ satisfies Theorem \ref{Thm-non-linear-Pesin-2}(2).
We have $y_k=\Psi_{x_0}(z,F_k(z))$ and so
$g_{x_0}^+(y_k)=\Psi_{x_1}(\overline{z},*)$ where
$\overline{z}=D_sz+H_1(z,F_k(z))$. Then
\begin{align*}
&\,\|{\overline z}\|  \leq \|{D_sz}\| +\|{H_1(z,F_k(z))} \| 
\leq \|{D_sz}\| +\|{H(z,F_k(z))}\| \\
&\leq \| {D_s}\| \|{z}\|+\|{H(0,0)}\|+
\|{dH}\|_{C^0(B[2\eta_0])} \|{(z,F_k(z))}\| \\
&\leq e^{-\chi r_\Lambda(x_0)}\eta_0+\ve \eta_0+\ve(3\eta_0)^{\beta/3}2\eta_0 
\leq (e^{-\chi \inf(r_\Lambda)}+2\ve)\eta_0 \\
& \leq (e^{-\chi\inf(r_\Lambda)}+2\ve)e^{\ve}\eta_1 
\end{align*}
is smaller than $\eta_1$ for $\ve>0$ small enough.
\medskip
Now, for fixed $k\geq 1$ write $\varrho_k=\Theta^s_{x_0,y_k}(v)$ and define $w_\ell\in N^s_{x_\ell}$
by the equality $\Theta^s_{x_\ell,G_\ell(y_k)}(w_\ell)={dG_\ell}(\varrho_k)$,
for all $\ell\geq 0$.
Since $G_{n_k}(y_k)\in G_{n_k}(W^s_k)\subset W^s$, we have
that $\tfrac{\vertiii{{dG}_{n_k}(\varrho_k)}}{\vertiii{w_{n_k}}}\leq L_1$.
By Lemma \ref{improvement-lemma}, we get that
$\tfrac{\vertiii{{dG}_{n_k-1}(\varrho_k)}}{\vertiii{w_{n_k-1}}}\leq L_1$ and, repeating this
procedure, that $\tfrac{\vertiii{\varrho_k}}{\vertiii{w_{0}}}\leq L_1$.
By Lemma \ref{ALP-Lemma 6.1}(1), $\vertiii{\varrho_k}\leq L_1\vertiii{w_0}\leq 2L_1 s(x_0)$
and so
\begin{equation}\label{estimate-finite-norm}
\vertiii{\Theta^s_{x_0,x}(v)}\leq 2L_1 s(x_0).
\end{equation}
Defining $L=L(v_0):=3L_1s(x_0)$ and applying Lemma \ref{ALP-Lemma 6.1}(1) again,
it follows that $\vertiii{v}<L$ for all $v\in T_{x}V^s$ with $\|{v}\|=1$,
and so $s(x)<L$.

\medskip
The proof for $V^u[\un v]$ is identical. This then proves that $x$ satisfies (NUH3). 
We now verify property (NUH1).
Define $N^s_{x}=T_{x}(V^s[\un v])$ and
$N^u_{x}=T_{x}(V^u[\un v])$.
The first condition of (NUH1) is proved as follows. Let $t_n = r_n (\un v)$. For $n \geq 0$, we have $0 \leq -t_{-n} \leq 2 n \sup(r_\Lambda)$, hence it is enough to prove that $\liminf\limits_{n\rightarrow\infty}\frac{1}{n}\log\| \Phi^{t_{-n}} v \| >0$ for all $v \in N^s_x \backslash \{0 \}$, which follows from the third estimate of Theorem \ref{Thm-stable-manifolds}(3).

Now we focus on the second condition of (NUH1). Differently from the case of diffeomorphisms, this is not a straightforward consequence of $s(x)<\infty$, but almost.
We claim that $e^{\chi t} \|\Phi^tv\| \rightarrow 0$ as $t \to + \infty$. Once this is proved, it is clear that  
$\limsup\limits_{t \to+\infty} \frac{1}{t} \log \|\Phi^t v \| \leq - \chi$.
By contradiction, suppose that $\lim\limits_{t \to +\infty} e^{\chi t} \| \Phi^t v \| \neq 0$. Then there is $C>0$ and a sequence $t_k \rightarrow +\infty$ such that $e^{\chi t_k} \| \Phi^{t_k} v \| > C$ for all $k>0$. Assuming that $t_{k+1}>t_k+\rho$, we have
$$
\int_{0}^\infty e^{2 \chi t} \| \Phi^t v\|^2dt \geq \sum_{k\geq 0} \int_{t_k}^{t_k+\rho} e^{2\chi t} \| \Phi^t v\|^2dt\geq \sum_{k\geq 0}\rho e^{2\chi t_k}e^{-4\rho}\|\Phi^{t_k}v\|^2>\rho e^{-4\rho}\sum_{k\geq 0}C^2=\infty 
$$
which contradicts the fact that $s(x)<\infty$. This completes the proof of (NUH1). The proof of (NUH2) is identical.
\end{proof}

\subsection{Control of $d$ and $C^{-1}$}
The control of $d$ follows by Lemma \ref{Lemma-admissible-manifolds} and by recalling that Pesin charts are 2--Lipschitz. This proves part (1) of Theorem \ref{Thm-inverse}.
Now we prove part (2). Recall that $x=\pi(\un v)$ for $\un v\in\Sigma^\#$.
We proved in the last section that $x\in {\rm NUH}$, 
i.e. there is a splitting $N_x=N^s_{x}\oplus N^u_{x}$
satisfying (NUH1)--(NUH3). To control $C^{-1}$, we need to control the Lyapunov
inner product for vectors in $N^s_{x}$ and $N^u_{x}$. We explain
how to make the control in $N^s_{x}$ (the control in $N^u_{x}$ is analogous).
Write $\un v=\{v_n\}_{n\in\Z}$ and $\Theta_n=\Theta_{x_n,G_n(x)}$. Without loss of generality,
assume that $v_{0}$ repeats infinitely often in the future, i.e. there is an increasing sequence
$\{n_k\}_{k\geq 1}$ s.t. $v_{n_k}=v_0$ for all $k\geq 1$.
Let $L=3L_1s(x_0)$ as in the proof of Proposition \ref{Prop-finite-norm-2},
and let $\xi>0$ s.t. $L=e^{\xi}$. Since $L_1\geq e^{\sqrt{\ve}}$ and $s(x_0)\geq \sqrt{2}$,
we have $\xi>\sqrt{\ve}$. We claim that
\begin{equation}\label{estimate-for-improvement}
\frac{\vertiii{v}}{\vertiii{\Theta_{n_k}(v)}}={\rm exp}[\pm \xi],
\ \text{ for all }v\in N^s_{x_0}\backslash\{0\}.
\end{equation}
By a normalization, we just need to check this for $\|{v}\|=1$.
We proved in Proposition \ref{Prop-finite-norm-2}
that $\vertiii{\Theta_{n_k}(v)}\leq 2L_1s(x_0)$, see estimate
(\ref{estimate-finite-norm}). On one hand, applying Lemma \ref{ALP-Lemma 6.1}(1) we have
$\tfrac{\vertiii{v}}{\vertiii{\Theta_{n_k}(v)}}\leq 2s(x_0)<L$,
and on the other hand
$\tfrac{\vertiii{v}}{\vertiii{\Theta_{n_k}(v)}}\geq \frac{\sqrt{2}}{2L_1s(x_0)}>L^{-1}$,
which proves (\ref{estimate-for-improvement}). Now fix $k\geq 1$.
Using (\ref{estimate-for-improvement}), apply
Lemma \ref{improvement-lemma} along the path
$v_{n_{k-1}}\overset{\ve}{\to}\cdots\overset{\ve}{\to}v_{n_k}$.
Since the ratio does not get worse for all edges $v_{\ell}\overset{\ve}{\to}v_{\ell+1}$
and it improves a fixed amount in the last edge $v_{n_{k-1}}\overset{\ve}{\to}v_{n_{k-1}+1}$,
we conclude that
$\tfrac{\vertiii{v}}{\vertiii{\Theta_{n_{k-1}}(v)}}={\rm exp}\left[\pm (\xi-Q( x_0)^{\beta/4})\right]$
for all $v\in N^s_{x_0}\backslash\{0\}$. Repeating this procedure until reaching $v_0$, we obtain
at least $k$ improvements, as long as the ratio remains outside $[{\rm exp}(-\sqrt{\ve}),{\rm exp}(\sqrt{\ve})]$.
Taking $k\to+\infty$, we conclude that
$\frac{\vertiii{v}}{\vertiii{\Theta_0(v)}}={\rm exp}[\pm \sqrt{\ve}]$
for all $v\in N^s_{x_0}\backslash\{0\}$.
Since $v_{n_k}=v_0$, we obtain similarly that
$\frac{\vertiii{v}}{\vertiii{\Theta_{n_k}(v)}}={\rm exp}[\pm \sqrt{\ve}]$
for all $v\in N^s_{x_0}\backslash\{0\}$. Finally, given $n\in\Z$, let $n_k>n$ and apply
Lemma \ref{improvement-lemma} along the path $v_{n}\overset{\ve}{\to}\cdots\overset{\ve}{\to}v_{n_k}$
to conclude that
\begin{equation}\label{comparison-C-1}
\frac{\vertiii{v}}{\vertiii{\Theta_n(v)}}={\rm exp}[\pm \sqrt{\ve}],
\ \text{ for all }v\in N^s_{x_n}\backslash\{0\}.
\end{equation}
By a similar argument, we get that
\begin{equation}\label{comparison-C-2}
\frac{\vertiii{v}}{\vertiii{\Theta_n(v)}}={\rm exp}[\pm \sqrt{\ve}],
\ \text{ for all }v\in N^u_{x_n}\backslash\{0\}.
\end{equation}

We now prove part (2) of Theorem \ref{Thm-inverse}.
For simplicity, assume $n=0$ and write $\Theta=\Theta_0$.
If $v=v^s+v^u\in N^s_{x_0}\oplus N^u_{x_0}$,
then $\Theta(v)=\Theta(v^s)+\Theta(v^u)\in N^s_{x}\oplus N^u_{x}$.
By the calculation made in the proof of Lemma \ref{Lemma-linear-reduction}(1) 
and estimates (\ref{comparison-C-1}) and (\ref{comparison-C-2}),
\begin{equation}\label{comparison-C-3}
\frac{\|{C(x_0)^{-1}v}\|^2}{\|{C(x)^{-1}\Theta(v)}\|^2}=
\frac{\vertiii{v^s}^2+\vertiii{v^u}^2}{\vertiii{\Theta(v^s)}^2+\vertiii{\Theta(v^u)}^2}={\rm exp}[\pm 2\sqrt{\ve}]
\end{equation}
and so $\frac{\|C(x_0)^{-1}\|}{\|{C(x)^{-1}\Theta}\|}={\rm exp}[\pm\sqrt{\ve}]$.
By Lemma \ref{ALP-Lemma 6.1}(2), if $\ve>0$ is small enough then
$\|{\Theta^{\pm 1}}\|={\rm exp}[\pm\sqrt{\ve}]$, hence
$\frac{\|C(x_0)^{-1}\|}{\|C(x)^{-1}\|}={\rm exp}[\pm2\sqrt{\ve}]$.

\subsection{Control of $Q, p^s,p^u$ and proof that $x \in {\rm NUH^\#}$}
Now we prove parts (3) and (4) of Theorem \ref{Thm-inverse}. We begin controlling $Q$. As usual, let $n=0$.
Recall that
$$
Q(x)= \ve^{6/\beta}\|{C(x)^{-1}}\|^{-48/\beta}.
$$
By part (2),
$\tfrac{\|C(x_0)^{-1}\|^{-48/\beta}}{\|C(x)^{-1}\|^{-48/\beta}}={\rm exp}\left[\pm \tfrac{96\sqrt{\ve}}{\beta}\right]$.
Hence $\tfrac{Q(x_0)}{Q(x)}={\rm exp}\left[\pm \tfrac{96\sqrt{\ve}}{\beta}\right]$, which is better than the claimed estimate when $\ve>0$ is small enough.

Now we prove Part (4). Once we get this, it follows that
$x\in\nuh^\#$. 
Write  $z_n=\vf^{r_n(\un v)}(x)$.
The control of $p^{s/u}_n$ consists on proving that it is comparable to $p^{s/u}(z_n)$.
To have the control from below, we will use that $\{\Psi_{x_n}^{p^s_n,p^u_n}\}_{n\in\Z}\in\Sigma^\#$
implies that the parameters $p^{s/u}_n$ are almost maximal infinitely often.
Proposition \ref{Prop-Z-par}(3) is the statement of maximality for $p^{s/u}(z_n)$.
The statement for $p^{s/u}_n$ is in the next lemma.
For simplicity of notation, write $T_k=T(v_k,v_{k+1})$.

\begin{lemma}\label{Lemma-max-in-chart}
If $\{\Psi_{x_n}^{p^s_n,p^u_n}\}_{n\in\Z}\in\Sigma^\#$ then
$\min\{e^{\ve T_n}p^s_{n+1},e^{-\ve}\ve Q(x_n)\}=e^{-\ve}\ve Q(x_n)$
for infinitely many $n>0$, and $\min\{e^{\ve T_n}p^u_n,e^{-\ve}\ve Q(x_{n+1})\}=e^{-\ve}\ve Q(x_{n+1})$
for infinitely many $n<0$.
\end{lemma}

\begin{proof}
The strategy is the same used in the proof of Proposition \ref{Prop-Z-par}(3). 
We prove the first statement (the second is identical). By contradiction, assume that
there exists $n\in\Z$ such that $\min\{e^{\ve T_N}p^s_{N+1},e^{-\ve}\ve Q(x_N)\}=e^{\ve T_N}p^s_{N+1}$
for all $N\geq n$. By (GPO2), it follows that $p^s_N\geq e^{\ve(T_N-p^s_N)}p^s_{N+1}$ for all $N\geq n$.
Let $\lambda=\exp{}[\ve^{1.5}]$,
then $\ve(T_N-p^s_N)\geq \ve[\inf(r_{\wh\Lambda})-\ve]>\ve^{1.5}$ when $\ve>0$ is sufficiently small.
Hence $p^s_N> \lambda p^s_{N+1}$ for all $N\geq n$, and so $p^s_n\geq \lambda^{N-n}p^s_N$
for all $N\geq n$. This is a contradiction, since $p^s_n<\ve$ and $\limsup\limits_{N\to+\infty}p^s_N>0$.
\end{proof}

Now we prove Theorem \ref{Thm-inverse}(4).
We will prove the statement for $p^s_n$ and $p^s(z_n)$
(the proof for $p^u_n$ and $p^u(z_n)$ is identical). 

\medskip
\noindent
{\bf Step 1.} $p^s_n\geq e^{-\sqrt[3]{\ve}}p^s(z_n)$ for all $n\in\Z$.

\medskip
We divide the proof into two cases, according to whether $n$ satisfies Lemma \ref{Lemma-max-in-chart}
or not. Assume first that it does, i.e. $\min\{e^{\ve T_n}p^s_{n+1},e^{-\ve}\ve Q(x_n)\}=e^{-\ve}\ve Q(x_n)$.
By (GPO2), we have $p^s_n\geq e^{-\ve p^s_n}e^{-\ve}\ve Q(x_n)\geq e^{-2\ve}\ve Q(x_n)$. 
By Theorem \ref{Thm-inverse}(3), it follows that
$$
p^s_n\geq  e^{-2\ve}\ve Q(x_n)\geq e^{-2\ve-O(\sqrt{\ve})}\ve Q(z_n)
\geq e^{-2\ve-O(\sqrt{\ve})}p^s(z_n)
\geq e^{-\sqrt[3]{\ve}}p^s(z_n).
$$

\medskip
Now assume that $n$ does not satisfy Lemma \ref{Lemma-max-in-chart}. Take the
smallest $m>n$ that satisfies Lemma \ref{Lemma-max-in-chart}. Hence
$\min\{e^{\ve T_k}p^s_{k+1},e^{-\ve}\ve Q(x_k)\}=e^{\ve T_k}p^s_{k+1}$
for $k=n,\ldots,m-1$. By (GPO2), we get that $p^s_k\geq e^{\ve(T_k-p^s_k)}p^s_{k+1}>\lambda p^s_{k+1}$
for $k=n,\ldots,m-1$. Therefore
$p^s_k\leq \lambda^{n-k}p^s_n$ for $k=n,\ldots,m-1$. Recalling that $t_k=r_k(\{\Psi_{x_n}^{p^s_n,p^u_n}\}_{n\in\Z})$ and writing $\Delta_k=(t_{k+1}-t_k)-T_k\geq 0$, 
the last estimate gives two consequences:
\begin{enumerate}[$\circ$]
\item $\displaystyle\sum_{k=n}^{m-1}p^s_k<\ve$: indeed,
$$
\sum_{k=n}^{m-1}p^s_k\leq p^s_n\sum_{k=n}^{m-1}\lambda^{n-k}\leq \ve^{\frac{6}{\beta}}\frac{1}{1-\lambda^{-1}}
<2\ve^{\frac{6}{\beta}-1.5}<\ve,
$$
since $\lim\limits_{\ve\to 0}\tfrac{\ve^{1.5}}{1-\lambda^{-1}}=1$.
\item $\displaystyle\sum_{k=n}^{m-1} \Delta_k<\ve$: since the transition time from $x_k$ to $x_{k+1}$
is 1--Lipschitz (Lemma \ref{lemma-local-coord}(3)),
$$
\sum_{k=n}^{m-1}\Delta_k\leq 2\sum_{k=n}^{m-1}p^s_k<4\ve^{\frac{6}{\beta}-1.5}<\ve.
$$
\end{enumerate}
Using that $p^s_k\geq e^{\ve(T_k-p^s_k)}p^s_{k+1}=e^{-\ve(p^s_k+\Delta_k)}e^{\ve(t_{k+1}-t_k)}p^s_{k+1}$
for $k=n,\ldots,m-1$, we get that
\begin{align*}
&\ p^s_n\geq \exp{}\left[-\ve\sum_{k=n}^{m-1}p^s_k-\ve\sum_{k=n}^{m-1}\Delta_k\right]e^{\ve(t_m-t_n)}p^s_m\\
&\geq \exp{}\left[-2\ve^2-2\ve-O(\sqrt{\ve})\right]e^{\ve(t_m-t_n)}p^s(z_m)
\geq e^{-\sqrt[3]{\ve}}p^s(z_n),
\end{align*}
where in the last inequality we used Proposition \ref{Prop-Z-par}(2).

\medskip
\noindent
{\bf Step 2.} $p^s(z_n)\geq e^{-\sqrt[3]{\ve}}p^s_n$ for all $n\in\Z$.

\medskip
The motivation for this inequality is that $p^s(z_n)$ grows at least as much as $p^s_n$, since
$p^s(z_n)$ satisfies the recursive equality $p^s(z_n)=\min\{e^{\ve(t_{n+1}-t_n)}p^s(z_{n+1}),\ve Q(z_n)\}$ while by (GPO2) we have the recursive inequality $p^s_n\leq \min\{e^{\ve T_n}p^s_{n+1},\ve Q(x_n)\}$
and $t_{n+1}-t_n\geq T_n$. For ease of notation, let $n=0$ (the general case is identical).
By the above recursive equality and inequality, we have
$$p^s(z_0)=\ve\inf\{e^{\ve t_n}Q(z_n):n\geq 0\}\ \text{ and }\
p^s_0\leq \ve\inf\{e^{\ve(T_0+\cdots+T_{n-1})}Q(x_n):n\geq 0\}.
$$
Using part (3) and that $t_n=\sum\limits_{k=0}^{n-1}(t_{k+1}-t_k)\geq \sum\limits_{k=0}^{n-1}T_k$, 
we conclude that
$$
\ p^s(z_0)=\ve\inf\{e^{\ve t_n}Q(z_n):n\geq 0\}\geq e^{-\sqrt[3]{\ve}}\ve \inf\{e^{\ve(T_0+\cdots+T_{n-1})}Q(x_n):n\geq 0\}=e^{-\sqrt[3]{\ve}} p^s_0.
$$
Steps 1 and 2 conclude the proof of Part (4). In particular, since $\{\Psi_{x_n}^{p^s_n,p^u_n}\}_{n\in\Z}\in\Sigma^\#$,
it follows that $x\in\nuh^\#$.

\subsection{Control of $\Psi_{x_n}^{-1}\circ \Psi_{\vf^{r_n(\un v)}(x)}$ and $\Psi_{\vf^{r_n(\un v)}(x)}^{-1} \circ \Psi_{x_n}$}

We do the case $n=0$. The other cases are analogous. Let $\eta=p^s_0\wedge p^u_0$,
$\Theta=\Theta_{x_0,x}$, and $P=P_{x_0,x}$.
Since $x\in\nuh^\#$, the Pesin chart $\Psi_x$ is well-defined. Additionally, by parts (1)--(4) of Theorem \ref{Thm-inverse}, the parameters of $\Psi_{x_0}^{p^s_0,p^u_0}$ and $\Psi_x^{q^s(x),q^u(x)}$ are almost the same. This will be enough to establish part (5) of Theorem \ref{Thm-inverse}.

We start proving that the compositions are well-defined in the respective domains.
\begin{enumerate}[$\circ$]
\item $\Psi_{x}^{-1}\circ\Psi_{x_0}$ is well-defined in $R[10Q(x_0)]$: we have
$$
\Psi_{x_0}(R[10Q(x_0)])\subset B(x_0,20Q(x_0))\subset B(x,20Q(x_0)+d(x_0,x))
\subset R[\mathfrak r],
$$
where in the last inclusion we used that
$20Q(x_0)+d(x_0,x)<25Q(x_0)\ll
25\ve<\mathfrak r$.
By Lemma \ref{lemma3.7}, $\Psi_{x}^{-1}\circ\Psi_{x_0}$ is well-defined in $R[10Q(x_0)]$.
\item $\Psi_{x_0}^{-1}\circ \Psi_x$ is well-defined in $R[10Q(x)]$: same as above, changing the roles of $x$ and $x_0$.
\end{enumerate}

The next step is to represent $\Psi_{x}^{-1}\circ\Psi_{x_0}$ as required. 
We have $\Psi_{x}^{-1}\circ\Psi_{x_0}=C(x)^{-1}\circ \Pi\circ C(x_0)$,
where $\Pi=\exp{x}^{-1}\circ \exp{x_0}$. The composition
$C(x)^{-1}\circ \Theta\circ C(x_0)$ has norm close to one. Indeed, 
by (\ref{comparison-C-3}) we have
$\|C(x)^{-1}\circ \Theta\circ C(x_0)v\|=e^{\pm \sqrt{\ve}}\|C(x_0)^{-1}\circ C(x_0)v\|=
e^{\pm\sqrt{\ve}}\|{v}\|$. By the polar decomposition for matrices,
$C(x)^{-1}\circ \Theta\circ C(x_0)=OR$ where $O$ is an orthogonal matrix
and $R$ is positive symmetric with $\|{Rv}\|=e^{\pm \sqrt{\ve}}\|{v}\|$ for all $v\in\R^d$.
Since $C(x)^{-1}\circ \Theta\circ C(x_0)$ preserves the splitting $\R^{d_s(x)}\times\R^{d_u(x)}$,
the same holds for $O$. Also, diagonalizing $R$ and estimating its eigenvalues,
we get that if $\ve>0$ is small enough then $\|R-{\rm Id}\|\leq 4\sqrt{\ve}$,
see details in \cite[pp. 100]{O18}.
Define $\delta=(\Psi_{x}^{-1}\circ\Psi_{x_0})(0)\in\R^d$, and $\Delta:R[10Q(x_0)]\to \R^d$
s.t. $\Psi_{x}^{-1}\circ\Psi_{x_0}=\delta+O+\Delta$.
We start estimating $d\Delta$. For $z\in R[10Q(x_0)]$,
\begin{align*}
&\,(d\Delta)_z=C(x)^{-1}\circ (d\Pi)_{C(x_0)z}\circ C(x_0)-O\\
&=C(x)^{-1}\circ \underbrace{\left[(d\Pi)_{C(x_0)z}-\Theta\right]}_{=:E}\circ C(x_0)+OR-O.
\end{align*}
To estimate $E$, observe that:
\begin{enumerate}[$\circ$]
\item By Lemma \ref{ALP-Lemma 6.1}(2), $\|{P-\Theta}\|\leq \tfrac{1}{2}\eta^{15\beta/48}$.
\item By assumption (Exp3),
\begin{align*}
&\, \|{(d\Pi)_{C(x_0)z}-P}\|=
\|{\widetilde{(d\exp{x}^{-1})_{\Psi_{x_0}(z)}(d\exp{x_0})_{C(x_0)z}}-{\rm Id}}\| \\
&=\|{\widetilde{(d\exp{x}^{-1})_{\Psi_{x_0}(z)}(d\exp{x_0})_{C(x_0)z}}-
\widetilde{(d\exp{x_0}^{-1})_{\Psi_{x_0}(z)}(d\exp{x_0})_{C(x_0)z}}}\| \\
&\leq \|{\widetilde{(d\exp{x}^{-1})_{\Psi_{x_0}(z)}}-
\widetilde{(d\exp{x_0}^{-1})_{\Psi_{x_0}(z)}}}\| \cdot\|{(d\exp{x_0})_{C(x_0)z}}\| \\
&\leq 2\mathfrak{K}d(x,x_0)
\end{align*}
which, by part (1), is bounded by
$\frac{2 \mathfrak{K} \eta}{50} \ll\frac{1}{2} \eta^{15 \beta/48}$.
\end{enumerate}
Hence, $\|{E}\|<\eta^{15\beta/48}$ and so, by part (2),
\begin{align*}
&\|{C(x)^{-1}\circ E\circ C(x_0)}\|  \leq \|{C(x)^{-1}}\|\eta^{15\beta/48} \\
& \leq e^{2\sqrt{\ve}}\|{C(x_0)^{-1}}\|\eta^{15\beta/48}
\leq e^{2\sqrt{\ve}}\ve^{1/8}\eta^{14\beta/48}\ll \sqrt{\ve}.
\end{align*}
Since $\|{OR-O}\|=\|{R-{\rm Id}}\|\leq 4\sqrt{\ve}$, we conclude that
$\|{(d\Delta)_z}\|\leq 5\sqrt{\ve}$. In particular, since $\Delta(0)=0$, we have
$\|{\Delta(z)}\|\leq \|{d\Delta}\|_{C^0}\|{z}\|\leq 5\sqrt{\ve}\|{z}\|$.

We now estimate $\|{\delta}\|$. Let $\overline{z} \in \R^d$ s.t. $\Psi_{x_0}(\overline{z})=x$.
We have $0=(\Psi_{x}^{-1}\circ\Psi_{x_0})(\overline{z})=\delta+O\overline{z}+\Delta(\overline{z})$
and so $\delta=-O\overline{z}-\Delta(\overline{z})$. By Lemma \ref{Lemma-admissible-manifolds}
we have $\|{\overline{z}}\|< 250^{-1}\eta$, therefore for $\ve>0$ small
$$
\|{\delta}\| \leq \|{O\overline{z}}\|+\|{\Delta(\overline{z})}\|\leq (1+5\sqrt{\ve})\|{\overline{z}}\|
\leq \tfrac{1+5\sqrt{\ve}}{250}\eta<50^{-1}\eta.
$$

The final step is to represent $\Psi_{x_0}^{-1}\circ\Psi_{x}$.
We have $\Psi_{x_0}^{-1}\circ\Psi_x=C(x_0)^{-1}\circ \Pi^{-1}\circ C(x)$. Changing $v$ by $\Theta(v)$ in \eqref{comparison-C-3} allows us to similarly prove that $\|C(x_0)^{-1}\circ \Theta^{-1}\circ C(x)v\|=e^{\pm\sqrt{\ve}}\|v\|$. Since the estimates used above for $\Pi,\Theta$ also hold for $\Pi^{-1},\Theta^{-1}$ (see {(Exp3) and Lemma 
\ref{ALP-Lemma 6.1}(2)), we can write $(\Psi_{x_0}^{-1}\circ\Psi_{x})(z)=\delta+Oz+\Delta(z)$ where $O,\Delta$ satisfy the same estimates. Finally, letting $z=0$, we obtain that $\overline{z}=\delta$ and so by Lemma \ref{Lemma-admissible-manifolds} we conclude that $\|\delta\|<250^{-1}\eta<50^{-1}\eta$.

\section{A countable locally finite section}\label{Section-locally-finite-section}

We summarize our discussion from the previous sections:
\begin{enumerate}[$\circ$]
\item We constructed a countable family $\mathfs A$ of $\ve$--double charts,
see Theorem \ref{Thm-coarse-graining}.
\item Letting $\Sigma$ be the TMS defined by $\mathfs A$ with the edge condition
defined in Section \ref{ss.pseudo.orbits}, we constructed a H\"older continuous map
$\pi:\Sigma\to \widehat\Lambda$ that ``captures'' all orbits in $\nuh^\#$, see
Propositions \ref{Prop-pi} and \ref{Prop-pi_R}. The map $\pi$ is defined as
$\{\pi(\un v)\}:=V^s[\un v]\cap V^u[\un v]$.
\item Although $\pi$ is not finite-to-one, we solved the inverse problem by analyzing 
when $\pi$ loses injectivity, see Theorem \ref{Thm-inverse}.
\end{enumerate}
We now employ this information to construct a countable family $\mathfs Z$ of subsets of $\widehat\Lambda$ s.t.:
\begin{enumerate}[$\circ$]
\item The union of elements of $\mathfs Z$, from now on also denoted by $\mathfs Z$,
is a section that contains $\Lambda\cap\nuh^\#$.
\item $\mathfs Z$ is {\em locally finite}: each point $x\in\mathfs Z$ belongs to at most finitely many
rectangles $Z\in\mathfs Z$.
\item Every element $Z\in \mathfs Z$ is a {\em rectangle}: each point $x\in Z$ has
{\em invariant fibres} $W^s(x,Z)$, $W^u(x,Z)$ in $Z$,
and these fibres induce a local product structure on $Z$.
\item $\mathfs Z$ satisfies a {\em symbolic Markov property}.

\end{enumerate} 
In this section, all statements assume that $0<\ve\ll \rho\ll 1$, so we will omit this information.

\subsection{The Markov cover $\mathfs Z$}
Let $\mathfs Z:=\{Z(v):v\in\mathfs A\}$, where
$$
Z(v):=\{\pi(\un v):\un v\in\Sigma^\#\text{ and }v_0=v\}.
$$
Using admissible manifolds, we define {\em invariant fibres} inside each $Z\in\mathfs Z$. Let $Z=Z(v)$.

\medskip
\noindent
{\sc $s$/$u$--fibres in $\mathfs Z$:} Given $x\in Z$, let $W^s(x,Z):=V^s[\{v_n\}_{n\geq 0}]\cap Z$
be the {\em $s$--fibre} of $x$ in $Z$ for some (any) $\un v=\{v_n\}_{n\in\Z}\in\Sigma^\#$
such that $\pi(\un v)=x$ and $v_0=v$. Similarly, let $W^u(x,Z):=V^u[\{v_n\}_{n\leq 0}]\cap Z$ be
the {\em $u$--fibre} of $x$ in $Z$.

\medskip
By Proposition \ref{Prop-disjointness}, the above definitions do not depend on the choice of $\un v$, 
and any two $s$--fibres ($u$--fibres) in $Z$ either coincide or are disjoint. We also
define $V^s(x,Z):=V^s[\{v_n\}_{n\geq 0}]$ and $V^u(x,Z):=V^u[\{v_n\}_{n\leq 0}]$.
Note that:
\begin{enumerate}[$\circ$]
\item $V^{s/u}(x,Z)$ are smooth curves, while $W^{s/u}(x,Z)$ are usually fractal sets.
\item $V^{s/u}(x,Z)$ are {\em not} subsets of $Z$, while $W^{s/u}(x,Z)$ are.
\end{enumerate}

\subsection{Fundamental properties of $\mathfs Z$}\label{subsec-fundpropZ}

\medskip
Although $\mathfs Z$ is usually a fractal set (and thus not a proper section),
we can still define its Poincar\'e return map. If $x=\pi(\un v)\in\mathfs Z$ with
$\un v\in\Sigma^\#$ then $\vf^{r_n(\un v)}(x)=\pi[\sigma^n(\un v)]\in\mathfs Z$ for all $n\in\N$.
Define $r_{\mathfs Z}:\mathfs Z\to (0,\rho)$ by $r_{\mathfs Z}(x):=\min\{t>0:\vf^t(x)\in\mathfs Z\}$. 

\medskip
\noindent
{\sc The return map $H$:} It is the map $H:\mathfs Z\to\mathfs Z$ defined by $H(x):=\vf^{r_{\mathfs Z}}(x)$.

\medskip
Below we collect the main properties of $\mathfs Z$.

\begin{proposition}\label{Prop-Z}
The following are true.
\begin{enumerate}[{\rm (1)}]
\item {\sc Covering property:} $\mathfs Z$ is a cover of $\Lambda\cap\nuh^\#$.
\item {\sc Local finiteness:} For every $Z\in\mathfs Z$,
$$
\#\left\{Z'\in\mathfs Z:\left[\bigcup_{|n|\leq 1}H^n(Z)\right]\cap Z'\neq\emptyset\right\}<\infty.
$$
\item {\sc Local product structure:} For every $Z\in\mathfs Z$ and every $x,y\in Z$, the intersection
$W^s(x,Z)\cap W^u(y,Z)$ consists of a single point, and this point belongs to $Z$.
\item {\sc Symbolic Markov property:} If $x=\pi(\un v)\in\mathfs Z$ with
$\un v=\{v_n\}_{n\in\Z}=\{\Psi_{x_n}^{p^s_n,p^u_n}\}_{n\in\Z}\in\Sigma^\#$, then
\begin{align*}
&g_{x_0}^+(W^s(x,Z(v_0)))\subset W^s(g_{x_0}^+(x),Z(v_1))\text{ and }\\
&g_{x_1}^{-}(W^u(g_{x_0}^+(x),Z(v_1)))\subset W^u(x,Z(v_0)).
\end{align*}
\end{enumerate}
\end{proposition}

Before discussing the proof, we use part (3) to introduce the following definition: 
for $x,y\in Z$, let $[x,y]_Z:=$ intersection point of $W^s(x,Z)$ and $W^u(y,Z)$, and
call it the {\em Smale bracket} of $x,y$ in $Z$.

\begin{proof}
We have $\mathfs Z=\pi[\Sigma^\#]$. Since $\pi[\Sigma^\#]\supset\Lambda\cap\nuh^\#$
by Proposition \ref{Prop-pi}(3), it follows that $\mathfs Z$ contains $\Lambda\cap\nuh^\#$.
This proves part (1).

\medskip
\noindent
(2) Write $Z=Z(\Psi_x^{p^s,p^u})$, and take $Z'=Z(\Psi_y^{q^s,q^u})$ such that
$$\left[\bigcup_{|n|\leq 1}H^n(Z)\right]\cap Z'\neq\emptyset.$$ We will estimate the ratio
$\tfrac{p^s\wedge p^u}{q^s\wedge q^u}$. By assumption, there is $x\in Z$ such that $x'=H^n(x)\in Z'$
for some $|n|\leq 1$. Let $\un v\in\Sigma^\#$ with $v_0=\Psi_x^{p^s,p^u}$ such that $x=\pi(\un v)$.
Recalling that $p^{s/u}(x)=p^{s/u}(x,\mathcal T,0)$ for $\mathcal T=\{r_n(\un v)\}_{n\in\Z}$, 
the following holds:
\begin{enumerate}[$\circ$]
\item $x\in Z$, hence by Theorem \ref{Thm-inverse}(4) we have
$\tfrac{p^s}{p^s(x)}=e^{\pm\sqrt[3]{\ve}}\text{ and }\tfrac{p^u}{p^u(x)}=e^{\pm\sqrt[3]{\ve}}$,
and so $\tfrac{p^s\wedge p^u}{p^s(x)\wedge p^u(x)}=e^{\pm\sqrt[3]{\ve}}$.
By Proposition \ref{Prop-Z-par}(1), we have $\tfrac{p^s(x)\wedge p^u(x)}{q(x)}=e^{\pm\mathfrak H}$.
The conclusion is that $\tfrac{p^s\wedge p^u}{q(x)}=e^{\pm(\sqrt[3]{\ve}+\mathfrak H)}$.
\item $x'\in Z'$, hence by the same reason $\tfrac{q^s\wedge q^u}{q(x')}=e^{\pm(\sqrt[3]{\ve}+\mathfrak H)}$.
\item $x'=\vf^t(x)$ with $|t|\leq \rho$, hence by Lemma \ref{Lemma-q} we have
$\tfrac{q(x)}{q(x')}=e^{\pm 2\ve}$.
\end{enumerate}
Altogether, we conclude that $\tfrac{p^s\wedge p^u}{q^s\wedge q^u}=e^{\pm2(\sqrt[3]{\ve}+\ve+\mathfrak H)}$,
and so
$$
\left\{Z'\in\mathfs Z:\left[\bigcup_{|n|\leq 1}H^n(Z)\right]\cap Z'\neq\emptyset\right\}\subset\left\{\Psi_y^{q^s,q^u}\in\mathfs A:(q^s\wedge q^u)\geq e^{-2(\sqrt[3]{\ve}+\ve+\mathfrak H)}(p^s\wedge p^u)\right\}.
$$
The last set above is finite, by Theorem \ref{Thm-coarse-graining}(1).

\medskip
\noindent
(3) We proceed as in \cite[Prop. 10.5]{Sarig-JAMS}.
Let $Z=Z(v)$, and take
$x,y\in Z$, say $x=\pi(\un v),y=\pi(\un w)$ with $\un v,\un w\in\Sigma^\#$, where
$\un v=\{v_n\}_{n\in\Z}=\{\Psi_{x_n}^{p^s_n,p^u_n}\}_{n\in\Z}$ and
$\un w=\{w_n\}_{n\in\Z}=\{\Psi_{y_n}^{q^s_n,q^u_n}\}_{n\in\Z}$ with $v_0=w_0=v$.
We let $z=\pi(\un u)$ where $\un u=\{u_n\}_{n\in\Z}$ is defined by
$$
u_n=\left\{\begin{array}{ll}v_n&,n\geq 0\\ w_n&,n\leq 0.\end{array}\right.
$$
We claim that $\{z\}=W^s(x,Z)\cap W^u(y,Z)$. To prove this, first remember that
$V^s[\{u_n\}_{n\geq 0}]\cap V^u[\{u_n\}_{n\geq 0}]$ intersects at a single point
(Lemma \ref{Lemma-admissible-manifolds}), and that $z$ belongs to such intersection.
Therefore, it is enough to show that $z\in\pi[\Sigma^\#]$, which is clear since
$\un u\in\Sigma^\#$.

\medskip
\noindent
(4) Proceed exactly as in \cite[Prop. 10.9]{Sarig-JAMS}.
\end{proof}

Let $Z=Z(v), Z'=Z(w)$ where $v=\Psi_x^{p^s,p^u},w=\Psi_y^{q^s,q^u}\in\mathfs A$, and assume
that $Z\cap \vf^{[-2\rho,2\rho]}Z'\neq\emptyset$. Let $D,D'$ be the 
connected components of $\widehat\Lambda$ such that $Z\subset D$ and $Z'\subset D'$. We wish to 
compare $s$--fibres of $Z$ with $u$--fibres of $Z'$ and vice-versa. 
To do that, we apply the holonomy maps $\mathfrak q_D$ and $\mathfrak q_{D'}$. Given $z\in Z,z'\in Z'$, define
\begin{align*}
\{[z,z']_Z\}&:=V^s(z,Z)\cap \mathfrak q_D[V^u(z',Z')]\\
\{[z,z']_{Z'}\}&:=\mathfrak q_{D'}[V^s(z,Z)]\cap V^u(z',Z').
\end{align*}
The next proposition proves that $[z,z']_Z$ and $[z,z']_{Z'}$ consist of single points, and some
compatibility properties that will be used in the next section. 

\begin{proposition}\label{Prop-overlapping-charts}
Let $Z=Z(v), Z'=Z(w)$ where $v=\Psi_x^{p^s,p^u},w=\Psi_y^{q^s,q^u}\in\mathfs A$, and assume
that $Z\cap \vf^{[-2\rho,2\rho]}Z'\neq\emptyset$. Let $D,D'$ be the 
connected components of $\widehat\Lambda$ such that $Z\subset D$ and $Z'\subset D'$.
The following are true.
\begin{enumerate}[{\rm (1)}]
\item $\mathfrak q_{D'}\circ \Psi_x(R[\tfrac{1}{2}(p^s\wedge p^u)])\subset \Psi_y(R[q^s\wedge q^u])$.
\item If $z\in Z$ with $z'=\mathfrak q_{D'}(z)\in Z'$, then $\mathfrak q_{D'}[W^{s/u}(z,Z)]\subset V^{s/u}(z',Z')$.
\item If $z\in Z,z'\in Z'$ then $[z,z']_Z$, $[z,z']_{Z'}$ are points with
$[z,z']_Z=\mathfrak q_{D}([z,z']_{Z'})$.
\end{enumerate}
\end{proposition}

When $M$ is compact and $f$ is a surface diffeomorphism, this result corresponds to \cite[Lemmas 10.8 and 10.10]{Sarig-JAMS}. For flows in three dimensions, this is \cite[Proposition 7.2]{BCL23}, and a very similar strategy of proof applies in our setting: Theorem \ref{Thm-non-linear-Pesin} remains valid when replacing $g_x^+$ with $\mathfrak q_{D'}$, allowing us to control the composition
$\Psi_y^{-1}\circ\mathfrak q_{D'}\circ\Psi_x$. The main difference compared to \cite{Sarig-JAMS,BCL23} is that the functions involved are no longer close to the identity, but rather close to orthogonal linear maps that preserve the splitting $\mathbb{R}^{d_s} \oplus \mathbb{R}^{d_u}$, as in \cite{O18}. The details of this construction are provided in Appendix \ref{Appendix-proofs}.

Additionally, we will require further information about the Smale product structure in nearby charts.

\begin{proposition}\label{Prop-overlapping-charts-2}
Let $Z, Z',Z''$ such that $Z\cap \vf^{[-2\rho,2\rho]}Z'\neq\emptyset$, $Z\cap \vf^{[-2\rho,2\rho]}Z''\neq\emptyset$.
Assume that there is $z'\in Z'$ such that $\vf^t(z')\in Z''$ for some $|t|\leq 2\rho$. For every $z\in Z$, it holds
$$
[z,z']_Z=[z,\vf^t(z')]_Z.
$$
\end{proposition}

Note that $[z,z']_Z$ is defined by $Z,Z'$ while $[z,\vf^t(z')]_Z$ is defined by 
$Z,Z''$. The equality shows a compatibility of the Smale product along small flow displacements. 
It holds because such displacements barely change the sizes of invariant
fibres, hence the unique intersection is preserved. The proof is in Appendix \ref{Appendix-proofs}.

\section{A refinement procedure}\label{Section-refinement}

Up to now, we have constructed a countable family $\mathfs Z$ of subsets of $\widehat\Lambda$
with the following properties:
\begin{enumerate}[$\circ$]
\item The union of elements of $\mathfs Z$, also denoted by $\mathfs Z$,
is a section that contains $\Lambda\cap\nuh^\#$.
\item $\mathfs Z$ is locally finite: each point $x\in\mathfs Z$ belongs to at most finitely many
rectangles $Z\in\mathfs Z$.
\item Every element $Z\in \mathfs Z$ is a rectangle: each point $x\in Z$ has
invariant fibres $W^s(x,Z)$, $W^u(x,Z)$ in $Z$,
and these fibres induce a local product structure on $Z$.
\item $\mathfs Z$ satisfies a {\em symbolic} Markov property.
\end{enumerate}
In this section, we will refine $\mathfs Z$ to obtain a countable family of disjoint sets
$\mathfs R$ that satisfy a {\em geometrical} Markov property.
We stress the difference from a symbolic to a geometrical Markov property:
by Proposition \ref{Prop-Z}(4), $g_{x_0}^\pm$ satisfies a symbolic Markov property;
our goal is to obtain a Markov property for the first return map $H$.
In general the orbit of $x$ can intersect $\mathfs Z$
between $x$ and $g_{x_0}^+(x)$, in which case we will have that $g_{x_0}^+(x)\neq H(x)$.
Therefore the symbolic Markov property
of Proposition \ref{Prop-Z}(4) does not directly translate into a geometrical Markov
property for $H$. To obtain a geometrical Markov property, we will use a refinement procedure
developed by Bowen \cite{Bowen-Symbolic-Flows}, motivated by the work of
Sina{\u\i} \cite{Sinai-Construction-of-MP,Sinai-MP-U-diffeomorphisms}. The difference
from our setup to Bowen's is that, while in Bowen's case all families are finite, in ours
they are usually countable. Fortunately, as implemented in \cite{Sarig-JAMS,BCL23}, the refinement procedure
works well for countable covers with the local finiteness property, which we have in
Proposition \ref{Prop-Z}(2).

\subsection{The Markov partition $\mathfs R$}
We first see that the map $g_{x_0}^+$ can be deduced from $H$ by a bounded iteration.

\begin{lemma}\label{l.time}
There exists $N\geq 1$ such that for any $x=\pi(\un v)\in\mathfs Z$ there exists $0< n < N$
such that $g_{x_0}^+(x)=H^n(x)$.
\end{lemma}

The statement and proof are the same as \cite[Lemma 8.1]{BCL23}.
Proposition \ref{Prop-Z}(4) then implies that for every $x\in\mathfs Z$ there are 
$0<k,\ell<N$ such that $H^k(x)$ satisfies a Markov property in the stable direction and
$H^{-\ell}(x)$ satisfies a Markov property in the unstable direction.

At this point, it is worth mentioning the method that Bowen used
to construct Markov partitions for Axiom A flows \cite{Bowen-Symbolic-Flows}:
\begin{enumerate}[(1)]
\item[(1)] Fix a global section for the flow; inside this section, construct a finite family of rectangles
(sets that are closed under the Smale bracket operation). Let $H$ be the Poincar\'e return map
of this family.
\item[(2)] Apply the method of Sina{\u\i} of successive approximations to get a new family
of rectangles $\mathfs Z$ with the following property: if $H$ is the Poincar\'e return map of $\mathfs Z$,
then for every $x\in\mathfs Z$ there are $k,\ell>0$ such that $H^k(x)$ satisfies a Markov property in the stable
direction and $H^{-\ell}(x)$ satisfies a Markov property in the unstable direction. In addition, there is a global
constant $N>0$ such that $k,\ell<N$.
\item[(3)] Apply a refinement procedure to $\mathfs Z$ such that the resulting partition $\mathfs R$ is a
disjoint family of rectangles satisfying the Markov property for $H$.
\end{enumerate}
So far, we have implemented steps (1) and (2) above,
with the difference that while Bowen used the method of successive approximations, we used
the method of $\ve$--gpo's. It remains to establish step (3), and we will do this closely
following Bowen \cite{Bowen-Symbolic-Flows}, as already done in \cite{BCL23}. Fortunately, the arguments made in \cite{BCL23} are abstract enough to work equally well in higher dimension, so in the remainder of this section we state the definitions and main results and only discuss proofs that require modifications. 

For each $Z\in\mathfs Z$, let
$$
\mathfs I_Z:=\left\{Z'\in\mathfs Z:\vf^{[-\rho,\rho]}Z\cap Z'\neq\emptyset\right\}.
$$
By Theorem \ref{Thm-inverse}, $\mathfs I_Z$ is finite.
Let $D$ be the connected component of $\widehat\Lambda$ such that $Z\subset D$.
By continuity, having chosen $\ve\ll \rho\ll 1$
the following property holds:
\begin{equation}\label{e.2rho}
\text{If $Z'\in\mathfs I_Z$ then $Z'\subset \vf^{[-2\rho,2\rho]}D$.}
\end{equation}
Therefore $\mathfrak q_D(Z')$ is a well-defined subset of $D$.
For each $Z'\in\mathfs I_Z$ we consider the partition of $Z$ into four subsets as follows:
\begin{align*}
E_{Z,Z'}^{su}&=\{x\in Z: W^s(x,Z)\cap\mathfrak{q}_{D}(Z')\neq\emptyset,
W^u(x,Z)\cap\mathfrak{q}_{D}(Z')\neq\emptyset\}\\
E_{Z,Z'}^{s\emptyset}&=\{x\in Z: W^s(x,Z)\cap\mathfrak{q}_{D}(Z')\neq\emptyset,
W^u(x,Z)\cap\mathfrak{q}_{D}(Z')=\emptyset\}\\
E_{Z,Z'}^{\emptyset u}&=\{x\in Z: W^s(x,Z)\cap\mathfrak{q}_{D}(Z')=\emptyset,
W^u(x,Z)\cap\mathfrak{q}_{D}(Z')\neq\emptyset\}\\
E_{Z,Z'}^{\emptyset\emptyset}&=\{x\in Z: W^s(x,Z)\cap\mathfrak{q}_{D}(Z')=\emptyset,
W^u(x,Z)\cap\mathfrak{q}_{D}(Z')=\emptyset\}.
\end{align*}
Call this partition
$\mathfs P_{Z,Z'}:=\{E_{Z,Z'}^{su},E_{Z,Z'}^{s\emptyset},E_{Z,Z'}^{\emptyset u},E_{Z,Z'}^{\emptyset\emptyset}\}$.
Clearly, $E^{su}_{Z,Z'}=Z\cap \mathfrak{q}_D(Z')$.

\medskip
\noindent
{\sc The partition $\mathfs E_Z$:} It is the coarser partition of $Z$ that refines
all of $\mathfs P_{Z,Z'}$, $Z'\in\mathfs I_Z$.

\medskip
To define a partition of $\mathfs Z$, we define an equivalence relation on $\mathfs Z$.

\medskip
\noindent
{\sc Equivalence relation $\simN$ on $\mathfs Z$:} For $x,y\in\mathfs Z$, we write
$x\simN y$ if for any $|k|\leq N$:
\begin{enumerate}[aa)]
\item[(i)] For all $Z\in\mathfs Z$: $H^k(x)\in Z\Leftrightarrow H^k(y)\in Z$.
\item[(ii)] For all $Z\in\mathfs Z$ such that $H^k(x),H^k(y)\in Z$, the points $H^k(x),H^k(y)$
belong to the same element of $\mathfs E_Z$. 
\end{enumerate}

\medskip
Clearly $\simN$ is an equivalence relation in $\mathfs Z$, hence it defines a partition of $\mathfs Z$.
Before proceeding, let us state a fact that will be used in the sequel: if $x\simN y$ with $x\in Z=Z(\Psi_{x_0}^{p^s_0,p^u_0})\in\mathfs Z$, then there exists $|k|\leq N$ such that
$g_{x_0}^+(x)=H^k(x)$ and $g_{x_0}^+(y)=H^k(y)$. To see this, write
$x=\pi(\un v)$ with $v_0=\Psi_{x_0}^{p^s_0,p^u_0}$,
and let $D'$ be the connected component of $\widehat\Lambda$ with $Z(v_1)\subset D'$.
On one hand, $g_{x_0}^+(y)=\mathfrak q_{D'}(y)$. On the other hand,
since $H^k(x)\in Z(v_1)\subset D'$ for some $|k|\leq N$, the definition of $\simN$ implies that
$H^k(y)\in Z(v_1)\subset D'$, hence $H^k(y)=\mathfrak q_{D'}(y)$. A similar result holds for $g_{x_0}^-$.

\medskip
\noindent
{\sc The Markov partition $\mathfs R$:} It is the partition of $\mathfs Z$ whose elements are the
equivalence classes of $\simN$.

\medskip
By definition, $\mathfs R$ is a refinement of $\mathfs Z$.

\begin{lemma}\label{Lemma-local-finite}
The partition $\mathfs R$ satisfies the following properties.
\begin{enumerate}[{\rm (1)}]
\item For every $Z\in\mathfs Z$, $\#\{R\in\mathfs R:R\subset \vf^{[-\rho,\rho]}Z\}<\infty$.
\item For every $R\in\mathfs R$, $\#\{Z\in\mathfs Z:R\subset \vf^{[-\rho,\rho]}Z\}<\infty$.
\end{enumerate}
\end{lemma}

\begin{proof} In three dimensions, this is \cite[Lemma 8.2]{BCL23}, and the same proof applies. Since it is short, we include it.

\medskip
\noindent
(1) Start noting that, for every $Z\in\mathfs Z$,
$\#\{R\in\mathfs R:R\subset Z\}\leq 4^{\#\mathfs I_Z}$. Hence
$$
\#\{R\in\mathfs R:R\subset \vf^{[-\rho,\rho]}Z\}\leq
\sum_{Z'\in\mathfs I_Z}\#\{R\in\mathfs R:R\subset Z'\}
\leq \sum_{Z'\in\mathfs I_Z}4^{\#\mathfs I_{Z'}}<+\infty
$$
since the last summand is the finite sum of finite numbers.

\medskip
\noindent
(2) For any $Z'\in\mathfs Z$ such that $Z'\supset R$, we have
$\{Z\in\mathfs Z:R\subset \vf^{[-\rho,\rho]}Z\}\subset \mathfs I_{Z'}$.
Since each $\mathfs I_{Z'}$ is finite, the result follows.
\end{proof}

The final step in the refinement procedure is to show that $\mathfs R$ is a Markov
partition for the map $H$, in the sense of Sina{\u\i} \cite{Sinai-MP-U-diffeomorphisms}. 

\medskip
\noindent
{\sc $s$/$u$--fibres in $\mathfs R$:} Given $x$ in $R\in\mathfs R$, we define the {\em $s$--fibre}
and {\em $u$--fibre} of $x$ by:
\begin{align*}
W^s(x,R):=
\bigcap_{Z\in\mathfs Z:Z\supset R} V^s(x,Z)\cap R
\, \text{, }\quad W^u(x,R):=
\bigcap_{Z\in\mathfs Z:Z\supset R} V^u(x,Z)\cap R.
\end{align*}

By Proposition \ref{Prop-disjointness}, any two $s$--fibres ($u$--fibres) in $\mathfs R$ either coincide or are disjoint.

\begin{proposition}\label{Prop-R}
The following are true.
\begin{enumerate}[{\rm (1)}]
\item {\sc Product structure:} For every $R\in\mathfs R$ and every $x,y\in R$, the intersection
$W^s(x,R)\cap W^u(y,R)$ is a single point, and this point is in $R$. Denote it by $[x,y]$.
\item {\sc Hyperbolicity:} If $z,w\in W^s(x,R)$ then $d(H^n(z),H^n(w))\xrightarrow[n\to\infty]{}0$, and
if $z,w\in W^u(x,R)$ then $d(H^n(z),H^n(w))\xrightarrow[n\to-\infty]{}0$. The rates are exponential.
\item {\sc Geometrical Markov property:} Let $R_0,R_1\in\mathfs R$. If $x\in R_0\cap H^{-1}(R_1)$ then 
$$
H(W^s(x,R_0))\subset W^s(H(x),R_1)\, \text{ and }\, H^{-1}(W^u(H(x),R_1))\subset W^u(x,R_0).
$$
\end{enumerate}
\end{proposition}

For flows in three dimensions, this statement is \cite[Prop. 8.3]{BCL23}. Since the proof of part (3) there does not treat all cases, we have decided to include the proof of all cases for completeness.

\begin{proof}
Parts (1) and (2) are proved as in \cite[Prop. 8.3]{BCL23}. Let us prove part (3), first recasting the case treated in \cite{BCL23} and then proving the case not covered there.

Fix $R_0,R_1\in\mathfs R$ and $x\in R_0\cap H^{-1}(R_1)$.
We check that $H(W^s(x,R_0))\subset W^s(H(x),R_1)$ (the other inclusion is proved
similarly). Let $y\in W^s(x,R_0)$. By Proposition~\ref{Prop-overlapping-charts}(2)
and the definition of $W^s(H(x),R_1)$, it is enough to check that $H(x)\simN H(y)$.
Since $x\simN y$, we already know that $H^k(x),H^k(y)$ satisfy the properties (i) and (ii)
defining the relation $\simN$ when $-N\leq k\leq N$, hence it is enough to prove that this is also
true for $k=N+1$. The property (ii) for $k=N$ says that
$H^{N}(x), H^{N}(y)$ belong to the same elements of the partitions
$\mathfs E_Z$. We claim that this implies that $H^{N+1}(x),H^{N+1}(y)$ belong to the same
sets $Z\in \mathfs Z$, which gives (i) for $k=N+1$. To see this, let $Z'\in\mathfs Z$
such that $H^{N+1}(x)\in Z'$, and let $D'$ be the connected component of $\widehat\Lambda$
that contains $Z'$.
Let $Z\in\mathfs Z$ containing $H^N(x),H^N(y)$. Noting that 
$H^N(x)\in E^{su}_{Z,Z'}$, it follows from property (ii) for $k=N$ that 
$H^N(y)\in E^{su}_{Z,Z'}$, hence $\mathfrak q_{D'}(H^N(y))\in Z'$. If
$\mathfrak q_{D'}(H^N(y))=H^{N+1}(y)$, the claim is proved. If not, 
there is $Z''\in\mathfs Z$ such that $H^{N+1}(y)\in Z''$,
and so repeating the same argument with the roles of $x,y$ interchanged
gives that $\mathfrak q_{D''}(H^N(x))\in Z''$, a contradiction since 
the time transition from $Z$ to $Z''$ is smaller than time transitions from $Z$ to $Z'$.
Hence property (i) for $k=N+1$ is proved, and
it remains to prove property (ii) for $k=N+1$.

Let $Z\in \mathfs Z$ be a rectangle which contains $H^{N+1}(x),H^{N+1}(y)$
and let $D$ be the connected component of $\widehat\Lambda$ that contains $Z$.
We need to show that $H^{N+1}(x),H^{N+1}(y)$ belong to the same element of $\mathfs E_Z$.
We first note that $W^s(H^{N+1}(x),Z)=W^s(H^{N+1}(y),Z)$:
since $x,y$ belong to the same $s$--fibre of a rectangle in $\mathfs Z$,
this can be checked by applying Proposition~\ref{Prop-overlapping-charts}(2) inductively.
In particular, we have the following property:
\begin{equation}\label{e.stable}
\forall Z'\in \mathfs I_Z,\quad
W^s(H^{N+1}(x),Z)\cap \mathfrak{q}_{D}(Z')\neq \emptyset
\iff W^s(H^{N+1}(y),Z)\cap \mathfrak{q}_{D}(Z')\neq \emptyset.
\end{equation}
We then prove the analogous property for the sets $W^u(H^{N+1}(x),Z)$, $W^u(H^{N+1}(y),Z)$.

Write $H^{N+1}(x)=\pi(\un v)$ with
$\un v=\{v_n\}_{n\in\Z}=\{\Psi_{x_n}^{p^s_n,p^u_n}\}_{n\in\Z}\in\Sigma^\#$ and $Z=Z(v_0)$.
By Lemma~\ref{l.time}, there exists $0\leq k\leq N$ such that 
the point $\widetilde x:=H^k(x)$ coincides with $\pi[\sigma^{-1}(\un v)]$.
The rectangle $\widetilde Z:=Z(v_{-1})$ contains $\widetilde x$. By the induction assumption, the point $\widetilde y:=H^k(y)$ also belongs to $\widetilde Z$.

Let us consider $Z'\in  \mathfs I_Z$ and assume for instance that
$W^u(H^{N+1}(x),Z)\cap \mathfrak{q}_{D}(Z')$ contains a point $z$
(the case when $W^u(H^{N+1}(y),Z)\cap \mathfrak{q}_{D}(Z')\neq \emptyset$ is treated analogously).
Let $|s|<2\rho$ s.t. $\varphi^s(z)\in Z'$.
The symbolic Markov property in Proposition~\ref{Prop-Z}(4)
implies that the image of $W^u(\widetilde x,\widetilde Z)$ under $g^+_{x_{-1}}$
contains $W^u(H^{N+1}(x),Z)$, hence the point $z$.
In particular, the backward orbit of $z$ under the flow intersects $W^u(\widetilde x,\widetilde Z)$
at some point $\widetilde z=\vf^{\wt s}(z)$.
Here is where we make a distinction between two cases.

\medskip
\noindent
{\sc Case 1: $|s-\wt s|>\rho$.}

\medskip
This is the case treated in the proof of \cite[Prop. 8.3]{BCL23}.
Figure \ref{figure-markov} contains the points we will define below. Write $\varphi^s(z)=\pi(\un w)$ with $\un w=\{w_n\}_{n\in\Z}\in\Sigma^\#$ and
$Z'=Z(w_0)$. Since all transition times of holonomy maps 
are bounded by $\rho$, necessarily the piece of orbit $\vf^{[0,\rho]}(\widetilde z)$
contains some $\pi[\sigma^{-b}(\un w)]$ with $b\geq 1$. Let $b\geq 1$ and $0\leq \widetilde s'\leq \rho$
with $\pi[\sigma^{-b}(\un w)]=\varphi^{\widetilde s'}(\widetilde z)$.
Consequently the rectangle
$\widetilde Z':=Z(w_{-b})$ belongs to $ \mathfs I_{\widetilde Z}$. 
Moreover, $\widetilde z$ belongs to the intersection between $W^u(\widetilde x,\widetilde Z)$ and 
$\mathfrak q_{\wt D}(\widetilde Z')$, where $\widetilde D$ is the connected component of $\widehat\Lambda$ 
containing $\widetilde Z$.

\begin{figure}[hbt!]
\centering
\def\svgwidth{14cm}
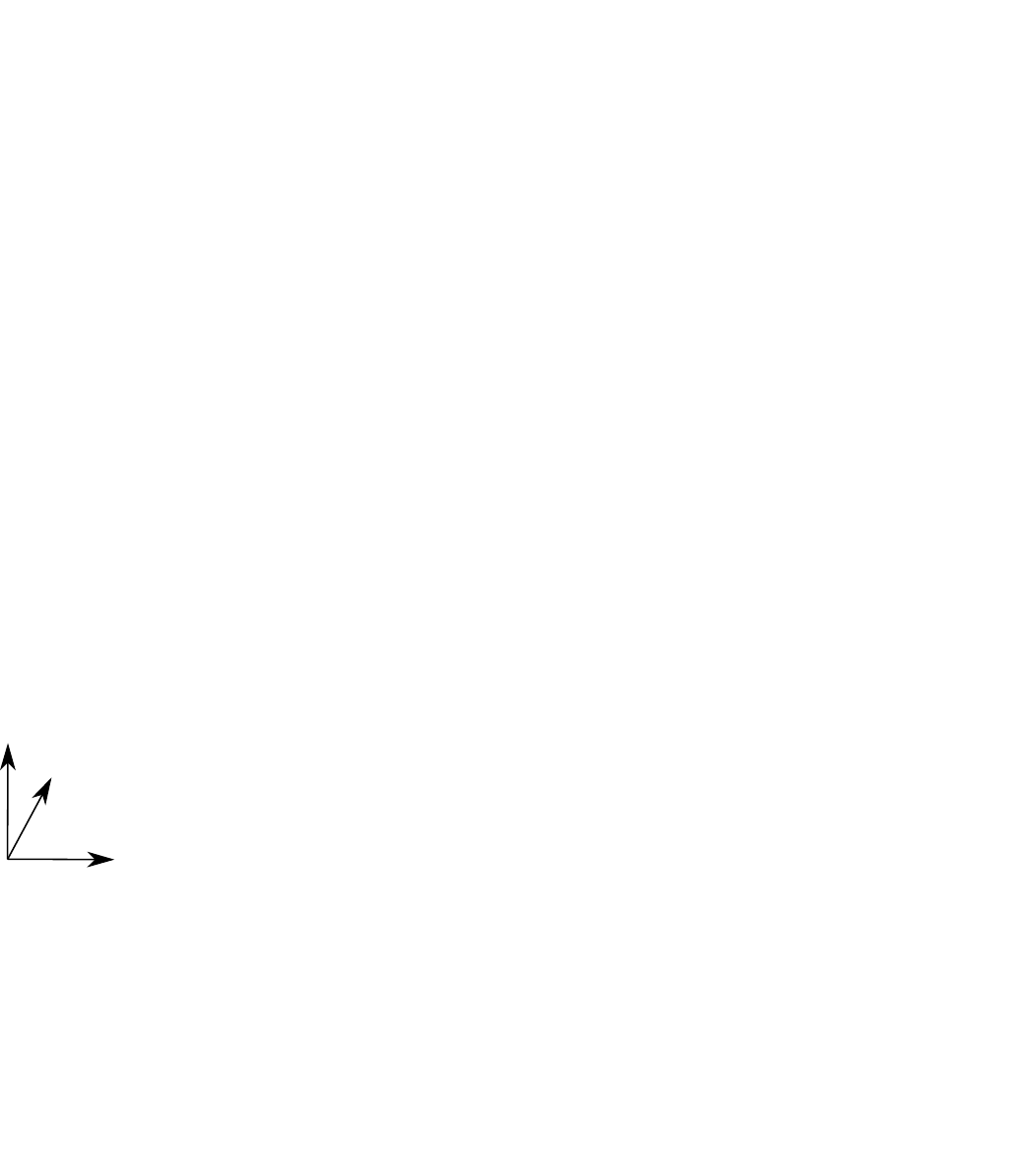\caption{Case 1: $|s-\wt s|>\rho$.}
\label{figure-markov}
\end{figure}

Again by the induction assumption, the point $\widetilde y:=H^k(y)$ belongs to the same element of the partition $\mathfs P_{\widetilde Z, \widetilde Z'}$ as $\widetilde x$.
Since $W^u(\widetilde x,\widetilde Z)$ intersects $\mathfrak q_{\widetilde D}(\widetilde Z')$,
the $u$--fibre $W^u(\widetilde y,\widetilde Z)$ intersects it as well at some point $\widetilde t$.
Note that $[\widetilde z, \widetilde t]_{\widetilde Z}=[\widetilde z, \widetilde y]_{\widetilde Z}$
also belongs to $W^u(\widetilde y,\widetilde Z)$ and to
$\mathfrak q_{\widetilde D}(\widetilde Z')$ (this last property follows from
Proposition \ref{Prop-overlapping-charts}(3), noting that 
$\widetilde z,\widetilde t\in \widetilde Z\cap \mathfrak q_{\widetilde D}(\widetilde Z')$),
hence we can replace $\widetilde t$
by any point in $W^u(\widetilde y,\widetilde Z)\cap \mathfrak q_{\widetilde D}(\widetilde Z')$.
Take $\widetilde t :=[\widetilde z,\widetilde y]_{\widetilde Z}$.

Let $|r|\leq 2\rho$ such that $\varphi^r(\widetilde t)\in W^s(\vf^{\widetilde s'}(\widetilde z), \widetilde Z')$.
The symbolic Markov property in Proposition~\ref{Prop-Z}(4) then implies that
its forward orbit under the flow will meet the rectangles
$Z(w_{-b})$,\dots, $Z(w_0)=Z'$, hence in particular it meets $Z'$. 

Note that $\widetilde z\in \widetilde Z=Z(v_{-1})$ and $z=g^+_{x_{-1}}(\widetilde z)\in Z=Z(v_0)$.
The same property holds for $\widetilde y$ and $H^{N+1}(y)=g^+_{x_{-1}}(\widetilde y)$ since
the points $H^i(x)$ and $H^i(y)$ belong to the same rectangles in $\mathfs Z$ for each $i=k,\dots,N+1$.
Using Proposition~\ref{Prop-overlapping-charts}(3), it follows that
the image of $\widetilde t=[\widetilde z,\widetilde y]_{\widetilde Z}$
by $g^+_{x_{-1}}$ belongs to $Z$ and coincides with the Smale product
$[z,H^{N+1}(y)]_Z$.

The properties found in the two previous paragraphs imply that
$W^u(H^{N+1}(y),Z)$ intersects $\mathfrak{q}_{D}(Z')$
at a point of the orbit of $\widetilde t$, contained in $W^s(z,Z)$.
In particular, the intersection $W^u(H^{N+1}(y),Z)\cap \mathfrak{q}_{D}(Z')$ is non-empty.
We have thus shown:
\begin{equation}\label{e.unstable}
\forall Z'\in \mathfs I_Z,\quad
W^u(H^{N+1}(x),Z)\cap \mathfrak{q}_{D}(Z')\neq \emptyset
\iff  W^u(H^{N+1}(y),Z)\cap \mathfrak{q}_{D}(Z')\neq \emptyset.
\end{equation}
Properties~\eqref{e.stable} and~\eqref{e.unstable} mean that
$H^{N+1}(x)$ and $H^{N+1}(y)$ belong to the same element of $\mathfs E_Z$
for any rectangle $Z\in \mathfs Z$ containing $H^{N+1}(x),H^{N+1}(y)$.
This concludes the proof that $H(x)\simN H(y)$ in this case.

\medskip
\noindent
{\sc Case 2: $|s-\wt s|\leq \rho$.}

\begin{figure}[hbt!]
\centering
\def\svgwidth{15cm}
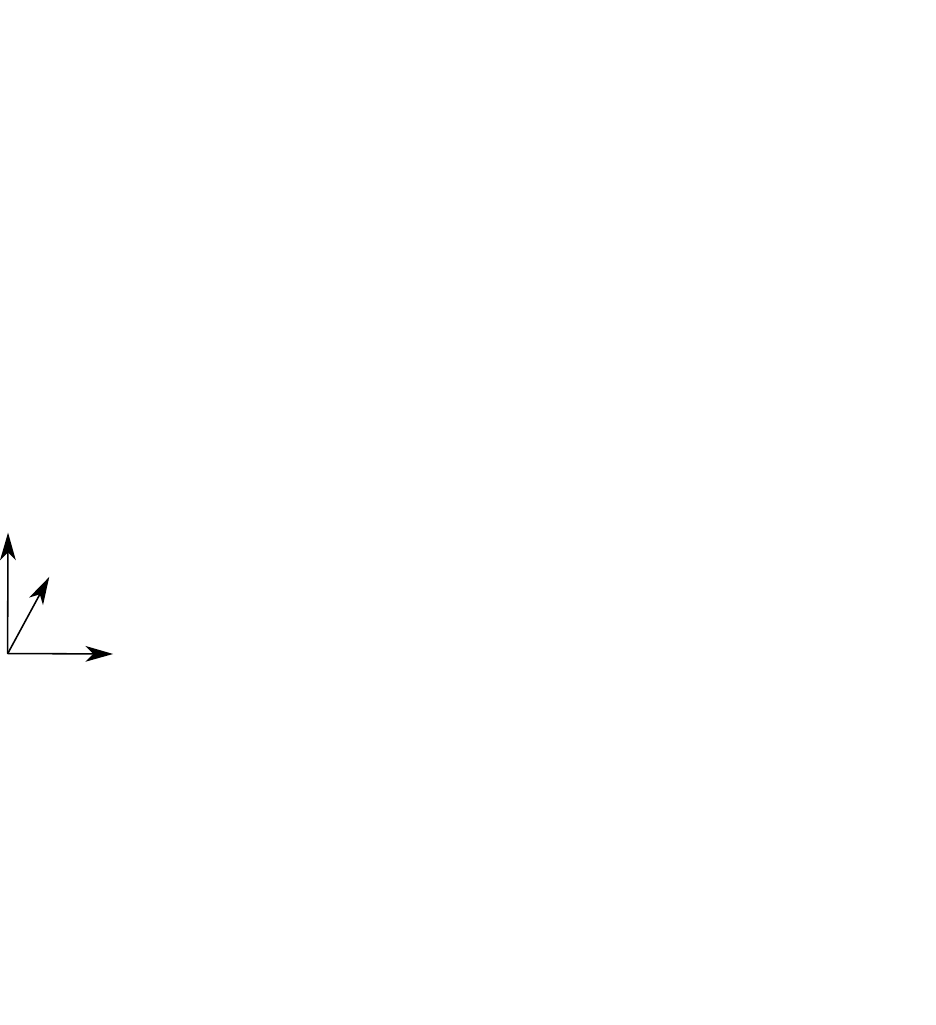\caption{Case 2: $|s-\wt s|\leq \rho$.}
\label{figure-markov-2}
\end{figure}

\medskip
We have that $\wt z\in \wt Z\cap \vf^{[-\rho,\rho]}Z'$, and so $Z'\in \mathfs  I_{\wt Z}$.
Now we adapt the proof of Case 1 as follows.
By the induction assumption, the point $\widetilde y:=H^k(y)$ belongs to the same element of the partition $\mathfs P_{\widetilde Z, Z'}$ as $\widetilde x$.
Since $W^u(\widetilde x,\widetilde Z)$ intersects $\mathfrak q_{\widetilde D}(Z')$,
the $u$--fibre $W^u(\widetilde y,\widetilde Z)$ intersects it as well at some point $\widetilde t$. 
As in Case 1, we can take 
$\widetilde t :=[\widetilde z,\widetilde y]_{\widetilde Z}$.
Using Proposition~\ref{Prop-overlapping-charts}(3) for the points $z\in Z$ and $\wt y\in \wt Z$, 
it follows that
the image of $\widetilde t=[\widetilde z,\widetilde y]_{\widetilde Z}$
by $g^+_{x_{-1}}$ belongs to $Z$ and coincides with the Smale product $[z,H^{N+1}(y)]_Z$. As in Case 1, we conclude the validity of property \eqref{e.unstable}.
This completes the proof that $H(x)\simN H(y)$ in this case, and of part (3) of the proposition.
\end{proof}

\section{A finite-to-one extension}

In this section, we construct a finite-to-one extension and deduce the Main Theorem.
We rely on the
family of disjoint sets $\mathfs R$ satisfying a geometrical Markov property.
This family was obtained in the previous section
as a refinement of the family $\mathfs Z$ constructed in Section \ref{Section-locally-finite-section}, which
was itself induced by the coding $\pi$ introduced in Section \ref{ss.first.coding}.
One important property of $\mathfs Z$ is that, due to the inverse theorem
(Theorem \ref{Thm-inverse}), it satisfies a local finiteness property, see Proposition \ref{Prop-Z}(2).
We use these objects to construct a symbolic coding of the return map $H$.

\subsection{A detailed statement}\label{s.detailed}

The theorem below implies the Main Theorem and includes additional properties that will
be useful for some applications, including the one we will obtain in Section \ref{sec.homoclinic}.
We begin defining a Bowen relation for flows. This notion was formalized for diffeomorphisms in \cite{Boyle-Buzzi},
and the following is an adaptation for flows, introduced in \cite{BCL23}. We refer to \cite{Buzzi-JMD} for a discussion on the notion.

\newcommand\hS{\widehat S}

Let  $T_r: S_r\to S_r$ be a suspension flow over a symbolic system $S$ that is an extension
of some flow $U:X\to X$ by a semiconjugacy map $\pi:S_r\to X$,  i.e.
$U^t\circ\pi=\pi\circ T^t_r$ for all $t\in\R$.\\

\noindent
{\sc Bowen relation:} A \emph{Bowen relation} $\sim$ for $(T_r,\pi,U)$ is a symmetric binary
relation on the alphabet of $S$ satisfying the following two properties:
\begin{enumerate}[i,]
\item[{\rm (i)}] $\forall\omega,\omega'\in S_r,\;\; \pi(\omega)=\pi(\omega')\implies \operatorname{v}(\omega)\sim\operatorname{v}(\omega')$, where $\operatorname{v}(x,t):=x_0$ for $x\in S$;
\item[{\rm (ii)}] $\exists \gamma>0$ with the following property:
$$\forall\omega,\omega'\in S_r,\;\; \left[ \forall t\in\R,\; \operatorname{v}(T_r^t\omega)\sim\operatorname{v}(T_r^t\omega') \right] \implies \left[ \exists |t|<\gamma\text{ s.t. } \pi(\omega)= U^t(\pi(\omega')) \right].
$$
  \end{enumerate}

\begin{theorem}\label{t.main}
Let $X$ be a non-singular $C^{1+\beta}$ vector field ($\beta>0$) on a closed manifold $M$.
Given $\chi>0$, there exists a locally compact topological Markov flow
$(\widehat \Sigma_{\widehat r},\widehat\sigma_{\widehat r})$ with graph
$\widehat{\mathfs G}=(\widehat V,\widehat E)$ and roof function $\widehat{r}$
and a map $\widehat \pi_{\widehat r}:\widehat \Sigma_{\widehat r}\to M$ such that
$\widehat \pi_{\widehat r}\circ {\widehat\sigma}_{\widehat r}^t=\vf^t\circ\widehat \pi_{\widehat r}$ for all $t\in\R$, and satisfying:
\smallskip
\begin{enumerate}[{\rm (1)}]
\item $\widehat r$ and $\widehat \pi_{\widehat r}$ are H\"older continuous.
\smallskip

\item $\widehat \pi_{\widehat r}[\widehat \Sigma_{\widehat r}^\#]=\nuh^\#$ has full measure for every $\chi$--hyperbolic measure; for every ergodic $\chi$--hyperbolic measure $\mu$,
there is an ergodic $\widehat\sigma_{\widehat r}$--invariant measure $\overline \mu$ 
on $\widehat \Sigma_{\widehat r}$
such that $\overline \mu\circ\widehat \pi_{\widehat r}^{-1}=\mu$ and $h_{\overline{\mu}}(\widehat \sigma_{\widehat{r}})=h_\mu(\vf)$.
\smallskip

\item If
$(\un R,t)\in \widehat \Sigma_{\widehat r}^{\#}$ satisfies
$R_n=R$ and $R_m=S$ for infinitely many $n<0$ and $m>0$, then $\operatorname{Card}\{z\in \widehat \Sigma_{\widehat r}^\#:\widehat \pi_{\widehat r}(z)=\widehat \pi_{\widehat r}(\un R,t)\}$
is bounded by a number $C(R,S)$, depending only on $R,S$.
\smallskip

\item\label{i.splitting} There is $\lambda>0$ and for $x\in \widehat \pi_{\widehat r}(\widehat \Sigma_{\widehat r})$ there is a unique splitting
$N_{x}=N^s_{x}\oplus N^u_{x}$ such that:
\begin{align*}
\limsup_{t\to +\infty} \tfrac{1}{t}\log \|\Phi^t|_{N^s_{x}}\|\leq -\lambda
\quad \text{ and }\quad\limsup_{t\to +\infty} \tfrac{1}{t}\log\|\Phi^{t}|_{N^s_{\varphi^{-t}(x)}}\|\leq -\lambda\\
\quad \limsup_{t\to +\infty} \tfrac{1}{t}\log\|\Phi^{-t}|_{N^u_{x}}\|\leq -\lambda
\quad \text{ and }\quad\limsup_{t\to +\infty} \tfrac{1}{t}\log \|\Phi^{-t}|_{N^u_{\varphi^t(x)}}\|\leq -\lambda.
\end{align*}
The splitting is $\Phi$--equivariant, and the maps $z\mapsto N^{s/u}_{\widehat \pi_{\widehat r}(z)}$ are H\"older continuous on $\widehat \Sigma_{\widehat r}$.
\smallskip

\item\label{i.manifold} There is $\alpha>0$ and for every $z\in \widehat \Sigma_{\widehat r}$
there are $C^1$ submanifolds $V^{cs}(z),V^{cu}(z)$ passing through $x:=\widehat \pi_{\widehat r}(z)$
such that:
\begin{enumerate}[{\rm (a)}]
\item
$T_{x}V^{cs}(z)=N^s_x+\mathbb{R}\cdot X(x)$ and $T_{x}V^{cu}(z)=N^{u}_x+\mathbb{R}\cdot X(x)$.
\item For all $y\in V^{cs}(z)$, there is $\tau\in \mathbb{R}$ such that
$d(\varphi^t(x),\varphi^{t+\tau}(y))\leq e^{-\alpha t}$, $\forall t\geq 0$.
\item For all $y\in V^{cu}(z)$, there is $\tau\in \mathbb{R}$ such that
$d(\varphi^{-t}(x),\varphi^{-t+\tau}(y))\leq e^{-\alpha t}$, $\forall t\geq 0$.
\end{enumerate}
\smallskip

\item\label{i.Bowen} There is a symmetric binary relation $\sim$ on the alphabet $\widehat V$
satisfying:
\begin{enumerate}[{\rm (a)}]
\item For any $R\in \widehat V$, the set $\{S\in \widehat V:R\sim S\}$ is finite.
\item The relation $\sim$ is a Bowen relation for $(\widehat\sigma_{\widehat r},\widehat\pi_{\widehat r}|_{\widehat\Sigma^\#_{\widehat r}},\vf^t)$.
\end{enumerate}
\smallskip

\item\label{i.canonical} There exists a measurable set $\mathfs R$ with a measurable partition
indexed by $\widehat V$, which we denote by $\{R:R\in\widehat V\}$, such that:
\begin{enumerate}[{\rm (a)}]
\item The orbit of any point $x\in \nuh^\#$ intersects $\mathfs R$.
\item The first return map $H\colon \mathfs R\to \mathfs R$ induced by $\varphi$ is a well-defined bijection.
\item For any $x\in \mathfs R$, if $\un R=\{R_n\}_{n\in\Z}$ satisfies $H^n(x)\in R_n$ for all $n\in\Z$,
then $(\un R,0)\in \widehat \Sigma^\#_{\widehat r}$ and $\widehat \pi_{\widehat r}(\un R,0)=x$.
\end{enumerate}

\item\label{i.lift}
For any compact transitive invariant hyperbolic set $K\subset M$ whose ergodic $\vf$--invariant
measures are all $\chi$--hyperbolic, there is a {\em transitive} invariant compact set 
$X\subset \widehat \Sigma_{\widehat r}$ such that $\widehat\pi_{\widehat r}(X)=K$.
\end{enumerate}
\end{theorem}

For flows in three dimensions, this is \cite[Theorem 9.1]{BCL23}.
Part~\eqref{i.Bowen}
is a combinatorial characterization of the noninjectivity of  the coding.
It is an adaptation for flows of the \emph{Bowen property}, which was
introduced in \cite{Boyle-Buzzi} for diffeomorphisms and motivated by the work of
Bowen \cite{Bowen-Regional-Conference}. Note that, in contrast to \cite{Bowen-Regional-Conference},
we {\em do not} claim that the flow restricted to $\widehat\pi_{\widehat r}[\widehat\Sigma_{\widehat r}^\#]$
is topologically equivalent to the corresponding quotient dynamics.

The relation $\sim$ will be the \emph{affiliation}, which will be introduced in Section~\ref{s.finite-one},
following a similar notion introduced in \cite{Sarig-JAMS}. 
Note that the assumption 
$\bigl[\operatorname{v}(\widehat \sigma_{\widehat r}^t(z))\sim \operatorname{v}(\widehat\sigma_{\widehat r}^t(z'))$
for all $t\in \mathbb{R}\bigr]$ consists of countably many
affiliation conditions: if $z=(\un R,s)$ and $z'=(\un S,s')$, then varying $t$ in the interval $[\widehat r_n(\un R),\widehat r_{n+1}(\un R))$
provides $i\leq \tfrac{\sup(\widehat r)}{\inf(\widehat r)}$ affiliations of the form
$R_n\sim S_{m+1},\ldots,R_n\sim S_{m+i}$.

Part~\eqref{i.canonical} provides for any
$x\in  \nuh^\#$ a particular pair $(\un R,t)\in \widehat \Sigma_{\widehat r}^\#$
such that $\widehat \pi_{\widehat r}(\un R,t)=x$
($t$ is the smallest non-negative number such that $\varphi^{-t}(x)\in \mathfs R$). We call
the pair $(\un R,t)$ the \emph{canonical lift} of $x$.
This is a measurable embedding of $\nuh^\#$ into $\widehat\Sigma_{\widehat r}$.
\color{black}

Part~\eqref{i.lift} is a version of \cite[Proposition 3.9]{BCS-MME} in our context, and the proof
is very similar, see Section~\ref{ss.conclusion}.

\subsection{Second coding}\label{s.second}
Let $\widehat{\mathfs G}=(\widehat V,\widehat E)$ be the oriented graph with vertex set
$\widehat V=\mathfs R$ and edge set $\widehat E=\{R\to S:R,S\in\mathfs R\text{ s.t. }H(R)\cap S\neq\emptyset\}$,
and let $(\widehat\Sigma,\widehat\sigma)$ be the TMS induced by $\widehat{\mathfs G}$.
We note that the ingoing and outgoing degree of every vertex in $\widehat\Sigma$ is finite. We show this
for the outgoing edges, since the proof for the ingoing edges is analogous. Fix $R\in\mathfs R$, and fix
$Z\in\mathfs Z$ such that $Z\supset R$. 
If $(R,S)\in\widehat E$ then $\vf^{[0,\rho]}(R)\cap S\ne\emptyset$, hence for any $Z'\in\mathfs Z$ with $S\subset Z'$, we have $Z'\in\mathfs I_Z$. In particular,
$$
\#\{(R,S)\in\widehat E\}\leq \sum_{Z'\in\mathfs I_Z}\#\{S\in\mathfs R:S\subset Z'\}<+\infty,
$$
since both $\mathfs I_Z$ and each $\{S\in\mathfs R:S\subset Z'\}$ are finite sets
(see Lemma \ref{Lemma-local-finite}(1)).

For $\ell\in\Z$ and a path $R_m\to\cdots\to R_n$ on $\widehat{\mathfs G}$ define
$$
_\ell[R_m,\ldots,R_n]:=H^{-\ell}(R_m)\cap\cdots\cap H^{-\ell-(n-m)}(R_n),
$$
the set of points whose itinerary
under $H$ from $\ell$ to $\ell+(n-m)$ visits the rectangles $R_m,\ldots,R_n$ respectively.
The crucial property that
gives the new coding is that $_\ell[R_m,\ldots,R_n]\neq\emptyset$. This follows by induction, using the
Markov property of $\mathfs R$ (Proposition \ref{Prop-R}(3)).

The map $\pi$ defines similar sets: for $\ell\in\Z$ and a path
$v_m\overset{\ve}{\to}\cdots\overset{\ve}{\to}v_n$ on $\Sigma$, let
$$
Z_\ell[v_m,\ldots,v_n]:=\{\pi(\un w):\un w\in\Sigma^\#\text{ and }w_\ell=v_m,\ldots,w_{\ell+(n-m)}=v_n\}.
$$
There is a relation between these sets we just defined. Before stating such a relation, we will define the coding of $H$,
and then collect some of its properties.

\medskip
\noindent
{\sc The map $\widehat\pi:\widehat\Sigma\to M$:} Given $\un R=\{R_n\}_{n\in\Z}\in\widehat\Sigma$,
$\widehat\pi(\un R)$ is defined by the identity
$$
\{\widehat\pi(\un R)\}:=\bigcap_{n\geq 0}\overline{_{-n}[R_{-n},\ldots,R_n]}.
$$

\medskip
Note that $\widehat\pi$ is well-defined, because the right hand side is an intersection of
nested compact sets with diameters going to zero.
The proposition below states relations between $\Sigma$ and $\widehat\Sigma$, and between $\pi$ and $\widehat\pi$.
For $\un v=\{\Psi_{x_n}^{p^s_n,p^u_n}\}_{n\in\Z}\in\Sigma$, let
$$
G_{\un v}^n=\left\{
\begin{array}{ll}
g_{x_{n-1}}^+\circ\cdots\circ g_{x_0}^+ &,n\geq 0\\
g_{x_{n+1}}^-\circ\cdots\circ g_{x_0}^- &,n<0.
\end{array}
\right.
$$
Recall the integer $N$ introduced in Lemma \ref{l.time}.

\begin{proposition}\label{Prop-relation-codings}
For each $\un R=(R_n)_{n\in\Z}\in\widehat\Sigma$ and $Z\in\mathfs Z$ with
$Z\supset R_0$, there are an $\ve$--gpo
$\un v=\{v_k\}_{k\in\Z}\in\Sigma$ with $Z(v_0)=Z$ and a sequence
$(n_k)_{k\in\Z}$ of integers with $n_0=0$ and $1\le n_k-n_{k-1}\le N$ for all $k\in\Z$ such that:
\begin{enumerate}[{\rm (1)}]
\item For each $k\ge1$, $$ _{n_{-k}}[R_{n_{-k}},\ldots,R_{n_k}]\subset Z_{-k}[v_{-k},\ldots,v_k].$$
In particular, $\widehat\pi(\un R)=\pi(\un v)$.
Moreover, $R_{n_k}\subset Z(v_k)$ for all $k\in\Z$.

\item The map $\widehat\pi$ is H\"older continuous over $\widehat\Sigma$. In fact,
$\{v_i\}_{|i|\le k}$ depends only on $\{R_j\}_{|j|\le kN}$ for each $k\ge1$.

\item If $\un R\in {\widehat \Sigma}^\#$, then $\un v\in {\Sigma}^\#$.

\item The two codings have the same regular image: $\pi[\Sigma^\#]=\widehat\pi[\widehat\Sigma^\#]$.
\end{enumerate}
\end{proposition}

For diffeomorphisms, the above lemma is \cite[Lemma 12.2]{Sarig-JAMS}. The difference from the case
of diffeomorphisms relies on our definitions of $\mathfs G$ and $\widehat{\mathfs G}$. While the edges of
$\widehat{\mathfs G}$ correspond to possible time evolutions of $H$, the edges of $\mathfs G$ correspond
to $\ve$--overlaps. In particular, not every edge of $\widehat{\mathfs G}$ corresponds to an edge of $\mathfs G$,
and this is the reason we have to introduce the sequence $(n_k)_{k\in\Z}$. 
In fact, each edge $v_k\to v_{k+1}$ of $\mathfs G$ corresponds to a sequence of edges $R_{n_k}\to\dots\to R_{n_{k+1}}$ of $\widehat{\mathfs G}$. For three dimensional flows, this statement is \cite[Prop. 9.2]{BCL23}, and the same proof applies to high dimenson.

We now define the topological Markov flow (TMF) and coding that satisfy the Main Theorem.
For that, recall the definition of TMF in Section \ref{Section-Preliminaries}.

\medskip
\noindent
{\sc The triple $({\widehat\Sigma}_{\widehat r},{\widehat\sigma}_{\widehat r},{\widehat\pi}_{\widehat r})$:} 
The topological Markov flow $({\widehat\Sigma}_{\widehat r},{\widehat\sigma}_{\widehat r})$ is the suspension of $(\widehat\Sigma,\widehat\sigma)$ by the roof function $\widehat r:\widehat\Sigma\to(0,\rho)$ defined by
$$
 \widehat r(\un R) := \min\{t>0:\vf^t(\widehat\pi(\un R)) = \widehat\pi(\widehat\sigma(\un R))\},
$$
and the factor map ${\widehat\pi}_{\widehat r}:\widehat\Sigma_{\widehat r}\to M$ is 
given by $\widehat\pi_{\widehat r}(\un R,s):=\vf^s(\widehat\pi(\un R))$.

\medskip
As claimed above, we have $\sup(\widehat r)<\rho$. Indeed, by 
Proposition \ref{Prop-relation-codings} there is $\un v=\{v_n\}_{n\in\Z}\in\Sigma$
such that $\widehat\pi(\un R)=\pi(\un v)$, and 
there are integers $n_{-1}<0<n_1$ such that  
$ _{n_{-1}}[R_{n_{-1}},\ldots,R_{n_1}]\subset Z_{-1}[v_{-1},v_0,v_1]$, hence
$\widehat r(\un R)\leq \widehat r_{n_1}(\un R)=r(\un v)<\rho$.
The rest of this section is devoted to proving that
$({\widehat\Sigma}_{\widehat r},{\widehat\sigma}_{\widehat r},{\widehat\pi}_{\widehat r})$
satisfies Theorem~\ref{t.main}. We start with some fundamental properties.

\begin{proposition}\label{Prop-pi_r}
The following holds for all $\ve>0$ small enough.
\begin{enumerate}[{\rm (1)}]
\item $\widehat r:\widehat\Sigma\to(0,\rho)$ is well-defined and H\"older continuous.
\item $\widehat\pi_{\widehat r}\circ\widehat\sigma_{\widehat r}^t=\vf^t\circ\widehat \pi_{\widehat r}$,
for all $t\in\R$.
\item $\widehat\pi_{\widehat r}$ is H\"older continuous with respect to the Bowen-Walters distance.
\item $\widehat\pi_{\widehat r}[\widehat\Sigma_{\widehat r}^\#]=\nuh^\#$.
\end{enumerate}
\end{proposition}

For three dimensional flows, this is \cite[Prop. 9.3]{BCL23} and the same proof applies to high dimension.
By Proposition \ref{Prop-adaptedness}, the above proposition establishes Part
(1) and the first half of Part (2) of Theorem~\ref{t.main}. In the next sections, we focus on proving the remaining statements.

\subsection{The map $\widehat\pi_r$ is finite-to-one}\label{s.finite-one}

Given $Z\in\mathfs Z$, remember that $\mathfs I_Z=\{Z'\in\mathfs Z:\vf^{[-\rho,\rho]}Z\cap Z'\neq\emptyset\}$.
The loss of injectivity of $\widehat\pi_{\widehat r}$ is related to the following notion.

\medskip
\noindent
{\sc Affiliation:} We say that two rectangles $R,S\in\mathfs R$ are {\em affiliated}, and write
$R\sim S$, if there are
$Z,Z'\in\mathfs Z$ such that $R\subset Z$, $S\subset Z'$ and $Z'\in\mathfs I_Z$. 
This is a symmetric relation.

\begin{lemma}\label{Lemma-affiliation}
If $\widehat\pi(\un R)=\vf^t[\widehat\pi(\un S)]$ with $\un R,\un S\in\widehat\Sigma^\#$ and $|t|\leq\rho$, then $R_0\sim S_0$.
More precisely, if $\un v,\un w\in\Sigma^\#$ are such that $\pi(\un v)=\widehat\pi(\un R)$ and $\pi(\un w)=\widehat\pi(\un S)$, then $R_0\subset Z(v_0)$ and $S_0\subset Z(w_0)$ with $Z(w_0)\in\mathfs I_{Z(v_0)}$.
\end{lemma}

\begin{proof}
For three dimensional flows, this is \cite[Lemma 9.4]{BCL23}. Since its proof is short, we reproduce it here.
Let $y=\widehat\pi(\un R)$ and $z=\widehat\pi(\un S)$, so that $y=\vf^{t}(z)$.
Applying Proposition \ref{Prop-relation-codings} to $\un R$ and $\un S$, there are
two $\ve$--gpo's $\un v,\un w\in \Sigma^\#$ such that:
\begin{enumerate}[$\circ$]
\item $\pi(\un v)=y$ and $R_0\subset Z(v_0)$,
\item $\pi(\un w)=z$ and $S_0\subset Z(w_0)$.
\end{enumerate}
The lemma thus follows with $Z=Z(v_0)$ and $Z'=Z(w_0)$, since $\vf^t(z)\in Z(v_0)$.
\end{proof}

For each $R\in\mathfs R$, define
$$
A(R):=\{(S,Z')\in\mathfs R\times\mathfs Z:R\sim S\text{ and }S\subset Z'\}\text{ and }N(R):=\#A(R).
$$
We can use Lemma \ref{Lemma-local-finite} and proceed as in the proof of \cite[Lemma 12.7]{Sarig-JAMS}
to show that $N(R)<\infty$, $\forall R\in\mathfs R$. Having this in mind, we are able to state the finiteness-to-one
property of $\widehat\pi_{\widehat r}$, i.e. part (3) of the Main Theorem and of Theorem~\ref{t.main}.

\begin{theorem}\label{Thm-finite-extension}
Every $x\in \widehat\pi_{\widehat r}[\widehat\Sigma_{\widehat r}^\#]$ has finitely many
$\widehat\pi_{\widehat r}$--preimages inside $\widehat\Sigma_{\widehat r}^\#$.
More precisely, if $x=\widehat\pi_{\widehat r}(\un R,t)$ with $R_n=R$ for infinitely many $n>0$
and $R_n=S$ for infinitely many $n<0$, then
$\#\{(\un S,t')\in\widehat\Sigma^\#_{\widehat r}:\widehat\pi_{\widehat r}(\un S,t')=x\}\leq N(R)N(S)$.
\end{theorem}

For three dimensional flows, this statement is \cite[Theorem 9.5]{BCL23}. Its proof, which consists of an adaptation of \cite[Theorem 12.8]{Sarig-JAMS} for three dimensional flows, works equally well in high dimension. The idea, due to Bowen \cite[pp. 13--14]{Bowen-Regional-Conference} and commonly known as the ``Bowen diamond'', explores the (non-uniform) expansiveness of $\vf$, expressed in terms of the uniqueness of shadowing (Proposition \ref{Prop-shadowing}). Assuming by contradiction that $x$ has more than $N(R)N(S)$ pre-images in $\widehat\Sigma_{\widehat r}^\#$, two of them must coincide in arbitrarly large positions in the past and future. Using the compatibility of the Smale bracket under holonomy maps (Propositions \ref{Prop-overlapping-charts} and \ref{Prop-overlapping-charts-2}) and the geometrical Markov property (Proposition \ref{Prop-R}), we thus obtain a contradiction.

\subsection{Conclusion of the proof of Theorem~\ref{t.main}}\label{ss.conclusion}

Except for arguments involving the angles between $N^s$ and $N^u$, this section is the same as \cite[Section 9.4]{BCL23}. As it is the conclusion of the construction developed in the article, we include it for completeness.

We already proved Part (1) and the first half of Part (2). Also, Theorem \ref{Thm-finite-extension}
establishes Part (3). For the second half of Part (2), we note that
every point of $\nuh^\#$ has a finite and nonzero number of lifts to $\widehat \Sigma^\#_{\widehat r}$,
hence every ergodic $\chi$--hyperbolic measure on $M$, which is supported in $\nuh^\#$,
can be lifted to an ergodic $\widehat\sigma_{\widehat r}$--invariant measure $\overline\mu$,
exactly as in the argument performed in~\cite[Section 13]{Sarig-JAMS}. This concludes the proof of 
Part (2) of Theorem~\ref{t.main}.

We now prove the remaining Parts (4)--(8) stated in Theorem~\ref{t.main}.

\paragraph{\bf Part~(\ref{i.splitting}).} Using Theorem~\ref{Thm-stable-manifolds}, 
we define $N^{s/u}_{z}$ as follows:
\begin{enumerate}[$\circ$]
\item  For $z=(\un R,0)\in\widehat\Sigma_{\widehat r}$, fix some $Z \in \mathfs Z$ such that $R_0\subset Z$ and let $V^{s/u}(z):=V^{s/u}(\widehat{\pi}(z),Z)$. Then define $N^{s/u}_{z}:=T_{\widehat\pi(\un R)}V^{s/u}(z)$. Note that $V^{s/u}(z)$ depends on the choice of $Z$, but $N^{s/u}_{z}$ does not. By definition, $N_{\wh\pi(\un R)}=N^s_z\oplus N^u_z$.
\item For $z=(\un R,t)\in\widehat\Sigma_{\widehat r}$, define
$N^{s/u}_{z}=\Phi^t\left(N^{s/u}_{(\un R,0)}\right)$.
Since $\Phi$ is an isomorphism, $N_{\widehat\pi_{\widehat r}(\un R,t)}=N^{s}_{z}\oplus N^{u}_{z}$.
\end{enumerate}
The geometrical Markov property of Proposition \ref{Prop-R}(3) 
implies that the families $\{N^{s/u}_z\}$ are invariant under $\Phi$.
The convergence rates along $N^{s/u}_z$ follow from Theorem \ref{Thm-stable-manifolds}(3),
taking $\lambda:=\tfrac{2\chi}{3}-\tfrac{\beta\ve}{6\inf(r_\Lambda)}$. 
These estimates show, in particular, that these spaces only depend on $x:=\widehat\pi_{\widehat r}(z)$,
hence one can set $N^{s/u}_x:=N^{s/u}_z$.
Finally, the H\"older continuity follows from 
Theorem \ref{Thm-stable-manifolds}(4). This concludes the proof of part~(\ref{i.splitting}).
\medskip

\paragraph{\bf Part~\eqref{i.manifold}.}
For $z=(\un R,t)\in\widehat\Sigma_{\widehat r}$, 
one defines the manifolds $V^{cs/cu}(z):=\varphi^{[t-1,t+1]}(V^{s/u}(\un R,0))$.
By construction, $V^{cs/cu}(z)$ is tangent to $N^{s/u}_z+\mathbb{R}\cdot X(\wh{\pi}_{\wh{r}}(z))$.
Setting $\alpha:=\tfrac{\chi\inf(r_\Lambda)}{4 \sup(r_\Lambda)}$, by Proposition~\ref{Prop-center-stable}, for any $y\in V^{cs}(z)$ there exists $\tau\in \mathbb{R}$
such that $d(\varphi^t(\widehat\pi_{\widehat r}(z)),\varphi^{t+\tau}(y))\leq e^{-\alpha t}$
for all $t\geq 0$. The same holds for $V^{cu}(z)$, thus concluding the proof of Part~\eqref{i.manifold}.

\medskip

\paragraph{\bf Part~\eqref{i.canonical}.} The proof of this part is almost automatic.
The measurable set $\mathfs Z=\mathfs R$ contains $\Lambda\cap \nuh^\#$, hence
the orbit of any point $x\in \nuh^\#$ intersects $\mathfs R$, which proves item (a).
Item (b) was proved in the beginning of Section~\ref{subsec-fundpropZ}.
Finally, any $x\in \mathfs R$ defines $\{R_n\}_{n\in\Z}$
such that $H^n(x)\in R_n$ for all $n\in\Z$.
In particular, $H(R_n)\cap R_{n+1}\neq \emptyset$ for all $n\in\Z$ and so
$\un R=\{R_n\}\in\widehat\Sigma$. Since $\mathfs R=\pi[\Sigma^\#]$, we also
have $x=\pi(\un v)$ for some $\un v=\{v_n\}_{n\in\Z}\in\Sigma^\#$. For each $k\in\Z$,
the point $\pi[\sigma^k(\un v)]$ is a return of $x$ to $\mathfs R$, hence
there is an increasing sequence such that $\pi[\sigma^k(\un v)]=H^{n_k}(x)$.
Therefore $R_{n_k}\subset Z(v_k)$.
Using that $\un v\in\Sigma^\#$ and Lemma \ref{Lemma-local-finite}(1), it follows that
$\un R\in\widehat\Sigma^\#$.
\medskip

\paragraph{\bf Part~\eqref{i.lift}.}
Assume $K\subset M$ is a compact, transitive, invariant, hyperbolic set
such that all ergodic $\vf$--invariant measures supported by it are $\chi$--hyperbolic. Let
$TK=E^s\oplus X\oplus E^u$ be the continuous hyperbolic splitting.
Proceeding as in \cite[Proposition 2.8]{BCS-MME}, there are constants $C_0=C_0(K)>0$ and $\kappa>\chi$
such that
$$
\|d\vf^t v^s\|\leq C_0e^{-\kappa t}\|v^s\| \text{ and } \|d\vf^{-t} v^u\|\leq C_0e^{-\kappa t}\|v^u\|\text{,\ \ for all }v^s\in E^s_K,
v^u\in E^u_K\text{ and }t\geq 0.
$$ 
Since $\mathfrak p\restriction_{E^{s/u}}:E^{s/u}\to N^{s/u}$
is an isomorphism, and since the maps $x\in K\to E^{s/u}_x$ and $x\in M\to N^{s/u}_x$ are continuous, we have $\|\mathfrak p_x^{\pm 1}\|=e^{\pm{\rm const}}$ for all $x\in K$.
Hence, there is $C_1=C_1(K)>0$ s.t.
$$
\|\Phi^t v^s\|\leq C_1e^{-\kappa t}\|v^s\| \text{ and } \|\Phi^{-t} v^u\|\leq C_1e^{-\kappa t}\|v^u\|\text{,\ \ for all }v^s\in N^s_K,
v^u\in N^u_K\text{ and }t\geq 0.
$$ 
This clearly implies that there is a constant 
$C_2>0$ s.t. $\vertiii{v}\leq C_2\|v\|$ for
all $v\in N^{s/u}_K$. Therefore, for non-zero $v=v^s+v^u\in N^s_K\oplus N^u_K$ we have
$$
\tfrac{\vertiii{v}}{\|v\|}=\tfrac{\sqrt{\vertiii{v^s}^2 + \vertiii{v^u}^2}}{\|v^s +v^u\|}\leq C_2\tfrac{\sqrt{\|v^s\|^2 + \|v^u\|^2}}{\|v^s +v^u\|}\leq 
C_2\tfrac{\|v^s\|+\|v^u\|}{\|v^s +v^u\|}\cdot
$$
Since $\inf_{x\in K}\angle (N^s_x,N^u_x)>0$, this later fraction has an upper bound $C_3$, thus 
by Lemma \ref{Lemma-linear-reduction} we get 
that $\|C(x)^{-1}\|\leq C_2C_3$ for all $x\in K$. Thus $\inf_{x\in K}Q(x)>0$, which then implies that 
$\inf_{x\in K}q(x)>0$.
In particular, $K\subset\nuh^\#$.
This property is enough to reproduce the method of proof of \cite[Prop. 3.9]{BCS-MME}, as follows.
We recall that $X\subset\widehat\Sigma_{\widehat r}$ is $\widehat\sigma_{\widehat r}$--invariant if 
$\widehat\sigma_{\widehat r}^t(X)=X$ for all $t\in\R$.

\medskip
\noindent
{\sc Step 1:} There is a $\widehat\sigma_{\widehat r}$--invariant compact set $X_0\subset \widehat\Sigma_{\widehat r}$ such that $\widehat\pi_{\widehat r}(X_0)\supset K$.

\begin{proof}[Proof of Step $1$]
For each $x\in K\cap\mathfs R$, consider its canonical coding $\un R(x)=\{R_n(x)\}_{n\in\Z}$.
Since $\inf_{x\in K}q(x)>0$, $K$ intersects finitely many rectangles of $\mathfs R$. Hence there is a finite
set $V_0\subset\mathfs R$ such that $R_0(x)\in V_0$ for all $x\in K\cap\mathfs R$.
By invariance, the same happens for all $n\in\Z$, i.e.
$R_n(x)\in V_0$ for all $x\in K\cap\mathfs R$. Therefore 
the subshift $\Sigma_0$ induced by $V_0$, which is compact since $V_0$ is finite,
satisfies $\widehat\pi(\Sigma_0)\supset K\cap\mathfs R$.
Let $X_0$ be the TMF defined by $(\Sigma_0,\sigma)$ with roof function $\widehat r\restriction_{\Sigma_0}$.
Saturating the last inclusion under $\vf$ and using part~(7)(a), we conclude that
$\widehat\pi_{\widehat r}(X_0)\supset K$.
\end{proof}

\medskip
\noindent
{\sc Step 2:} There is a transitive $\widehat\sigma_{\widehat r}$--invariant compact subset
$X\subset X_0$ such that $\widehat\pi_{\widehat r}(X)=K$.

\begin{proof}[Proof of Step $2$] Among all compact $\widehat\sigma_{\widehat r}$--invariant sets
$X\subset X_0$ with $\widehat\pi_{\widehat r}(X)\supset K$,
consider one which is minimal for the inclusion (it exists by Zorn's lemma). We claim that such an $X$ satisfies
Step 2. 
To see that, let $z\in K$ whose forward orbit is dense in $K$, let
$x\in X$ be a lift of $z$,
and let $Y$ be the $\omega$--limit set of the forward orbit of $x$,
$$
Y=\{y\in \widehat\Sigma_{\widehat r}:\exists t_n\to+\infty \text{ s.t. }\widehat\sigma_{\widehat r}^{t_n}(x)\to y\}.
$$
For any $n\geq 1$, the set
$Y_n:=\{\sigma^t_{\widehat r}(x), t\geq n\}\cup Y\subset X$ is compact and forward invariant.
Hence the projection $\widehat\pi_{\widehat r}(Y_n)$ is compact and contains
$\{\varphi^t(z), t\geq n\}$. Since the forward orbit of $z$ is dense in $K$,
we have $\widehat\pi_{\widehat r}(Y_n)\supset K$.
Taking the intersection over $n$, one deduces that the projection of the
$\sigma^t_{\widehat r}$--invariant compact set $Y$ contains $K$.
By the minimality of $X$, it follows that $X=Y$.
\end{proof}

This concludes the proof of Part~\eqref{i.lift}.
\bigskip

\paragraph{\bf Part~\eqref{i.Bowen}, items (a) and (b)-(i).} We will use the affiliation relation. Item (a) of Part~\eqref{i.Bowen}, the local finiteness of the affiliation, was proved at the beginning of Section~\ref{s.finite-one}.
Item (b) claims that the affiliation $\sim$ is a Bowen relation. This splits into two properties (i)~and~(ii).

To prove item (i) of the Bowen relation,  let $(\un R,t),(\un S,s)\in \widehat \Sigma^\#_{\widehat r}$ with 
$\widehat \pi_{\widehat r}(\un R,t)=\widehat \pi_{\widehat r}(\un S,s)$, i.e.
$\widehat \pi(\un R)=\vf^{s-t}\widehat \pi(\un S)$. 
Since $|s-t|\leq\sup(\widehat r)\le \rho$, Lemma~\ref{Lemma-affiliation} implies that $R_0\sim S_0$.
\bigskip

\paragraph{\bf Part~\eqref{i.Bowen}, item (b)-(ii).}
We turn to property (ii) of the Bowen relation. We take $\gamma=3\rho$.
Let $z,z'\in\widehat\Sigma_{\widehat r}^\#$ 
such that $\operatorname{v}(\widehat\sigma_{\widehat r}^t z)\sim\operatorname{v}(\widehat\sigma_{\widehat r}^tz')$
for all $t\in\R$. By flowing the two orbits, we can assume that $z=(\un R,0)$ and $z'=(\un S,s)$.
Let $x=\widehat\pi(\un R)$ and $y=\widehat\pi(\un S)$. We wish to show that $x=\vf^{t+s}(y)$ for some 
$|t|< \gamma$. We will deduce from the affiliation condition that the orbit of $y$ must be shadowed by an
$\varepsilon$--gpo that shadows $x$.  By Proposition~\ref{Prop-shadowing}, the two orbits are equal
and the time shift between $x$ and $\vf^s(y)$ will be easily bounded.

To do this, we first apply Proposition \ref{Prop-relation-codings}(1) and get $\varepsilon$--gpo's
$\un v,\un w\in \Sigma^\#$ such that $x=\widehat\pi(\un R)=\pi(\un v)$ and $y=\widehat\pi(\un S)=\pi(\un w)$ 
with $R_0\subset Z(v_0)$ and $S_0\subset Z(w_0)$. Moreover, there are increasing integer sequences $(n_i)_{i\in\Z}$, $(\widetilde m_i)_{i\in\Z}$ such that $R_{n_i}\subset Z(v_i)$ and $S_{\widetilde m_i}\subset Z(w_i)$. For each $i\in\Z$, we locate affiliated symbols in the codings of $x$ and $y$ as follows.

We start with $\vf^{t_i}(x)\in Z(v_i)$ for $t_i=r_i(\un v)=\widehat r_{n_i}(\un R)$.
We have $\widehat\sigma^{t_i}_{\widehat r}(\un R,0)=(\widehat\sigma^{n_i}(\un R),0)$,
hence $\operatorname{v}(\widehat\sigma_{\widehat r}^{t_i}(z))=R_{n_i}$.
We also have $\widehat\sigma^{t_i}_{\widehat r}(\un S,s)=(\widehat\sigma^{\ell_i}(\un S),t_i+s-\widehat r_{\ell_i}(\un S))$,
where $\ell_i$ is the unique integer such that $\widehat r_{\ell_i}(\un S)\le t_i+s<\widehat r_{\ell_i+1}(\un S)$.
Thus $\operatorname{v}(\widehat\sigma_{\widehat r}^{t_i}(z'))=S_{\ell_i}$ and, by assumption, $R_{n_i}\sim S_{\ell_i}$.

Let $a_i\in\Z$ be the largest integer such that $m_i:=\widetilde m_{a_i}\le\ell_i$. Hence, $S_{m_i}\subset Z(w_{a_i})$.
We have $R_{n_i}\subset Z(v_i)\subset D_i$ and likewise $S_{m_i}\subset Z(w_{a_i})\subset E_i$ for some unique connected components $D_i,E_i$  of the section $\widehat \Lambda$.

Write $\Psi_{X_i}^{P^s_i,P^u_i}$ for $v_i$ and $\Psi_{Y_i}^{Q^s_i,Q^u_i}$ for $w_{a_i}$ for all $i\in\Z$.
Finally, set $\widetilde y_i:=\pi(\sigma^{a_i}\un w)\in Z(w_{a_i})$  and $y_i:=\mathfrak q_{D_i}(\widetilde y_i)$.
We are going to show that, for all $i\in\Z$:
 \begin{enumerate}[(1)]\label{id.for.bowen}
  \item[(1)] $y_i$ is well-defined, and for $i=0$ we have $y_0=\vf^{u}(\widetilde y_0)$
  with $|u|\leq 2\rho$;
  \item[(2)] $y_{i+1}=g_{X_i}^+(y_i)$. 
 \end{enumerate}
Proposition~\ref{Prop-shadowing} will then imply that 
$x=y_0=\vf^u(\widetilde y_0)=\vf^u(y)=\vf^{u-s}(\widehat\pi_{\widehat r}(\un S,s))$, where
$|u-s|\le 2\rho+\sup(\widehat r)<3\rho$. 
Property (ii) and therefore the Bowen relation claimed by Part (6)(b) will be established.

It remains to prove the above identities. As in \cite{BCL23}, they require checking that some holonomies along the flow are compatible. The idea is that affiliation implies
that charts have comparable parameters and their images fall inside $\widehat\Lambda$
far from its boundary. The claims below and their proofs are the same as those in \cite[Section 9.4]{BCL23}.

\medskip
\noindent
{\sc Claim 1:} Let $Z_1,Z_2\in\mathfs Z$ such that $Z_1\cap \vf^{[-\rho,\rho]}Z_2\neq\emptyset$.
Write $Z_i=Z(\Psi_{x_i}^{p^s_i,p^u_i})$ and let $D_i$ be the connected component
of $\widehat\Lambda$ containing $Z_i$. 
Then $\tfrac{p^s_1\wedge p^u_1}{p^s_2\wedge p^u_2}=e^{\pm(O(\sqrt[3]{\ve})+O(\rho))}$ and  
$$
\mathfrak q_{D_1}(\Psi_{x_2}(R[c(p^s_2\wedge p^u_2)]))\subset \Psi_{x_1}(R[2c(p^s_1\wedge p^u_1)])
$$
for all $1\leq c\leq 64$.\\

\begin{proof}[Proof of Claim $1$.] Same as Proposition \ref{Prop-overlapping-charts}(1).
\end{proof}

\medskip
\noindent
{\sc Claim 2:} Let $R_1,R_2\in\mathfs R$ such that $R_1\sim R_2$. For $i=1,2$,
let $D_i$ be the connected component of
$\widehat\Lambda$ containing $R_i$, and let $Z_i=Z(\Psi_{x_i}^{p^s_i,p^u_i})\in\mathfs Z$ 
such that $Z_i\supset R_i$. 
Then $\tfrac{p^s_1\wedge p^u_1}{p^s_2\wedge p^u_2}=e^{\pm(O(\sqrt[3]{\ve})+O(\rho))}$ and
$$
\mathfrak q_{D_1}(\Psi_{x_2}(R[c(p^s_2\wedge p^u_2)]))\subset \Psi_{x_1}(R[8c(p^s_1\wedge p^u_1)]).
$$
for all $1\leq c\leq 16$.

\begin{proof}[Proof of Claim $2$.]
Same as in \cite{BCL23}, applying Claim 1 three times.
\end{proof}
\medskip
\noindent
{\sc Claim 3:} Let $R_1,R_2,R_3\in\mathfs R$ such that $R_1\sim R_2$ and $R_2\sim R_3$.
For $i=1,2,3$, let $D_i$ be the connected component of
$\widehat\Lambda$ containing $R_i$, and let $Z_i=Z(\Psi_{x_i}^{p^s_i,p^u_i})\in\mathfs Z$ 
such that $Z_i\supset R_i$. 
Then $\tfrac{p^s_3\wedge p^u_3}{p^s_1\wedge p^u_1}=e^{\pm(O(\sqrt[3]{\ve})+O(\rho))}$ and
$$
(\mathfrak q_{D_1}\circ \mathfrak q_{D_2})(\Psi_{x_3}(R[c(p^s_3\wedge p^u_3)]))=
\mathfrak q_{D_1}(\Psi_{x_3}(R[c(p^s_3\wedge p^u_3)]))\subset 
\Psi_{x_1}(R[64c(p^s_1\wedge p^u_1)])
$$
for all $1\leq c\leq 2$.

\begin{proof}[Proof of Claim $3$.]
Same as in \cite{BCL23}, applying Claim 2 twice.
\end{proof}

Now we apply the above claims to our particular situation. Write $v_i=\Psi_{x_i}^{p^s_i,p^u_i}$
and $w_i=\Psi_{z_i}^{q^s_i,q^u_i}$, so that $Q^{s/u}_i=q^{s/u}_{a_i}$.

\medskip
\noindent
{\sc Claim 4:} Let $i\in\Z$. We have
$$
\mathfrak q_{E_{i+1}}(\Psi_{Y_{i}}(R[Q^s_{i}\wedge Q^u_{i}]))\subset
\Psi_{Y_{{i+1}}}(R[2(Q^s_{{i+1}}\wedge Q^u_{{i+1}})]).
$$

\begin{proof}[Proof of Claim $4$.] 
By Lemma \ref{Lemma-minimum} and Claim 3, 
$$
\tfrac{q^s_{a_{i+1}}\wedge q^u_{a_{i+1}}}{q^s_{a_i}\wedge q^u_{a_i}}=
\tfrac{q^s_{a_{i+1}}\wedge q^u_{a_{i+1}}}{p^s_{i+1}\wedge p^u_{i+1}}\cdot 
\tfrac{p^s_{i+1}\wedge p^u_{i+1}}{p^s_{i}\wedge p^u_{i}}\cdot
\tfrac{p^s_{i}\wedge p^u_{i}}{q^s_{a_i}\wedge q^u_{a_i}}=e^{\pm(O(\sqrt[3]{\ve})+O(\rho))}.
$$
This estimate allows to apply the same proof of Proposition \ref{Prop-overlapping-charts}(1), 
and so we can obtain the claimed inclusion in the same way.
\end{proof}

\medskip
\noindent
{\sc Claim 5:} Let $i\in\Z$. Restricted to the set $\Psi_{Y_{i}}(R[Q^s_{i}\wedge Q^u_{i}])$, we have the equality
$\mathfrak q_{D_{i+1}}\circ \mathfrak q_{E_{i+1}}=\mathfrak q_{D_{i+1}}=g_{X_i}^+\circ \mathfrak q_{D_i}$.

\begin{proof}[Proof of Claim $5$.] It is enough to prove the equality for $i=0$, i.e. that
$\mathfrak q_{D_1}\circ\mathfrak q_{E_1} =\mathfrak q_{D_1}=g_{X_0}^+\circ \mathfrak q_{D_0}$ 
when restricted to $\Psi_{Y_0}(R[Q^s_0\wedge Q^u_0])$. 
By Claim 4, $\mathfrak q_{E_1}[\Psi_{Y_0}(R[Q^s_0\wedge Q^u_0])]\subset \Psi_{Y_1}(R[2(Q^s_1\wedge Q^u_1)])$.
Applying Claim 3 with $c=2$ to the triple $(R_{n_1},S_{\ell_1},S_{m_1})$, we get that
$\mathfrak q_{D_1}[\Psi_{Y_1}(R[2(Q^s_1\wedge Q^u_1)])]$
is well-defined, hence $\mathfrak q_{D_1}\circ \mathfrak q_{E_1}=\mathfrak q_{D_1}$ when
restricted to $\Psi_{Y_0}(R[Q^s_0\wedge Q^u_0])$. On the other hand, applying Claim 3
with $c=1$ to the triple $(R_{n_0},S_{\ell_0},S_{m_0})$, we have that
$\mathfrak q_{D_0}[\Psi_{Y_0}(R[Q^s_0\wedge Q^u_0])]\subset \Psi_{X_0}(R[64(P^s_0\wedge P^u_0)])$.
By definition, $g_{X_0}^+=\mathfrak q_{D_1}$ when restricted to $R[64(P^s_0\wedge P^u_0)]$.
Therefore, $g_{X_0}^+\circ \mathfrak q_{D_0}=\mathfrak q_{D_1}$ when
restricted to $\Psi_{Y_0}(R[Q^s_0\wedge Q^u_0])$. This proves Claim 5.
\end{proof}

We now complete the proof of identities (1) and (2) of page \pageref{id.for.bowen}, 
which in turn will complete the proof of part (6) of Theorem \ref{t.main}.
For that, we use the claims we just proved.

Firstly we check that $y_i:=\mathfrak q_{D_i}(\widetilde y_i)$ is well-defined. By assumption $R_{n_i}\sim S_{\ell_i}$,
and by construction the orbit of $y$ between $S_{m_i}$ and $S_{\ell_i}$ flows for a time at most $\sup(r)<\rho$,
hence $S_{\ell_i}\sim S_{m_i}$. This allows us to apply Claim 3 for $c=1$ and get 
that $y_i:=\mathfrak q_{D_i}(\widetilde y_i)$ is well-defined. To calculate the time displacement for
$i=0$, recall that $m_0=\ell_0=0$. Since $R_0\sim S_0$, inclusion (\ref{e.2rho}) 
implies that $y_0=\vf^u(\widehat y_0)$ with $|u|\leq 2\rho$.

Finally, Claim 5 implies that 
 $$
   g_{X_i}^+(y_i)= g_{X_i}^+\circ \mathfrak q_{D_i}(\widetilde y_i)=\mathfrak q_{D_{i+1}}\circ \mathfrak q_{E_{i+1}}(\widetilde y_i)
    = \mathfrak q_{D_{i+1}}(\widetilde y_{i+1})=y_{i+1},
  $$
finishing  the proof of Theorem~\ref{t.main}.

\section{Homoclinic classes of measures}\label{sec.homoclinic}

In this final section, we prove Theorem \ref{thm.homoclinic} stated in the Introduction,
as well as Corollary~\ref{cor.local-uniq}.

\subsection{The homoclinic relation}

For any hyperbolic measure $\mu$ and $\mu$--a.e. $x$,
the stable set $W^s(x)$ of the orbit of $x$ is
the set of points $y$ such that there exists an increasing homeomorphism
$h\colon \mathbb{R}\to \mathbb{R}$ satisfying $d(\varphi^t(x),\varphi^{h(t)}(y))\to 0$
as $t\to +\infty$. This is an injectively immersed submanifold which is tangent to $E^s_x\oplus X(x)$ 
and invariant under the flow. 
We define similarly the unstable manifold $W^u(x)$ by considering past orbits.

\medskip
\noindent
{\sc Homoclinic relation of measures:} 
We say that two hyperbolic measures $\mu,\nu$ are
\emph{homoclinically related} if for $\mu$--a.e. $x$ and $\nu$--a.e. $y$
there exist transverse intersections $W^s(x)\pitchfork W^u(y)\ne\emptyset$
and  $W^u(x)\pitchfork W^s(y)\ne\emptyset$.

\medskip
Note that necessarily $\mu$ and $\nu$ have the same index,
i.e. the invariant manifolds $W^{s/u}(x)$ and
$W^{s/u}(y)$ have the same dimension for $\mu$--a.e. $x$ and $\nu$--a.e. $y$.
Since any hyperbolic periodic orbit supports a (unique) ergodic measure,
the above homoclinic relation is also defined between
hyperbolic periodic orbits, in which case it coincides with the classical notion,
see, e.g. \cite{Newhouse-Lectures-dynamical-systems}.

\begin{proposition}
The homoclinic relation is an equivalence relation among ergodic hyperbolic measures.
\end{proposition}

For three dimensional flows, this is \cite[Prop. 10.1]{BCL23}, and the same proof applies. The only property that is not direct is to check the transitivity of the relation, which uses the following standard lemma, whose proof is sketched in \cite{BCL23} and works equally well in any dimension. Recall that $d_{s/u}(x)$ is the dimension of $E^{s/u}_x$.

\medskip
\noindent
{\bf Inclination lemma.}
\emph{For any $\chi$--hyperbolic measure $\mu$,
there is a set $Y\subset M$ of full $\mu$--measure satisfying the following:
if $x\in Y$, $D\subset W^u(x)$ is a $(d_u(x)+1)$--dimensional disc 
and $\Delta$ is a $(d_u(x)+1)$--dimensional disc tangent to $X$ having a transverse intersection point with $W^s(x)$,
then there are discs $\Delta_k\subset \varphi_{(k,+\infty)}(\Delta)$ which converge to $D$ in the $C^1$ topology.}

In order to prove the proposition, let us consider three measures $\mu_1,\mu_2,\mu_3$
such that $\mu_1,\mu_2$ are homoclinically related and $\mu_2,\mu_3$ are homoclinically related.
For each measure $\mu_i$, let $x_i$ be a point in the full measure set implied by the homoclinic relation.
In particular, there exists a disc $\Delta\subset W^u(x_1)$ which intersects transversally $W^s(x_2)$
and a disc $D\subset W^u(x_2)$ which intersects transversally $W^s(x_3)$.
By the inclination lemma, the orbit of $\Delta$ contains discs that converge to $D$ for the $C^1$ topology.
This proves that $W^u(x_1)$ has a transverse intersection point with $W^s(x_3)$.
The same argument shows that $W^u(x_3)$ has a transverse intersection with $W^s(x_1)$.
Hence $\mu_1$ and $\mu_3$ are homoclinically related.

\medskip
\noindent
{\sc Homoclinic classes of measures:}
The equivalence classes for the homoclinic relation on the set of hyperbolic measures
are called \emph{homoclinic classes of measures.}

\subsection{Proof of Theorem \ref{thm.homoclinic}}
The proof is essentially the same as \cite[Theorem 1.1]{BCL23}, which in turn follows closely the argument in~\cite[Section 3]{BCS-MME}.
We consider the setting of the Main Theorem and especially a topological Markov flow
$(\widehat \Sigma_{\widehat r},\widehat\sigma_{\widehat r})$ satisfying the properties
stated in Theorem~\ref{t.main}.

We begin with some preliminary lemmas.
The first two correspond to properties (C6), (C7) in~\cite{BCS-MME}.

\begin{lemma}\label{C6}
For any two ergodic measures supported on a common irreducible component of $\widehat \Sigma_{\widehat r}$,
their projections under $\widehat\pi_{\widehat r}$ are hyperbolic ergodic measures that are homoclinically related.
\end{lemma}

For three dimensional flows, this is \cite[Lemma 10.2]{BCL23} and the same proof applies in high dimension.

\begin{lemma}\label{C7}
For any $\chi’>0$, the set of ergodic measures on $\widehat \Sigma_{\widehat r}$
whose projection is $\chi’$--hyperbolic is open for the weak--* topology.
\end{lemma}

\begin{proof}
For three dimensional flows, this is \cite[Lemma 10.3]{BCL23}, whose proof is inspired by \cite[Prop. 3.7]{BCS-MME}. The proof below combines the arguments of these two references. 
Let $\overline\mu$ be an ergodic probability measure on $\widehat \Sigma_{\widehat r}$ such that its projection $\mu=\overline\mu\circ \wh\pi_{\wh r}^{-1}$ is $\chi'$--hyperbolic. We wish to show that if $\overline\nu$ is close to $\overline\mu$ in the weak--* topology, then all Lyapunov exponents of $\nu=\overline\nu\circ \wh\pi_{\wh r}^{-1}$ in the stable direction $E^s$ are smaller than $-\chi'$. Since the same applies to the unstable direction, the proof will follow.

Let $\lambda<\chi'$ such that the Lyapunov exponents of $\mu$ along $E^s$ are smaller than $-\lambda$. By the proof of Proposition \ref{Prop-NUH}, the Lyapunov exponents of the cocycles $(d\vf^t)_{t\in\R}$ and $(\Phi^t)_{t\in\R}$ coincide almost everywhere for every $\vf$--invariant probability measure.
Therefore, for a fixed $\delta>0$, there is $T=T(\delta)>0$ s.t.
$$
A:=\left\{x\in\wh\Sigma_{\wh r}:\|\Phi^T v\|<e^{-\lambda T}\|v\|\text{ for all non-zero }v\in N^s_{\wh\pi_{\wh r}(x)}\right\}
$$
has $\overline\mu$--measure larger than $1-\delta$. By Theorem \ref{s.detailed}(4), the map $x\in\wh\Sigma_{\wh r}\mapsto N^s_{\wh\pi_{\wh r}(x)}$ is continuous and so $\overline\nu(A)>1-\delta$ for $\overline\nu$ close to $\overline\mu$ in the weak--* topology. Since $\delta>0$ can be chosen arbitrarily small, it follows that $\tfrac{1}{T}\int_{\wh \Sigma_{\wh r}} \log\|\Phi^t\restriction_{N^s}\|d\overline\nu<-\chi'$.
By the subadditivity of $t\in\R\mapsto \|\Phi^t\restriction_{N^s}\|$, it follows that
$$
\lim_{t\to+\infty} \frac{1}{t}\int \log \|\Phi^t\restriction_{N^s}\|d\overline\nu\leq 
\frac{1}{T}\int \log \|\Phi^T\restriction_{N^s}\|d\overline\nu<-\chi'.
$$
Since the limit above equals the largest exponent along $N^s$, we conclude that all exponents of $(\Phi^t)_{t\in\R}$ along $N^s$ are smaller than $-\chi'$, and so the same holds for the exponents of $(d\vf^t)_{t\in\R}$ along $E^s$.
\end{proof}

\begin{lemma}
There exists an irreducible component $\widehat\Sigma'_{\widehat r}\subset \widehat \Sigma_{\widehat r}$
to which one can lift all $\chi$--hyperbolic periodic orbits that are homoclinically related to $\mu$.
\end{lemma}

In other words, there is an irreducible component that lifts periodic orbits. For three dimensional flows, this is \cite[Lemma 10.4]{BCL23} and the same proof applies here.

Now we complete the proof of Theorem \ref{thm.homoclinic}.
Let $\nu$ be a $\chi$--hyperbolic ergodic measure that is homoclinically related to $\mu$.
By Theorem~\ref{t.main}(2), there exists an ergodic lift $\overline\nu$ of $\nu$ to
$\widehat \Sigma_{\widehat r}$.
Consider a point $ q\in \widehat \Sigma_{\widehat r}$ that is recurrent (such that there exists a sequence of forward iterates
$\widehat \sigma^{t_i}_{\wh r}(q)$ which converges to $q$) and generic for $\overline\nu$,
and let $x=\widehat\pi_{\widehat r}(q)$.

Using the recurrence of $q$, there is a sequence of periodic points $ q^i$ in $\widehat \Sigma_{\widehat r}$
which converge to $q$ (hence are in a same irreducible component)
and whose orbits weak--* converge to $\overline\nu$.
By Lemma~\ref{C7} the projections of these periodic orbits are $\chi$--hyperbolic and
by Lemma~\ref{C6} they are homoclinically related to $\mu$. Therefore there are periodic
orbits $p^i$ in the irreducible component $\widehat\Sigma'_{\widehat r}$
which have the same projections as the periodic orbits $ q^i$.

Write $q^i=(\underline R^i,t^i)$ and $p^i=(\underline S^i,s^i)$.
Since $(q^i)$ is converging
and $\widehat \Sigma$ is locally compact,
the sequence $(\underline R^i)$ is relatively compact.
The Bowen property of Theorem~\ref{t.main}(6) implies that
$\operatorname{v}(\widehat \sigma_{\widehat r}^t(q^i))\sim \operatorname{v}(\widehat\sigma_{\widehat r}^t(p^i))
\text{ for all }t\in \mathbb{R}$ so, by the local finiteness of the affiliation,
the sequence $(\underline S^i)$ is relatively compact.
This implies that $(p^i)$ is relatively compact and (up to taking a subsequence)
converges to some $p\in \widehat\Sigma'_{\widehat r}$.
By continuity of the projection, $\widehat\pi_{\widehat r}(p)=\widehat\pi_{\widehat r}(q)=x$.

We claim that $p\in  \widehat{\Sigma'}_{\widehat r}^\#$.
This follows from the fact that $q$ is recurrent and that the Bowen relation is locally finite.
More precisely, there are some vertex $A\in\widehat V$ and integers $m_k,n_k\to\infty$ such that $q_{m_k}=q_{-n_k}=A$. 
In particular, for each $k\ge1$ we have $q^i_{m_k}=q^i_{-n_k}=A$ for all large $i$.
Hence $p^i_{m_k},p^i_{-n_k}$ are related to $A$, and so they belong to the set $\{B\in\widehat V:B\sim A\}$.
This set is finite, hence some symbol must repeat
as required and this passes to the limit $p$, proving the claim.

We have thus proved that $\nu$--almost every point has a lift in $\widehat{\Sigma'}_{\widehat r}^\#$.
The finiteness-to-one property of Theorem~\ref{t.main}(3) and
the same averaging argument used in the proof of Theorem~\ref{t.main}(2) imply
that $\nu$ has a lift in $\widehat\Sigma'_{\widehat r}$. Considering the ergodic decomposition,
we can choose an ergodic lift, as claimed.
Theorem \ref{thm.homoclinic} is now proved.\qed

\subsection{Proof of Corollary~\ref{cor.local-uniq}}
We fix an admissible continuous potential $\psi:M\to\R$, meaning that $\psi\circ\pi_r:\Sigma_r\to\R$
is H\"older continuous for every $\pi_r$ satisfying the Main Theorem.
Let $\mathcal H$ be some homoclinic class of hyperbolic ergodic measures.
We will deduce from Theorem~\ref{thm.homoclinic} that there is at most one
$\nu\in\mathcal H$ such that  
$$
h(\vf,\nu)+\int \psi  d\nu=\sup\left\{h(\vf,\eta) +\int \psi  d\eta:\eta\in\mathcal H\right\}.
$$
Let  $\nu,\nu'\in\mathcal H$ be two measures with this property.  They are both hyperbolic, 
hence $\chi$--hyperbolic for some $\chi>0$. For one such fixed parameter $\chi$,
let $\pi_r:\Sigma_r\to M$ be the coding given by the Main Theorem.

By Theorem~\ref{thm.homoclinic}, there is an irreducible component $\Sigma'_r$ of $\Sigma_r$
to which both $\nu$ and $\nu'$ lift. Since the factor map $\pi_r$ preserves the entropy
and since the projection of any ergodic measure on $\Sigma'_r$ is homoclinically related 
to $\nu$ and $\nu'$ by Lemma~\ref{C6},
the two lifts $\overline\nu,\overline{\nu}'$ are equilibrium states of the lifted potential $\psi\circ\pi_r$ on $\Sigma'_r$.
But a bounded H\"older continuous function with finite pressure defined on 
an irreducible topological Markov flow has at most one equilibrium state (this is proved in \cite[Corollary 3.2]{LLS16},
see also Claim 1 in the proof of Theorem 3.1 of the same paper).
Hence $\nu=\nu'$, which proves Corollary~\ref{cor.local-uniq}.

\appendix

\section{Standard proofs}\label{Appendix-proofs}

\renewcommand\thetheorem{\Alph{section}.\arabic{theorem}}

Recall that we are assuming $\|\nabla X\|\leq 1$, and that this implies two facts:
\begin{enumerate}[$\circ$]
\item Every Lyapunov exponent of $\vf$ has absolute value $\leq 1$, hence we consider $\chi\in (0,1)$.
\item $\|\Phi^t\|=e^{\pm (\rho+|t|)}$ for all $t\in\R$, see estimate (\ref{Norm of ILPF}).
\end{enumerate}

\begin{proof}[Proof of Lemma \ref{Lemma-linear-reduction}.]
(1)  Let $v=v^s+v^u$, where $v^s\in N^s_x$
and $v^u\in N^u_x$. Then
$$\vertiii{v}^2 = \vertiii{v^s}^2 + \vertiii{v^u}^2 \geq  \tfrac{2}{1- \chi} (\| v^s \|^2 + \| v^u \|^2). $$
Using the arithmetic-quadratic mean inequality and the triangle inequality,
$$\vertiii{v}^2 \geq \tfrac{2}{1- \chi} (\| v^s\|^2 + \| v^u \|^2) \geq \tfrac{1}{1- \chi}( \| v^s \| +  \| v^u \|)^2\geq \tfrac{1}{1-\chi} \|v \|^2.$$

Recalling the definition of $C(x)$, we have $\| C(x)^{-1} v \|^2 = \vertiii{v}^2 \geq \frac{1}{1-\chi} \| v \|^2$. Setting $w=C(x)^{-1}v$ yields $\| C(x) w \| \leq \sqrt{1- \chi} \|w\|$, showing that $C(x)$ is a contraction.

The formula for $\|C(x)^{-1}\|$ is direct:
if $v=v^s+v^u\in N^s_x\oplus N^u_x$ then 
$\|C(x)^{-1}v\|^2=\vertiii{v}^2= \vertiii{v^s}^2+\vertiii{v^u}^2$ and so
$$
\|C(x)^{-1}\|^2=\sup_{v\in N_x\setminus\{0\}}\frac{\vertiii{v}^2}{\|v\|^2}=\sup_{v\in N_x\setminus\{0\}}\frac{\vertiii{v^s}^2+\vertiii{v^u}^2}{\|v\|^2}\cdot
$$
This proves part (1).

\medskip
\noindent
(2) Start noting that, by the definition of $C(\cdot)$, we have
$$
\R^{d_s(x)}\times\{0\} \xrightarrow[]{C(x)} N^s_x \xrightarrow[]{\ \Phi^t\ }
N^s_{\vf^t(x)} \xrightarrow[]{C(\vf^t(x))^{-1}} \R^{d_s(\vf^t(x))}\times \{0\}.
$$
Similarly,
$$
\{0\}\times\R^{d_u(x)} \xrightarrow[]{C(x)} N^u_x \xrightarrow[]{\ \Phi^t\ }
N^u_{\vf^t(x)} \xrightarrow[]{C(\vf^t(x))^{-1}} \{0\}
\times\R^{d_u(\vf^t(x))}
$$
and so $D(x,t)$ has the block form 
$$
D(x,t)=\left[\begin{array}{cc}D_s(x,t) & 0 \\
0 & D_u(x,t)
\end{array}\right],
$$
where $D_s(x,t)$ has dimension $d_s(x)$ and $D_u(x,t)$ has dimension $d_u(x)$. It remains to estimate $D_{s/u}(x,t)$.
Observe that if $v_1\in \mathbb{R}^{d_s(x)}\times\{0\}$
then for $v^s=C(x)v_1$ the  definition of the Lyapunov inner product implies that
$\|v_1\|=\|C(x)^{-1}v^s\|=\vertiii{v^s}$ and 
$\|D_s(x,t)v_1\|=\|[C(\vf^t(x))^{-1}\circ\Phi^t](v^s)\|=\vertiii{\Phi^t v^s}$. Since analogous equations hold for $v_2\in \{0\}\times \mathbb R^{d_u(x)}$, we just need to estimate the ratios $\frac{\vertiii{\Phi^tv^{s/u}}^2}{\vertiii{v^{s/u}}^2}$ for non-zero $v^{s/u}\in N^{s/u}_x$. Define $\kappa(t):=e^{-\chi t}\left[1-e^{2\rho}\left(1-e^{-2(1-\chi)t}\right) \right]^{1/2}$.

\medskip
\noindent
{\sc Claim 1:}  For all nonzero $v^{s/u}\in N^{s/u}$ and $t\geq 0$, it holds 
$$\kappa(t)<\tfrac{\vertiii{\Phi^tv^s}}{\vertiii{v^s}}<e^{-\chi t}\ \text{ and }\ e^{\chi t}<\tfrac{\vertiii{\Phi^tv^u}}{\vertiii{v^u}}<\kappa(t)^{-1}.$$

\medskip
Noticing that $\kappa(t) \geq e^{-4 \rho}$ for $0\leq t\leq 2\rho$, this claim clearly implies part (2).

\begin{proof}[Proof of Claim $1$]
We prove the estimate for $v^s$ (the argument for $v^u$ is analogous, by symmetry). Given $v^s\in N^s$, decomposing the 
integral defining $\vertiii{v^s}$ into two parts, we have
\begin{equation*}
\begin{split}
&\vertiii{v^s}^2 = 4 e^{2 \rho}\int_0^t e^{2 \chi t'} \| \Phi^{t'} v^s\|^2 dt'+4 e^{2 \rho}\int_t^\infty e^{2 \chi t'} \| \Phi^{t'} v^s\|^2 dt'\\
&=4 e^{2 \rho}\underbrace{\int_0^t e^{2 \chi t'} \| \Phi^{t'} v^s\|^2 dt'}_{=:A}+
e^{2\chi t}\vertiii{\Phi^t v^s}^2
\end{split}
\end{equation*}
and so
\begin{equation}\label{Lemma3.2,auxiliar1}
    \tfrac{\vertiii{\Phi^t v^s}^2}{\vertiii{v^s}^2}= e^{-2 \chi t} \left( 1 - \tfrac{4 e^{2 \rho} A}{ \vertiii{v^s}^2} \right)\cdot
\end{equation}
Recalling that $\chi\in (0,1)$, this estimate already gives the upper bound 
$\tfrac{\vertiii{\Phi^t v^s}}{\vertiii{v^s}}<e^{-\chi t}$. For the lower bound, we estimate the ratio $\frac{A}{\vertiii{v^s}^2}$ from above. The idea is to decompose $\vertiii{v^s}^2$ into a sum of integrals
$$
\vertiii{v^s}^2 = 4 e^{2 \rho} \int_0^{\infty} e^{2 \chi t'} \| \Phi^{t'} v^s \|^2 dt'
    = 4 e^{2 \rho} \sum_{j \geq 0} \int_{jt}^{(j+1)t} e^{2 \chi t'} \| \Phi^{t'} v^s \|^2 dt'.
$$
and estimate each integral in terms of $A$.
By the change of variables $t'=jt+r$, we have 
\begin{align*}
&\,\int_{jt}^{(j+1)t} e^{2 \chi t'} \| \Phi^{t'} v^s \|^2 dt'
=e^{2\chi jt}\int_0^t e^{2\chi r}\|\Phi^{jt}\Phi^r v^s\|^2dr
\geq e^{2\chi jt}\int_0^t e^{2\chi r-2jt-2\rho}\|\Phi^r v^s\|^2dr\\
&=A e^{-2(1-\chi)jt-2\rho}
\end{align*}
and so 
$$
\vertiii{v^s}^2\geq 4e^{2\rho}A\sum_{j\geq 0}e^{-2(1-\chi)jt-2\rho}= \tfrac{4A}{1-e^{-2(1-\chi)t}},
$$
thus giving that $\tfrac{A}{\vertiii{v^s}^2}\leq \tfrac{1-e^{-2(1-\chi)t}}{4}$. Therefore
$$\tfrac{\vertiii{\Phi^t v^s}^2}{\vertiii{v^s}^2} \geq e^{-2 \chi t} \left[ 1- e^{2 \rho}\left( 1-e^{-2(1-\chi)t}\right) \right]=\kappa(t)^2,$$
which is the required lower bound.
\end{proof}

\medskip
\noindent
(3) We begin with the following estimate.

\medskip
\noindent
{\sc Claim 2:}  For all non-zero $v\in N$ and $t\in\R$, it holds 
$$
\tfrac{\vertiii{\Phi^tv}}{\vertiii{v}}=e^{\pm(\rho+|t|)}.
$$

\medskip
Contrary to Claim 1, the above estimate holds for every non-zero vector and also for negative values of $t$.

\begin{proof}[Proof of Claim $2$]
We first prove the result for $v\in N^s$.
By the estimate $\|\Phi^t\|=e^{\pm (\rho+|t|)}$, we have 
$$
\vertiii{\Phi^t v}^2=4 e^{2\rho}\int_0^\infty e^{2\chi t'}\|\Phi^{t'}\Phi^t v\|^2dt'=
e^{\pm 2(\rho+|t|)} 4 e^{2\rho}\int_0^\infty e^{2\chi t'}\|\Phi^{t'}v\|^2dt'=e^{\pm 2(\rho+|t|)}\vertiii{v}^2
$$
and we obtain the estimate. Analogously, the estimate holds for $v\in N^u$. Finally, for a general $v=v^s+v^u$, using the $\Phi$--invariance of the decomposition $N^s\oplus N^u$, we have
$$
\vertiii{\Phi^t v}^2=\vertiii{\Phi^tv^s}^2+\vertiii{\Phi^tv^u}^2=
e^{\pm 2(\rho+|t|)}\left(\vertiii{v^s}^2+\vertiii{v^u}^2\right)
=e^{\pm 2(\rho+|t|)}\vertiii{v}^2,
$$
which concludes the proof of the claim.
\end{proof}

Now we prove part (3). By symmetry, it is enough to prove the upper estimate $\tfrac{\|C(\vf^t(x))^{-1}\|}{\|C(x)^{-1}\|}\leq e^{2(\rho+|t|)}$. Fix a non-zero $w\in N$ and let $v=\Phi^{-t}w$. We have
\begin{align*}
\tfrac{\|C(\vf^t(x))^{-1}w\|/\|w\|}{\|C(x)^{-1}\|}
\leq \tfrac{\|C(\vf^t(x))^{-1}w\|\cdot \|v\|}{\|w\|\cdot \|C(x)^{-1}v\|}=
\tfrac{\vertiii{w}}{\vertiii{v}}\cdot \tfrac{\|v\|}{\|w\|}\leq e^{2(\rho+|t|)},
\end{align*}
where in the last inequality we used the estimate $\|\Phi^t\|=e^{\pm (\rho+|t|)}$ and Claim 2. Since $w$ is arbitrary, part (3) follows.
\end{proof}

\begin{proof}[Proof of Theorem \ref{Thm-non-linear-Pesin}]
Remember that $B_x = B(x, 2\mathfrak{r})$. If $\ve >0$ is small enough, then by Lemma \ref{lemma3.7}(1)
$$
\Psi_x (R[10 Q(x)]) \subset B (x, 20 \sqrt{2} Q(x))\subset B_x
$$
and inside this ball conditions (Exp1)--(Exp4) are satisfied. We start showing that $f_x^+:R[10 Q(x)] \to \R^d$ is well-defined. Since $C(x)$ is a contraction, we have $C(x)(R[10 Q(x)])\subset B_x[10 \sqrt{2} Q(x)]$. Using that $C(f(x))^{-1}$ is globally defined and (Exp1), it is enough to show that
$$
(g_x^+ \circ \exp{x})(B_x[10 \sqrt{2} Q(x)]) \subset B_{f(x)}.
$$
For $\ve >0$ small, we have:
\begin{enumerate}[$\circ$]
    \item By (Exp2), $\exp{x}$ maps $B_x [10\sqrt{2} Q(x)]$ diffeomorphically into $B(x, 20 \sqrt{2} Q(x))$. 
    \item Since $20 \sqrt{2} Q(x) < 2 \mathfrak{r}$, we have $B(x, 20 \sqrt{2} Q(x)) \subset B_x$ and so, by Lemma \ref{Lemma-map-g}, $g_x^+$ maps $B(x, 20 \sqrt{2} Q(x))$ diffeomorphically into $B(f(x), 40 \sqrt{2} Q(x))$.
    \item Since $40 \sqrt{2} Q(x) <2 \mathfrak{r}$, we have $B(f(x), 40 \sqrt{2} Q(x)) \subset B_{f(x)}$.
\end{enumerate}
Hence $f_x^+ : R[10 Q(x)] \rightarrow \R^d$ is a diffeomorphism onto its image.

Now we verify parts (1)--(2). Using the equalities $d(\Psi_x)_0 = C(x)$, $d(\Psi_{f(x)})_0 = C(f(x))$ and Lemma \ref{Lemma-map-g}, we have that
$$
d(f_x^+)_0 = C(f(x))^{-1} \circ \Phi^{r_\Lambda (x)} \circ C(x).
$$
By Lemma \ref{Lemma-linear-reduction},
$d(f_x^+)_0 = \begin{bmatrix} D_s (x) && \\ && D_u (x) \end{bmatrix}$ with $e^{-4\rho}\leq \|D_s (x)\|,\|D_u(x)^{-1}\| \leq e^{- \chi r_{\Lambda} (x)}$
and so part (1) is proved.

\medskip
\noindent
(2) Items (a)--(b) are automatic, hence we focus on (c).

\medskip
\noindent
{\sc Claim:} $\| d(f_x^+)_{v_1} - d(f_x^+)_{v_2} \| \leq \frac{\ve}{3} \| v_1 - v_2 \|^{\frac{\beta}{2}}$ for all $v_1, v_2 \in R[10 Q(x)]$.

\medskip
Before proving the claim, we show how it implies (c). If $\ve>0$ is small enough, then $R[10 Q(x)] \subset B_x[1]$. Applying the claim for $v_2 =0$, we have $\|d H_v \| \leq \frac{\ve}{3} \| v \|^{\frac{\beta}{2}} < \frac{\ve}{3}$. By the mean value inequality, $\| H(v) \| \leq \frac{\ve}{3} \| v \| \leq \frac{\ve}{3}$ and so $\| H \|_{C^{1 + \frac{\beta}{2}}}< \ve$ (the norm is taken in $R[10 Q(x)]$).

\begin{proof}[Proof of the claim] Fix $L >\Hol{\beta} (dg_x^+)$. For $i=1,2$, write $w_i = C(x) v_i$ and put
$$A_i = \widetilde{d(\exp{f(x)}^{-1})_{(g_x^+ \circ \exp{x})(w_i)}}\,, \ B_i = \widetilde{d (g_x^+)_{\exp{x} (w_i)}}\,, \ C_i = \widetilde{d(\exp{x})_{w_i}}.$$
We  estimate $\|A_1 B_1 C_1-A_2 B_2 C_2\|$.
\begin{enumerate}[$\circ$]
\item By (Exp2), $\|A_i\|\leq 2$. By (Exp2), (Exp3) and Lemma \ref{Lemma-map-g}:
\begin{align*}
\|A_1-A_2\|\leq \mathfrak Kd((g_x^+\circ \exp{x})(w_1),(g_x^+\circ \exp{x})(w_2))
\leq 4\mathfrak K\|w_1-w_2\|.
\end{align*}
\item By Lemma \ref{Lemma-map-g}, $\|B_i\|\leq 2$.
By (Exp2) and Lemma \ref{Lemma-map-g}:
$$
\|B_1-B_2\|\leq Ld(\exp{x}(w_1),\exp{x}(w_2))^{\beta}\leq 2L\|w_1-w_2\|^\beta.
$$
\item By (Exp2), $\|C_i\|\leq 2$. By (Exp3), $\|C_1-C_2\|\leq \mathfrak K\|w_1-w_2\|$.
\end{enumerate}
Hence
\begin{align*}
   &\, \| A_1 B_1 C_1 - A_2 B_2 C_2 \| 
   \leq \|(A_1-A_2)B_1C_1\|+\|A_2(B_1-B_2)C_1 \|+\|A_2B_2 (C_1-C_2)\|\\
   & \leq 16\mathfrak{K}\|w_1-w_2\|+8L\| w_1 - w_2 \|^\beta + 4\mathfrak{K} \| w_1 - w_2 \|
   \leq 28 \mathfrak{K} L \| w_1 - w_2 \|^\beta
\end{align*}
and so
\begin{align*}
&\,\|d(f_x^+)_{v_1}-d(f_x^+)_{v_2}\| \leq \|C(f(x))^{-1} \|\|A_1 B_1 C_1-A_2 B_2 C_2\| \|C(x)\|\\
&\leq 28 \mathfrak{K}L\|C(f(x))^{-1}\|\|w_1-w_2\|^\beta
\leq 28 \mathfrak{K}L\|C(f(x))^{-1}\|\|v_1-v_2\|^\beta.
\end{align*}

Using estimate \ref{Relation-Q(x)} and that $\| v_1 - v_2 \| \leq 20 \sqrt{2} Q(x)$, we conclude that for $\ve > 0$ small:
\begin{align*}
  & \ 28 \mathfrak{K} L \| C(f(x))^{-1} \| \| v_1 - v_2 \|^{\beta /2} \leq 800 \mathfrak{K} L \| C(f(x))^{-1} \| Q(x)^{\beta / 2} \\
    & \leq 800 \mathfrak{K} L e^{144 \rho}  \| C(f(x))^{-1} \| Q(f(x))^{\beta/2} 
    \leq 800 \mathfrak{K} L e^{144 \rho}  \ve^{3}
 \leq \ve.
\end{align*}
The proof of the claim is complete.
\end{proof}
This finishes the proof of the theorem.
\end{proof}

\begin{remark}\label{rmk-holonomy}
The only property of $g_x^+$ used in the above proof is Lemma \ref{Lemma-map-g}.
Since any holonomy map $\mathfrak q_{D_j}$ also satisfies this lemma, we conclude that $\mathfrak q_{D_j}$
satisfies a statement analogous to Theorem \ref{Thm-non-linear-Pesin}. We will use this
observation in the proof of Proposition \ref{Prop-overlapping-charts}.
\end{remark}

\begin{proof}[Proof of Proposition \ref{Lemma-overlap}]
Recall that $C_i = \widetilde{C(x_i)}$ for $i=1,2$.

\medskip
\noindent
(1) We have $\| C_1^{-1} - C_2^{-1} \| = C_1^{-1} (C_2 - C_1) C_2^{-1}$, hence
$$
\| C_1^{-1} - C_2^{-1} \| \leq \| C_1^{-1} \|\cdot\| C_2^{-1} \|\cdot\|C_1 - C_2 \| \leq \ve^{1/4} (\eta_1 \eta_2)^{4 - \beta / 48} \ll \frac{1}{2} (\eta_1 \eta_2)^3,
$$
which gives the first estimate. Additionally, 
$$
\left| \tfrac{\|C_1^{-1}\|}{\| C_2^{-1}\|} - 1 \right| \leq \| C_1^{-1} - C_2^{-1} \| \ll \frac{1}{2}(\eta_1 \eta_2)^3
$$
and so $\frac{\|C_1^{-1}\|}{\| C_2^{-1}\|} = e^{\pm (\eta_1 \eta_2)^3}$.

\medskip
\noindent
(2) By the last estimate, 
$$
\tfrac{Q(x_1)}{Q(x_2)} = \left( \tfrac{\|C_1^{-1}\|}{\| C_2^{-1}\|} \right)^{-48/ \beta} = e^{ \pm \frac{48}{\beta}(\eta_1 \eta_2)^3} = e^{ \pm (\eta_1 \eta_2 )^2}.
$$

\medskip
\noindent
(3) We prove that $\Psi_{x_1} (R[e^{-2 \ve} \eta_1]) \subset \Psi_{x_2} (R[\eta_2])$. If $v \in R[e^{-2 \ve} \eta_1]$ then $\|C(x_1)v\| \leq \sqrt{2}e^{-2\ve}\eta_1<2\mathfrak{r}$ (since $\ve$ is small enough), hence by (Exp1):
$$\Sas (C(x_1) v, C(x_2) v) \leq 2 ( d(x_1, x_2) + \| C_1 v - C_2 v \|) \leq 2 (\eta_1 \eta_2)^4.$$
By (Exp2), $d(\Psi_{x_1} (v), \Psi_{x_2} (v)) < 4 (\eta_1 \eta_2)^4$, thus $\Psi_{x_1}(v) \in B(\Psi_{x_2}(v), 4(\eta_1 \eta_2)^4)$. By Lemma \ref{lemma3.7}, $B(\Psi_{x_2} (v), 4(\eta_1 \eta_2)^4) \subset \Psi_{x_2} (B)$ where $B \subset \R^d$ is the ball with center $v$ and radius $8 \| C_2^{-1} \| (\eta_1 \eta_2)^4$, hence it is enough to show that $B \subset R[\eta_2]$. If $w \in B$ then $\| w \|_{\infty} \leq \| v \|_{\infty} + 8 \| C_2^{-1} \| (\eta_1 \eta_2)^4 \leq (e^{- \ve} + 8 \ve^{1/8}) \eta_2< \eta_2$ for $\ve >0$ small enough.

\medskip
\noindent
(4) The proof that $\Psi_{x_2}^{-1} \circ \Psi_{x_1}$ is well-defined in $R[\mathfrak r]$ is similar to the proof of (3), the only difference being the last calculation: if $\ve > 0$ is small, then for $w \in B$ it holds
$$
\|w\|\leq \| v \| + 8 \|C_2^{-1} \| (\eta_1 \eta_2)^4 \leq \sqrt{2} \mathfrak{r} + 8 (\eta_1 \eta_2)^3 \leq (\sqrt{2} + 8 \ve^{1/8}) \mathfrak{r} < 2 \mathfrak{r}.
$$
Therefore $B$ is contained in the ball of radius $2\mathfrak{r}$ and center $0$ in $\R^d$, and restricted to this ball $\Psi_{x_2}$ is a diffeomorphism onto its image. It remains to estimate the $C^2$ norm of $\Psi_{x_2}^{-1} \circ \Psi_{x_1}-{\rm Id}$. We have:
\begin{align*}
   & \,\Psi_{x_2}^{-1} \circ \Psi_{x_1} - {\rm Id}  = C(x_2)^{-1} \circ \exp{x_2}^{-1} \circ \exp{x_1} \circ C(x_1) - {\rm Id} \\
    & = [C_2^{-1} \circ P_{x_2,x_1}] \circ [\exp{x_2}^{-1} \circ \exp{x_1} - P_{x_1,x_2}] \circ [P_{x_1,x_1} \circ C_1] + C_2^{-1} (C_1 - C_2) \\
    & = [C_2^{-1} \circ P_{x_2,x_1}] \circ [\exp{x_2}^{-1} - P_{x_1,x_2} \circ \exp{x_1}^{-1}] \circ \Psi_{x_1} + C_2^{-1} (C_1 - C_2).
\end{align*}
We will calculate the $C^2$ norm of $[\exp{x_2}^{-1} - P_{x_1,x_2} \circ \exp{x_1}^{-1}] \circ \Psi_{x_1}$ in the domain $R[ \mathfrak{r}]$. By Lemma \ref{lemma3.7},
$\|d \Psi_{x_1} \|_{C^0} \leq 2$ and $\Lip ( d \Psi_{x_1}) \leq \mathfrak{K}$.
Call $\Theta := \exp{x_2}^{-1} - P_{x_1, x_2} \circ \exp{x_1}^{-1}$. For $\ve >0$ small, in $B_{x_1}$ we have:
\begin{enumerate}[$\circ$]
    \item By (Exp2), $\| \Theta (y) \| \leq 2 \Sas(\exp{x_2}^{-1}(y),\exp{x_1}^{-1}(y)) \leq 4 d(x_1, x_2) <\ve^{6/ \beta}(\eta_1 \eta_2)^3$ and therefore $\| \Theta \circ \Psi_{x_1} \|_{C^0} <\ve^{6/ \beta} (\eta_1 \eta_2)^3$.
    \item By (Exp3), $\| d \Theta_y \| = \| \tau (x_2, y) - \tau (x_1, y) \| \leq \mathfrak{K} d(x_1, x_2) < \mathfrak{K}\ve^{12/ \beta} (\eta_1 \eta_2)^3$. Therefore $\| d \Theta \|_{C^0} < \mathfrak{K}\ve^{12 / \beta} (\eta_1 \eta_2)^3$ and $\| d( \Theta \circ \Psi_{x_1} ) \|_{C^0} \leq 2\mathfrak{K} \ve^{12/ \beta} (\eta_1 \eta_2)^3 < \ve^{6 / \beta} (\eta_1 \eta_2)^3$.
    \item By(Exp4), 
 $$
       \| \widetilde{ d \Theta_y } - \widetilde{ d \Theta_z} \| = \| \tau (x_2,y) - \tau(x_1, y) - [\tau (x_2, z) - \tau (x_1, z)] \| \leq \mathfrak{K} d(x_1, x_2) d(y,z),
 $$
hence $\Lip(d \Theta) \leq \mathfrak{K} d(x_1, x_2)<\mathfrak{K}(\eta_1\eta_2)^4$.
    \item Using that $\Lip(d(\Theta_1 \circ \Theta_2)) \leq \| d \Theta_1 \|_{C^0}\Lip (d \Theta_2)+\Lip (d \Theta_1) \| d \Theta_2 \|_{C^0}^2$, we have
    \begin{align*}
        &\,\Lip (d( \Theta \circ \Psi_{x_1}))  \leq \| d \Theta \|_{C^0} \Lip( d \Psi_{x_1}) + \Lip(d \Theta) \| d \Psi_{x_1} \|_{C^0}^2 \\
        & < \mathfrak{K}\ve^{6/\beta} (\eta_1 \eta_2)^3  + 4 \mathfrak{K}(\eta_1\eta_2)^4 < \ve^{3 / \beta} (\eta_1 \eta_2)^3.
    \end{align*}
\end{enumerate}
This implies that $\| \Theta \circ \Psi_{x_1} \|_{C^2} < 3\ve^{3 / \beta} (\eta_1 \eta_2)^3$, therefore
$$
\| C_2^{-1} \circ P_{x_2, x_1} \circ \Theta \circ \Psi_{x_1} \|_{C^2} \leq \| C_2^{-1} \| 3\ve^{3 / \beta} (\eta_1 \eta_2)^3 \leq \ve^{3/\beta} (\eta_1 \eta_2)^2.
$$
Thus
$\| \Psi_{x_2}^{-1} \circ \Psi_{x_1} - {\rm Id} \|_{C^2}  \leq \ve^{3 / \beta} (\eta_1 \eta_2)^2 + \| C_2^{-1} \| (\eta_1 \eta_2)^4 < 2\ve^{3 / \beta}(\eta_1 \eta_2)^2 \ll \ve (\eta_1 \eta_2)^2$. This finishes the proof of the proposition.
\end{proof}

\begin{proof}[Proof of Proposition \ref{Prop-center-stable}]
Write $v_n=\Psi_{x_n}^{p^s_n,p^u_n}$ with $\Psi_{x_0}^{p^s_0,p^u_0}=\Psi_x^{p^s,p^u}$. 
The idea is the same used in the proof of \cite[Proposition 4.8]{BCL23}: $\Delta$ is the cumulative shear of a point of $V^s$ under iterations
of the maps $g_{x_n}^+$. Recall that $g_{x_n}^+=\vf^{T_n}$ where $T_n:B_{x_n}\to\R$
is a $C^{1+\beta}$ function with $T_n(x_n)=r_\Lambda(x_n)$, $G_0={\rm Id}$,
$G_n:=g_{x_{n-1}}^+\circ\cdots\circ g_{x_0}^+$ for $n\geq 1$, and $\tau_n:B^{d_s (x)}[p^s]\to\R$ for $n\geq 0$ is the map defined by
$$
\tau_n(w):=\sum_{k=0}^{n-1}T_k(G_k[\Psi_x(w,F(w))]),
$$
equal to the flow displacement of $\Psi_x(w,F(w))$
under the maps $g_{x_0}^+,g_{x_1}^+$,\ldots, $g_{x_{n-1}}^+$. Define $\Delta_n:B^{d_s (x)}[p^s]\to\R$
by $\Delta_n(w):=\tau_n(w)-\tau_n(0)$ for $n\geq 0$,
and $\Delta:B^{d_s (x)}[p^s]\to\R$ by $\Delta(w):=\lim\limits_{n\to+\infty}\Delta_n(w)$. We have the following estimates:
\begin{enumerate}[$\circ$]
\item $\Lip{}(T_n)<1$, by Lemma \ref{lemma-local-coord}(3).
\item $\|\Delta-\Delta_n\|_{C^0}<\ve e^{-\frac{\chi\inf(r_\Lambda)}{2}n}$ for all $n\geq 0$, since
\begin{align*}
&\, \|\Delta-\Delta_n\|_{C^0}\leq \sum_{k=n}^\infty \|T_k(G_k[\Psi_x(\cdot,F(\cdot))])-T_k(G_k[\Psi_x(0,F(0))])\|_{C^0}\\
&\overset{!}{\leq}\sum_{k=n}^\infty \Lip{}(T_k) 6 p^s e^{-\frac{\chi\inf(r_\Lambda)}{2}k}
\leq \frac{6 p^s}{1-e^{-\frac{\chi\inf(r_\Lambda)}{2}}}e^{-\frac{\chi\inf(r_\Lambda)}{2}n}\overset{!!}{<}
\ve e^{-\frac{\chi\inf(r_\Lambda)}{2}n},
\end{align*}
where in $\overset{!}{\leq}$ we used Theorem~\ref{Thm-stable-manifolds}(3) and
in $\overset{!!}{<}$ we used that
$\frac{6 p^s}{1-e^{-\frac{\chi\inf(r_\Lambda)}{2}}}<\frac{6 \ve^{6/\beta}}{1-e^{-\frac{\chi\inf(r_\Lambda)}{2}}}<\ve$
for small enough $\ve>0$.
\end{enumerate}

\medskip
Let $\wt V^s:=\{\vf^{\Delta(w)}[\Psi_x(w,F(w))]: w \in B^{d_s (x)}[p^s] \}$. Fix
$\wt y,\wt z\in\wt V^s$, say 
$$
\wt y=\vf^{\Delta(w_0)}[\Psi_x(w_0,F(w_0))]=\vf^{\Delta(w_0)}(y)\ \text { and }\
\wt z=\vf^{\Delta(w_1)}[\Psi_x(w_1,F(w_1))]=\vf^{\Delta(w_1)}(z)
$$
with $w_0,w_1\in B^{d_s (x)}[p^s]$ and $y,z\in V^s$.
Fix $t\geq 0$, and take the unique $n\geq 0$ such that $\tau_{n-1}(0)<t\leq\tau_n(0)$. For this $n$,
write $\Delta=\Delta_n+E$, where $\|E\|_{C^0}<\ve e^{-\frac{\chi\inf(r_\Lambda)}{2}n}$.
We have
$$
\vf^t(\wt y)=\vf^{t+\Delta(w_0)}(y)=\vf^{t+\Delta_n(w_0)+E(w_0)}(y)=
\vf^{t-\tau_n(0)+E(w_0)}[G_n(y)],
$$
and similarly $\vf^t(\wt z)=\vf^{t-\tau_n(0)+E(w_1)}[G_n(z)]$, therefore $d(\vf^t(\wt y),\vf^t(\wt z))$ is bounded by
\begin{align*}
&\ d(\vf^{t-\tau_n(0)+E(w_0)}[G_n(y)],\vf^{t-\tau_n(0)+E(w_0)}[G_n(z)])+\\
&\ d(\vf^{t-\tau_n(0)+E(w_0)}[G_n(z)],\vf^{t-\tau_n(0)+E(w_1)}[G_n(z)])\\
&\leq \sup_{|\zeta|\leq 1}\Lip{}(\vf^\zeta) d(G_n(y),G_n(z))+\|X\|_{C^0}|E(w_0)-E(w_1)|\\
&\leq \left[6 p^s\sup_{|\zeta|\leq 1}\Lip{}(\vf^\zeta)+2\ve \|X\|_{C^0}\right]e^{-\frac{\chi\inf(r_\Lambda)}{2}n}
\leq e^{-\frac{\chi\inf(r_\Lambda)}{2}n}
\end{align*}
for $\ve>0$ small. Since $t\leq \tau_n(0)\leq 2n\sup(r_\Lambda)$,
we get that $d(\vf^t(\wt y),\vf^t(\wt z))\leq e^{-\frac{\chi\inf(r_\Lambda)}{4\sup(r_\Lambda)}t}$.
\end{proof}

\begin{proof}[Proof of Proposition \ref{Prop-disjointness}]
The proof is the same as \cite[Proposition 4.9]{BCL23}, and we include it here for completeness. Let $\un v^+=\{v_n\}_{n\geq 0}$ and $\un w^+=\{w_n\}_{n\geq 0}$ be positive $\ve$--gpo's,
with $v_0=\Psi_x^{p^s,p^u}$ and $w_0=\Psi_x^{q^s,q^u}$.
Write $V^s=V^s[\un v^+]$ and $U^s=V^s[\un w^+]$.
If $V^s\cap U^s=\emptyset$, we are done, 
so assume there is $z\in V^s\cap U^s$. Assuming without loss of generality
that $q^s\leq p^s$, we will prove that $U^s\subset V^s$. The proof will follow from three claims as
in \cite[Prop. 6.4]{Sarig-JAMS}. Write $\un v^+=\{\Psi_{x_n}^{p^s_n,p^u_n}\}_{n\geq 0}$.
We continue using the same terminology used in the proof of the previous proposition, with $g_{x_ n}^+=\vf^{T_n}$ for $n\geq 0$, $G_0={\rm Id}$,
and $G_n=g_{x_{n-1}}^+\circ\cdots\circ g_{x_0}^+$ for $n\geq 1$.

\medskip
\noindent
{\sc Claim 1:} If $n$ is large enough then $G_ n(V^s)\subset \Psi_{x_n}(R[\tfrac{1}{2}Q(x_n)])$.

\begin{proof}[Proof of Claim $1$.] Same as \cite[Prop. 6.4]{Sarig-JAMS}, using that the representation 
of $g_{x_n}^+$ in Pesin charts satisfies Theorem \ref{Thm-non-linear-Pesin-2}.
\end{proof}

\medskip
\noindent
{\sc Claim 2:} If $n$ is large enough then $G_ n(U^s)\subset \Psi_{x_n}(R[Q(x_n)])$.

\begin{proof}[Proof of Claim $2$.] Lift $U^s$ to a curve $\wt U^s$ passing through $z$
and satisfying Proposition \ref{Prop-center-stable}. Fix $n\geq 0$, and let
$t_n=\sum_{k=0}^{n-1}T_k(G_k(z))$ be the total flow time of $z$ under $G_n$.
Let $z_n=G_n(z)=\vf^{t_n}(z)$. If $D\subset\widehat\Lambda$ is the disc containing $x_n$ then 
$$
G_n(U^s)=\mathfrak q_D[\vf^{t_n}(\wt U^s)].
$$
Let $c:=\inf(r_\Lambda)^2/4\sup(r_\Lambda)$. Since $\mathfrak q_D$ is $2$--Lipschitz (Lemma \ref{lemma-local-coord}(2)),
Lemma \ref{Lemma-map-g} and Proposition \ref{Prop-center-stable}
imply that
$$
{\rm diam}(G_n(U^s))={\rm diam}(\mathfrak q_D[\vf^{t_n}(\wt U^s)])\leq 2\,{\rm diam}(\vf^{t_n}(\wt U^s))
\leq 2e^{-\frac{\chi\inf(r_\Lambda)}{4\sup(r_\Lambda)}t_n}\leq 2e^{-\chi c n},
$$
since $t_n\geq \inf(r_\Lambda)n$. Hence
$\Psi_{x_n}^{-1}[G_n(U^s)]$ is contained in the ball with center $\Psi_{x_n}^{-1}(z_n)$
and radius $4\|C(x_n)^{-1}\|e^{-\chi c n}$. Since by Claim 1 we have
$\Psi_{x_n}^{-1}(z_n)\in R[\tfrac{1}{2}Q(x_n)]$, it is enough to prove that 
$4\|C(x_n)^{-1}\|e^{-\chi c n}<\tfrac{1}{2}Q(x_n)$. Using that
$Q(x_n)<\|C(x_n)^{-1}\|^{-1}$, it is enough to prove
that $8Q(x_n)^{-2}e^{-\chi cn}<1$. We claim that $Q(x_n)^{-2}e^{-\chi cn}$
converges to zero exponentially fast as $n$ increases. Indeed, by Lemma \ref{Lemma-minimum}
we have
$Q(x_n)\geq p^s_n\wedge p^u_n\geq e^{-2\ve n}(p^s_0\wedge p^u_0)$
and so 
$$
Q(x_n)^{-2}e^{-\chi cn}\leq e^{4\ve n}(p^s_0\wedge p^u_0)^{-2}e^{-\chi cn}=
(p^s_0\wedge p^u_0)^{-2}e^{-(\chi c-4\ve)n}
$$
which converges to zero if $\ve>0$ is small enough.
\end{proof}

By Theorem \ref{Thm-stable-manifolds}(1), we conclude that
$G_n(U^s)\subset V^s[\{\Psi_{x_k}^{p^s_k,p^u_k}\}_{k\geq n}]$ for all large $n$. 

\medskip
\noindent
{\sc Claim 3:} $U^s\subset V^s$.

\begin{proof}[Proof of Claim $3$.] Fix $n$ large enough so that
$G_n(U^s)\subset V^s[\{\Psi_{x_k}^{p^s_k,p^u_k}\}_{k\geq n}]$, 
and proceed as in Claim 3 of \cite[Prop. 6.4]{Sarig-JAMS}.
\end{proof}

This concludes the proof of the proposition.
\end{proof}
\begin{proof}[Proof of Proposition \ref{Prop-overlapping-charts}]
Let $d_{s/u}=d_{s/u}(x)$, $z\in Z$, $z'=\vf^t(z)\in Z'$ with $|t|\leq2\rho$, and assume that $Z'\subset D'$.
Define $\Upsilon:=\Psi_y^{-1}\circ \mathfrak q_{D'}\circ\Psi_x$. We will write $\Upsilon$
as a small perturbation of an isometry $O$ that preservers the splitting $\R^{d_s}\oplus\R^{d_u}$. For ease of notation, write $p:=p^s\wedge p^u$
and $q:=q^s\wedge q^u$.
By Lemma \ref{Lemma-q}, Proposition \ref{Prop-Z-par}(1), and Theorem \ref{Thm-inverse}(4):
$$
\tfrac{p}{q}=\tfrac{p}{p^s(z)\wedge p^u(z)}\cdot\tfrac{p^s(z)\wedge p^u(z)}{q(z)}\cdot\tfrac{q(z)}{q(z')}\cdot
\tfrac{q(z')}{p^s(z')\wedge p^u(z')}\cdot \tfrac{p^s(z')\wedge p^u(z')}{q}=e^{\pm[O(\sqrt[3]{\ve})+O(\rho)]}.$$
We write
$\Upsilon=
(\Psi_y^{-1}\circ \Psi_{z'})\circ(\Psi_{z'}^{-1}\circ\mathfrak q_{D'}\circ\Psi_z)\circ(\Psi_z^{-1}\circ \Psi_x)$.
By Theorem \ref{Thm-inverse}(5), we have:
\begin{enumerate}[$\circ$]
\item $\Psi_y^{-1}\circ\Psi_{z'}=O_1+\Delta_1(v)$ where 
$\|\Delta_1(0)\|<50^{-1}q$, and $\|d\Delta_1\|_{C^0}<5\sqrt{\ve}$ on $R[10Q(z')]$.
\item $\Psi_z^{-1}\circ \Psi_x= O_2+\Delta_2(v)$ where 
$\|\Delta_2(0)\|<50^{-1}p$, and $\|d\Delta_2\|_{C^0}<5\sqrt{\ve}$ on $R[10Q(x)]$.
\end{enumerate}
Firstly, we prove that $\Psi_{z'}^{-1} \circ \mathfrak{q}_{D'} \circ \Psi_z$ is a perturbation of an orthogonal linear map.

\medskip
\noindent
{\sc Claim 1:}
$\Psi_{z'}^{-1}\circ\mathfrak q_{D'}\circ\Psi_z=O_3+ \Delta_3$, where $O$ is an orthogonal linear map preserving the splitting $\R^{d_s}\oplus \R^{d_u}$, $\|\Delta_3(0)\|=0$ and $\|d\Delta_3\|_{C^0}=O(\rho)+O(\ve)
$ on $R[10Q(z)]$.

\begin{proof}

Applying the same method of proof of Theorem \ref{Thm-non-linear-Pesin} to
$\mathfrak q_{D'}$ (see Remark \ref{rmk-holonomy}), we get that 
$\Psi_{z'}^{-1}\circ\mathfrak q_{D'}\circ\Psi_z$ can be written in the form
$\begin{bmatrix}
    D_s & \\ & D_u
\end{bmatrix}+H$,
    where $D_s, D_u,H$ satisfy Theorem \ref{Thm-non-linear-Pesin}(2) with $\rho$ changed to $2\rho$.

In the following, we proceed similarly to the proof of Theorem \ref{Thm-inverse}(5), which in turn is inspired by \cite[Thm 4.13(3)]{O18}.
By Lemma \ref{Lemma-linear-reduction}(2), 
$\|C(z')^{-1}\Phi^t C(z)(v)\|\leq e^{8\rho}\|v\|$
for all $v \in \R^d$, hence by the polar decomposition for matrices we can write $C(z')^{-1}\Phi^t C(z)=O_3R$, where:
\begin{enumerate}[$\circ$]
\item $O_3$ is an orthogonal linear map that preserves the splitting $\mathbb{R}^{d_s} \oplus \mathbb{R}^{d_u}$;
\item $R$ is a positive symmetric matrix preserving the splitting $\mathbb{R}^{d_s} \oplus \mathbb{R}^{d_u}$ with $\|R - {\rm Id}\| =O(\rho)$.
\end{enumerate}
Write $\Psi^{-1}_{z'}\circ \mathfrak q_{D'}\circ \Psi_z=O_3+\Delta_3$. Note that $\Delta_3(0)=0$ and
\begin{align*}
  &\ \Delta_3 = C(z')^{-1}(\exp{z'}^{-1}\circ \mathfrak q_{D'}\circ \exp{z}-\Phi^t)C(z)+(C(z')^{-1}\Phi^t C(z)-O_3)\\
    &=C(z')^{-1}(\exp{z'}^{-1} \circ \mathfrak q_{D'}\circ \exp{z}-\Phi^t)C(z)+O_3(R-{\rm Id}).
\end{align*}
To estimate $\|d(\Delta_3)_v\|$, we analyze the derivative of each term separately:
\begin{enumerate}[$\circ$]
\item $\|d(O_3(R-{\rm Id}))_v\|=\|O_3(R-{\rm Id})\|=\|R-{\rm Id}\|=O(\rho)$.
\item Using Lemma \ref{Lemma-map-g}, we have 
\begin{align*}
&\ d(\exp{{z'}}^{-1} \circ \mathfrak q_{D'}\circ \exp{z}-\Phi^t)_v=d(\exp{z'}^{-1} \circ \mathfrak q_{D'}\circ \exp{z})_v-\Phi^t\\
&=d(\exp{z'}^{-1} \circ \mathfrak q_{D'}\circ \exp{z})_v -d(\exp{z'}^{-1} \circ \mathfrak q_{D'} \circ \exp{z})_0.
\end{align*}
Since  $\exp{z'}^{-1} \circ \mathfrak q_{D'}\circ \exp{z}$ is $C^{1+\beta}$, we have $d(\exp{z'}^{-1} \circ \mathfrak q_{D'}\circ \exp{z}-\Phi^t)_v\leq {\rm const}\cdot\|v\|^\beta$, and so
$\|d(C(z')^{-1}(\exp{z'}^{-1} \circ\mathfrak q_{D'}\circ \exp{z}-\Phi^t)C(z))_v\|\leq {\rm const}\cdot\|C(z')\|^{-1}\cdot\|v\|^\beta=O(\ve)$.
\end{enumerate}
This completes the proof of Claim 1.
\end{proof}

We now proceed to prove that $\Upsilon$ is a perturbation of an orthogonal linear map. By Claim 1,
\begin{align*}
    &\,\Upsilon = (O_1 + \Delta_1)(O_3 + \Delta_3)(O_2 + \Delta_2)\\
    &= \underbrace{O_1 O_3 O_2}_{=:O} + \underbrace{O_1 O_3 \Delta_2 + O_1 \Delta_3 (O_2 + \Delta_2) + \Delta_1 (O_3+\Delta_3)(O_2+\Delta_2)}_{=:\Delta}\\
    &=: O+\Delta,
\end{align*}
where $O=O_1O_3O_2$. We  estimate $\|d\Delta\|_{C^0}$ on $R[5Q(x)]$. Letting $v_2=(O_2+\Delta_2)(v)$ and $v_3=(O_3+\Delta_3)(O_2+\Delta_2)(v)$, we have
$$
d\Delta_v=O_1O_3d(\Delta_2)_v+O_1d(\Delta_3)_{v_2}(O_2+d(\Delta_2)_v)+d(\Delta_1)_{v_3}(O_3+d(\Delta_3)_{v_2})(O_2+d(\Delta_2)_{v}).
$$
Assuming momentarily that $v_2\in R[10Q(z)]$ and $v_3\in R[10Q(z')]$, we then have that
\begin{align*}
&\,\|d\Delta_v\|\leq \|d\Delta_2\|_{C^0}+2\|d\Delta_3\|_{C^0}+4\|d\Delta_1\|_{C^0}=
O(\rho)+O(\ve^{1/2}).
\end{align*}
Now we show that $v_2,v_3$ are in the aforementioned sets:
\begin{enumerate}[$\circ$]
\item $\|v_2\|\leq \|v\|+\|\Delta_2(v)\|\leq \|\Delta_2(0)\|+\left[1+{\rm Lip}(\Delta_2)\right]\|v\|\leq 50^{-1}p+\left[1+O(\ve^{1/2})\right]5\sqrt{2}Q(x)\leq 5\left[\sqrt{2}+250^{-1}+O(\ve^{1/2})\right]Q(x)$ which, by Theorem \ref{Thm-inverse}(3), gives us that
$$
\|v_2\|\leq 5\left[\sqrt{2}+250^{-1}+O(\ve^{1/2})\right]e^{\sqrt[3]{\ve}}Q(z)<10Q(z).
$$
\item Since $v_3=(O_3+\Delta_3)(v_2)$ and $\Delta_3(0)=0$, proceeding as above and using the estimate for $\|v_2\|$ implies
\begin{align*}
&\,\|v_3\|\leq \left[1+{\rm Lip}(\Delta_3)\right]\|v_2\|
\leq 5\left[1+O(\rho)+O(\ve)\right]\left[\sqrt{2}+250^{-1}+O(\ve^{1/2})\right]e^{\sqrt[3]{\ve}}Q(z)\\
&=5\left[\sqrt{2}+250^{-1}+O(\rho)+O(\ve^{1/2})\right]e^{O(\sqrt[3]{\ve})+O(\rho)}Q(z')<10Q(z').
\end{align*}
\end{enumerate}
We also estimate $\|\Delta(0)\|$. Using the above estimates for $v=0$ and Theorem \ref{Thm-inverse}(4),
\begin{align*}
&\,\|\Delta(0)\|=\|(O_1+\Delta_1)(v_3)\|\leq \|\Delta_1(0)\|+\left[1+{\rm Lip}(\Delta_1)\right]\|v_3\|\\
&\leq 50^{-1}q+\left[1+O(\ve^{1/2})\right]\left[ 1+O(\rho)+O(\ve)\right]50^{-1}p\\
&\leq 50^{-1}q+50^{-1}\left[1+O(\rho)+O(\ve^{1/2})\right]e^{O(\sqrt[3]{\ve})+O(\rho)}q\\
&= 50^{-1}\left[2+O(\rho)+O(\ve^{1/3})\right]q < \tfrac{3}{50}q.
\end{align*}
We have thus shown that $\|\Upsilon(0)\|<\tfrac{3}{50}q$ and 
$\|d\Upsilon\|_{C^0}\leq 1+O(\rho)+O(\ve^{1/2})$ on $R[5Q(x)]$.
Now we prove the proposition.

\noindent
(1) We have $\Upsilon(R[\frac{1}{2}p]) \subset \Upsilon\left(B_0 \left( \frac{1}{\sqrt{2}}p \right) \right) \subset B_{\Upsilon(0)} \left[ \frac{1}{\sqrt{2}} \Lip(\Upsilon) p \right] \subset B$ where $B \subset \R^d$ is the ball with center 0 and radius $\|\Upsilon(0)\|+\frac{1}{\sqrt{2}} \Lip(\Upsilon) p $. By the estimates obtained above,
\begin{align*}
       &\, \|\Upsilon(0)\|+\tfrac{1}{\sqrt{2}}\Lip(\Upsilon)p  \leq \tfrac{3}{50}q + \tfrac{1}{\sqrt{2}}[1+O(\rho)+ O(\ve^{1/2})]p\\ &\leq \left(\tfrac{3}{50} + \tfrac{1}{\sqrt{2}}[1+O(\rho)+ O(\ve^{1/3})]\right)q
\end{align*}
and the last expression above is smaller than $q$, 
since $\tfrac{3}{50} + \tfrac{1}{\sqrt{2}}[1+O(\rho)+ O(\ve^{1/3})]<1$ for $\ve \ll\rho \ll 1$. Hence $B \subset B_0[q] \subset R[q]$. \\

\noindent   
(2) Fix $z\in Z$ such that $z'=\mathfrak q_{D'}(z)\in Z'$. We will show that
$\mathfrak q_{D'}[W^{s}(z,Z)]\subset V^{s}(z',Z')$ (the other inclusion is identical).
Write $W=\mathfrak q_{D'}[W^s(z,Z)]$ and $V=V^s(z',Z')$. We wish to show that $W\subset V$.
Let $\un v=\{\Psi_{x_n}^{p^s_n,p^u_n}\}_{n\in\Z},\un w=\{\Psi_{y_n}^{q^s_n,q^u_n}\}_{n\in\Z}$
such that $z=\pi(\un v)$ and $z'=\pi(\un w)$. For $n\geq 0$, let
$G^n_{\un v}=g_{x_{n-1}}^+\circ\cdots\circ g_{x_0}^+ \text{ and } G^n_{\un w}=g_{y_{n-1}}^+\circ\cdots\circ g_{y_0}^+$.
By Theorem \ref{Thm-stable-manifolds}(1), we need to show that
$G^n_{\un w}[W]\subset\Psi_{y_n}(R[10Q(y_n)])$ for all $n\geq 0$.

Fix $n\geq 0$. If $z'=\vf^t(z)$, $|t|\leq2\rho$, then there is a unique $m\geq 0$ such that
$r_m(\un v)<r_n(\un w)+t\leq r_{m+1}(\un v)$. Let $D_k$ be the disc containing
$\vf^{r_n(\un w)}(z')$. We claim that $G^n_{\un w}\circ{\mathfrak q}_{D'}=\mathfrak q_{D_k}\circ G^m_{\un v}$
wherever these maps are well-defined. To see this, firstly note that these maps are both
of the form $\vf^{\tau}$ for some continuous function $\tau$. Secondly, we claim that they coincide
at $z$. Indeed, $(G^n_{\un w}\circ{\mathfrak q}_{D'})(z)=G^n_{\un w}(z')=\vf^{r_n(\un w)}(z')$ and
$(\mathfrak q_{D_k}\circ G^m_{\un v})(z)=\mathfrak q_{D_k}[\vf^{r_m(\un v)}(z)]$.
Writing $\vf^{r_n(\un w)}(z')=z_n'$ and $\vf^{r_m(\un v)}(z)=z_m$, we have $z_n'=\vf^{t'}(z_m)$
for $t'=r_n(\un w)+t-r_m(\un v) \in (0,\rho]$, therefore $\mathfrak q_{D_k}(z_m)=z_n'$. Hence
$G^n_{\un w}[W]=(G^n_{\un w}\circ{\mathfrak q}_{D'})[W^s(z,Z)]=(\mathfrak q_{D_k}\circ G^m_{\un v})[W^s(z,Z)]
\subset \mathfrak q_{D_k}[W^s(\vf^{r_m(\un v)}(z),Z(v_m))]$, where we used Proposition \ref{Prop-Z}(4)
in the last inclusion. Since
$W^s(\vf^{r_m(\un v)}(z),Z(v_m))\subset \Psi_{x_m}(R[10^{-2}(p^s_m\wedge p^u_m)])$,
part (1) gives that $\mathfrak q_{D_k}[W^s(\vf^{r_m(\un v)}(z),Z(v_m))]\subset \Psi_{y_n}(R[q^s_n\wedge q^u_n])$,
and this last set is contained in $\Psi_{y_n}(R[10Q(y_n)])$.\\

\noindent    
(3) When $M$ has dimension 3, this result is shown in \cite[Proposition 7.2(3)]{BCL23}, where the authors adapt \cite[Lemma 10.8]{Sarig-JAMS} to the context of flows. In both cases, the change of coordinates $\Upsilon$ is a small perturbation of the identity, allowing control over the geometry of admissible manifolds. When $M$ is a closed manifold of arbitrary finite dimension, a similar approach is made in \cite[Lemma 5.8]{O18}, where the change of coordinates $\Upsilon$ is shown to be a small perturbation of an isometry that preserves the splitting $\mathbb{R}^{d_s} \oplus \mathbb{R}^{d_u}$. Since in our setting we also obtained this property, the same method of proof applies and so $[z, z']_{Z'}$ is well-defined. Similarly,
$[z,z']_Z$ is well-defined.

It remains to prove that $[z,z']_Z=\mathfrak q_D([z,z']_{Z'})$.
To see this, observe that the composition $\mathfrak q_D\circ\mathfrak q_{D'}$ is the identity
where it is defined, hence
$$
\mathfrak q_D([z,z']_{Z'})=\mathfrak q_D(\mathfrak q_{D'}[V^s(z,Z)]\cap V^u(z',Z'))
= V^s(z,Z)\cap \mathfrak q_D[V^u(z',Z')]=[z,z']_Z.
$$
This completes the proof of the proposition.
\end{proof}

\begin{proof}[Proof of Proposition \ref{Prop-overlapping-charts-2}]
Let $Z, Z',Z''$ such that $Z\cap \vf^{[-2\rho,2\rho]}Z'\neq\emptyset$, $Z\cap \vf^{[-2\rho,2\rho]}Z''\neq\emptyset$,
and assume that $z'\in Z'$ such that $\vf^t(z')\in Z''$ for some $|t|\leq 2\rho$. 
We need to show that for every $z\in Z$ it holds
$$
[z,z']_Z=[z,\vf^t(z')]_Z.
$$
For this, we will show that:
\begin{enumerate}[$\circ$] 
\item  $\mathfrak q_{D''}[V^u(z',Z')]$ and $V^u(\vf^t(z'),Z'')$ coincide in a small window, where $D''$ is the connected component of $\widehat\Lambda$ with $Z''\subset D''$. 
\item If $Z=Z(\Psi_x^{p^s,p^u})$ and $G$ is the representing function of $V^s(z,Z)$,
then $[z,z']_Z=\Psi_x(s,G(s))$ for some $|s|\leq \tfrac{1}{3}(p^s\wedge p^u)$.
\end{enumerate}
The precise statements are in the next claims. Write $Z'=Z(\Psi_y^{q^s,q^u})$,
$p=p^s\wedge p^u$ and $q=q^s\wedge q^u$, and
let $D$ be the connected components of $\widehat\Lambda$ with $Z\subset D$.
Since $d_{s/u}(x) = d_{s/u}(z)=d_{s/u}(z')=d_{s/u}(y)$, we will denote $B^{d_{s/u}(\cdot)}[r]$ and $\R^{d_{s/u}(\cdot)}$ simply by $B^{d_{s/u}}[r]$ and $\R^{d_{s/u}}$, respectively.
 
\medskip
\noindent
{\sc Claim 1:} $\mathfrak q_D[V^u(z',Z')\cap \Psi_y(R[\tfrac{1}{2}q])]$ contains
$\Psi_x\{(H(w),w):w \in B^{d_u}[\tfrac{1}{3}p]\}$ for some function $H:B^{d_u}[\tfrac{1}{3}p]\to\R^{d_s}$
such that $\|H(0)\|<\tfrac{3}{50}p$ and $\|dH\|_{C^0}<\tfrac{1}{2}$.
Additionally, $[z,z']_Z=\Psi_x(s,G(s))$ for some $|s|\leq \tfrac{1}{3}p$.

\medskip
\noindent
{\sc Claim 2:} Recalling that $D''$ is the connected components of $\widehat\Lambda$ such that $Z''\subset D''$,
then
$$
\mathfrak q_{D''}[V^{s/u}(z',Z')\cap \Psi_y(R[\tfrac{1}{2}q])]\subset V^{s/u}(z'',Z'').
$$

Once these claims are proved, the result follows:
Claim 2 implies that $\mathfrak q_D[V^u(z',Z')\cap \Psi_y(R[\tfrac{1}{2}q])]
\subset \mathfrak q_D[V^u(z'',Z'')]$ and so by Claim 1
we conclude that
\begin{align*}
&\ \{[z,z']_Z\}=V^s(z,Z)\cap \mathfrak q_D[V^u(z',Z')\cap \Psi_y(R[\tfrac{1}{2}q])]\\
&\subset V^s(z,Z)\cap \mathfrak q_D[V^u(z'',Z'')]=\{[z,z'']_Z\}.
\end{align*}

\begin{proof}[Proof of Claim $1$]
With the estimates obtained in the beginning of the proof of Proposition \ref{Prop-overlapping-charts},
we just need to proceed as in the proof of \cite[Lemma 10.8]{Sarig-JAMS}.
We include the details for completeness. By the proof of Proposition \ref{Prop-overlapping-charts},
$\Upsilon:=\Psi_x^{-1}\circ \mathfrak q_D\circ \Psi_y=O+\Delta$ where:
\begin{enumerate}[$\circ$]
\item $O = (O^s,O^u)$ is a linear orthogonal map with $O^{s/u}: \R^{d_{s/u}} \to \R^{d_{s/u}}$.
\item $\|d\Delta\|_{C^0}\leq O(\rho)+O(\ve^{1/2})$.
\item $\|\Delta(0)\|\leq \tfrac{2}{50}\left[1+O(\rho)+O(\ve^{1/3})\right]p$.
\end{enumerate} 
In particular, $\|\Delta\|_{C^0(R[p])}\leq \tfrac{2}{50}\left[1+O(\rho)+O(\ve^{1/3})\right]p$.
 Write $\Delta=(\Delta_1,\Delta_2)$, and let $F$ be the representing function of $V^u(z',Z')$, i.e.
$V^u(z',Z')=\Psi_y\{(F(v),v): v \in B^{d_u}[ q^u]\}$. Hence
$V^u(z',Z')\cap \Psi_y(R[\tfrac{1}{2}q])=\Psi_y\{(F(v),v):v \in B^{d_u}[\frac{1}{2}q]\}$, and since
$\mathfrak q_D\circ \Psi_y=\Psi_x\circ\Upsilon$ we have
\begin{align*}
&\ \mathfrak q_D[V^u(z',Z')\cap \Psi_y(R[\tfrac{1}{2}q])]=
(\Psi_x\circ\Upsilon)\left\{(F(v),v): v \in B^{d_u}[\tfrac{1}{2}q]\right\}\\
&=\Psi_x\left\{(O^sF(v)+\Delta_1(F(v),v),O^uv+\Delta_2(F(v),v)):v \in B^{d_u}[\tfrac{1}{2}q]\right\}.
\end{align*}
We represent the pair inside $\Psi_x$ above as a graph on the second coordinate.
Call $\tau(v):=O^uv+\Delta_2(F(v),v)$. We have:
\begin{enumerate}[$\circ$]
\item $\|\tau(0)\|=\|\Delta_2(F(0),0)\|\leq \|\Delta(F(0),0)\|\leq \|\Delta(0)\|+\|d\Delta\|_{C^0}\|F(0)\|\leq
\tfrac{2}{50}[1+O(\rho)+O(\ve^{1/3})]p+[O(\rho)+O(\ve^{1/2})]10^{-3}q \leq
\tfrac{2}{50}[1+O(\rho)+O(\ve^{1/3})]p$.
\item $\|d\tau_w\|=1\pm\|d\Delta\|_{C^0}(1+\|dF\|_{C^0})=1\pm[O(\rho)+O(\ve^{1/2})](1+\ve)\leq 1+O(\rho)+O(\ve^{1/3})$
for every $w \in B^{d_u}[\tfrac{1}{2}q]$.
\end{enumerate}

Now we prove that $\tau({B^{d_u}[\frac{1}{2}q]}) \supset B^{d_u}[\frac{1}{3}p]$.
First, notice that $\tau$ is injective: if $v,v'$ satisfy $\tau(v)=\tau(v')$, then $O^uv+\Delta_2(F(v),v) = O^uv'+\Delta_2(F(v'),v')$ and so 
\begin{align*}
&\,\|O^u(v-v')\|=\|\Delta_2(F(v),v)-\Delta_2(F(v'),v')\|\leq \|d\Delta\|_{C^0}(1+\|dF\|_{C^0})\|v-v'\|\\
&\leq 2[O(\rho)+O(\ve^{1/2})]\|v-v'\|,
\end{align*}
which implies $v=v'$ since $0<\ve\ll \rho\ll 1$.

Next, we show that for every $z \in B^{d_u}[\frac{1}{3}p]$ there is $v \in B^{d_u}[\tfrac{1}{2}q]$
such that $\tau(v)=z$. This is equivalent to $v$ being a fixed point of the map $T_z(v) = (O^u)^{-1}[z-\Delta_2(F(v),v)]$.
We will verify this via the Banach fixed-point theorem. For each $z$, the map $T_z$ is contraction, since
$$\|T_z(v)-T_z(v')\|=\|\Delta_2(F(v),v)-\Delta_2(F(v'),v')\| \leq \underbrace{2[O(\rho)+O(\ve^{1/2})]}_{\ll 1}\|v-v'\|.$$
Furthermore, for every $z \in B^{d_u}[\frac{1}{3}p]$, the map $T_z$ takes $B^{d_u}[\frac{1}{2}q]$ into itself:
\begin{align*}
&\,\|T_z(v)\| \leq \|z\| + \|\Delta_2(F(v),v)\| \leq \|z\| + \|\Delta(0)\| + \|d\Delta\|_{C^0}\|\left(\|F(0)\|+[1+\|dF\|_{C^0}]\|v\|\right) \\
& \leq \tfrac{1}{3}p + \tfrac{2}{50}[1+O(\rho)+O(\ve^{1/3})]p+[O(\rho)+O(\ve^{1/2})]2q \leq \tfrac{1}{2}q.
\end{align*}
Therefore, each $T_z$ has a unique fixed point.

Now, we write the first coordinate $F(v)+\Delta_1(F(v),v)$ as a function of $\tau$.
Start noticing that, since $\tau$ is injective, it has an inverse $\theta:\tau(B^{d_u}[\tfrac{1}{2}q])\to 
B^{d_u}[\tfrac{1}{2}q]$ such that $\|d\theta\|_{C^0}=1+O(\rho)+O(\ve^{1/3})$. This follows from calculations analogous to \cite[p.103]{ALP}.
In particular,
$$
\|\theta(0)\|=\|\theta(0)-\theta(\tau(0))\|\leq \|\theta'\|_{C^0}\|\tau(0)\|\leq \tfrac{2}{50}[1+O(\rho)+O(\ve^{1/3})]p<\tfrac{1}{5}p.
$$
Defining $H:B^{d_u}[\frac{1}{3}p]\to\R^{d_s}$ by
$$
H(\tau)=O^sF(v)+\Delta_1(F(v),v)=O^sF(\theta(\tau))+\Delta_1(F(\theta(\tau)),\theta(\tau)),
$$
we have:
\begin{enumerate}[$\circ$]
\item $\|H(0)\|\leq \|F(\theta(0))\|+\|\Delta_1(F(\theta(0)),\theta(0))\|\leq \|F(0)\|+\|dF\|_{C^0}\|\theta(0)\|+\|\Delta\|_{C^0}\leq
10^{-3}q+\ve \tfrac{1}{5}p+\tfrac{2}{50}\left[1+O(\rho)+O(\ve^{1/3})\right]p<\tfrac{3}{50}p$.
\item $\|dH\|_{C^0}\leq \|dF\|_{C^0}\|d\theta\|_{C^0}+\|d\Delta\|_{C^0}(1+\|dF\|_{C^0})\|d\theta\|_{C^0}\leq 
2\ve+2[O(\rho)+O(\ve^{1/2})][1+\ve]=O(\rho)+O(\ve^{1/2})$
which is smaller than $\tfrac{1}{2}$ for $\rho,\ve>0$ small.
\end{enumerate}
This proves the first part of Claim 1. For the second part, note that
$\|H(\tau)\|\leq \|H(0)\|+\|dH\|_{C^0}\|\tau\|\leq \tfrac{3}{50}p+\tfrac{1}{2}\cdot\tfrac{1}{3}p<\tfrac{1}{3}p$, thus
$H:B^{d_u}[\frac{1}{3}p]\to B^{d_s}[\tfrac{1}{3}p]$ is a contraction.
We have $[z,z']_Z=\Psi_x(v,G(v))$,
where $v$ is the unique $v\in B^{d_s}[p^s]$ such that $(v,G(v))=(H(\tau),\tau)$. Necessarily 
$H(G(v))=v$, i.e. $v$ is a fixed point of $H\circ G$. Using the admissibility of $G$ and the above estimates,
the restriction of $H\circ G$ to $B^{d_s}[\tfrac{1}{3}p]$ is a contraction into 
$B^{d_s}[\tfrac{1}{3}p]$, and so it has a unique fixed point in this interval, proving that
$\|v\|\leq \tfrac{1}{3}p$.
\end{proof}

\begin{proof}[Proof of Claim $2$]
The proof is very similar to the proof of Proposition \ref{Prop-disjointness}.
Let us prove the inclusion for $V^s$.
Let $V^s=V^s(z'',Z'')=V^s[\un v^+]$ with $\un v^+=\{\Psi_{y_n}^{q^s_n,q^u_n}\}$,
and let $G_n=g_{y_{n-1}}^+\circ\cdots\circ g_{y_0}^+$.
Let  $U^s=\mathfrak q_{D''}[V^s(z',Z')\cap \Psi_y(R[\tfrac{1}{2}q])]$.
By Proposition \ref{Prop-overlapping-charts}(1), $U^s\subset \Psi_{y_0}(R[q^s_0\wedge q^u_0])$.
Now we proceed as in the proof of Proposition \ref{Prop-disjointness} to get that:
\begin{enumerate}[$\circ$]
\item If $n$ is large enough then $G_n(U^s)\subset \Psi_{y_n}(R[Q(y_n)])$: this is exactly 
Claim 2 in the proof of Proposition \ref{Prop-disjointness}.
\item $U^s\subset V^s$: this is exactly Claim 3 in the proof of Proposition \ref{Prop-disjointness}.
\end{enumerate}
Hence Claim 2 is proved.
\end{proof}
The proof of the proposition is complete.
\end{proof}

\bibliographystyle{alpha}
\bibliography{bib}

\begin{thebibliography}{RHRHTU11}

\bibitem[ALP24]{ALP}
Ermerson Araujo, Yuri Lima, and Mauricio Poletti.
\newblock Symbolic dynamics for nonuniformly hyperbolic maps with singularities
  in high dimension.
\newblock {\em Mem. Amer. Math. Soc.}, 301(1511):vi+117, 2024.

\bibitem[AW67]{Adler-Weiss-PNAS}
R.~L. Adler and B.~Weiss.
\newblock Entropy, a complete metric invariant for automorphisms of the torus.
\newblock {\em Proc. Nat. Acad. Sci. U.S.A.}, 57:1573--1576, 1967.

\bibitem[BB17]{Boyle-Buzzi}
M.~Boyle and J.~Buzzi.
\newblock The almost {B}orel structure of surface diffeomorphisms, {M}arkov
  shifts and their factors.
\newblock {\em J. Eur. Math. Soc.}, 19:2739--2782, 2017.

\bibitem[BCFT18]{BCFT}
K.~Burns, V.~Climenhaga, T.~Fisher, and D.~J. Thompson.
\newblock Unique equilibrium states for geodesic flows in nonpositive
  curvature.
\newblock {\em Geom. Funct. Anal.}, 28(5):1209--1259, 2018.

\bibitem[BCL25]{BCL23}
J{\'e}r{\^o}me Buzzi, Sylvain Crovisier, and Yuri Lima.
\newblock Symbolic dynamics for large non-uniformly hyperbolic sets of three
  dimensional flows.
\newblock {\em Advances in Mathematics}, 479:110410, 2025.

\bibitem[BCS22]{BCS-MME}
J\'er\^ome Buzzi, Sylvain Crovisier, and Omri Sarig.
\newblock Finiteness of measures maximizing the entropy for surface
  diffeomorphisms.
\newblock {\em Ann. of Math.}, 195:421--508, 2022.

\bibitem[BMW12]{Burns-Masur-Wilkinson}
K.~Burns, H.~Masur, and A.~Wilkinson.
\newblock The {W}eil-{P}etersson geodesic flow is ergodic.
\newblock {\em Ann. of Math.}, 175:835--908, 2012.

\bibitem[BO18]{O18}
Snir Ben~Ovadia.
\newblock Symbolic dynamics for non-uniformly hyperbolic diffeomorphisms of
  compact smooth manifolds.
\newblock {\em Journal of Modern Dynamics}, 13, 2018.

\bibitem[BO20]{Ova20}
Snir Ben~Ovadia.
\newblock The set of points with markovian symbolic dynamics for non-uniformly
  hyperbolic diffeomorphisms.
\newblock {\em Ergodic Theory and Dynamical Systems}, 41(11):3244--3269, 2020.

\bibitem[Bow70]{Bowen-MP-Axiom-A}
Rufus Bowen.
\newblock Markov partitions for {A}xiom {${\rm A}$} diffeomorphisms.
\newblock {\em Amer. J. Math.}, 92:725--747, 1970.

\bibitem[Bow73]{Bowen-Symbolic-Flows}
Rufus Bowen.
\newblock Symbolic dynamics for hyperbolic flows.
\newblock {\em American journal of mathematics}, 95(2):429--460, 1973.

\bibitem[Bow78]{Bowen-Regional-Conference}
Rufus Bowen.
\newblock {\em On {A}xiom {A} diffeomorphisms}.
\newblock American Mathematical Society, Providence, R.I., 1978.
\newblock Regional Conference Series in Mathematics, No. 35.

\bibitem[Bow08]{Bowen-LNM}
R.~Bowen.
\newblock {\em Equilibrium states and the ergodic theory of {A}nosov
  diffeomorphisms}, volume 470 of {\em Lecture Notes in Mathematics}.
\newblock Springer-Verlag, Berlin, 2008.

\bibitem[Buz97]{Buzzi-IJM}
J.~Buzzi.
\newblock Intrinsic ergodicity of smooth interval maps.
\newblock {\em Israel J. Math.}, 100:125--161, 1997.

\bibitem[Buz05]{Buzzi-Invent}
J.~Buzzi.
\newblock Subshifts of quasi-finite type.
\newblock {\em Invent. Math.}, 159:369--406, 2005.

\bibitem[Buz20]{Buzzi-JMD}
J\'{e}r\^{o}me Buzzi.
\newblock The degree of {B}owen factors and injective codings of
  diffeomorphisms.
\newblock {\em J. Mod. Dyn.}, 16:1--36, 2020.

\bibitem[BW72]{Bowen-Walters-Metric}
R.~Bowen and P.~Walters.
\newblock Expansive one-parameter flows.
\newblock {\em J. Diff. Equations}, 12:180--193, 1972.

\bibitem[CKP20]{CKP-20}
Dong Chen, Lien-Yung Kao, and Kiho Park.
\newblock Unique equilibrium states for geodesic flows over surfaces without
  focal points.
\newblock {\em Nonlinearity}, 33(3):1118--1155, 2020.

\bibitem[CKW21]{CKW}
Vaughn Climenhaga, Gerhard Knieper, and Khadim War.
\newblock Uniqueness of the measure of maximal entropy for geodesic flows on
  certain manifolds without conjugate points.
\newblock {\em Adv. Math.}, 376:Paper No. 107452, 44, 2021.

\bibitem[Din25]{Dinowitz}
Emma Dinowitz.
\newblock Personal communication.
\newblock 2025.

\bibitem[dJEP25]{JEP2025}
Ygor de~Jesus, Marcielis Espitia, and Gabriel Ponce.
\newblock Homoclinic classes for flows: ergodicity and srb measures.
\newblock 2025.

\bibitem[GR19]{Gelfert-Ruggiero}
Katrin Gelfert and Rafael~O. Ruggiero.
\newblock Geodesic flows modelled by expansive flows.
\newblock {\em Proc. Edinb. Math. Soc. (2)}, 62(1):61--95, 2019.

\bibitem[GR23]{Gelfert-Ruggiero-23}
Katrin Gelfert and Rafael~O. Ruggiero.
\newblock Geodesic flows modeled by expansive flows: compact surfaces without
  conjugate points and continuous {G}reen bundles.
\newblock {\em Ann. Inst. Fourier (Grenoble)}, 73(6):2605--2649, 2023.

\bibitem[Hof79]{Hofbauer-PMM}
Franz Hofbauer.
\newblock On intrinsic ergodicity of piecewise monotonic transformations with
  positive entropy.
\newblock {\em Israel J. Math.}, 34:213--237, 1979.

\bibitem[Kat80]{KatokIHES}
A.~Katok.
\newblock Lyapunov exponents, entropy and periodic orbits for diffeomorphisms.
\newblock {\em Inst. Hautes \'Etudes Sci. Publ. Math.}, 51:137--173, 1980.

\bibitem[Kni98]{Knieper-Rank-One-Entropy}
Gerhard Knieper.
\newblock The uniqueness of the measure of maximal entropy for geodesic flows
  on rank {$1$} manifolds.
\newblock {\em Ann. of Math. (2)}, 148:291--314, 1998.

\bibitem[KSSV06]{Kunzinger-flow}
M.~Kunzinger, H.~Schichl, R.~Steinbauer, and J.~A. Vickers.
\newblock Global gronwall estimates for integral curves on riemannian
  manifolds.
\newblock {\em Rev. Mat. Complut.}, 19:133--137, 2006.

\bibitem[Lim20]{Lima-AIHP}
Yuri Lima.
\newblock Symbolic dynamics for one dimensional maps with nonuniform expansion.
\newblock {\em Annales de l'Institut Henri Poincar{\'e} C, Analyse non
  lin{\'e}aire}, 37(3):727--755, 2020.

\bibitem[LLS16]{LLS16}
Fran\c~cois Ledrappier, Yuri Lima, and Omri Sarig.
\newblock Ergodic properties of equilibrium measures for smooth three
  dimensional flows.
\newblock {\em Comment. Math. Helv.}, 91(1):65--106, 2016.

\bibitem[LM18]{Lima-Matheus}
Yuri Lima and Carlos Matheus.
\newblock Symbolic dynamics for non-uniformly hyperbolic surface maps with
  discontinuities.
\newblock In {\em Annales scientifiques de l'{\'E}cole Normale Sup{\'e}rieure},
  volume~51, pages 1--38, 2018.

\bibitem[LOP]{LOP-24}
Yuri Lima, Davi Obata, and Mauricio Poletti.
\newblock Measures of maximal entropy for non-uniformly hyperbolic maps.
\newblock arXiv:2405.04676.

\bibitem[LP25]{Lima-Poletti}
Yuri Lima and Mauricio Poletti.
\newblock Homoclinic classes of geodesic flows on rank 1 manifolds.
\newblock {\em Proc. Amer. Math. Soc.}, 153(4):1611--1620, 2025.

\bibitem[LS19]{Lima-Sarig}
Yuri Lima and Omri~M Sarig.
\newblock Symbolic dynamics for three-dimensional flows with positive
  topological entropy.
\newblock {\em Journal of the European Mathematical Society (EMS Publishing)},
  21(1), 2019.

\bibitem[Mam24]{Mamani-24}
Edhin~Franklin Mamani.
\newblock Geodesic flows of compact higher genus surfaces without conjugate
  points have expansive factors.
\newblock {\em Nonlinearity}, 37(5):Paper No. 055019, 22, 2024.

\bibitem[MR]{Mamani-Ruggiero}
Edhin~F. Mamani and Rafael Ruggiero.
\newblock Expansive factors for geodesic flows of compact manifolds without
  conjugate points and with visibility universal covering.
\newblock arXiv:2311.02698.

\bibitem[New80]{Newhouse-Lectures-dynamical-systems}
Sheldon~E. Newhouse.
\newblock Lectures on dynamical systems.
\newblock In {\em Dynamical systems ({C}.{I}.{M}.{E}. {S}ummer {S}chool,
  {B}ressanone, 1978)}, volume~8 of {\em Progr. Math.}, pages 1--114.
  Birkh\"{a}user, Boston, Mass., 1980.

\bibitem[Pes76]{Pesin-Izvestia-1976}
Ja.~B. Pesin.
\newblock Families of invariant manifolds that correspond to nonzero
  characteristic exponents.
\newblock {\em Izv. Akad. Nauk SSSR Ser. Mat.}, 40(6):1332--1379, 1440, 1976.

\bibitem[PYY]{pacifico-yang-yang}
Maria~Jose Pacifico, Fan Yang, and Jiagang Yang.
\newblock Equilibrium states for the classical lorenz attractor and
  sectional-hyperbolic attractors in higher dimensions.
\newblock arXiv:2209.10784, to appear on Duke Math. J.

\bibitem[Rat69]{Ratner-MP-three-dimensions}
M.~E. Ratner.
\newblock Markov decomposition for an {U}-flow on a three-dimensional manifold.
\newblock {\em Mat. Zametki}, 6:693--704, 1969.

\bibitem[Rat73]{Ratner-MP-n-dimensions}
M.~Ratner.
\newblock Markov partitions for {A}nosov flows on {$n$}-dimensional manifolds.
\newblock {\em Israel J. Math.}, 15:92--114, 1973.

\bibitem[RHRHTU11]{HHTU-CMP}
F.~Rodriguez~Hertz, M.~A. Rodriguez~Hertz, A.~Tahzibi, and R.~Ures.
\newblock Uniqueness of {SRB} measures for transitive diffeomorphisms on
  surfaces.
\newblock {\em Comm. Math. Phys.}, 306(1):35--49, 2011.

\bibitem[Sar13]{Sarig-JAMS}
Omri Sarig.
\newblock Symbolic dynamics for surface diffeomorphisms with positive entropy.
\newblock {\em Journal of the American Mathematical Society}, 26(2):341--426,
  2013.

\bibitem[Sin68a]{Sinai-Construction-of-MP}
Ya~G Sinai.
\newblock Construction of markov partitions.
\newblock {\em Functional Analysis and Its Applications}, 2(3):245--253, 1968.

\bibitem[Sin68b]{Sinai-MP-U-diffeomorphisms}
Ya~G Sinai.
\newblock Markov partitions and c-diffeomorphisms.
\newblock {\em Functional Analysis and its applications}, 2(1):61--82, 1968.

\bibitem[Wu]{Wu}
Weisheng Wu.
\newblock On ergodic properties of geodesic flows on uniform visibility
  manifolds without conjugate points.
\newblock arXiv:2405.11635.

\bibitem[Zan25]{Yuntao2025}
Yuntao Zang.
\newblock Measures of maximal entropy for ${C}^\infty$ three-dimensional flows.
\newblock 2025.

\end{thebibliography}

\end{document}